\documentclass[12pt, a4paper, BCOR=8mm]{scrbook}

\usepackage{amsmath, amsthm, amsfonts, amssymb}
\usepackage[applemac]{inputenc}
\usepackage[T1]{fontenc}
\usepackage[nottoc]{tocbibind} 
\usepackage{color}
\usepackage{slashed} 
\usepackage[all, cmtip]{xy}
\usepackage{paralist} 
\usepackage{hyperref} 
\usepackage{chngcntr} 
\usepackage{graphicx} 
\usepackage{subfigure}
\usepackage[ngerman, english]{babel}
\usepackage{microtype}
\usepackage{bm} 
\usepackage{xfrac} 

\newcommand{\IU}{\mathcal{U}^\ast_{-\infty}}
\newcommand{\C}{C_{-\infty}^\ast}
\newcommand{\IN}{\mathbb{N}}
\newcommand{\IZ}{\mathbb{Z}}
\newcommand{\IQ}{\mathbb{Q}}
\newcommand{\IR}{\mathbb{R}}
\newcommand{\IC}{\mathbb{C}}
\newcommand{\IH}{\mathcal{H}}
\newcommand{\injrad}{\operatorname{inj-rad}}
\newcommand{\Rm}{\operatorname{Rm}}

\newcommand{\LLip}{L\text{-}\operatorname{Lip}}
\newcommand{\IB}{\mathfrak{B}}
\newcommand{\IK}{\mathfrak{K}}
\newcommand{\frakD}{\mathfrak{D}}
\newcommand{\frakC}{\mathfrak{C}}
\newcommand{\fraklC}{\mathfrak{lC}}
\newcommand{\frakCr}{\mathfrak{Cr}}
\newcommand{\supp}{\operatorname{supp}}
\newcommand{\diam}{\operatorname{diam}}
\newcommand{\Hom}{\operatorname{Hom}}
\newcommand{\Mat}{\operatorname{Mat}}
\newcommand{\End}{\operatorname{End}}
\newcommand{\id}{\operatorname{id}}
\newcommand{\kernel}{\operatorname{ker}}
\newcommand{\cokernel}{\operatorname{coker}}
\newcommand{\image}{\operatorname{im}}

\newcommand{\card}{\#}
\newcommand{\Vect}{\operatorname{Vect}}
\newcommand{\Idem}{\operatorname{Idem}}
\newcommand{\closure}{\operatorname{cl}}
\newcommand{\trace}{\operatorname{tr}}
\newcommand{\ch}{\operatorname{ch}}
\newcommand{\vol}{\operatorname{vol}}
\newcommand{\ind}{\operatorname{ind}}
\newcommand{\tind}{\operatorname{t-ind}}
\newcommand{\aind}{\operatorname{a-ind}}
\newcommand{\op}{\text{op}}
\newcommand{\pt}{\text{pt}}
\newcommand{\Symb}{\operatorname{Symb}}

\newcommand{\ICl}{\operatorname{\mathbb{C}\mathit{l}}}

\newcommand{\Frechet}{Fr\'{e}chet }
\newcommand{\Poincare}{Poincar\'{e} }
\newcommand{\Folner}{F{\o}lner }
\newcommand{\Spakula}{\v{S}pakula }
\newcommand{\spinc}{spin$^c$ }

\newcommand*{\largecdot}{\raisebox{-0.25ex}{\scalebox{1.2}{$\cdot$}}}

\DeclareMathOperator{\hatotimes}{\hat{\otimes}}
\DeclareMathOperator{\barotimes}{\bar{\otimes}}

\newtheorem{thm}{Theorem}[chapter]
\newtheorem*{thm*}{Theorem}
\newtheorem{chernthm}[thm]{Chern Character Isomorphism Theorem}
\newtheorem{roesthm}[thm]{Roe's Index Theorem}
\newtheorem*{roesthm*}{Roe's Index Theorem}

\newtheorem*{atiyahsingerthm*}{Atiyah--Singer Index Theorem}
\newtheorem*{mainthm*}{Main Theorem of this Thesis}
\newtheorem{indexthm}[thm]{Index Theorem}
\newtheorem{cor}[thm]{Corollary}
\newtheorem{lem}[thm]{Lemma}
\newtheorem{prop}[thm]{Proposition}
\newtheorem{conj}[thm]{Conjecture}
\theoremstyle{remark}
\newtheorem{rem}[thm]{Remark}
\newtheorem{example}[thm]{Example}
\newtheorem{examples}[thm]{Examples}
\theoremstyle{definition}
\newtheorem{defn}[thm]{Definition}

\numberwithin{equation}{chapter}

\counterwithout{footnote}{chapter}

\begin{document}

\pagestyle{empty}

\begin{figure}
\vspace*{-0.4in}
\hspace*{3.0in}
\subfigure{\includegraphics[scale=0.4]{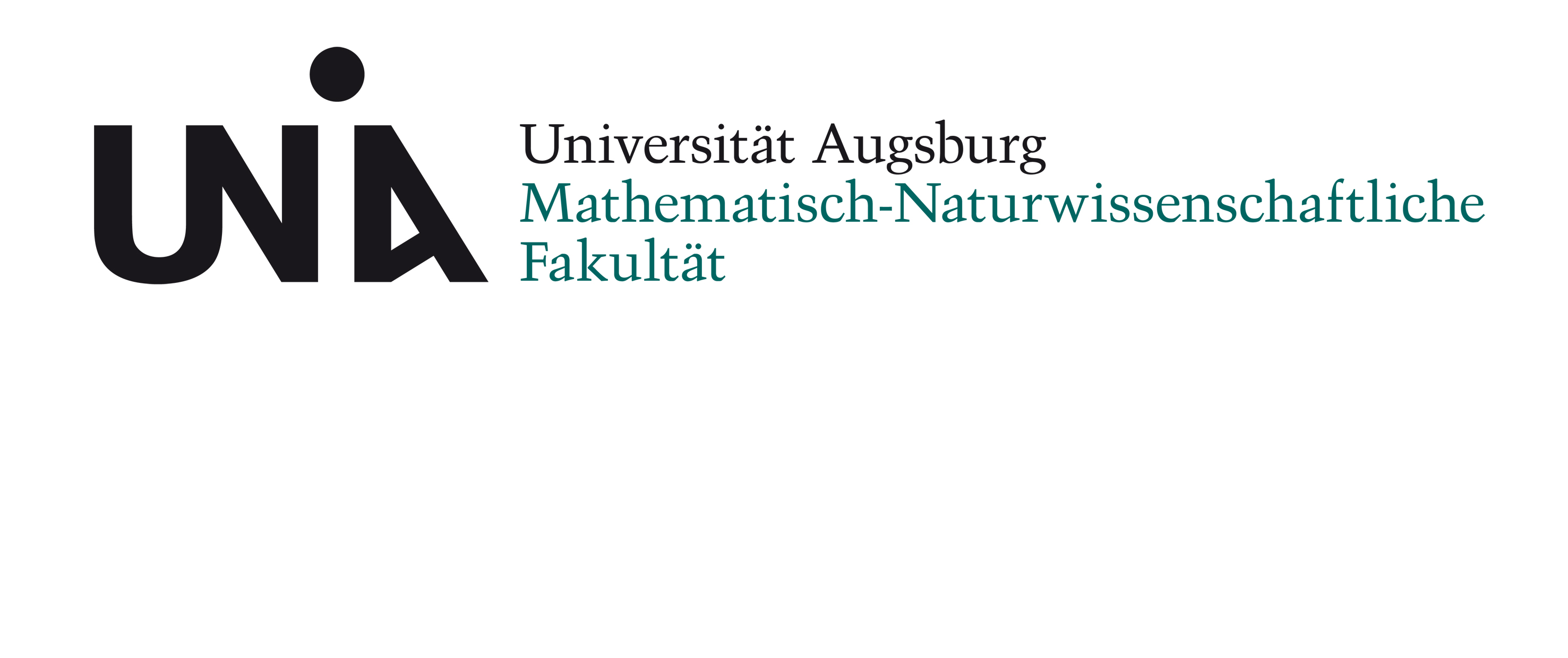}}
\end{figure}

\vspace*{11.0in}

\smash{
\rlap{$\hspace*{-4.0in}$
\includegraphics[scale=1.0]{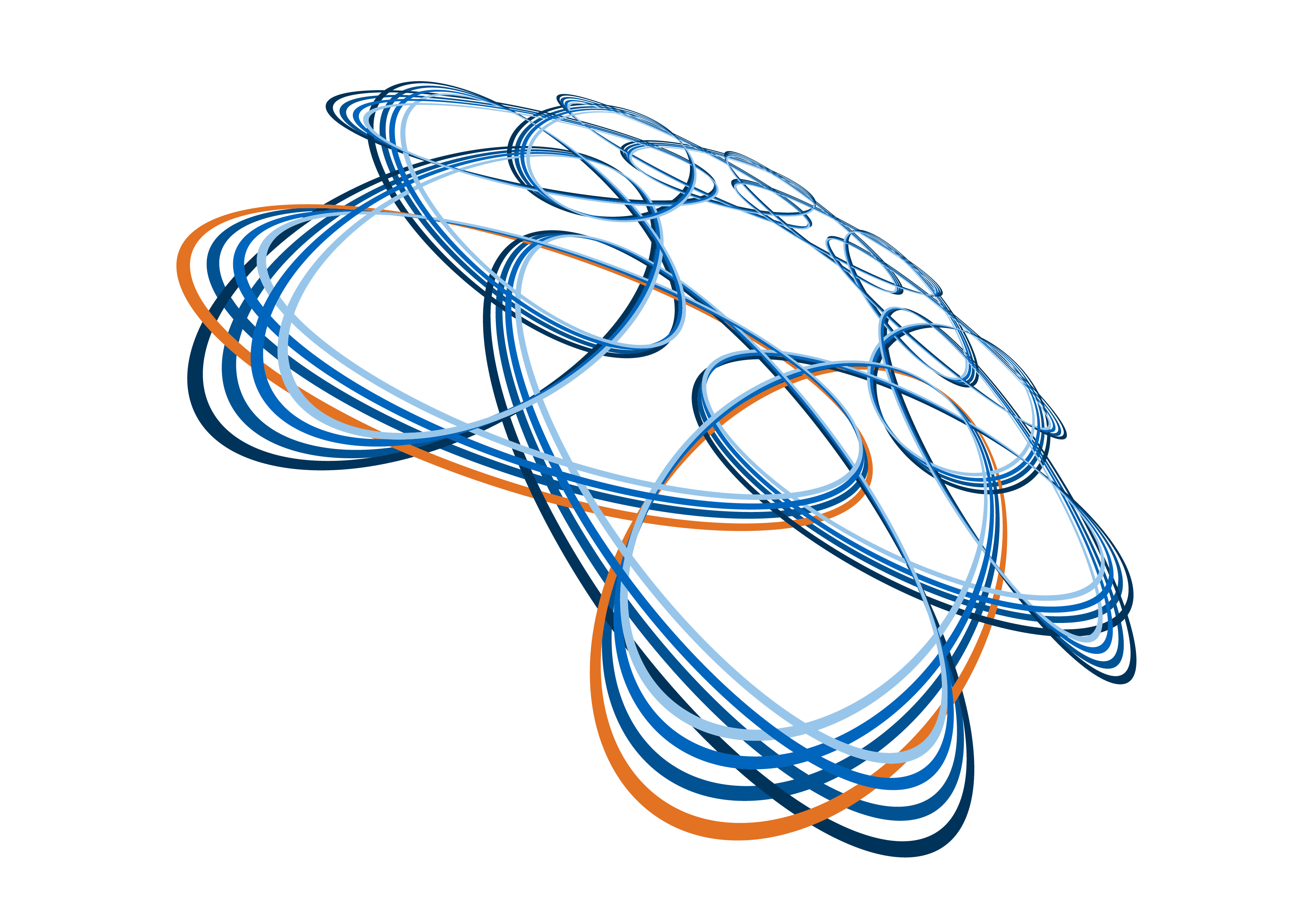}
}
}

\vspace*{-11.0in}

\begin{center}
\Huge
\textbf{Indices of pseudodifferential operators on open manifolds}
\end{center}

\begin{center}
\Large
Alexander Engel
\end{center}

\newpage \mbox{} \newpage

\vspace*{1.86in}

\begin{center}
\Huge
\textbf{Indices of pseudodifferential operators on open manifolds}
\end{center}

\vspace*{1.0in}
\begin{center}
\large
Dissertation\\
zur Erlangung des Doktorgrades (Dr.\ rer.\ nat.)\\
an der Mathematisch-Naturwissenschaftlichen Fakultät\\
der Universität Augsburg\\
\mbox{}\\
vorgelegt von\\
Alexander Engel, M.Sc.\ (hons)\\
\mbox{}\\
betreut von\\
Prof.\ Dr.\ Bernhard Hanke\\
(Universität Augsburg)
\end{center}

\newpage

\mbox{}
\vfill
\large
\begin{center}
\parbox{0cm}
{
\begin{tabbing}
Zweitgutachter: \ \= Prof.\ Dr.\ Bernhard Hanke \kill
Erstgutachter: \> Prof.\ Dr.\ Bernhard Hanke\\
\> (Universität Augsburg)\\
Zweitgutachter: \> Prof.\ John Roe, D.Phil.\\
\> (Pennsylvania State University, USA)\\
\\
Tag der mündlichen Prüfung: \ 27.\ Oktober 2014
\end{tabbing}
}
\end{center}

\normalsize
\newpage

\pagestyle{headings}
\pagenumbering{roman}
\setcounter{page}{5}

\begin{center}\large\textbf{Abstract}\end{center}\normalsize

We generalize Roe's Index Theorem for operators of Dirac type on open manifolds to elliptic pseudodifferential operators.

To this end we first introduce a novel class of pseudodifferential operators on manifolds of bounded geometry which is more general than similar classes of pseudodifferential operators defined by other authors.

We revisit \v{S}pakula's uniform $K$-homology and show that our elliptic pseudodifferential operators naturally define classes there. Furthermore, we use the uniform coarse assembly map to relate this classes to the index classes of these operators in the $K$-theory of the uniform Roe algebra and therefore establish a new and very fruitful link between uniform $K$-homology and Roe's Index Theorem.

Our investigation of uniform $K$-homology goes on with constructing the external product for it and deducing homotopy invariance of uniform $K$-homology.

The next major result is the identification of the dual theory of uniform $K$-homology: the uniform $K$-theory. We give a simple definition of uniform $K$-theory for all metric spaces and in the case of manifolds of bounded geometry we give an interpretation via vector bundles of bounded geometry over the manifold. This opens up the door for Chern--Weil theory and we define a Chern character map from uniform $K$-theory of a manifold of bounded geometry to its bounded de Rham cohomology.

We introduce a type of Mayer--Vietoris argument for these uniform (co-)homology theories which enables us to show firstly, that the Chern character induces an isomorphism modulo torsion from the uniform $K$-theory to the bounded de Rham cohomology, and secondly, that we have \Poincare duality between uniform $K$-theory and uniform $K$-homology if the manifold is spin$^c$.

\Poincare duality together with the relation of uniform $K$-homology to the index theorem of Roe mentioned above directly leads to a generalization of the index theorem to elliptic pseudodiffential operators.

Finally, using homotopy invariance of uniform $K$-homology we derive important results about the uniform coarse Baum--Connes conjecture establishing it equally important as the usual coarse Baum--Connes conjecture.

\newpage \thispagestyle{empty} \mbox{} \newpage

\begin{center}\large\textbf{Acknowledgements}\end{center}\normalsize

First and foremost I would like to thank my doctoral advisor Professor Bernhard Hanke for his support and guidance. Without his encouragement throughout the last years, this thesis would not have been possible.

Further, I want to express my gratitude to Professor John Roe from Pennsylvania State University with whom I had many inspirational conversations especially during my research stay there in summer 2013. I would also like to thank the Professors Alexandr Sergeevich Mishchenko and Evgeni Vadimovich Troitsky from Lomonosov Moscow State University, with whom I had the pleasure to work with during my stay in Moscow from September 2010 to June 2011.

It is a pleasure for me to thank all my fellow graduate students at the University of Augsburg for the delightful time I spent with them.

I gratefully acknowledge the financial and academic support from the graduate program TopMath of the Elite Network of Bavaria, the TopMath Graduate Center of TUM Graduate School at Technische Universität München and the German National Academic Foundation (Studienstiftung des deutschen Volkes).

\newpage \thispagestyle{empty} \mbox{} \newpage

\tableofcontents

\clearpage
\pagenumbering{arabic}
\setcounter{page}{1}

\chapter{Introduction}

\section*{Index theorem on open manifolds}

Let us first recall the famous Atiyah--Singer Index Theorem:

\begin{atiyahsingerthm*}[{\cite{atiyah_singer_1}}]
For any elliptic pseudodifferential operator $P$ over an oriented compact manifold $M$ without boundary we have
\[\operatorname{index}(P) = \int_M \operatorname{top-index}(P),\]
where $\operatorname{index}(P)$ is the Fredholm index of $P$ and $\operatorname{top-index}(P) \in H^m(M)$ is the cohomological index class of $P$.
\end{atiyahsingerthm*}

Roe generalized in \cite{roe_index_1} the index theorem to non-compact manifolds and we will explain this generalization now. First we will treat the cohomological side of the index theorem: if the manifold is non-compact, the top-dimensional cohomology $H^m(M)$ of it vanishes, i.e., we need to find another receptacle for the topological index class. Roe's idea was to use the bounded de Rham cohomology $H^m_{b, \mathrm{dR}}(M)$ of $M$ which is defined analogously as the usual one but using only differential forms $\alpha$ that are bounded in the norm $\sup_{x \in M} \|\alpha(x)\| + \|d \alpha(x)\|$. For this definition to make sense, the manifold $M$ must be equipped with a Riemannian metric, which is the first difference to the compact case (where the index theorem is independent of any metric on $M$). This reliance on a Riemannian metric manifests itself in, e.g., the fact that $H^m_{b, \mathrm{dR}}(\IR^m) \not= 0$, but $H^m_{b, \mathrm{dR}}(\mathbb{H}^m) = 0$.

This example also shows that if we use bounded de Rham cohomology as a receptacle for the index classes, we can not get an index theorem for, e.g., the hyperbolic space. The vanishing of the top-dimensional bounded de Rham cohomology is closely related to the amenability of the manifold. In fact, it is amenable if and only if the top-dimensional bounded de Rham cohomology does not vanish. So Roe's generalization of the Atiyah--Singer Index Theorem does only hold for amenable manifolds. That this is indeed a conceptual problem and not a problem of the techniques that Roe uses in the proof, was shown by Block and Weinberger in \cite{block_weinberger_1}.

Having found a receptacle for the index classes, we now have to find a way to evaluate them on the manifold (since plainly integrating them would often give infinity as a result). Here Roe's idea was to use an averaging procedure: we choose a compact exhaustion $(M_i)_i$ (where each $M_i$ is an embedded submanifold with boundary and of codimension $0$) of the manifold $M$ and consider the sequence $\frac{1}{\vol M_i} \int_{M_i} \alpha$, where $\alpha$ is a top-dimensional differential form. If $\alpha$ is a bounded form, then the above sequence is bounded, and choosing a functional $\tau \in (\ell^\infty)^\ast$ we may evaluate it on the above sequence to get an averaged integral of $\alpha$. We also need that this averaged integral descends to classes, i.e., we need that it vanishes on derivatives of bounded forms. So if $\alpha = d \beta$, we get $\frac{1}{\vol M_i} \int_{M_i} \alpha = \frac{1}{\vol M_i} \int_{\partial M_i} \beta \le \frac{\vol \partial M_i}{\vol M_i} \| \beta \|_\infty$, i.e., in order that the averaged integral vanishes on all derivatives of bounded forms, the exhaustion $(M_i)_i$ of $M$ must satisfy $\frac{\vol \partial M_i}{\vol M_i} \to 0$.\footnote{And furthermore, the functional $\tau \in (\ell^\infty)^\ast$ must be associated to a free ultrafilter on $\IN$, i.e., if we evaluate $\tau$ on a bounded sequence, we get the limit of some convergent subsequence. This is needed so that we may exploit the property $\frac{\vol \partial M_i}{\vol M_i} \to 0$ of the \Folner exhaustion.} Fortunately, the existence of such an exhaustion is exactly the definition of amenability of a manifold,\footnote{One generally demands $\frac{\vol B_r(\partial M_i)}{\vol M_i} \to 0$ for all $r > 0$ since the compact sets $M_i$ are usually not required to be submanifolds with a smooth boundary and of codimension $0$. Such exhaustions are called \Folner exhaustions.} and we have seen above that we have anyway to assume that our manifold is amenable (i.e., we do not get further restrictions here on the index theorem).

So the topological side of Roe's Index Theorem has the following form: given an operator $D$ of Dirac type over an amenable manifold $M$, the topological index of $D$ is defined as $\ind_\tau(I_t(D))$, where $\ind_\tau$ is the averaged integral with respect to a choice of \Folner exhaustion of $M$ and a choice of functional $\tau \in (\ell^\infty)^\ast$ associated to a free ultrafilter on $\IN$. We use another symbol $I_t(\largecdot)$ instead $\tind(\largecdot)$ in order to denote the topological index, since it is now defined via a different method than Atiyah and Singer do it: the index class $I_t(D) \in H^m_{b, \mathrm{dR}}(M)$ of the operator $D$ is defined via the asymptotic expansion of the integral kernel of the operator $e^{-tD^2}$. That it is indeed a bounded form needs that the manifold $M$ has bounded geometry, i.e., that its injectivity radius is uniformly positive and that its curvature tensor and all its covariant derivatives are bounded. This is an additional restriction on the index theorem, but since it is also used crucially at other places in the proof, we can not circumvent it. The crucial use of the asymptotic expansion of the operator $e^{-tD^2}$ is the fact why Roe can prove his index theorem only for operators of Dirac type, and is the reason why we have to use totally different methods than Roe to extend his index theorem to pseudodifferential operators.

To explain the analytic side of Roe's Index Theorem, let us first rephrase it in the compact case. Given an elliptic, graded operator $D$, its analytic index is defined as the Fredholm index $\dim \kernel D^+ - \dim \cokernel D^+$ of the positive part $D^+$ of $D$. On a compact manifold, the eigenvalues of elliptic operators are discrete. So we may find a bump function $f$ with $f(0) = 1$ and which is zero on all non-zero eigenvalues of $D^2$. Then $f(D^2)$ is the projection operator onto the kernel of $D^2$, which is the kernel of $D$ itself, and $\dim \kernel D$ may be rewritten as $\trace f(D^2)$. The analytic index of $D$ is then given as $\trace_\epsilon f(D^2)$, where $\epsilon$ is the grading operator and $\trace_\epsilon$ the graded trace. Since $f$ is a Schwartz function, the operator $f(D^2)$ is a smoothing operator and therefore represented by a smooth integral kernel $k_{f(D^2)}(x,y)$ over $M \times M$. So we finally rewrite the analytic index of $D$ as $\int_M \trace_\epsilon k_{f(D^2)}(x,x) dM$. But this expression may be now generalized to non-compact manifolds: given an elliptic operator $D$ over a manifold $M$ of bounded geometry, the operator $f(D^2)$, where $f$ is a Schwartz function with $f(0) = 1$, is a smoothing operator and therefore represented by a bounded smooth integral kernel $k_{f(D^2)}(x,y)$ over $M \times M$. Choosing a compact exhaustion $(M_i)_i$ of $M$, we define the analytic index of $D$ as the evaluation of a functional $\tau \in (\ell^\infty)^\ast$ on the bounded sequence $\frac{1}{\vol M_i} \int_{M_i} \trace_\epsilon k_{f(D^2)}(x,x) dM$. Though now the eigenvalues of the operator $D$ need not be discrete, i.e., we usually can not find a Schwartz function $f$ with $f(0) = 1$ which vanishes on every non-zero eigenvalue of $D^2$, Roe's Index Theorem leads to the fact that the analytic index is independent of the choice of such a function $f$. So the analytic index does not only depend on the zero-eigenvalue of $D$, but on the germ of the spectrum of $D$ around $0$. Roe discusses implications of this in \cite{roe_index_2}.

So we are finally able to state Roe's theorem:

\begin{roesthm*}[{\cite{roe_index_1}}]
Let $M$ be an amenable and oriented manifold of bounded geometry and let $D$ be a graded operator of Dirac type over $M$. Then
\[\ind_\tau(I_a(D)) = \ind_\tau(I_t(D)).\]
Here $I_t(D) \in H^m_{b, \mathrm{dR}}(M)$ is the topological index class of $D$, $I_a(D) \in K_0^{\mathrm{alg}}(\mathcal{U}_{-\infty}(M))$ its analytic index class, and $\ind_\tau$ is the averaged integral.
\end{roesthm*}

We have not explained the algebra $\mathcal{U}_{-\infty}(M)$ since it is not needed to do this for the understanding of the theorem: the above described procedure to compute the analytic index of $D$ is exactly the value $\ind_\tau(I_a(D))$.\footnote{But we will come back to this algebra later in the introduction in our discussion of pseudodifferential operators.} But let us mention that in order for the averaged integral (which is used to calculate the analytic index) to descend to the algebraic $K$-theory of $\mathcal{U}_{-\infty}(M)$, it is necessary to show that it is a trace, i.e., vanishes on commutators. To do this, we need again the amenability of $M$. So similarly as for the averaged integral on top-dimensional bounded de Rham forms, amenability is needed to show that the averaged integral descends to classes.

Roe proved his index theorem only for operators of Dirac type, since he has to extract the topological index class out of the asymptotic expansion of the integral kernel of $e^{-tD^2}$. But the Atiyah--Singer Index Theorem is valid much more generally for elliptic pseudodifferential operators. But for them the asymptotic expansion does in general not exist, i.e., Roe's proof breaks down for more general operators than of Dirac type.

But of course nevertheless the question arises whether there is a generalization of his index theorem to pseudodifferential operators. This is the main topic of this thesis and, in fact, we will be successful:

\begin{mainthm*}
Let $M$ be an amenable, even-dimensional spin$^c$ manifold of bounded geometry and let $P \in \Psi \mathrm{DO}_?^k(S)$ be a graded, elliptic and symmetric pseudodifferential operator of positive order $k > 0$.

Then the topological and analytic index classes of $P$ coincide:
\[\tind(P) = \aind(P) \in \bar{H}^m_{b, \mathrm{dR}}(M).\]
\end{mainthm*}

Since now we compute the analytic and topological index classes differently than Roe and also than Atiyah and Singer, we again use other symbols for them: $\tind(\largecdot)$ and $\aind(\largecdot)$ instead of $I_t(\largecdot)$ and $I_a(\largecdot)$. Note that we may even state the equality of the classes and not only the equality of the evaluations of them under averaged integrals (though this is not due to our techniques that we use in our proof and in fact, Roe could have shown this too).

The restriction to even dimensions and to graded operators is a strong one, but there is a possibility to overcome it: our restriction relies on the contruction of evaluation maps on the even, uniform $K$-homology $K_0^u(M)$ of $M$, and in order to generalize our theorem to all dimensions and to ungraded operators, we would need to construct evaluation maps on $K_1^u(M)$. This is probably possible, though we do not do it in this thesis, since it would lead us too far astray from the main techniques of this thesis. But see our more thorough discussion of this in Section \ref{sec:partitioned_manifold}.

At last, let us mention that though the main theorem of our thesis is to generalize Roe's Index Theorem to pseudodifferential operators, it will be apparent in our discussion further down of the techniques used to prove this generalization that also the development of these techniques is a major result of this thesis.

To elaborate more on this, let us mention that almost all applications of index theory to geometry are using operators of Dirac type (like the question about existence of positive scalar metrics on spin manifolds), i.e., generalizing Roe's Index Theorem does not give us new applications in geometry. But from the techniques used to prove our main theorem we will get a more thorough understanding of Roe's Index Theorem and especially its relation to the Baum--Connes conjecture.

After completing this thesis the author was made aware of Wang's Ph.D.\ thesis \cite{wang_thesis} which was published as \cite{wang}. His main theorem is an $L^2$-index theorem for properly supported, elliptic pseudodifferential operators over complete Riemannian manifolds on which a locally compact, unimodular groups acts properly and cocompactly. Though his theorem applies to a different class of operators, resp. manifolds (being properly supported is a strong restriction on the operators, since it especially not requires one to `ìnvent'' the uniformity condition as we have to do this in this thesis, and our theorem does not need an action of a group on the manifold), it uses the same idea for its proof: one has to find a Dirac operator which has the same ``higher index'' as the pseudodifferential operator and then one uses the heat kernel method for proving the needed index theorem for Dirac operators. Note that finding a Dirac operator having the same analytic index as the pseudodifferential operator is done in Wang's thesis by using the symbol class of the pseudodifferential operator and in our thesis by proving \Poincare duality between uniform $K$-homology and uniform $K$-theory. The connection between these two approaches (in the case of compact manifolds) is the Dolbeault operator on the disk bundle of $M$ giving a $K \! K$-equivalence between $K$-homology of $M$ and compactly supported $K$-theory of the tangent bundle $TM$ of $M$.

\section*{Pseudodifferential operators on open manifolds}

To the surprise of the author the first problem was to find a suitable definition of pseudodifferential operators on open manifolds. Recall that on a compact manifold a pseudodifferential operator is defined as an operator that locally (i.e., in a chart) looks like one on $\IR^m$, and this suffices: one can then show that such an operator has extensions to a continuous operator on Sobolev spaces, that one can freely compose such operators (i.e., that the composition of two pseudodifferential operators is again a pseudodifferential operator), and so on. But on open manifolds such a definition is not good. As an example, consider the operator $x^2 \cdot \tfrac{d}{dx}$ on $\IR$. It is not a pseudodifferential operator on Euclidean space since its symbol $p(x, \xi) = x^2 \cdot \xi$ is unbounded in $x$. And in fact, this operator does not have an extension to a continuous operator $H^1(\IR) \to L^2(\IR)$. But locally this operator does look like a pseudodifferential one, since every local symbol itself is bounded in $x$. To solve this problem we could of course just demand that the local bounds on the symbol should be dominated by a single global bound. But the problem is that in the definition of pseudodifferential operators there are also bounds imposed for the derivatives of the symbol of the operator. Choosing ``bad'' coordinate charts and partitions of unity we could achieve that the true pseudodifferential operator $1 \cdot \tfrac{d}{dx}$ on $\IR$ would not be recognized as one since a bad choice of charts could distort the derivatives of its symbol arbitrarily high with respect to the different charts of the ``bad'' cover of $\IR$.

The solution to this problem is to restrict to Riemannian manifolds that have bounded geometry, i.e., such that their injectivity radius is uniformly positive and their curvature tensor and its derivatives are bounded. On such manifolds we have the following very nice property of normal coordinate charts of a fixed radius less than the injectivity radius: the derivatives of the change-of-coordinates maps are uniformly bounded, i.e., we have bounds that are independent of the position of the normal coordinate charts in the manifold. For such manifolds we may define pseudodifferential operators as ones which locally in normal coordinate charts look like pseudodifferential operators on $\IR^m$ and such that the bounds imposed on the local symbols are independent of the choice of normal coordinate chart. Due to the above property of change-of-coordinates maps this is now well-defined.

This local definition of pseudodifferential operators was already given by Kordyukov in \cite{kordyukov}, by Shubin in \cite{shubin} and by Taylor in \cite{taylor_pseudodifferential_operators_lectures}. And these are, to the surprise of the author, the only three instances that the author is aware of and where such pseudodifferential operators were investigated. Up to now we have described pseudodifferential operators only locally, but we need also a certain global restriction (e.g., in order to prove that such operators compose). Kordyukov and Shubin impose that their operators must have finite propagation, i.e., that there is an $R > 0$ such that the integral kernel $k(x,y)$ of the pseudodifferential operator vanishes for all $x,y$ with $d(x,y) > R$ (note that pseudodifferential operators always have an integral kernel that is smooth outside the diagonal). And Taylor requires more generally an exponential decay of the integral kernel at infinity, and often this decay should be faster than the volume growth of the manifold.

Our definition is a priori more general than Taylor's\footnote{For some results he has to require that the exponential decay of the integral kernels is faster than the volume growth of the manifold. In these cases our class of operators becomes a priori larger. If it is indeed larger is an open question, see the discussion in Section \ref{sec:equality_smooth_uniform_roe_algebras}.} and is inspired by Roe's work: we require that our pseudodifferential operators are quasilocal\footnote{An operator $A \colon H^r(E) \to H^s(F)$ is \emph{quasilocal}, if there is some function $\mu\colon \IR_{> 0} \to \IR_{\ge 0}$ with $\mu(R) \to 0$ for $R \to \infty$ and such that for all $L \subset M$ and all $u \in H^r(E)$ with $\supp u \subset L$ we have $\|A u\|_{H^s, M - B_R(L)} \le \mu(R) \cdot \|u\|_{H^r}$.}. So at the end the definition of pseudodifferential operators on open manifold that we give is novel.

Let us summarize some facts why we prefer it over the other definitions mentioned above: firstly, it is to the knowledge of the author the most general definition of pseudodifferential operators on open manifolds.

Secondly, the pseudodifferential operators of degree $-\infty$, i.e., the smoothing ones, may be completely described: they are exactly the quasilocal smoothing operators. Furthermore, we may also define a slightly smaller class of our operators by restricting the behaviour of the integral kernels at infinity a bit more (concretely, demanding that the ``parts at infinity'' of the operators may be approximated by operators of finite propagation speed in a suitable \Frechet topology). In this case the operators of degree $-\infty$ form exactly the smooth uniform Roe algebra. It is especially nice to know how the operators of order $-\infty$ look like since elliptic operators are invertible modulo operators of order $-\infty$ and therefore have abstract index classes in the even $K$-group of them.\footnote{Recall that in the statement of Roe's Index Theorem we used the analytic index class of an operator $[D] \in K_0^{\mathrm{alg}}(\mathcal{U}_{-\infty}(M))$. This is exactly the abstract index class that we mention here. This immediately establishes that the abstract index class of an operator of Dirac type (which is of course a pseudodifferential operator in our sense) resulting from our considerations coincides with the analytic index class as Roe has defined it.} So it is necessary to understand the operators of order $-\infty$. Now the completion of the smooth uniform Roe algebra is the uniform Roe algebra $C_u^\ast(M)$ of the manifold (hence its name ``\emph{smooth} uniform Roe algebra''), i.e., it is well-investigated. Of course the completion of the operators of order $-\infty$ and having finite propagation (i.e., of the smoothing operators in the sense of Kordyukov and Shubin) is also the uniform Roe algebra. But the problem here is that these operators do not form a local $C^\ast$-algebra, i.e., their operator $K$-groups do not necessarily need to coincide with the ones of the uniform Roe algebra. But we will see that our operators of order $-\infty$ do form a local $C^\ast$-algebra. So this is the second advantage of our definition of pseudodifferential operators over the other definitions.

And thirdly, recall that in order to compute Roe's analytic index of an operator $D$ of Dirac type, we have to consider the operator $f(D^2)$, where $f$ is a Schwartz function with $f(0) = 1$. Now usually $f(D^2)$ will not have finite propagation, but it will be a quasilocal operator. This was proven by Roe and we will generalize this crucial fact to pseudodifferential operators. So this means that we stay in our class of operators when computing analytic indices, but this is not true if we would work with the definitions of the other authors. Note that the proof of the fact that $f(P)$ is quasilocal requires substantial analysis and is one of our key technical lemmas.

Let us summarize some aspects of the above discussion in the following theorem.

\begin{thm*}
We will introduce two slightly different, new class of pseudodifferential operators on manifolds of bounded geometry, which are a priori more general than the above mentioned classes of other authors and which coincide on compact manifolds with the usual one.

Our operators have the usual properties that they have on compact manifolds, e.g., they extend to bounded operators on corresponding Sobolev spaces and they compose, i.e., the product of two pseudodifferential operators is again a pseudodifferential operator.

Furthermore, the operators of order $-\infty$ may be completely described: they are in one version the quasilocal smoothing operators and in the other version they form the smooth uniform Roe algebra. In both cases they form a local $C^\ast$-algebra and the completion of the latter is the uniform Roe algebra of the manifold.

And last, the slightly bigger class of both is closed under functional calculus with Schwartz functions: if $P$ is a symmetric and elliptic operator we will show that it is essentially self-adjoint, and if $f$ is a Schwartz function, then $f(P)$ will be a pseudodifferential operator of order $-\infty$.
\end{thm*}

\section*{Outline of our arguments}

We will now give a brief overview over the proof that we use to show our main theorem. As we have already said, we need something completely novel since Roe's proof only works for operators of Dirac type.

We start with explaining our proof in the compact setting, where all the arguments that we give now are of course already well-known. First we recall that the analytic index may be transformed into a map on $K$-homology: crushing the compact manifold to a point, we get an induced map $K_0(M) \to K_0(\pt)$, and since $K_0(\pt) \cong \IZ$ we get for every class in $K_0(M)$ an integer-value. Of course, if we have a graded, elliptic pseudodifferential operator $P$, then this integer-value associated to the class $[P] \in K_0(M)$ is the Fredholm index of $P$. Furthermore, it is possible to define a topological index map $K_0(M) \to \IZ$ and the Atiyah--Singer Index Theorem may be then rephrased by saying that both maps $K_0(M) \to \IZ$ coincide. The cohomological expression in the index theorem follows from applying the Chern character.

Now suppose that we already know the Atiyah--Singer Index Theorem for operators of Dirac type. How can we extend it to pseudodifferential operators? Well, by the above reformulation of the index theorem, we already know it for all elliptic pseudodifferential operators whose $K$-homology class may be represented by operators of Dirac type! So the only remaining part is to show that indeed \emph{every} class coming from pseudodifferential operators is represented in this way.

This remaining part is well-known for compact manifolds: if the manifold is a \spinc manifold, then the cap product $\largecdot \cap [M]\colon K^\ast(M) \to K_{m-\ast}(M)$ with the fundamental class $[M] \in K_m(M)$ ssociated to the \spinc structure is an isomorphism. This is the \Poincare Duality Theorem for $K$-homology and $K$-theory. Since $K_1(\pt) = 0$, it suffices to restrict to even-dimensional manifolds (this corresponds to the fact that on an odd-dimensional manifold the index of any differential operator is zero). Furthermore, applying the formal $2$-periodicity $K_{p}(M) \cong K_{p+2}(M)$, we may rewrite \Poincare duality as $K^0(M) \cong K_0(M)$. Now the fundamental class $[M] \in K_0(M)$ is given by the class of the Dirac operator $D$ associated to the \spinc structure of $M$, topological $K$-theory $K^0(M)$ consists of formal differences $[E] - [F]$ of isomorphism classes of vector bundles over $M$, and the cap product is in this case given by forming the twisted Dirac operator: $([E] - [F]) \cap [M] = [D_E] - [D_F]$. So we see that on compact \spinc manifolds \emph{every} $K$-homology class is represented by a difference of operators of Dirac type. So the Atiyah--Singer Index Theorem for operators of Dirac type immediately extends to all elliptic pseudodifferential operators.

So we now have to transport this line of arguments somehow into the non-compact world, i.e., we have to find a suitable $K$-homology theory, a corresponding dual $K$-theory, and we have to prove \Poincare duality for them. And in fact, this is exactly what we will do in this thesis.

\section*{Uniform $K$-homology}

The suitable $K$-homology theory that we need in order to execute the above discussed line of reasoning was provided by \Spakula who defined in his PhD thesis \cite{spakula_thesis} uniform $K$-homology $K_\ast^u(M)$. The reason why it looks more promising to try it instead of the usual $K$-homology is that \Spakula constructed a uniform coarse assembly map $\mu_u \colon K_\ast^u(M) \to K_\ast(C_u^\ast(M))$ to the $K$-theory of the uniform Roe algebra $C_u^\ast(M)$ of $M$. So we immediately have a connection to Roe's Index Theorem: recall that elliptic pseudodifferential operators have an abstract index class in the $K$-theory of the quasilocal smoothing operators. Since this algebra is a local $C^\ast$-algebra and its completion is the uniform Roe algebra, we get an abstract index class in $K_0(C_u^\ast(M))$. Of course we will show that graded elliptic pseudodifferential operators naturally define classes in $K_0^u(M)$ and then we will show that the image of these classes under the uniform coarse assembly map coincides with the abstract index classes of the elliptic operators. Since for operators of Dirac type this abstract index class is the one that Roe considers, we immediately get a connection of uniform $K$-homology to Roe's Index Theorem. This connection is furthermore enhanced by the fact that for amenable manifolds we may define index maps $\ind_\tau \colon K_0^u(M) \to \IR$ and these maps coincide under the uniform coarse assembly map with the analytic index maps of Roe that we have mentioned above and that are defined using averaging integrals.

So the first major accomplishment of the author of this thesis with respect to the uniform $K$-homology of \Spakula is to establish it as the main receptacle for classes of elliptic operators in order to conduct index theory with them in the direct tradition of Atiyah and Singer. Note that though we also have the usual $K$-homology which also accepts classes of elliptic operators, we may in general not calculate any indices with it, i.e., we do not have evaluation maps to the reals in the non-compact case for usual $K$-homology. Note that the index maps $\ind_\tau \colon K_0^u(M) \to \IR$ that we define are exactly the analytic index maps we talked about in the above section where we outlined our arguments, i.e., which are part of the rephrasing of the Atiyah--Singer Index Theorem by saying that the analytic and topological index maps on $K$-homology coincide.

If course, if our manifold is compact, uniform $K$-homology is exactly the usual $K$-homology and the analytic index map is the usual one, i.e., the Fredholm index.

Let us summarize these major results of this thesis:

\begin{thm*}
If $M$ is a manifold of bounded geometry and $P$ a graded elliptic pseudodifferential operator, then $P$ naturally defines a class in the uniform $K$-homology $K_0^u(M)$ of $M$. This class depends only on the principal symbol class of $P$.

Moreover, the uniform $K$-homology class of $P$ is mapped under the uniform coarse assembly map $\mu_u \colon K_\ast^u(M) \to K_\ast(C_u^\ast(M)) \cong K_\ast(\Psi \mathrm{DO}_u^{-\infty}(M))$ to the abstract index class of $P$ arising from the fact that $P$ is invertible modulo smoothing operators.

If $M$ is amenable, we will define analytic index maps $K_0^u(M) \to \IR$ which coincide with the analytic index maps of Roe under the uniform coarse assembly map.
\end{thm*}

In the above theorem we have mentioned that the uniform $K$-homology class of $P$ depends only on the principal symbol class of $P$. This expresses a certain stability of the index class (in this case independence from lower order terms) and therefore also a certain stability of the analytic indices itself. But the most perfect form of stability is achieved by homotopy invariance: recall that if two compact manifolds are homotopy equivalent, then their $K$-homology groups are isomorphic. This shows that the analytic index is resistent against homotopies---a result that is not expected since the definition of $K$-homology involves a lot of analysis. In fact, this homotopy invariance may be seen as an instance of the index theorem itself: if we rip the evaluation to integers away from the Atiyah--Singer Index Theorem, the remaining statement is exactly the homotopy invariance.

So it is of course extremly desirable to prove homotopy invariance of uniform $K$-homology, since, as we have said, this is the major part of the index theorem. To put the homotopy invariance more into the light, recall that our main theorem states the equality of the analytic and topological index classes. But this can only be formulated for amenable manifolds since only here the top-dimensional bounded de Rham cohomology does not vanish. So we have no main theorem in the non-amenable case. But as explained, we may use instead the homotopy invariance of uniform $K$-homology as our main theorem for non-amenable manifolds. Though we do not have any classes to compare in this case, we still have everything else of the index theorem. We will persue this path to index theory in the second-to-last chapter, where we will discuss uniform coarse indices. There homotopy invariance will be crucially used at several places.

The proof of homotopy invariance for usual $K$-homology utilizes the exterior product. This product may be seen as the most crucial construction for $K$-homology, since from it not only homotopy invariance may be deduced, but also Bott periodicity which is essential to the definition of the topological index map on $K$-homology. Our method for proving homotopy invariance of uniform $K$-homology will be the same, i.e., we will construct the exterior product and then deduce homotopy invariance. So the next major result of this thesis is the following theorem:

\begin{thm*}
We will construct an associative product
\[\times \colon K_{p_1}^u(X_1) \otimes K_{p_2}^u(X_2) \to K^u_{p_1 + p_2}(X_1 \times X_2)\]
with all the usual properties that the external product on usual $K$-homology has.

From this we will conclude the homotopy invariance of uniform $K$-homology.
\end{thm*}

At last, let us get to the difference of uniform $K$-homology and usual $K$-homology and what this has to do with our definition of pseudodifferential operators. Atiyah extracted in \cite{atiyah_global_theory_elliptic_operators} the pseudolocality of operators over manifolds as the main index theoretic property. Since this property may be defined on any space, not only manifolds (and indeed, also far more general for $C^\ast$-algebras) this led to the definition of $K$-homology as we know it via Fredholm modules. In fact, under Paschke duality and since the $K$-theory of the locally compact operators vanishes we have that $K$-homology is isomorphic to the $K$-theory of the quotient of pseudolocal operators modulo locally compact operators. This corresponds to the abstraction of the following properties of pseudodifferential operators on Euclidean space: operators of order zero are pseudolocal and operators of negative order are locally compact.

Now uniform $K$-homology is the $K$-theory of the quotient of the uniformly pseudolocal operators modulo the uniformly locally compact one, i.e., the only difference is a built-in uniformity. We will show that our pseudodifferential operators of order zero are uniformly pseudolocal and that our pseudodifferential operators of negative orders are uniformly locally compact. This is due to our requirement in the definition of them that the bounds on the local symbols should be independent of the location of the normal coordinate chart in the manifold. So we see that uniform $K$-homology resembles more closely the $K$-theory of the quotient of operators of order zero modulo operators of negative order than usual $K$-homology does. Maybe this explains why we get an index theorem which is similar to the original one of Atiyah and Singer if we use uniform $K$-homology, and why an analogous theorem for usual $K$-homology of non-compact manifolds is not in sight, though of course usual $K$-homology is used in many other ways in index theory like in the Baum--Connes conjecture.

\section*{Uniform $K$-theory and duality}

So we have found a suitable homology theory for our endeavour, we now have to find its corresponding dual theory. To get an idea where to look for, we recall the compact case: here the dual theory of $K$-homology is topological $K$-theory. It consists of formal differences $[E] - [F]$ of isomorphism classes of vector bundles over $M$ and the \Poincare duality map is given in degree zero by forming the twisted operator: $[E] \cap [D] = [D_E]$. This is exactly what we need in order to extend the index theorem from operators of Dirac type to pseudodifferential operators, since twisted operators $D_E$ are of Dirac type. So to generalize to the non-compact case, we have to figure out what vector bundles are needed in order to define twisted operators. And the answer is easy: if we have an operator $D$ of Dirac type associated to a Dirac bundle $S$ of bounded geometry, a vector bundle $E$ must also have bounded geometry in order that the twisted bundle $S \otimes E$ does also possess bounded geometry. Then Roe's Index Theorem also applies to the twisted operator $D_E$.

So we need to introduce a $K$-theory on manifolds of bounded geometry that consists of vector bundles of bounded geometry. Though we could do this directly, i.e., completely analogously as for topological $K$-theory of compact spaces, we will persue another path to a definition of uniform $K$-theory: recall that we have the isomorphism $K^\ast(M) \cong K_\ast(C(M))$ between the topological $K$-theory of the compact space $M$ and the operator $K$-theory of the algebra $C(M)$ of continuous functions on $M$. This isomorphism is given in the following way: if we have a vector bundle $E$, there is a vector bundle $F$ such that $E \oplus F$ is a trivial bundle. Then we associate to $E$ the idempotent matrix with entries in $C(M)$ given by the projection matrix of the trivial bundle onto the subspace $E$. Conversely, given an idempotent matrix over $M$, its image will be a vector bundle over $M$.

Generalizing to manifolds of bounded geometry, the author had the following idea: if the entries of the idempotent matrix are functions from $C_b^\infty(M)$, i.e., all their derivatives are bounded, then the image of this matrix equipped with the induced metric and connection (where we equip the trivial bundle with a fixed metric and the flat connection) should have bounded geometry. The intuition behind this idea is that since the entries of the idempotent matrix have bounded derivatives, the subspace of the trivial bundle given by the image of that matrix shouldn't turn too fast around inside the trivial bundle, i.e., its curvature shouldn't get arbitrarily large. And in fact, this idea will be fruitful.

Since $C_b^\infty(M)$ is a local $C^\ast$-algebra and its completion is the algebra $C_u(M)$ of all bounded, uniformly continuos functions, we may therefore define uniform $K$-theory for all metric spaces as the $K$-theory of this algebra. This is especially nice since this enables us to define the cap product between uniform $K$-theory and uniform $K$-homology for all metric spaces (and not only for manifolds of bounded geometry where the idea for the definition of uniform $K$-theory comes from). Let us summarize this accomplishment of this thesis in the following theorem:

\begin{thm*}
We introduce a new version of $K$-theory, called ``uniform $K$-theory'', for all metric spaces as $K^p_u(X) := K_{-p}(C_u(X))$.

On a manifold of bounded geometry, we get the following interpretation of $K^0_u(M)$: it consists of formal differences $[E] - [F]$ of isomorphism classes (which take the metric and connection into account) of vector bundles of bounded geometry over $M$. Furthermore, \emph{every} vector bundle of bounded geometry defines a class in uniform $K$-theory.\footnote{We especially emphasized the word ``every'', since in order that $E$ defines a class in $K^0_u(M)$ we have to construct a complement bundle $F$ of bounded geometry such that $E \oplus F$ is isomorphic to the trivial bundle equipped with a fixed metric and flat connection. The crucial point here is the isomorphism and its inverse must be bounded against the occuring metrics and connections, and also all derivatives of them.}
\end{thm*}

Note that other authors have, of course, investigated similar versions of $K$-theory: Kaad investigated in \cite{kaad} Hilbert bundles of bounded geometry over manifolds of bounded geometry (the author thanks Magnus Goffeng for pointing to that publication, but this happened in the final stage of the authors work on his thesis, i.e., he was not aware of this publication until he almost finished his thesis). Dropping the condition that the bundles must have bounded geometry, there is a general result by Morye contained in \cite{morye} having as a corollary the Serre--Swan theorem for smooth vector bundles over (possibly non-compact) smooth manifolds. If one is only interested in the last mentioned result, there is also the short note \cite{sardanashvily} by Sardanashvily.

Atiyah and Hirzebruch showed in \cite{atiyah_hirzebruch} that the Chern character induces an isomorphism $K^\ast(X) \otimes \IQ \cong H^\ast(X; \IQ)$ if $X$ is a finite CW-complex. The question whether we have something similar for uniform $K$-theory immediately arises. Looking at the definition of the Chern character via Chern--Weil theory, we see that the Chern character maps from $K^0_u(M)$ into the bounded de Rham cohomology $H^{\mathrm{ev}}_{b, \mathrm{dR}}(M)$ of $M$. So the author asked himself if the Chern character induces an isomorphism modulo torsion between the uniform $K$-theory of a manifold of bounded geometry and its bounded de Rham cohomology, and was in fact able to show it. For the proof of this we will have to introduce slightly modified versions of uniform $K$-theory and bounded de Rham cohomology for open subsets of manifolds. This modified versions will then enable us to use Mayer--Vietoris sequences and an induction over an open cover of the manifold to assemble the isomorphism. In order to do only finitely many induction steps, we will have to consider at every step infinitely many, disjoint open balls at once, which will introduce even more technical baggage into the proof. But at the end we will succeed.

The above type of Mayer--Vietoris argument turns out to be very useful since our proof of \Poincare duality will be essentially the same proof, but with bounded de Rham cohomology replaced by uniform $K$-homology.

\begin{thm*}
We introduce a type of Mayer--Vietoris argument which is suitable for proving isomorphism theorems for uniform (co-)homology theories.

We execute this Mayer--Vietoris argument to show that the Chern character induces an isomorphism $K^\ast_u(M) \barotimes \IR \cong H^\ast_{b, \mathrm{dR}}(M)$ for manifolds of bounded geometry and that the cap product is an isomorphism $K^\ast_u(M) \cong K_{m - \ast}^u(M)$ if $M$ is spin$^c$.
\end{thm*}

Now we do have all that we need to generalize Roe's Index Theorem to pseudodifferential operators. And in fact, the actual proof of this generalization is now quite short and straightforward.

\section*{Uniform coarse indices}

Roe's Index Theorem does only hold for amenable manifolds since on non-amenable manifolds the top-dimensional bounded de Rham cohomology vanishes, i.e., we have no group anymore to accept the index classes of operators. Therefore also our generalization of Roe's Index Theorem to pseudodifferential operators does not hold for non-amenable manifolds. But we have already mentioned above that we may regard homotopy invariance as a form of an index theorem, i.e., our proof that uniform $K$-homology is homotopy invariant will be in the non-amenable case our main theorem.

We will of course elaborate in this thesis on that matter. Concretely, we will be concerned with the uniform coarse Baum--Connes conjecture which asserts that the uniform coarse assembly map $\mu_u \colon K_\ast^u(X) \to K_\ast(C_u^\ast(X))$ constructed by \Spakula is an isomorphism if $X$ is uniformly contractible and of bounded geometry. We may also formulate a version of this conjecture for countable, discrete groups equipped with a proper, left-invariant metric, and a corollary of the homotopy invariance of uniform $K$-homology will be that the conjecture for the universal cover of an aspherical manifold is equivalent to the conjecture for the fundamental group of the manifold.

Furthermore, \Spakula showed that the uniform coarse conjecture for a group is equivalent to the usual Baum--Connes conjecture for this group with certain coefficients, if the group is torsion-free. Another corollary of the homotopy invariance of uniform $K$-homology will be that we may drop the torsion-freeness in this statement. Since the Baum--Connes conjecture with coefficients is meanwhile proven for a large class of groups, we therefore get that the uniform coarse Baum--Connes conjecture does also hold for a large class of groups. Combining this with the first mentioned corollary relating the conjecture for groups to the conjecture for universal covers of aspherical manifolds, we see that for a fairly large class of manifolds the conjecture is true. So we have fulfilled our promise that homotopy invariance may be seen as a replacement for our main theorem: it establishes the uniform coarse Baum--Connes conjecture as a new index theorem (and not only in the non-amenable case, but in all cases).

Let us elaborate on the last sentence of the above paragraph. To show that we may use the uniform coarse Baum--Connes conjecture as an index theorem, we will derive an index theoretic obstruction to positive scalar curvature of spin manifolds in the $K$-theory of the uniform Roe algebra of the manifold, and we will use this to show that if the conjecture holds for a certain group, then aspherical manifolds with this fundamental group do not carry metrics of positive scalar curvature.

We have an analogous obstruction to positive scalar curvature metrics on spin manifolds in the $K$-theory of the usual Roe algebra and from this we may derive the same conclusion about non-existence of metrics of positive scalar curvature on aspherical manifolds using the usual coarse Baum--Connes conjecture. Moreover, it seems that we may also derive the analytic Novikov conjecture from the uniform coarse conjecture (though we do not prove this in this thesis explicitly). So both major implications of the coarse Baum--Connes conjecture also follow from the uniform coarse conjecture, establishing them equally important.

\section*{Short overview over each chapter}

Let us end our introduction with a short overview over each chapter explaining what we do there and what the crucial technical lemmas are.

\paragraph{Chapter \ref{chap:quasilocal_smoothing_operators}} We start with an investigation of quasilocal smoothing operators and of the smooth uniform Roe algebra, their $K$-theories and the definition of the analytic index maps. The reader should regard this chapter as an introduction to all the important notions and techniques used in this thesis: we will introduce bounded geometry and amenability, discuss the various implications and equivalent characterizations of amenability, and we will show how to define index maps on amenable manifolds using \Folner sequences. Furthermore, we will establish the important relation to the uniform Roe algebra and the index maps on it and therefore to Roe's index theorem.

\paragraph{Chapter \ref{chap:PDOs}} Here we introduce our classes of pseudodifferential operators and discuss all the usual, but crucial properties of them. Though all results in this chapter are novel since the author is the first to define this classes of pseudodifferential operators, the results are not surprising since they mimic the results already known for compact manifolds, and the proofs are also similar to the compact case (and therefore we will leave some of them out). In Section \ref{sec:uniformity_PDOs} we will then establish the important connection between our pseudodifferential operators and uniform $K$-homology: the latter relies on so-called uniformly pseudolocal and uniformly locally compact operators, and we will show that our pseudodifferential operators are such operators, if they have order $0$, resp. negative order. But the meat of this chapter is contained in its last section about functions of pseudodifferential operators: there the most of the non-trivial analysis of this thesis is done. The reason why we need to do this hard analysis is that we do not have anymore the finite propagation property of the operators $e^{itP}$ (recall that $e^{itP}$ has finite propagation if $P = D$ is an operator $D$ of Dirac type). This complicates things a lot. But instead we will prove the crucial lemma that $e^{itP}$ is a quasilocal operator. Indeed, this lemma is probably one of the most important ones in this thesis since without it any other result in the last section of this third chapter breaks down.

\paragraph{Chapter \ref{chapter:uniform_k_homology}} In this chapter we first revisit \v{S}pakula's uniform $K$-homology, which takes some time since we need many properties of it that \Spakula proved (like Paschke duality). Moreover, since the author is the first after \Spakula to work with uniform $K$-homology, the reader is probably unfamiliar with it. So we take our time to explain everything. The first important novel result in this chapter will be that elliptic pseudodifferential operators define classes in uniform $K$-homology. The proof of this requires also substantial analysis. The second half of the chapter is devoted to prove further results about uniform $K$-homology that \Spakula did not investigate: we construct analytic index maps on it, we construct the external product and from it we derive the homotopy invariance of uniform $K$-homology. It is not surprising that the construction of the external product is quite involved since this is already the case for usual $K$-homology and here we additionally have to worry about the uniformity condition all the time.

\paragraph{Chapter \ref{chap:uniform_k_theory}} This chapter introduces the completely novel uniform $K$-theory. After proving some basic properties of it, the first of two main results will be the interpretation of it on manifolds of bounded geometry, i.e., that in this case it consists of vector bundles of bounded geometry (completely analogous to the compact case, where $K$-theory consists of usual vector bundles over the space). The proof of this is quite lengthy since we always have to carry the uniformity with us, i.e., we have to do a lot of differential geometry to derive the necessary bounds on the derivatives of the curvature tensor. The last part of this chapter proves the second of the two main results in it, namely the Chern character isomorphism. This is a direct generalization of the fact that on compact spaces the Chern character induces an isomorphism modulo torsion between $K$-theory and cohomology.

\paragraph{Chapter \ref{chapter:index_theorem}} Here we assemble everything to a proof of our main theorem. We will discuss \spinc manifolds and the cap product between uniform $K$-theory and uniform $K$-homology and we will also prove \Poincare duality. Though this chapter contains the main theorem and the proof of \Poincare duality, it is quite a short and not very technical chapter. The reason is that most of the work has been already done in the previous chapters.

\paragraph{Chapter \ref{chap:uniform_coarse_indices}} In this chapter we will be concerned with the uniform coarse Baum--Conens conjecture. We will start with defining the involved assembly map from uniform $K$-homology to the $K$-theory of the uniform Roe algebra and then we will leap right into it: we will derive important properties about the uniform coarse assembly map from the homotopy invariance of uniform $K$-homology that we will have proved in this thesis. We will end this chapter with an application of the uniform theory to positive scalar curvature metrics and show by this that the uniform theory is in no way inferior to the usual theory.

\paragraph{Chapter \ref{chap:further_questions}} The last chapter is devoted to all the questions that arose out of this thesis but were left unanswered. Since both the pseudodifferential operators that we use in this thesis and the uniform $K$-theory are defined here for the first time, and since we are the first after \Spakula to work with his uniform $K$-homology, there are of course a lot questions left that one may investigate.

\chapter{Quasilocal smoothing operators}\label{chap:quasilocal_smoothing_operators}

This chapter will set the stage for most parts of the thesis. Our main goal is to define the algebra $\IU(E)$ of quasilocal smoothing operators acting on a vector bundle $E$ of bounded geometry over a manifold $M$ of bounded geometry, to define the smooth uniform Roe algebra $\C(E)$, to define analytic index maps $\ind_\tau \colon K_0(\IU(E)) \to \IR$ and on $K_0(\C(E))$, and last to identify the $K$-theory of $\C(E)$ naturally with the one of $C_u^\ast(Y)$, where $Y \subset M$ is a uniformly discrete quasi-lattice. The latter algebra is the uniform Roe algebra of $Y$.

The algebras $\IC(E) \subset \IU(E)$ consist of smoothing operators on the bundle $E$, i.e., operators $H^{-\infty}_\iota(E) \to H^\infty(E)$ between Sobolev spaces of infinite order, with an additional property controlling the behavior of the integral kernels of these operators at infinity. This last property is needed to define the index maps $\ind_\tau$ (Section \ref{sec:analytic_indices_quasiloc_smoothing}).

Our interest in this operator algebra is manifold: first of all and of the most importance is that their $K_0$-group is a receptacle for the analytic index classes of uniform elliptic operators. This was already exploited by Roe in \cite{roe_index_1}, where he proved his index theorem for operators of Dirac type over amenable manifolds, and this is also the reason why the maps $\ind_\tau \colon K_0(\IU(E)) \to \IR$ and on $K_0(\IC(E))$ that we construct in this chapter are called the analytic index maps.

Our second interest in the algebra $\IU(E)$, resp. $\IC(E)$, stems from the fact that it coincides with $\Psi \mathrm{DO}^{-\infty}(E)$, resp. $\Psi \mathrm{DO}_u^{-\infty}(E)$, the algebra of all pseudodifferential operators of order $-\infty$ (recall that we will give two slightly different definitions of pseudodifferential operators). The importance of this relies on index theory: an elliptic pseudodifferential operator has a parametrix, i.e., an inverse modulo $\Psi \mathrm{DO}_?^{-\infty}(E)$\footnote{The subscript ``?'' states that the statement applies to both versions of pseudodifferential operators.}. Therefore, it has an abstract index class in $K_0(\Psi \mathrm{DO}_?^{-\infty}(E))$, and for an operator of Dirac type this class coincides under the identification $\Psi \mathrm{DO}^{-\infty}(E) = \IU(E)$ with the analytic index class that Roe constructed.

Thirdly, our interest in $\C(E)$ relies also on the fact that its $K$-theory coincides with the one of the uniform Roe algebra $C_u^\ast(M)$ of the manifold $M$. This is a uniform analogue of the usual Roe algebra $C^\ast(M)$ which is of great importance in index theory and coarse topology since it is the receptacle for coarse indices of elliptic operators and part of the coarse Baum--Connes conjecture. So the identification of the $K$-theory of $\C(E)$ with the one of $C_u^\ast(M)$ allows us to use methods from coarse topology since $K_\ast(C_u^\ast(M))$ is, as its non-uniform counterpart, a coarse invariant of the space, i.e., $K_\ast(C_u^\ast(M)) = K_\ast(C_u^\ast(Y))$, where $Y \subset M$ is a uniformly discrete quasi-lattice. We will treat this interplay with coarse topology in the second part of this thesis, where we will also deal with the uniform analogue of the coarse Baum--Connes conjecture.

While doing all the above we will have to revisit important concepts that will be used throughout this thesis. First of all, we will have to define the notion of bounded geometry for Riemannian manifolds and for Hermitian vector bundles over them. Bounded geometry for manifolds means that the curvature tensor and all its covariant derivatives are bounded and that the injectivity radius is uniformly positive. This property is crucial for having nice equivalences between local and global definitions of the same things, e.g., the Sobolev spaces. Furthermore, the definition of pseudodifferential operators is not possible without bounded geometry, since there is just a local definition and in order to show that it is independent of the chosen coordinate charts and partitions of unity, we need bounded geometry. And last, bounded geometry ensures the existence of a uniformly discrete quasi-lattice in the manifold, i.e., it makes the application of coarse topological methods possible.

The second crucial concept that we will recall is amenability. It is the property that will give us our index maps into $\IR$ and there is no way round it as a counterexample of Block and Weinberger from \cite{block_weinberger_1} shows (i.e., they show that Roe's Index Theorem can not hold for non-amenable manifolds). This especially means that the index theorem for pseudodiffential operators that we develop in this first part of the thesis does only hold for amenable manifolds. So this is quite an important property and therefore it is not surprising that there are a lot of equivalent characterizations of it. We will state some of them, since we will need them at different places.

The last thing that we will revisit are the uniform Roe algebras that will play a major role in the second part of the thesis. But we will of course also use them in the first part, e.g., we will show in a later chapter that the closure with respect to the operator norm of the pseudodifferential operators of order $-\infty$ (at least, of one version of them) equals exactly the uniform Roe algebra. Furthermore, we will show that the analytic index maps that we may define on $K_0(\C(E))$ and on $K_0(C_u^\ast(Y))$ coincide under the canonical identification of these two $K$-groups and therefore we will draw a line from the smooth version of the index maps to the coarse versions.

So all in all this chapter will be concerned with revisiting important concepts that we will use later. Surely most, if not all, of the statements in this chapter are already known to the experts. But since the author could not find them all in the literature and for the convenience of the reader, we will give proofs for most of the results.

\section{Bounded geometry}

In this section we will revisit the important property of bounded geometry for Riemannian manifolds and for Hermitian vector bundles. Since the goal of this first part of the thesis is to generalize Roe's Index Theorem to pseudodifferential operators on manifolds without boundary,\footnote{Recall that boundary conditions are quite a delicate matter in index theory. Therefore we will save the investigation of these concepts for the future.} we will not bother ourselves with the definition of bounded geometry for manifolds with boundary.\footnote{The interested reader may consult \cite{schick_bounded_geometry_boundary} where Schick investigated the notion of bounded geometry for manifolds with boundary.}

Our definition of bounded geometry will be a global one. But the main focus of this section will be to discuss the important equivalent local characterizations of it that are crucially needed for, e.g., the definition of pseudodifferential operators.

\subsection*{Manifolds of bounded geometry}

\begin{defn}[Bounded geometry]
We will say that a Riemannian manifold $M$ has \emph{bounded geometry}, if
\begin{itemize}
\item the curvature tensor and all its derivatives are bounded, i.e., $\| \nabla^k \Rm (x) \| < C_k$ for all $x \in M$ and $k \in \IN_0$, and
\item the injectivity radius is uniformly positive, i.e., $\injrad_M(x) > \varepsilon > 0$ for all points $x \in M$ and for a fixed $\varepsilon > 0$.
\end{itemize}
\end{defn}

For a $1$-dimensional manifold the first requirement is of course vacuous, but not the second (so the manifold $(0,1)$ with the induced metric from $\IR$ is not of bounded geometry). It follows that the only connected, $1$-dimensional manifolds of bounded geometry are $\IR$ and $S^1$.

If $M$ has bounded geometry, then it must be complete (since the injectivity radius is uniformly positive). Furthermore, if $M$ is non-compact and of bounded geometry, then it must have infinite volume. And from the bound $\|\Rm(x)\| < C_0$ for all $x \in M$ we conclude that the sectional curvatures of $M$ are pinched between $-C_0$ and $C_0$.

\begin{examples}
There are plenty of examples of manifolds of bounded geometry:
\begin{itemize}
\item flat manifolds with a positive injectivity radius,
\item compact manifolds and coverings of them (equipped with the pull-back metric); more generally, coverings of manifolds of bounded geometry,
\item homogeneous manifolds with an invariant metric (this includes especially Euclidean and hyperbolic space) and
\item leafs in a foliation of a compact manifold (this is proved in \cite[lemma on page 91 and the paragraph thereafter]{greene}).
\end{itemize}
It is clear that products of manifolds of bounded geometry equipped with the product metric again have bounded geometry. Furthermore, perturbations of the metric in a compact region preserve bounded geometry (so especially connected sums of manifolds of bounded geometry are again of bounded geometry regardless of the metric that we put onto the region where we glued the manifolds).
\end{examples}

Greene also proved that there are no obstructions against admitting a metric of bounded geometry:

\begin{thm}[{\cite[Theorem 2']{greene}}]
Let $M$ be a smooth manifold without boundary. Then there exists on $M$ a metric of bounded geometry.
\end{thm}

We now come to the important equivalent characterizations of bounded geometry using local coordinates. The first one that we will state uses the Christoffel symbols, resp. the metric coefficients of $M$. As a reference for the first characterization via Christoffel symbols one may use, e.g., \cite[Proposition 2.4]{roe_index_1}, and for the second characterization via the metric coefficients, e.g., \cite[Theorem 2.4]{eichhorn_banach_manifold_structure}.

\begin{lem}\label{lem:bounded_geometry_christoffel_symbols}
A Riemannian manifold $(M^m, g)$ has bounded geometry if and only if there is a ball $B \subset \IR^m$ centered at $0 \in \IR^m$ such that
\begin{itemize}
\item $B$ is the domain of normal coordinates at all points $x \in M$ and
\item the Christoffel symbols $\Gamma_{ij}^k(y)$ of $M$ (considered as functions on $B$) and all their derivatives are bounded independently of $x \in M$, $y \in \exp_x(B)$ and $i,j,k$.
\end{itemize}
The second bullet point can be equivalently changed to
\begin{itemize}
\item the metric coefficients $g_{ij}(y)$ (considered as functions on $B$) and all their derivatives are bounded independently of $x \in M$ and $y \in \exp_x(B)$.
\end{itemize}
\end{lem}

The second equivalent characterization of bounded geometry in local coordinates is via the transition functions between overlapping normal coordinate charts. This is quite an important characterization since it will allow us to show that certain local definitions (like the one of pseudodifferential operators) are independent of the chosen coordinates.

\begin{lem}[{\cite[Appendix A1.1]{shubin}}]\label{lem:transition_functions_uniformly_bounded}
Let the injectivity radius of $M$ be positive.

Then the curvature tensor of $M$ and all its derivatives are bounded if and only if for any $0 < r < \injrad_M$ all the transition functions between overlapping normal coordinate charts of radius $r$ are uniformly bounded, as are all their derivatives (i.e., the bounds can be chosen to be the same for all transition functions).
\end{lem}

A direct proof that the uniform boundedness of the transition functions and all their derivatives gives a uniform bound for the metric coefficients and all their derivatives was given in \cite{MO_characterization_bounded_geometry} by Ivanov.

The last fact which we will need about manifolds of bounded geometry is the existence of ``nice'' coordinate covers and corresponding partitions of unity. A proof may be found in, e.g., \cite[Appendix A1.1]{shubin} (Shubin addresses the first statement about the existence of ``nice'' covers actually to the paper \cite{gromov_curvature_diameter_betti_numbers} of Gromov).

\begin{lem}\label{lem:nice_coverings_partitions_of_unity}
Let $M$ be a manifold of bounded geometry.

For every $0 < \varepsilon < \tfrac{\injrad_M}{3}$ there exists a countable covering of $M$ by balls $B_\varepsilon(x_i)$, i.e., $M = \bigcup B_\varepsilon(x_i)$, with the properties
\begin{itemize}
\item the midpoints $x_i \in M$ of the balls form a \emph{uniformly discrete set} in $M$ (concretely, $d(x_i, x_j) \ge \varepsilon$ for all $i \not= j$) and
\item the balls $B_{2\varepsilon}(x_i)$ with the double radius and the same centers form a \emph{uniformly locally finite cover} of $M$ (i.e., there is a bound $C_\varepsilon$ such that every $x \in M$ lies in at most $C_\varepsilon$ of the balls $B_{2\varepsilon}(x_i)$).
\end{itemize}

Furthermore, there exists a subordinate partition of unity $1 = \sum_i \varphi_i$, i.e., we have $\supp \varphi_i \subset B_{2\varepsilon}(x_i)$, such that in normal coordinates the functions $\varphi_i$ and all their derivatives are uniformly bounded (i.e., the bounds do not depend on $i$).
\end{lem}

\subsection*{Vector bundles of bounded geometry}

Now we come to the notion of bounded geometry for Hermitian vector bundles and its equivalent local characterizations. We restrict ourselves to complex vector bundles since later we will be working with complex $K$-theory, i.e., in order to translate uniform $K$-theory that we will investigate in Chapter \ref{chap:uniform_k_theory} to suitable isomorphism classes of vector bundles, we will have to use complex vector bundles. This is completely analogous to the case of compact spaces and usual topological $K$-theory. But note that everything in this section also applies to real vector bundles of bounded geometry.\footnote{In fact, it should be possible to treat everything in this thesis in the ``real'' (or ``Real'') setting.}

\begin{defn}[Vector bundles of bounded geometry]
Let $M$ be a Riemannian manifold and let $E \to M$ be a complex vector bundle equipped with a Hermitian metric and compatible connection. Then we say that \emph{$E$ has bounded geometry}, if the curvature tensor of $E$ and all its derivatives are bounded.
\end{defn}

\begin{examples}
As for manifolds, there are of course plenty of important examples of vector bundles of bounded geometry:
\begin{itemize}
\item flat bundles (so especially trivial bundles),
\item every bundle over a compact manifold,
\item if $\widetilde{M}$ is a covering of $M$, then the pull-back of a vector bundle of bounded geometry from $M$ to $\widetilde{M}$ has bounded geometry (where $\widetilde{M}$ is equipped with the pull-back metric), and
\item if $M$ has bounded geometry, then the tangent bundle $T M$ of $M$ is a real vector bundle of bounded geometry.
\end{itemize}
Furthermore, if $E$ and $F$ are two vector bundles of bounded geometry, then the dual bundle $E^\ast$, the direct sum $E \oplus F$, the tensor product $E \otimes F$ (and so especially also the homomorphism bundle $\Hom(E, F) = F \otimes E^\ast$) and all exterior powers $\Lambda^l E$ are also of bounded geometry. If $E$ is defined over $M$ and $F$ over $N$, then their external tensor product\footnote{The fiber of $E \boxtimes F$ over the point $(x,y) \in M \times N$ is given by $E_x \otimes F_y$.} $E \boxtimes F$ over $M \times N$ is also of bounded geometry.
\end{examples}

Analogously as for manifolds, there are no obstructions against equipping a complex vector bundle with a Hermitian metric and a compatible connection such that it becomes of bounded geometry. The proof (i.e., the construction of the metric and the connection) is done in ``nice'' local coordinates and partition of unity of $M$ and we have to use the equivalent local characterization of bounded geometry for vector bundles from Lemma \ref{lem:equiv_characterizations_bounded_geom_bundles}.

\begin{lem}
Let $M$ be a manifold of bounded geometry and let $E \to M$ be a complex vector bundle over $M$. Then we may equip $E$ with a Hermitian metric and a compatible connection of bounded geometry.
\end{lem}

If the manifold $M$ has bounded geometry, we have analogous equivalent local characterizations of bounded geometry for vector bundles as for manifolds (i.e., via Christoffel symbols and via transition functions). But before we state them let us briefly explain why we need $M$ to have bounded geometry for this characterizations:

\begin{rem}\label{rem:derivatives_uniformly_bounded}
In the following local characterizations of bounded geometry for vector bundles we will consider certain functions (e.g., the Christoffel symbols of the vector bundle $E$) defined locally in charts that constitute a cover of $M$ and we want to say that this functions and all their derivatives are uniformly bounded (i.e., the bounds can be chosen to be the same for all charts of the cover). But this notion does in general depend on the chosen cover, i.e., such a definition is not desirable. But on manifolds with bounded geometry we do have Lemma \ref{lem:transition_functions_uniformly_bounded} which roughly states the following: if we use normal coordinate charts of a fixed radius to cover $M$ than the notion of ``the functions and all their derivatives are uniformly bounded'' is well-defined, i.e., if we change the coordinate charts, the derivatives of the functions stay bounded if they were so before, and the stay unbounded, if they were unbounded. So on manifolds of bounded geometry such local definitions of boundedness of derivatives make sense.
\end{rem}

\begin{defn}[Synchronous framings]\label{defn:synchronous_framings}
Let us define the equivalent notion of ``normal coordinates'' for vector bundles: at a point $x \in M$ we choose an orthonormal frame for a vector bundle $E \to M$ and extend it to a framing of $E$ in $\exp_x(B)$, where $B \subset \IR^m$ is the domain of normal coordinates at all points of $M$, by parallel translation along radial geodesics. Such a framing is called \emph{synchronous}.
\end{defn}

Now we get to the equivalent local characterization of bounded geometry for vector bundles. The equivalence of the first two bullet points in the next lemma is stated in, e.g., \cite[Proposition 2.5]{roe_index_1}. Concerning the third bullet point, the author could not find any citable reference in the literature (though Shubin uses in \cite{shubin} this as the actual definition).

\begin{lem}\label{lem:equiv_characterizations_bounded_geom_bundles}
Let $M$ be a manifold of bounded geometry and $E \to M$ a vector bundle. Then the following are equivalent:

\begin{itemize}
\item $E$ has bounded geometry,
\item the Christoffel symbols $\Gamma_{i \alpha}^\beta(y)$ of $E$ with respect to synchronous framings (considered as functions on the domain $B$ of normal coordinates at all points) are bounded, as are all their derivatives, and this bounds are independent of $x \in M$, $y \in \exp_x(B)$ and $i, \alpha, \beta$, and
\item the matrix transition functions between overlapping synchronous framings are uniformly bounded, as are all their derivatives (i.e., the bounds are the same for all transition functions).
\end{itemize}
\end{lem}

\section{Uniform \texorpdfstring{$C^\infty$}{C-infty}-spaces and Sobolev spaces}

The purpose of this section is to define Sobolev spaces and to discuss some of their basic properties. We will define them first globally and then will give a local characterization for which we of course will need bounded geometry of the manifold. Note that on manifolds that do not have bounded geometry in a suitable sense, there are in general different, i.e., non-equivalent, definitions of Sobolev spaces.

Since one of the crucial properties of Sobolev spaces are the embedding theorems, we will first briefly discuss the target spaces of the embedding theorem which are called uniform $C^\infty$-spaces.

\subsection*{Uniform \texorpdfstring{$C^\infty$}{C-infty}-spaces}

As already mentioned, we will only give the definition of uniform $C^\infty$-spaces together with a local characterization on manifolds of bounded geometry, since we won't need much more. The interested reader is refered to, e.g., the papers \cite[Section 2]{roe_index_1} or \cite[Appendix A1.1]{shubin} of Roe and Shubin for more information regarding these uniform $C^\infty$-spaces.

Analogously as with vector bundles that are always assumed to be complex, all our functions are assumed to be complex-valued.

\begin{defn}[$C^r$-bounded functions]
Let $f \in C^\infty(M)$. We will say that $f$ is a \emph{$C_b^r$-function}, or equivalently that it is \emph{$C^r$-bounded}, if $\| \nabla^i f \|_\infty < C_i$ for all $0 \le i \le r$.
\end{defn}

If $M$ has bounded geometry, being $C^r$-bounded is equivalent to the statement that $|\partial^\alpha f(y)| < C_\alpha$ for every multiindex $\alpha$ with $|\alpha| \le r$ and in every ``nice'' coordinate chart (where the constants $C_\alpha$ are independent of the chart). In order for the last characterization to make sense, we have to use Lemma \ref{lem:transition_functions_uniformly_bounded}, and the equivalence of these two characterization follows from Lemma \ref{lem:bounded_geometry_christoffel_symbols}. See also Remark \ref{rem:derivatives_uniformly_bounded}.

Of course, the definition of $C^r$-boundedness and its equivalent characterization in normal coordinate charts for manifolds of bounded geometry make also sense for sections of vector bundles of bounded geometry (and so especially also for vector fields, differential forms and other tensor fields).

\begin{defn}[Uniform $C^\infty$-spaces]
\label{defn:uniform_frechet_spaces}
Let $E$ be a vector bundle of bounded geometry over $M$. We will denote the \emph{uniform $C^r$-space} of all $C^r$-bounded sections of $E$ by $C_b^r(E)$.

Furthermore, we define the \emph{uniform $C^\infty$-space $C_b^\infty(E)$} as the \Frechet space
\[C_b^\infty(E) := \bigcap_r C_b^r(E).\]
\end{defn}

\subsection*{Sobolev spaces}

Now we get to Sobolev spaces on manifolds of bounded geometry. Much of the following material is from \cite[Appendix A1.1]{shubin} and \cite[Section 2]{roe_index_1}, where an interested reader can find more thorough discussions of this matters.

Let $s \in C^\infty_c(E)$ be a compactly supported, smooth section of some complex vector bundle $E \to M$ with Hermitian metric and connection $\nabla$. For $k \in \IN_0$ we define the global $H^k$-Sobolev norm of $s$ by
\begin{equation}\label{eq:sobolev_norm}
\|s\|_{H^k}^2 := \sum_{i=0}^k \int_M \|\nabla^i s(x)\|^2 dx.
\end{equation}

\begin{defn}[Sobolev spaces $H^k(E)$]
Let $E$ be a complex vector bundle which is equipped with a Hermitian metric and a connection. The \emph{$H^k$-Sobolev space of $E$} is the completion of $C^\infty_c(E)$ in the norm $\|\largecdot\|_{H^k}$ and will be denoted by $H^k(E)$. It is a Hilbert space.
\end{defn}

If $E$ and $M^m$ both have bounded geometry than the Sobolev norm \eqref{eq:sobolev_norm} is equivalent to the local one given by
\begin{equation}\label{eq:sobolev_norm_local}
\|s\|_{H^k}^2 \stackrel{\text{equiv}}= \sum_{i=1}^\infty \|\varphi_i s\|^2_{H^k(B_{2\varepsilon}(x_i))},
\end{equation}
where the balls $B_{2\varepsilon}(x_i)$ and the subordinate partition of unity $\varphi_i$ are as in Lemma \ref{lem:nice_coverings_partitions_of_unity}, we have chosen synchronous framings and $\|\largecdot\|_{H^k(B_{2\varepsilon}(x_i))}$ denotes the usual Sobolev norm on $B_{2\varepsilon}(x_i) \subset \IR^m$.

This equivalence enables us to define the Sobolev norms for all $k \in \IR$. For $k < 0$ the Sobolev space $H^k(E)$ coincides with the dual of $H^{-k}(E)$, regarded as a space of distributional sections of $E$.

Assuming bounded geometry, the usual embedding theorems are true. For their statement recall Definition \ref{defn:uniform_frechet_spaces} of the uniform $C^r$-spaces $C_b^r(E)$.
\begin{thm}[{\cite[Theorem 2.21]{aubin_nonlinear_problems}}]\label{thm:sobolev_embedding}
Let $E$ be a vector bundle of bounded geometry over a manifold $M$ of bounded geometry. Then we have for all $k > r + m/2$ continuous embeddings
\[H^k(E) \subset C^r_b(E).\]
\end{thm}

We will now define and investigate the Sobolev spaces $H^\infty(E)$ and $H^{-\infty}(E)$ of infinite orders. They are crucial since they will allow us to define smoothing operators and hence the important algebra $\IU(E)$ in the next section.

\begin{defn}
We define the space
\[H^\infty(E) := \bigcap_{k \in \IN} H^k(E)\]
and equip it with the obvious \Frechet topology. By standard arguments we see that its topological dual space is given by
\[H^{-\infty}(E) := \bigcup_{k \in \IN} H^{-k}(E).\]
\end{defn}

Let us equip the space $H^{-\infty}(E)$ with the locally convex topology defined as follows: the \Frechet space $H^\infty(E) = \operatorname{\underleftarrow{\lim}} H^k(E)$ is the projective limit of the Banach spaces $H^k(E)$, so using dualization we may put on the space $H^{-\infty}(E)$ the inductive limit topology denoted $\iota(H^{-\infty}(E), H^\infty(E))$:
\[H^{-\infty}_\iota(E) := \operatorname{\underrightarrow{\lim}} H^{-k}(E).\]
It enjoys the following universal property: a linear map $A \colon H^{-\infty}_\iota(E) \to F$ to a locally convex topological vector space $F$ is continuous if and only if $A|_{H^{-k}(E)}\colon H^{-k}(E) \to F$ is continuous for all $k \in \IN$.

Later we will need to know how the bounded\footnote{A subset $B \subset H^{-\infty}_\iota(E)$ is \emph{bounded} if and only if for all open neighbourhoods $U \subset H^{-\infty}_\iota(E)$ of $0$ there exists $\lambda > 0$ with $B \subset \lambda U$.} subsets of $H^{-\infty}_\iota(E)$ look like, which is the content of the following lemma. In its proof we will deduce, solely for the enjoyment of the reader, some nice properties of the spaces $H^\infty(E)$ and $H^{-\infty}_\iota(E)$ that we won't need at all later.

\begin{lem}\label{lem:regular_inductive_limit}
The space $H^{-\infty}_\iota(E) := \operatorname{\underrightarrow{\lim}} H^{-k}(E)$ is a \emph{regular inductive limit}, i.e., for every bounded subset $B \subset H^{-\infty}_\iota(E)$ exists some $k \in \IN$ such that $B$ is already contained in $H^{-k}(E)$ and bounded there.\footnote{Note that the converse does always hold for inductive limits, i.e., if $B \subset H^{-k}(E)$ is bounded, then it is also bounded in $H^{-\infty}_\iota(E)$.}
\end{lem}

\begin{proof}
Since all $H^{-k}(E)$ are \Frechet spaces, we may apply the following corollary of Grothendieck's Factorization Theorem: the inductive limit $H^{-\infty}_\iota(E)$ is regular if and only if it is locally complete (see, e.g., \cite[Lemma 7.3.3(i)]{perezcarreras_bonet}). To avoid introducing more burdensome vocabulary, we won't define the notion of local completeness here since we will show something stronger: $H^{-\infty}_\iota(E)$ is actually complete\footnote{That is to say, every Cauchy net converges. In locally convex spaces, being Cauchy and to converge is meant with respect to each of the semi-norms simultaneously.}.

From \cite[Sections 3.(a \& b)]{bierstedt_bonet} we conclude the following: since each $H^k(E)$ is a Hilbert space, the \Frechet space $H^\infty(E)$ is the projective limit of reflexive Banach spaces and therefore totally reflexive\footnote{That is to say, every quotient of it is reflexive, i.e., the canonical embeddings of the quotients into their strong biduals are isomorphisms of topological vector spaces.}. It follows that $H^\infty(E)$ is distinguished, which can be characterized by $H^{-\infty}_\beta(E) = H^{-\infty}_\iota(E)$, where $\beta(H^{-\infty}(E), H^\infty(E))$ is the strong topology on $H^{-\infty}(E)$. Now without defining the strong topology we just note that strong dual spaces of \Frechet space are always complete.
\end{proof}

\section{Definition and basic properties}\label{sec:defn_quasiloc_smoothing_operators}

Now we are ready to get to the definition of quasilocal smoothing operators and to deduce their basic properties. We will deduce from the Schwartz Kernel Theorem that smoothing operators are representable as integral operators with a uniformly bounded smooth kernel. This will be the crucial thing that will enable us to define analytic indices of such operators, since we will use for this the diagonal of the kernels, i.e., we need especially that the kernels of the operators that we consider are not singular (therefore, we need that they are smoothing operators).

Furthermore, we need that $\IU(E)$ is a local $C^\ast$-algebra, so that we may consider without worries its operator $K$-theory, which will then coincide with the $K$-theory of its completion.

Recall that we said in the introduction to this chapter that we need the control at infinity of the integral kernels so that we may define our analytic index maps $\ind_\tau\colon K_0(\IU(E)) \to \IR$ in Section \ref{sec:analytic_indices_quasiloc_smoothing}. But above we said that in order to define the analytic index of an operator, we just need that it is a smoothing operator so that it has a non-singular, bounded integral kernel. Now the control at infinity of the integral kernels is needed so that the analytic index that we define for smoothing operators is a trace, i.e., vanishes on commutators. Only then it descends to a map on the $K$-theory.

\subsection*{Smoothing operators}

Let's start with the definition and basic properties of smoothing operators on manifolds of bounded geometry.

\begin{defn}[Smoothing operators]
Let $M$ be a manifold of bounded geometry and $E$ and $F$ two vector bundles of bounded geometry over $M$. We will call a continuous linear operator $A \colon H^{-\infty}_\iota(E) \to H^\infty(F)$ a \emph{smoothing operator}.
\end{defn}

Due to the universal property of the inductive limit topology $\iota(H^{-\infty}(E), H^\infty(E))$ on $H^{-\infty}(E)$ an operator $A$ is smoothing if and only if its restrictions $A|_{H^{-k}(E)}$ are continuous for all $k \in \IN$. Furthermore, we know that $A|_{H^{-k}(E)}\colon H^{-k}(E) \to H^\infty(F)$ is continuous if and only if it is bounded, i.e., for all $l \in \IN$ the operator norm $\|A\|_{-k,l}$ of $A|_{H^{-k}(E)}\colon H^{-k}(E) \to H^l(F)$ is finite. So we have arrived at the following characterization of smoothing operators:

\begin{lem}\label{lem:smoothing_operator_iff_bounded}
A linear operator $A \colon H^{-\infty}_\iota(E) \to H^\infty(F)$ is continuous if and only if it is bounded as an operator $H^{-k}(E) \to H^l(F)$ for all $k, l \in \IN$.
\end{lem}

Let us denote by $\IB(H^{-\infty}_\iota(E), H^\infty(E))$ the algebra of all smoothing operators from $E$ to itself. Due to the above lemma we may equip it with the countable family of norms $(\| \largecdot \|_{-k,l})_{k,l \in \IN}$ so that it becomes a \Frechet space\footnote{That is to say, a topological vector space whose topology is Hausdorff and induced by a countable family of semi-norms such that it is complete with respect to this family of semi-norms.}. Given smoothing operators $A$ and $B$, we have for all $k,l \in \IN$ the diagram
\[\xymatrix{H^{-k}(E) \ar[r]^{AB} \ar[d]_B & H^l(E) \\ H^1(E) \ar@{^{(}->}[r] & H^{-1}(E) \ar[u]^A}\]
from which we get $\|AB\|_{-k,l} \le \|B\|_{-k,1} \cdot C \cdot \|A\|_{-1,l}$, where $C$ is the norm of the inclusion $H^1(E) \hookrightarrow H^{-1}(E)$. Hence multiplication in $\IB(H^{-\infty}_\iota(E), H^\infty(E))$ is jointly continuous, i.e., we have proved that it is a \Frechet algebra\footnote{This is an algebra with a topology turning it into a \Frechet space with jointly continuous multiplication; see Definition \ref{defn:Frechet_subalgebras}. Note that we do not require that the semi-norms are submultiplicative.}:

\begin{lem}\label{lem:smoothing_operators_frechet}
The algebra $\IB(H^{-\infty}_\iota(E), H^\infty(E))$ equipped with the family of semi-norms $(\| \largecdot \|_{-k,l})_{k,l \in \IN}$ is a \Frechet algebra.
\end{lem}

Now let us get to the main property of smoothing operators that we will need, namely that they can be represented as integral operators with a uniformly bounded smooth kernel. Let $A \colon H^{-\infty}_\iota(E) \to H^\infty(F)$ be given. Then we get by the Sobolev Embedding Theorem \ref{thm:sobolev_embedding} a continuous operator $A \colon H^{-\infty}_\iota(E) \to C_b^\infty(F)$ and so may conclude by the Schwartz Kernel Theorem for regularizing operators\footnote{Note that the usual wording of the Schwartz Kernel Theorem for regularizing operators requires the domain $H^{-\infty}(E)$ to be equipped with the weak-$^\ast$ topology $\sigma(H^{-\infty}(E), H^{\infty}(F))$ and $A$ to be continuous against it. But one actually only needs the domain to be equipped with the inductive limit topology. To see this, one has to do the proof of the Schwartz Kernel Theorem for regularizing kernels directly, like it is done in, e.g., \cite[Theorem 3.18]{ganglberger}, and not deducing it from the usual kernel theorem for distributional kernels.} that $A$ has a smooth integral kernel $k_A \in C^\infty(F \boxtimes E^\ast)$, which is uniformly bounded as are all its derivatives, because of the bounded geometry of $M$ and the vector bundles $E$ and $F$, i.e., $k_A \in C_b^\infty(F \boxtimes E^\ast)$.
From the proof of the Schwartz Kernel Theorem for regularizing operators we also see that the assignment of the kernel to the operator is continuous against the \Frechet topology on $\IB(H^{-\infty}_\iota(E), H^\infty(F))$. Furthermore, note that due to Lemma \ref{lem:regular_inductive_limit} this \Frechet topology coincides with the topology of bounded convergence\footnote{A basis of neighbourhoods of zero for the topology of bounded convergence is given by the subsets $M(B, U) \subset \IB(H^{-\infty}_\iota(E), H^\infty(F))$ of all operators $T$ with $T(B) \subset U$, where $B$ ranges over all bounded subsets of $H^{-\infty}_\iota(E)$ and $U$ over a basis of neighbourhoods of zero in $H^\infty(F)$.} on $\IB(H^{-\infty}_\iota(E), H^\infty(F))$. Note that we need this equality of the topologies to cite \cite[Proposition 2.9]{roe_index_1} for the next proposition, i.e., so that our wording of it coincides with the wording in the cited proposition.

\begin{prop}[{\cite[Proposition 2.9]{roe_index_1}}]\label{prop:smoothing_op_kernel}
Let $A\colon H^{-\infty}_\iota(E) \to H^\infty(F)$ be a smoothing operator. Then $A$ can be written as an integral operator with kernel $k_A \in C_b^\infty(F \boxtimes E^\ast)$. Furthermore, the map
\[\IB(H^{-\infty}_\iota(E), H^\infty(F)) \to C_b^\infty(F \boxtimes E^\ast)\]
associating a smoothing operator its kernel is continuous.
\end{prop}

\subsection*{Quasilocal smoothing operators}

We will define now when an operator is quasilocal and then we will show that the algebra of all quasilocal smoothing operators is a local $C^\ast$-algebra.

Let $L \subset M$ be a subset. We denote by $\|\largecdot\|_{H^r, L}$ the seminorm on $H^r(E)$ given by
\[\|u\|_{H^r, L} := \inf \{ \|u^\prime\|_{H^r} \ | \ u^\prime \in H^r(E), u^\prime = u \text{ on a neighbourhood of }L\}.\]

\begin{defn}[Quasilocal operators, {\cite[Section 5]{roe_index_1}}]\label{defn:quasiloc_ops}
We will call a continuous operator $A \colon H^r(E) \to H^s(F)$ \emph{quasilocal}, if there is a function $\mu\colon \IR_{> 0} \to \IR_{\ge 0}$ with $\mu(R) \to 0$ for $R \to \infty$ and such that for all $L \subset M$ and all $u \in H^r(E)$ with $\supp u \subset L$ we have
\[\|A u\|_{H^s, M - B_R(L)} \le \mu(R) \cdot \|u\|_{H^r}.\]

Such a function $\mu$ will be called a \emph{dominating function for $A$}.

We will say that an operator $A\colon C_c^\infty(E) \to C^\infty(F)$ is a \emph{quasilocal operator of order $k$}\footnote{Roe calls such operators ``\emph{uniform} operators of order $k$'' in \cite[Definition 5.3]{roe_index_1}. But since ``uniform'' will have another meaning for us (see, e.g., the definition of uniform $K$-homology), we changed the name.} for some $k \in \IZ$, if $A$ has a continuous extension to a quasilocal operator $H^s(E) \to H^{s-k}(F)$ for all $s \in \IZ$.

A smoothing operator $A\colon H^{-\infty}_\iota(E) \to H^\infty(F)$ will be called \emph{quasilocal}, if $A$ is quasilocal as an operator $H^{-k}(E) \to H^l(F)$ for all $k,l \in \IN$ (from which it follows that $A$ is also quasilocal for all $k,l \in \IZ$).
\end{defn}

\begin{rem}
Note that for continuous operators $L^2(E) \to L^2(F)$ the notion of quasilocality coincides with the notion of ``approximate propagation'' from \cite[Remark after the proof of Lemma 3.5]{roe_index_coarse}.
\end{rem}

If we regard a smoothing operator $A$ as an operator $L^2(E) \to L^2(F)$, we get a uniquely defined adjoint $A^\ast\colon L^2(F) \to L^2(E)$. Its integral kernel will be given by
\[k_{A^\ast}(x,y) := k_A(y,x)^\ast \in C_b^\infty(E \boxtimes F^\ast),\]
where $k_A(y,x)^\ast \in F_y^\ast \otimes E_x$ is the dual element of $k_A(y,x) \in F_y \otimes E^\ast_x$.

\begin{defn}[cf. {\cite[Definition 5.3]{roe_index_1}}]\label{defn:quasiloc_smoothing}
We will denote the set of all quasilocal smoothing operators $A\colon H^{-\infty}_\iota(E) \to H^\infty(F)$ with the property that their adjoint operator $A^\ast$ is also a quasilocal smoothing operator $H^{-\infty}_\iota(F) \to H^\infty(E)$ by $\IU(E,F)$.

If $E=F$, we will just write $\IU(E)$.
\end{defn}

Roe defines in \cite[Definition 5.3]{roe_index_1} the algebra $\mathcal{U}_{-\infty}(E)$ instead of $\IU(E)$, i.e., he does not demand that the adjoint operator is also quasilocal smoothing. The reason why we do this, is that we want $\IU(E)$ to be a local $C^\ast$-algebra, i.e., we need that our algebra $\IU(E) \subset \IB(L^2(E))$ is closed under taking adjoints.

Since the composition of quasilocal operators is again a quasilocal operator (proved in \cite[Proposition 5.2]{roe_index_1}), we conclude that $\IU(E)$ is a $^\ast$-algebra. We will now show that it is naturally a \Frechet algebra, since then it will be easier for us to deduce that it is a local $C^\ast$-algebra (see Appendix \ref{chapter:appendix_A}). Recall that for $k,l \in \IN$ we denote by $\|A\|_{-k,l}$ the operator norm of $A\colon H^{-k}(E) \to H^l(E)$. We will need the statement that the limit of quasilocal operators is again quasilocal:

\begin{lem}\label{lem:limit_quasilocal}
If $(A_i)_i$ is a sequence of quasilocal operators $H^{-k}(E) \to H^l(E)$ converging in the norm $\|\largecdot \|_{-k,l}$, then the limit will also be a quasilocal operator.
\end{lem}

\begin{proof}
Let $(A_i)_i$ be a sequence of quasilocal operators $H^{-k}(E) \to H^l(E)$ converging in the norm $\|\largecdot \|_{-k,l}$ to an operator $A$. We have to show that $A$ is also quasilocal.

Let $L \subset M$ and $u \in H^{-k}(E)$ with $\supp u \subset L$. Let $\varepsilon > 0$ be given and we choose an $A_{i_1}$ with $\|A - A_{i_1}\|_{-k,l} < \varepsilon$. Then
\begin{align*}
\|Au\|_{H^l, M-B_R(L)} & \le \underbrace{\| (A-A_{i_1})u\|_{H^l, M-B_R(L)}}_{\le \|(A - A_{i_1})u\|_{H^l}} + \|A_{i_1} u\|_{H^l, M-B_R(L)}\\
& \le \big( \varepsilon + \mu_{i_1}(R) \big) \cdot \|u\|_{H^{-k}}.
\end{align*}
We choose an $A_{i_2}$ with $\|A - A_{i_2}\|_{-k,l} < \varepsilon / 2$ and we have an $R_{i_2}$ such that $\mu_{i_2} < \varepsilon / 2$ for all $R \ge R_{i_2}$. Now we set $\mu_A(R) := \varepsilon + \mu_{i_1}(R)$ for all $R \le R_{i_2}$.

Let $A_{i_3}$ correspond to $\varepsilon / 4$ with corresponding $R_{i_3}$ such that $\mu_{i_3}(R) < \varepsilon / 4$ for all $R \ge R_{i_3}$. Then we go on with the definition of $\mu_R$ and set $\mu_A(R) := \varepsilon / 2 + \mu_{i_2}(R)$ for all $R$ with $R_{i_2} \le R \le R_{i_3}$.

Continuing this way, we get a dominating function $\mu_A$ for $A$ with $\mu(R) \to 0$ as $R \to \infty$, i.e., we have shown that $A$ is quasilocal.
\end{proof}

Equipping $\IU(E)$ with the family of norms $(\|\largecdot\|_{-k,l}, \|\largecdot^\ast\|_{-k,l})_{k,l \in\IN}$, the $^\ast$-algebra $\IU(E)$ will be complete and the adjoint map $A \mapsto A^\ast$ will be continuous on it. Since we also know from Lemma \ref{lem:smoothing_operators_frechet} that multiplication is jointly continuous, we have proved that $\IU(E)$ is a \Frechet algebra:

\begin{lem}\label{lem:mathcalU_Frechet}
$\IU(E)$ is naturally a \Frechet algebra and the adjoint map $A \mapsto A^\ast$ on it is continuous.
\end{lem}

Recall that our goal is to show that $\IU(E)$ is a local $C^\ast$-algebra (Definition \ref{defn:local_Cstar_algebra}). Since we already know that it is a \Frechet algebra with a topology which is finer than the norm topology (since the family of norms used to define the \Frechet topology does include the usual operator norm), we conclude from Lemma \ref{lem:matrix_algebras_holomorphically_closed} that it suffices to show that $\IU(E)$ is closed under holomorphic functional calculus.

To do this we will use Lemma \ref{lem:stable_calculus_power_series}: let $A \in \IU(E)$ and let $f(x) = \sum_{i \ge 1} a_i x^i$ be a power series around $0 \in \IC$ with radius of convergence bigger than $\|A\|_{op} = \|A\|_{0,0}$. We have to show that $f(A) \in \IU(E)$. For $k,l \in \IN$ we have the estimate
\[\|A^{n+2}\|_{-k,l} \le \|A\|_{0,l} \cdot \|A\|_{0,0}^n \cdot \|A\|_{-k,0}\]
and so
\begin{align*}
\|f(A)\|_{-k,l} & \le \sum_{i \ge 1} |a_i| \cdot \|A^i\|_{-k,l}\\
& \le |a_1| \|A\|_{-k,l} + \sum_{i \ge 2} |a_i| \cdot \|A\|_{0,l} \cdot \|A\|_{0,0}^{i-2} \cdot \|A\|_{-k,0}\\
& < \infty.
\end{align*}
Since this holds for every $k,l \in \IN$, we conclude that $f(A)$ is a smoothing operator.

The above estimate also shows that the power series $f(A)$ converges in all norms $\|\largecdot\|_{-k,l}$ (though it is a priori only assumed to converge with respect to the operator norm $\|\largecdot\|_{0,0}$). Since every partial sum is a quasilocal operator, we get with the above Lemma \ref{lem:limit_quasilocal} that the limit $f(A)$ is also quasilocal.

Applying the same arguments also to the adjoint $A^\ast$ of $A$, we finally conclude that $f(A) \in \IU(E)$. This completes the proof of the lemma that $\IU(E)$ is a local $C^\ast$-algebra:

\begin{lem}\label{U*_-infty(E)_local_C*_algebra}
$\IU(E)$ is a local $C^\ast$-algebra.
\end{lem}

Another possible control of the integral kernel of a smoothing operator at infinity would be to demand that the kernel vanishes outside of a uniform neighbourhood of the diagonal. This leads to smoothing operators with finite propagation, as defined in a moment. Note that operators of finite propagation will play a crucial role in our discussion of uniform $K$-homology.

\begin{defn}[Finite propagation]\label{defn:finite_prop_speed}
Let $A\colon L^2(E) \to L^2(F)$ be an operator. We will say that $A$ has \emph{finite propagation}, if there exists an $R \in \IR_{\ge 0}$ such that $\supp As \subset \overline{B_R(\supp s)}$ for all sections $s \in L^2(E)$.

In this case we will call the smallest possible value of $R$ the \emph{propagation of the operator $A$}.
\end{defn}

Note that an operator with propagation $R$ is of course quasilocal, since we can find a dominating function $\mu$ with $\mu(r) = 0$ for all $r > R$.

If $A$ is a smoothing operator, then it has by Proposition \ref{prop:smoothing_op_kernel} a smooth integral kernel $k_A$. It is clear that $A$ has propagation at most $R$ if and only if the kernel satisfies $k_A(x,y) = 0$ for all $x,y \in M$ with $d(x,y) > R$. On the other hand, if we have some section $k \in C_b^\infty(F \boxtimes E^\ast)$ with $k(x,y) = 0$ for all $x,y \in M$ with $d(x,y) > R$ for some $R > 0$, then the integral operator $A_k$ defined by it is a smoothing operator with propagation at most $R$. Since in this case the same also holds for the adjoint operator $A_k^\ast = A_{k^\ast}$, we get $A_k \in \IU(E, F)$.

Given an operator $A \in \IU(E, F)$, it is tempting to truncate its kernel $k_A(x,y)$ outside an $R$-neighbourhood of the diagonal in order to approximate $A$ via operators with finite propagation. But if one goes on to prove this, one quickly comes to the point where one has to show that the kernel $k(x,y)$ is integrable (i.e., is in $L^1(M)$) with respect to either one of the variables. But this seems not necessarily to hold for quailocal smoothing operators. In fact, it is an open question to determine the closure of the smoothing operators with finite propagation, and it is conjectured that this closure coincides with the quasilocal smoothing operators if we impose certain restrictions on the manifold like having Property A or having finite asymptotic dimension.

\subsection*{Smooth uniform Roe algebra}

We have discussed above that the closure of the smoothing operators with finite propagation are not necessarily all quasilocal smoothing operators. But since exactly this closure has connections to the uniform Roe algebra, we will give it a separate name now:

\begin{defn}[Smooth uniform Roe algebra]\label{defn:smooth_uniform_roe_algebra}
Let $M$ be a manifold of bounded geometry and let $E$, $F$ be two vector bundles of bounded geometry over $M$. The closure of all smoothing operators $A \colon H_\iota^{-\infty}(E) \to H^\infty(F)$ with finite propagation\footnote{Note that if $A$ is smoothing and has finite propagation, then $A^\ast$ is also smoothing and also with finite propagation. This follows from representing $A$ as an integral operator with smooth, uniformly bounded kernel.} under the family of norms $(\|\largecdot\|_{-k,l}, \|\largecdot^\ast\|_{-k,l})_{k,l \in\IN}$ will be denoted $\C(E,F)$.

If $E = F$, we will write $\C(E)$ for $\C(E,E)$ and will call it the \emph{smooth uniform Roe algebra}.
\end{defn}

We certainly have $\C(E) \subset \IU(E)$ and $\C(E)$ is by definition a \Frechet algebra on which the adjoint map $A \mapsto A^\ast$ is continuous. Basically the same arguments that we used to prove that $\IU(E)$ is a local $C^\ast$-algebra can also be used to prove that $\C(E)$ is a local $C^\ast$-algebra.

So the $K$-theory of $\C(E)$ coincides with its norm closure, and we will see later that this norm closure is exactly the uniform Roe algebra $C_u^\ast(E)$. Therefore $\C(E)$ may be regarded as a ``smooth version'' of it, hence its name.

\section{Coarsely bounded geometry}\label{sec:coarsely_bounded_geometry}

In this section we introduce the highly important notion of coarsely bounded geometry for metric spaces. First note that ``coarsely bounded geometry''  is the same notion as ``bounded geometry'' that other authors use. Our reason for renaming it is that there will be another notion of bounded geometry, namely ``locally bounded geometry'' (Definition \ref{defn:locally_bounded_geometry}, originally introduced in \cite[Definition 3.1]{spakula_universal_rep} by \v{S}pakula) for metric spaces, that will be equally important for us later. As the names suggest, this notions of bounded geometry will control certain aspects of the metric space, either at the coarse level, i.e., at infinity, or at the local level. Note that manifolds of bounded geometry, if regarded as metric spaces, will be both of coarsely and locally bounded geometry, i.e., no confusion can arise by not distinguishing between ``coarsely'' and ``locally'' for manifolds. In fact, one could take the view that coarsely bounded geometry is the abstraction to metric spaces of the implications of bounded geometry for manifolds on the large scale, and that locally bounded geometry is the abstraction to metric spaces of the local implications of it.

The importance of coarsely bounded geometry stems from the fact that this seems to be the right class of spaces where we want to do coarse index theory. As an first example, the property of amenability may only be possessed by metric spaces having coarsely bounded geometry, and amenability is a crucial property for the theory that we develop in this first part of the thesis, i.e., only it allows us to define our index maps into $\IR$.

Furthermore, the coarse Baum--Connes conjecture is stated only for metric spaces having coarsely bounded geometry, since there are examples of spaces not having bounded geometry and not satisfying the coarse Baum--Connes conjecture (\cite{dranishnikov_ferry_weinberger} and \cite[Section 8]{yu_finite_asymptotic_dimension}). And regarding some positive results about the coarse Baum--Connes conjecture: Skandalis, Tu and Yu showed in \cite[Theorem 5.3]{skandalis_tu_yu} that if a metric space $X$ has bounded geometry, then $X$ having property A as defined by Yu in \cite[Definition 2.1]{yu_embedding_Hilbert_space} is equivalent to the uniform Roe algebra $C_u^\ast(X)$ being a nuclear $C^\ast$-algebra. Furthermore, Sako showed in his preprint \cite{sako} that we additionally have the equivalent characterizations of $X$ having property A via the exactness and also via the local reflexivity of $C_u^\ast(X)$. Now the importance of property A lies in the fact that if a metric space $X$ does have it, then $X$ admits a uniform embedding into a Hilbert space (\cite[Theorem 2.7]{yu_embedding_Hilbert_space}). And if a metric space $X$ with bounded geometry admits a uniform embedding into a Hilbert space, than Yu showed in \cite[Theorem 1.1]{yu_embedding_Hilbert_space} that $X$ satisfies the coarse Baum--Connes conjecture in the form of \cite[Conjecture 3.4]{yu_embedding_Hilbert_space}. So we see that the notion of coarsely bounded geometry is immanent to the coarse Baum--Connes conjecture and therefore to coarse index theory.

\begin{defn}[Coarsely bounded geometry]
\label{defn:coarsely_bounded_geometry}
Let $X$ be a metric space. We call a subset $\Gamma \subset X$ a \emph{quasi-lattice} if
\begin{itemize}
\item there is a $c > 0$ such that $B_c(\Gamma) = X$ (i.e., $\Gamma$ is \emph{coarsely dense}) and
\item for all $r > 0$ there is a $K_r > 0$ such that $\card(\Gamma \cap B_r(y)) \le K_r$ for all $y \in X$.
\end{itemize}
A metric space is said to have \emph{coarsely bounded geometry} if it admits a quasi-lattice.
\end{defn}

Note that if we have a quasi-lattice $\Gamma \subset X$, then there also exists a uniformly discrete quasi-lattice $\Gamma^\prime \subset X$. The proof of this is an easy application of the Lemma of Zorn: given an arbitrary $\delta > 0$ we look at the family $\mathcal{A}$ of all subsets $A \subset \Gamma$ with $d(x,y) > \delta$ for all $x,y \in A$. These subsets are partially ordered under inclusion of sets and every totally ordered chain $A_1 \subset A_2 \subset \ldots \subset \Gamma$ has an upper bound given by the union $\bigcup_i A_i \in \mathcal{A}$. So the Lemma of Zorn provides us with a maximal element $\Gamma^\prime \in \mathcal{A}$. That $\Gamma^\prime$ is a quasi-lattice follows from its maximality.

\begin{examples}\label{ex:coarsely_bounded_geometry}
Every manifold $M$ of bounded geometry is a metric space of coarsely bounded geometry: any maximal set $\Gamma \subset M$ of points which are at least a fixed distance apart (i.e., there is an $\varepsilon > 0$ such that $d(x, y) \ge \varepsilon$ for all $x \not= y \in \Gamma$) will do the job. We can get such a maximal set by invoking Zorn's lemma (this is actually the same proof as the one of Lemma \ref{lem:nice_coverings_partitions_of_unity}). Note that a manifold of bounded geometry will also have locally bounded geometry, so no confusion can arise by not distinguishing between ``coarsely'' and ``locally'' bounded geometry in the terminology for manifolds.

If $(X,d)$ is an arbitrary metric space that is bounded, i.e., $d(x,x^\prime) < D$ for all $x, x^\prime \in X$ and some $D$, then \emph{any} finite subset of $X$ will constitute a quasi-lattice.

Let $K$ be a simplicial complex of bounded geometry\footnote{That is, the number of simplices in the link of each vertex is uniformly bounded.}. If we equip $K$ with the metric derived from barycentric coordinates, then the set of all vertices becomes a quasi-lattice in $K$.
\end{examples}

Let us now briefly discuss the notion of ``coarse equivalence'' for metric spaces, which we already mentioned sometimes.

\begin{defn}
Let $f\colon X \to Y$ be a (not necessarily continuous) map. We call $f$ a \emph{coarse} map, if
\begin{itemize}
\item for all $R > 0$ there is an $S > 0$ such that we have
\[d(x_1, x_2) < R \Rightarrow d(f(x_1), f(x_2)) < S\]
for all $x_1, x_2 \in X$, and
\item the preimage of every bounded subset is bounded.
\end{itemize}

If we replace the second bullet point by the following one:
\begin{itemize}
\item for all $R > 0$ there is an $S > 0$ such that we have
\[d(f(x_1), f(x_2)) < R \Rightarrow d(x_1, x_2) < S\]
for all $x_1, x_2 \in X$,
\end{itemize}
then we will call $f$ a \emph{uniformly coarse} map\footnote{Roe calls such maps ``rough''.}.

Two (not necessarily continuous) maps $f, g \colon X \to Y$ are called \emph{close}, if there is an $R > 0$ such that $d(f(x), g(x)) < R$ for all $x \in X$.

Finally, two metric spaces $X$ and $Y$ are called \emph{coarsely equivalent}, if there are coarse maps $f \colon X \to Y$ and $g \colon Y \to X$ such that their composites are close to the corresponding identity maps. If $f$ and $g$ are uniformly coarse, then $X$ and $Y$ are called \emph{uniformly coarsely equivalent}.
\end{defn}

It is clear that if $\Gamma \subset X$ is a quasi-lattice, then $\Gamma$ and $X$ are uniformly coarsely equivalent.

From \cite[Exercise 1.12]{roe_lectures_coarse_geometry} we know the following: if $f \colon X \to Y$ is a uniformly coarse map and if $f(X) \subset Y$ is coarsely dense (i.e., $B_c(f(X)) = Y$ for some $c > 0$), then there is a uniformly coarse map $g \colon Y \to X$ such that $f \circ g$ and $g \circ f$ are close to the corresponding identities, i.e., $X$ and $Y$ are uniformly coarsely equivalent. Moreover, every coarse equivalence arises in this way, i.e., coarsely equivalent spaces are also uniformly coarsely equivalent.

The importance of coarse equivalences stems from the following fact: recall that we said that the $K$-groups of the algebra $\IU(E)$ are a natural receptacle for the analytic index classes of uniform elliptic operators. We know that $K_\ast(\IU(E))$ is isomorphic to the $K$-theory $K_\ast(C_u^\ast(M))$ of the uniform Roe algebra of the manifold $M$, and this is again isomorphic to $K_\ast(C_u^\ast(Y))$ for a uniformly discrete quasi-lattice $Y \subset M$ (Proposition \ref{prop:IU(E)_dense_Cu*(E)} together with Lemma \ref{lem:iso_discrete_versions_uniform_roe}). But $K_\ast(C_u^\ast(Y))$ depends only on the coarse equivalence class of $Y$: in \cite[Theorem 4]{brodzki_niblo_wright} Brodzki, Niblo and Wright showed that if two uniformly discrete metric spaces of coarsely bounded geometry are coarsely equivalent, then their uniform Roe algebras are stably isomorphic\footnote{They stated their theorem by saying that the uniform Roe algebras will be Morita equivalent, but the proof actually shows stable isomorphy of the two algebras. Note that the uniform Roe algebra is usually not separable, and for such algebras being Morita equivalent is in general not equivalent to being stably isomorphic (see \cite{brown_green_rieffel}).}, i.e., their $K$-groups will coincide. So we see that the analytic index classes of uniform elliptic operators ``live'' in $K$-groups that depend only on the coarse equivalence class of the manifold.

\section{Amenability}\label{sec:amenability}

Amenability is the property that restricts the generality of Roe's Index Theorem at most, and therefore also of our index theorem for pseudodifferential operators. The class of amenable manifolds, e.g., does not contain hyperbolic space, and so the index theorem that we develop here does not hold for it. This is not a problem with the techniques used to prove the theorem, but it is an inherent problem: Block and Weinberger showed in \cite{block_weinberger_1} that Roe's Index Theorem can not hold for non-amenable manifolds. This is mostly due to the observation that if $M$ is not amenable, then its top-dimensional bounded de Rham cohomology $H^n_{b, \mathrm{dR}}(M)$ vanishes (we will discuss this in the paragraph after Proposition \ref{prop:equiv_amenability_metric_spaces}).
But if there are no cohomological index classes, there is no index theorem in the classical sense\footnote{This means that there is no index theorem if one wants such a theorem to include expressions in cohomology classes. Of course there are also other index theorem that have nothing to do with cohomology, and these theorems are also applicable to non-amenable manifolds.}.

Let us explain what amenability means and why it is needed to define indices in the sense how we will do it. If we have a Riemannian manifold $M$ of infinite volume and a bounded function $f$ on it, we may do the following averaging procedure: we choose an exhaustion $(M_i)_i$ of $M$ via compact subsets and then we consider the sequence $\frac{1}{\vol M_i} \int_{M_i} f dM$. Since $f$ is bounded, this sequence is a bounded sequence, and we may evaluate a functional $\tau \in (\ell^\infty)^\ast$ on it. Up to this point no amenability is needed. Now we consider bounded top-dimensional forms $\alpha \in \Omega^n(M)$ (bounded means that $\| \alpha \| := \sup_{x \in M} \{\|\alpha(x)\| + \|d \alpha(x)\|\}$ is bounded). Again, we may average these forms over an exhaustion of $M$ by compact subsets (now they should be embedded submanifolds with boundary and of codimension $0$), and no amenability is needed for that. But if we want that the functional $\theta$ that we defined in this way on the bounded forms $\Omega^n_b(M)$ descends to cohomology classes, we need that it vanishes on bounded exact forms, i.e., that $\theta(d \beta) = 0$. And here amenability enters the game. The line of argumentation is roughly the following (see \cite[Proposition 6.5]{roe_index_1} for the concrete proof): using Stokes Theorem we get $\frac{1}{\vol M_i} \int_{M_i} d \beta = \frac{1}{\vol M_i} \int_{\partial M_i} \beta$, and that must become $0$. Now $\beta$ is bounded and therefore $\int_{\partial M_i} \beta \le \vol(\partial M_i) \cdot \|\beta\|$. Hence we need that the ratio $\vol(\partial M_i) / \vol(M_i)$ goes to $0$ in order to conclude $\theta(d \beta) = 0$ for all bounded forms $\beta$. So we have to choose a compact exhaustion $(M_i)_i$ of $M$ with concretely this property in order that $\theta$ descends to $H^n_{b, \mathrm{dR}}(M)$. Raising exactly this to a definition, gives us ``amenability''.

But let us first start with the notion of amenability for metric spaces and after that we will relate it to the above discussion for manifolds.

\begin{defn}[Amenable metric spaces]\label{defn:amenability_metric_spaces}
Let $(X, d)$ be a metric space of coarsely bounded geometry. For a subset $U \subset X$ we define its \emph{$r$-boundary $\partial_r U$} as
\[\partial_r U := \{x \in X \ | \ d(x, U) < r \text{ and }d(x, X - U) < r\}.\]

We will call a quasi-lattice $\Gamma \subset X$ \emph{amenable} if for any $r, \delta > 0$ there is a finite subset $U \subset \Gamma$ with
\[\frac{\card \partial_r U}{\card U} < \delta,\]
where $\partial_r U$ is computed in $\Gamma$. Note that $\card \partial_r U$ is finite, because $\Gamma$ is a quasi-lattice.
We will call the space $X$ \emph{amenable} if it admits an amenable quasi-lattice.
\end{defn}

Let us give elementary examples. Using the characterization of amenability via bounded de Rham cohomology, that we will discuss later, we may easily find further examples and non-examples of amenable spaces on our own. Furthermore, we will also discuss the relation between amenability of the universal cover of a space and the amenability of the fundamental group of the covered space, which will also give us a lot more examples.

\begin{examples}
If the quasi-lattice $\Gamma$ is finite, then the space is clearly amenable (in this case we may take $U := \Gamma$ for all $r, \delta > 0$ since then $\partial_r U = \emptyset$). So especially every bounded metric space is amenable.

Euclidean space $\IR^n$ is amenable, whereas hyperbolic space $\mathbb{H}^n$ for $n \ge 2$ is not amenable.
\end{examples}

It is clear that given some amenable quasi-lattice $\Gamma \subset X$ and a sequence $(r_i)_i$ of positive numbers, we can construct a sequence $(U_i)_i$ of finite subsets of $\Gamma$ with $\frac{\card \partial_{r_i} U_i}{\card U_i} \to 0$. Since $(r_i)_i$ is allowed to go to infinity, we can therefore get a sequence $(U_i)_i$ with the property that for all $r > 0$ we have $\frac{\card \partial_{r} U_i}{\card U_i} \to 0$. Such a sequence will be called a \emph{F{\o}lner sequence}. If $(U_i)_i$ is also an exhaustion, i.e., it satisfies $U_1 \subset U_2 \subset \ldots$ and $\bigcup U_i = \Gamma$, then it is called a \emph{F{\o}lner exhaustion} of $\Gamma$. Now will now show that we can always get such an exhaustion, if we have at least one sequence (at least if $\Gamma$ is countable):

\begin{lem}\label{lem:amenable_admits_exhaustion}
Let $\Gamma \subset X$ be an amenable, countable quasi-lattice. Then $\Gamma$ admits a F{\o}lner exhaustion.
\end{lem}

\begin{proof}
We will first show that for any $r > 0$ there is an exhaustion of $\Gamma$ by finite subsets $U_i$, such that
\[\frac{\card \partial_r U_i}{\card U_i} \stackrel{i \to \infty}\longrightarrow 0.\]

Let $c > 0$ be such that $B_c(\Gamma) = X$ and let $r > 2c$ (note that this is no restriction since it suffices to prove the statement for arbitrarily large $r$). Let us furthermore enumerate the set $\Gamma$, i.e., $\Gamma = \{\gamma_1, \gamma_2, \ldots\}$. Now we choose a subset $U \subset \Gamma$ with $\card U > 2 K_r$, where $K_r$ is such that $\card(\Gamma \cap B_{r}(y)) \le K_r$ for all $y \in M$ (note that this is possible since from $k > 2c$ we can conclude $\card \partial_r U \ge 1$), and $\tfrac{\card \partial_r U}{\card U} < 1/2$ and set $U_1 := U \cup \{\gamma_1\}$. Then we have $\card \partial_r U_1 \le \card \partial_r U + K_r$ and therefore
\[\frac{\card \partial_r U_1}{\card U_1} \le \frac{\card \partial_r U + K_r}{\card U} < \frac{1}{2} + \frac{1}{2}.\]

Now suppose that we have already constructed $U_1 \subset \ldots \subset U_k$ with $\tfrac{\card \partial_r U_i}{\card U_i} < \tfrac{1}{i}$ for all $1 \le i \le k$. We choose a new subset $U \subset \Gamma$ satisfying $\card U > 2 K_r (\card U_k + 1)$ and $\tfrac{\card \partial_r U}{\card U} < \tfrac{1}{2(k+1)}$ and define $U_{k+1} := U_k \cup U \cup \{\gamma_{k+1}\}$. Then we get the upper bound $\card \partial_r U_{k+1} \le \card \partial_r U + K_r (\card U_k + 1)$ and so
\[\frac{\card \partial_r U_{k+1}}{\card U_{k+1}} \le \frac{\card \partial_r U + K_r (\card U_k + 1)}{\card U} < \frac{1}{2(k+1)} + \frac{1}{2(k+1)}.\]

Since $\gamma_i \in U_i$, the subsets $U_1 \subset U_2 \subset \ldots$ are an exhaustion of $\Gamma$.

Note that it is also possible to increase $r$ on every step of the construction (and we can of course also decrease it), i.e., we can get an exhaustion satisfying
\[\frac{\card \partial_{r_k} U_{k}}{\card U_{k}} < \frac{1}{k},\]
where $r_k$ is any arbitrary sequence (especially it is allowed to go to infinity).
\end{proof}

As promised, we now come to the notion of amenability for Riemannian manifolds.

\begin{defn}[{\cite[Definition 6.1]{roe_index_1}}]
Let $M$ be a manifold of bounded geometry. A sequence of compact subsets $(M_i)_i$ of $M$ will be called a \emph{F{\o}lner sequence}\footnote{In \cite[Definition 6.1]{roe_index_1} such sequences were called \emph{regular}.} if for each $r > 0$ we have
\[\frac{\vol B_r(\partial M_i)}{\vol M_i} \stackrel{i \to \infty}\longrightarrow 0.\]

A F{\o}lner sequence $(M_i)_i$ will be called a \emph{F{\o}lner exhaustion}, if $(M_i)_i$ is an exhaustion, i.e., $M_1 \subset M_2 \subset \ldots$ and $\bigcup_i M_i = M$.
\end{defn}

\begin{lem}\label{lem:manifold_equiv_amenable_reg_exhaustion}
Let $M$ be a manifold of bounded geometry.

Then $M$ is amenable as a metric space if and only if it admits a F{\o}lner sequence. If this is the case, then $M$ also admits a F{\o}lner exhaustion.
\end{lem}

\begin{proof}
\textbf{($\boldsymbol{\Longrightarrow}$)} Let $M$ be amenable with amenable quasi-lattice $\Gamma \subset M$, let $c > 0$ be such that $B_c(\Gamma) = M$ and let $r > 2c$ be given.

Since $M$ has bounded geometry, there is a $C_1 > 0$ such that $\vol B_{2r}(x) \ge C_1$ and a $C_2 > 0$ such that $\vol B_{r}(\partial B_{2r}(x)) \le C_2$ for all $x \in M$.

Let $\delta > 0$ be given and let $U \subset \Gamma$ be the finite subset of $\Gamma$ satisfying $\frac{\card \partial_{5r} U}{\card U} < \delta$. We set $M_\delta := \bigcup_{x \in U} B_{2r}(x)$ which is a compact subset due to the finiteness of $U$. Then we get
\[\vol M_\delta \ge \frac{C_1}{K_{4r}} \cdot \card U,\]
where $K_{4r} > 0$ is such that $\card(\Gamma \cap B_{4r}(y)) \le K_{4r}$ for all $y \in M$. Furthermore
\[\vol B_{r}(\partial M_\delta) \le C_2 \cdot \card \partial_{5r} U,\]
since if $y \notin \partial_{5r} U$ then $B_{2r}(y) \cap B_r(\partial M_\delta) = \emptyset$, i.e., only points $y \in \partial_{5r} U$ can possibly contribute to $B_{r}(\partial M_\delta)$; see Figure \ref{fig:amenability_equiv}:

\begin{figure}[htbp]
\centering
\includegraphics[scale=0.45]{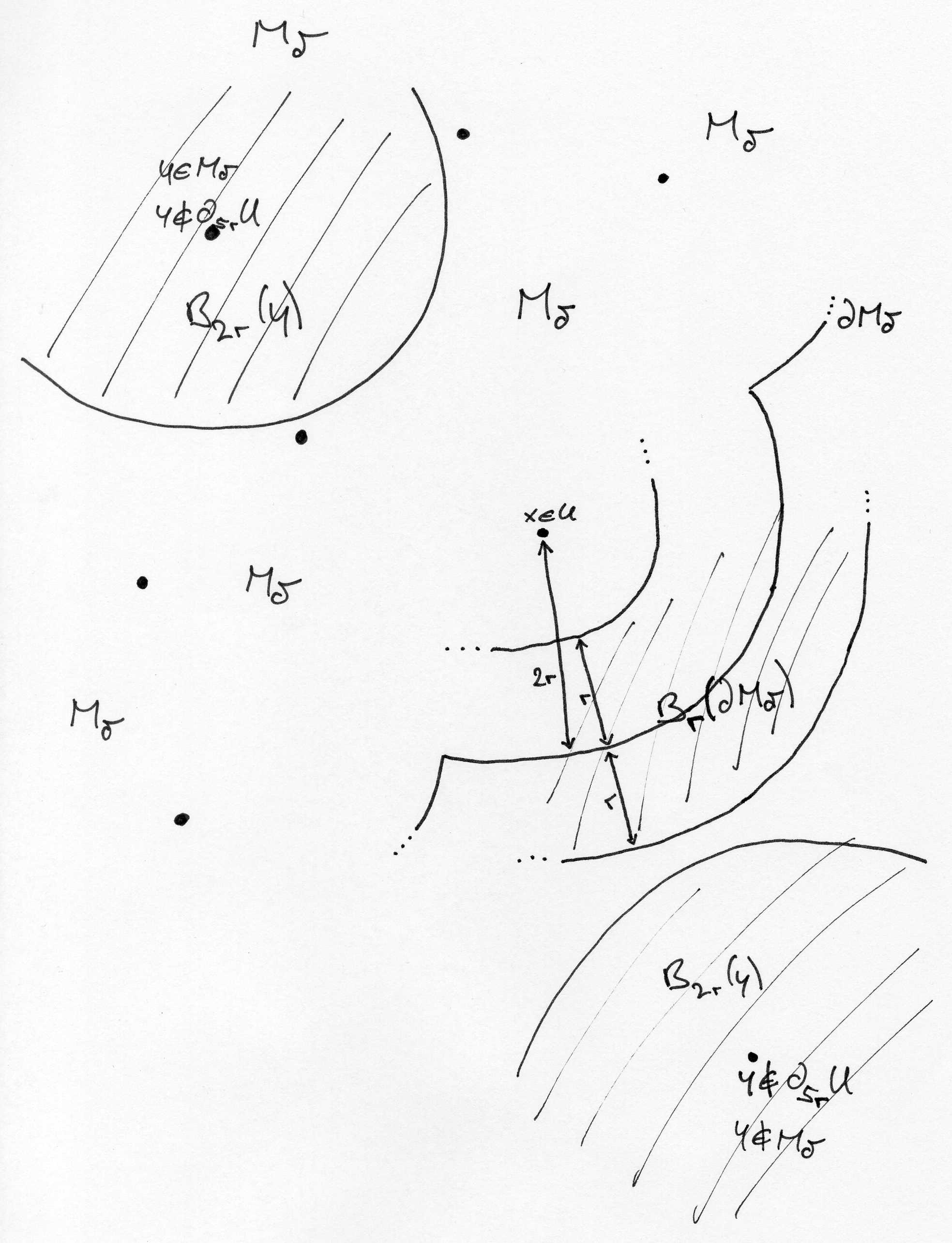}
\caption{Only points $y \in \partial_{5r} U$ can contribute to $B_r(\partial M_\delta)$.}
\label{fig:amenability_equiv}
\end{figure}

Putting this two estimates together we conclude
\[\frac{\vol B_r(\partial M_\delta)}{\vol M_\delta} \le \frac{C_2 \cdot \card \partial_{5r} U}{\frac{C_1}{K_{4r}} \cdot \card U} < \frac{C_2 K_{4r}}{C_1} \cdot \delta.\]

Now we choose via Lemma \ref{lem:amenable_admits_exhaustion} an exhaustion $U_i$ of $\Gamma$ satisfying $\tfrac{\card \partial_{3r_i} U_i}{\card U_i} < \tfrac{C_1}{C_2 K_{4r_i}}\cdot \tfrac{1}{i}$ for a sequence $(r_i)_i$ converging to infinity (note that $C_2$ does depend on $r_i$), and define $M_i := \bigcup_{x \in U_i} B_{2r_i}(x)$; then $\tfrac{\vol B_{r_i}(\partial M_i)}{\vol M_i} < \tfrac{1}{i}$. We have shown that if $M$ is amenable as a metric space, then it admits a \Folner exhaustion.

\textbf{($\boldsymbol{\Longleftarrow}$)} Let $M$ admit a F{\o}lner sequence, and we have to show that $M$ is amenable as a metric space. So let $\Gamma$ be a coarsely dense quasi-lattice in $M$ with $B_c(\Gamma) = M$, let $r, \delta > 0$ be given and let $\varepsilon > 0$ be arbitrary. As assumed, there is a compact $M_i \subset M$ with
\[\frac{\vol B_{r+c+\varepsilon}(\partial M_i)}{\vol M_i} < \delta \cdot \frac{C_1}{K_\varepsilon}\cdot \frac{1}{C_2},\]
where $C_1$ is a lower bound on $\vol B_\varepsilon(x)$ and $C_2$ an upper bound on $\vol B_{2c}(x)$ for all points $x \in M$.

We define $U := \Gamma \cap B_c(M_i)$ which is finite since $M_i$ is compact and $\Gamma$ a quasi-lattice. Since $\partial_r U \subset B_{r+c}(\partial M_i)$, we get the estimate
\[\card \partial_r U \cdot C_1 \le K_\varepsilon \cdot \vol B_{r+c+\varepsilon}(\partial M_i),\]
and since every $x \in M_i$ is contained in a ball $B_{2c}(\gamma)$ for a $\gamma \in U$, we get the estimate $\vol M_i \le \card U \cdot C_2$. So putting it together we have
\[\frac{\card \partial_r U}{\card U} \le \frac{K_\varepsilon \cdot \vol B_{r+c+\varepsilon}(\partial M_i)}{C_1} \cdot \frac{C_2}{\vol M_i} < \delta,\]
which completes the proof.
\end{proof}

At last, let us state that amenability is preserved under coarse equivalences of spaces. A fact that is straightforward to prove.

\begin{lem}\label{lem:amenability_coarse_equiv}
Let $X$ and $Y$ be metric spaces of coarsely bounded geometry and let them be coarsely equivalent.

Then $X$ is amenable if and only if $Y$ is.
\end{lem}

\subsection*{Characterizations of amenability}

We have said that we need amenability so that we have \Folner sequences and therefore suitable functionals on $H^n_{b, \mathrm{dR}}(M)$ to detect cohomological index classes. Now we may of course ask ourselves if the restriction to \Folner sequences is too much, i.e., it might a priori be the case that we have a non-trivial index class in $H^n_{b, \mathrm{dR}}(M)$ and $M$ is not amenable, i.e., we would not be able to detect that non-trivial class via a \Folner sequence (since that would make $M$ amenable). But we could maybe detect by other means, i.e., by another index theorem. But we already said that there is no cohomological index theorem for non-amenable manifolds, i.e., the case that we have a non-trivial index class in $H^n_{b, \mathrm{dR}}(M)$ for $M$ non-amenable should not occur. And in fact, we will show now that if $M$ is not amenable, then $H^n_{b, \mathrm{dR}}(M)$ vanishes, and that this is even an equivalent characterization. Furthermore, we will discuss some other equivalent characterizations of amenability further showing the importance of this notion.\footnote{True to the motto that if a property has a lot of different equivalent characterizations, then this property must be important.} Note that we won't explain all the occuring notions in the following discussion since that would lead us too far astray.

We will start with characterizations via different homology and cohomology theories. An important one was given by Block and Weinberger in \cite[Theorem 3.1]{block_weinberger_1}, where they showed that the amenability of $X$ is equivalent to $H^{\mathrm{uf}}_0(X) \not= 0$, where $H^{\mathrm{uf}}_\ast(X)$ denotes the uniformly finite homology of $X$ that they introduced there, and it is also equivalent to $\bar{H}^{\mathrm{uf}}_0(X) \not= 0$, where $\bar{H}^{\mathrm{uf}}_\ast(X)$ is \emph{reduced} uniformly finite homology (i.e., we divide by the \emph{closure} of the boundary operator in the definition of the homology groups). The equivalence of the latter to the other ones is mentioned in, e.g., \cite{block_weinberger_large_scale}.

Another characterization was given by \v{S}pakula using uniform $K$-homology: he showed in \cite[Theorem 11.2]{spakula_uniform_k_homology} that if we have a connected graph $X$ with vertex set $Y$, then $X$ is amenable if and only if the fundamental class $[Y] \in K_0^u(X)$ does not vanish. The technical restriction of $X$ being a connected graph is not crucial: every metric space of coarsely bounded geometry is coarsely equivalent to such a graph with vertex set $Y$ being a quasi-lattice of $X$ and we know that amenability is a property of the coarse equivalence class of a metric space (Lemma \ref{lem:amenability_coarse_equiv}).

There is also a characterization using the $K$-theory of uniform Roe algebras. In \cite{elek} Elek proved that for a finitely generated group $G$ amenability\footnote{We will discuss amenability of groups further down.} is equivalent to the non-vanishing of its fundamental class in the algebraic $K$-theory group $K_0^{\mathrm{alg}}(T(G))$ of its translation algebra $T(G)$. Note that the translation algebra of $G$ is the same as the uniform Roe algebra $C_u^\ast(G)$. Moreover, Elek's proof generalizes to vertex sets of connected graphs in the following sense: if $Y$ is a vertex set of a connected graph $X$, then $X$ is amenable if and only if $[1] \not= [0] \in K_0(C_u^\ast (Y))$, where $[1] \in K_0(C_u^\ast (Y))$ is the corresponding fundamental class. This characterization is connected to the one via uniform $K$-homology by the uniform coarse assembly map $\mu_u\colon K_\ast^u(X) \to K_\ast(C_u^\ast (Y))$ which was constructed in \cite[Section 9]{spakula_uniform_k_homology} by \v{S}pakula: indeed, we have $\mu_u([Y]) = [1]$.

Let us summarize the above discussion:

\begin{prop}\label{prop:equiv_amenability_metric_spaces}
Let $X$ be a connected graph and $Y$ its set of vertices such that it is a quasi-lattice for $X$.

Then the following are equivalent:
\begin{itemize}
\item $X$ is amenable,
\item $H_0^{\mathrm{uf}}(X) \not= 0$,
\item $\bar{H}_0^{\mathrm{uf}}(X) \not= 0$,
\item $[Y] \not= [0] \in K_0^u(X)$ and
\item $[1] \not= [0] \in K_0(C_u^\ast (Y))$.
\end{itemize}
\end{prop}

Let us now discuss the for us most important characterization of amenability via the bounded de Rham cohomology, which we are going to define first:

\begin{defn}[Bounded de Rham cohomology]\label{defn:bounded_de_rham}
Let $\Omega_b^p(M)$ denote the vector space of $p$-forms on $M$, which are bounded in the norm
\[\| \alpha \| := \sup_{x \in M} \{\|\alpha(x)\| + \|d \alpha(x)\|\}.\]
We define the \emph{bounded de Rham cohomology of $M$} as
\[ H_{b, \mathrm{dR}}^p(M) := \kernel d_p / \image d_{p-1}.\]
\end{defn}

\begin{rem}
Since in general the subspace $\image d_{p-1} \subset \kernel d_p$ is not closed, the induced norm on the bounded de Rham cohomology vector space is in general just a seminorm, i.e., in general there are elements with induced norm $0$ in $H_{b, \mathrm{dR}}^\ast(M)$. The bounded de Rham cohomology as we have defined it is sometimes called \emph{unreduced}. The \emph{reduced bounded de Rham cohomology} is then defined as
\[\bar{H}_{b, \mathrm{dR}}^p(M) := \kernel d_p / \closure(\image d_{p-1}).\]
\end{rem}

Now let $M$ be a connected and oriented $m$-dimensional manifold of bounded geometry. Then there is a duality isomorphism $H_{b, \mathrm{dR}}^m(M) \cong H_0^{\mathrm{uf}}(M; \IR)$, where the latter denotes the uniformly finite homology of Block and Weinberger. This particular isomorphism is mentioned in the remark at the end of Section 3 in \cite{block_weinberger_1} and proved in \cite[Lemma 2.2]{whyte}.\footnote{Alternatively, we could use the \Poincare duality isomorphism $H_{b, \mathrm{dR}}^i(M) \cong H_{m-i}^\infty(M; \IR)$ which is proved in \cite[Theorem 4]{attie_block_1}, where $H_{m-i}^\infty(M; \IR)$ denotes simplicial $L^\infty$-homology and $M$ is triangulated according to Theorem \ref{thm:triangulation_bounded_geometry}, and then use the fact that $H_0^\infty(M; \IR) \cong H_0^{\mathrm{uf}}(M; \IR)$ under this triangulation (for this we need the assumption that $M$ is connected).} Since we have a characterization of amenability using uniformly finite homology (see the above Proposition \ref{prop:equiv_amenability_metric_spaces}), we therefore also have a characterization of it via bounded de Rham cohomology. We are going to discuss this now more intimately.

First we introduce the following notions:

\begin{defn}[Closed at infinity, {\cite[Definition II.5]{sullivan}}]
A Riemannian manifold $M$ is called \emph{closed at infinity} if for every function $f$ on $M$ with $0 < C^{-1} < f < C$ for some $C > 0$, we have $[f \cdot dM] \not= 0 \in H_{b, \mathrm{dR}}^m(M)$ (where $dM$ denotes the volume form of $M$ and $m = \dim M$).
\end{defn}

\begin{defn}[Fundamental classes, {\cite[Definition 3.3]{roe_index_1}}]\label{defn:fundamental_class}
A \emph{fundamental class} for the manifold $M$ is a positive linear functional $\theta\colon \Omega^m_b(M) \to \IR$ such that $\theta(dM) \not = 0$ and $\theta \circ d = 0$.
\end{defn}

If the volume form $dM$ of $M$ does not lie in the closure of the boundary operator $d$, i.e., $dM \notin \closure (\image d_{m-1})$, then we can get a positive linear functional $\theta \colon \Omega_b^m(M) \to \IR$ with $\theta(dM) \not= 0$ and $\theta \circ d = 0$, i.e., a fundamental class for $M$, by using Hahn--Banach. On the contrary, if the manifold $M$ does admit such a fundamental class $\theta$, we may conclude that $M$ is closed at infinity (for the proof we use the positivity of $\theta$: $\theta(f \cdot dM) \ge \theta(C^{-1} \cdot dM) = C^{-1} \cdot \theta(dM) \not= 0$). In a moment (Corollary \ref{cor:volume_form_not_boundary_not_in_closure_of_such}) we will even see that having a fundamental class and being closed at infinity are equivalent (i.e., it can not happen that the volume form is not the boundary of some bounded form, but lies in the closure of the boundaries).

Now we relate the above defined two notions to amenability: if $M$ is amenable, then we know already from Lemma \ref{lem:manifold_equiv_amenable_reg_exhaustion} that it admits a F{\o}lner sequence. Given some F{\o}lner sequence for $M$, we can construct a fundamental class for $M$ out of it; this is proved in \cite[Propositions 6.4 \& 6.5]{roe_index_1}. But we have already convinced ourselves that the existence of such a fundamental class for $M$ implies that $M$ is closed at infinity. This means especially $H_{b, \mathrm{dR}}^m(M) \not= 0$. But since this is isomorphic to $H_0^{\mathrm{uf}}(M; \IR)$, we conclude that the latter does also not vanish. And since we know from the above Proposition \ref{prop:equiv_amenability_metric_spaces} that the non-vanishing of $H_0^{\mathrm{uf}}(M; \IR)$ is equivalent to the amenability of $M$, we have arrived at the start of our chain. This shows that all these properties are equivalent. Let us summarize this:

\begin{prop}
Let $M$ be a connected, orientable manifold of bounded geometry.

Then the following are equivalent:
\begin{itemize}
\item $M$ is amenable,
\item $M$ admits a F{\o}lner sequence,
\item $M$ admits a fundamental class and
\item $M$ is closed at infinity.
\end{itemize}
\end{prop}

From duality of $H_{b, \mathrm{dR}}^m(M)$ with $H_0^{\mathrm{uf}}(M; \IR)$ we get the following corollary: if the volume form is not the boundary of some bounded form, then it is also not in the closure of the boundaries.

\begin{cor}\label{cor:volume_form_not_boundary_not_in_closure_of_such}
If $[dM] \not= 0 \in H_{b, \mathrm{dR}}^m(M)$, then $[dM] \not= 0 \in \bar{H}_{b, \mathrm{dR}}^m(M)$, where the latter group denotes the reduced bounded de Rham cohomology of $M$.
\end{cor}

At last, we are going to relate amenability to the linear isoparametric inequality and to the volume growth of the manifold:

\begin{prop}[{\cite[Subsection 4.1]{gromov_hyperbolic_manifolds_groups_actions}}]
Let $M$ be a connected and orientable manifold of bounded geometry.

Then $M$ is not amenable if and only if $\vol(R) \le C \cdot \vol(\partial R)$ for all $R \subset M$ and a fixed constant $C > 0$.
\end{prop}

\begin{prop}[{\cite[Proposition 6.2]{roe_index_1}}]
If $M$ has subexponential growth\footnote{This means that there is a point $x_0 \in M$ such that for all $p > 0$ we have $e^{-pr} \vol(B_r(x_0)) \to 0$ as $r \to \infty$.}, then it admits a \Folner exhaustion.
\end{prop}

Note that the converse to the last proposition does not hold.

\subsection*{Amenability of groups}

We will now briefly discuss amenability of groups. Our main interest for this is the fact that the universal covering $\widetilde{M}$ of a compact Riemannian manifold $M$ equipped with the pull-back metric is amenable if and only if the fundamental group $\pi_1(M)$ of $M$ is amenable. And if this is the case, we may construct fundamental classes for $\widetilde{M}$ that respect this structure of $\widetilde{M}$ of a covering manifold (see Proposition \ref{prop:fundamental_group_amenable_nice_fundamental_classes} for the concrete statement).

We will state the \Folner condition as a definition for amenability of a countable, discrete group. Note that there are a lot of other equivalent characterizations of amenability of countable, discrete groups, but since we do not need them at all, we won't state them either.

\begin{defn}[Amenable groups]\label{defn:amenable_group}
Let $\Gamma$ be a countable, discrete group. We will call $\Gamma$ \emph{amenable}, if for each $\varepsilon > 0$ and each finite subset $\Sigma \subset \Gamma$ there is a finite subset $E \subset \Gamma$ such that
\[\card(E \cap \gamma E) \ge (1-\varepsilon) \card(E) \ \forall \gamma \in \Sigma.\]
\end{defn}

This group theoretic definition of amenability is related to amenability of metric spaces via the following observation: a finitely generated group $\Gamma$ is amenable as a group if and only if $\Gamma$ equipped with the word metric is amenable as a metric space.

The class of amenable groups is quite large. Let us state without proof certain subclasses that it includes and certain permanency properties:

\begin{examples}
All finite groups are amenable and also all abelian groups.

Furthermore, the class of amenable groups is closed under taking subgroups, quotients, group extensions by another amenable group (therefore, the class is also closed under taking finite direct products) and under direct limits. It follows that all solvable groups are amenable.

And last, similarly to the fact that manifolds of subexponential volume growth are amenable, groups with subexponential volume growth are also amenable.
\end{examples}

To give at least one non-example, if a countable group contains a free subgroup on two generators, then it is not amenable.

Let us now state the prior mentioned, well-known equivalence of amenability of the fundamental group of a manifold and of the amenability of the universal cover of this manifold:

\begin{prop}[{\cite{brooks}}]
Let $M$ be a compact Riemannian manifold and denote by $\widetilde{M}$ its universal cover equipped with the pull-back metric.

Then $\widetilde{M}$ is amenable if and only if $\pi_1(M)$ is an amenable group.
\end{prop}

At last, we state the following useful proposition: if $\pi_1(M)$ is amenable, we may construct fundamental classes $\theta$ for $\widetilde{M}$ that are compatible with the structure of $\widetilde{M}$ as a covering space in the following way:

\begin{prop}[{\cite[Proposition 6.6]{roe_index_1}}]\label{prop:fundamental_group_amenable_nice_fundamental_classes}
Let $M$ be a compact Riemannian manifold, denote by $\widetilde{M}$ its universal cover equipped with the pull-back metric, and let $\pi_1(M)$ be amenable.

Then $\widetilde{M}$ admits a fundamental class $\theta$ with the property
\[\theta(\pi^\ast \alpha) = \int_M \alpha\]
for every top-dimensional form $\alpha$ on $M$ and where $\pi \colon \widetilde{M} \to M$ is the covering projection.
\end{prop}

\section{Analytic indices}\label{sec:analytic_indices_quasiloc_smoothing}

Our goal is to describe an index map on $K_0(\IU(E))$ and on $K_0(\C(E))$ if the manifold $M$ is amenable. Let us explain the basic idea: if we have an operator $A$, then its Fredholm index is defined as $\dim (\kernel A) - \dim (\kernel A^\ast)$ if both quantities are finite. We may rewrite $\dim (\kernel A)$ as $\trace P_A$, where $P$ is the projection operator onto the kernel of $A$. Now we suppose that $P_A$ can be represented as an integral operator with continuous integral kernel $k_A(x,y)$ and then it is well-known that we have $\trace P_A = \int k_A(x,x)dx$. Though this reasoning only works if $A$ is Fredholm and if the domain of integration is compact, we may use it to define ``averaged analytic indices'': if $M$ is non-compact but the integral kernel of the projection operator $P_A$ onto the kernel of $A$ is bounded, we may form the bounded sequence $\frac{1}{\vol M_i} \int_{M_i} k_A(x,x) dM$ and then evaluate a functional $\tau \in (\ell^\infty)^\ast$ on it.

In this way we get an index map on $\IU(E)$, resp. on $\C(E)$, since every smoothing operator has a smooth and bounded integral kernel. Furthermore, we get this index map for every manifold $M$ regardless of its amenability. But now the amenability together with the control at infinity of the integral kernels of operators from $\IU(E)$, resp. from $\C(E)$, ensures us that this index map is a trace, i.e., vanishes on commutators. So it descends in this case to a map on $K_0(\IU(E))$, resp. on $K_0(\C(E))$, which is exactly what we need.

Let us formalize this discussion (everything that we say here for $\IU(E)$ also holds for $\C(E)$): recall from Lemma \ref{lem:manifold_equiv_amenable_reg_exhaustion} that amenability of $M$ is equivalent to $M$ admitting a F{\o}lner sequence $(M_i)_i$. Given $A \in \IU(E)$, we define a sequence $(m_i)_i$ via
\begin{equation}\label{eq:index_map_cont_version}
m_i := \frac{1}{\vol M_i} \int_{M_i} \trace k_A(x,x) dM,
\end{equation}
where $k_A \in C_b^\infty(E \boxtimes E^\ast)$ is the uniformly bounded integral kernel of $A$ which is provided by Proposition \ref{prop:smoothing_op_kernel}. So $(m_i)_i$ is a bounded sequence. Now if $\tau \in (\ell^\infty)^\ast$ is a functional associated to a free ultrafilter on $\IN$,\footnote{That is, if we evaluate $\tau$ on a bounded sequence, we get the limit of some convergent subsequence.} we can evaluate $\tau$ on $(m_i)_i$. Note that if $A$ is a self-adjoint operator, then $\tau(m_i) \in \IR$. Using that $(M_i)_i$ is a \Folner sequence, it was shown in \cite[Theorem 6.7]{roe_index_1} that this construction defines a trace\footnote{That is to say, it vanishes on commutators.} on $\IU(E)$ and therefore we get an induced map $\ind_\tau \colon K_0(\IU(E)) \to \IR$.

In the case of the trivial, one-dimensional bundle $\IC \to M$ we will write $\IU(M)$ for the quasilocal smoothing operators on it, and we will write $\C(M)$ for the smooth uniform Roe algebra of it. The main result of this section is the following proposition (where the first part of it about the existence of index maps was already proved in the above discussion).

\begin{prop}\label{prop:analytic_index_map_quasiloc_smoothing}
Let $M$ be an amenable manifold of bounded geometry and $E$ a vector bundle of bounded geometry over $M$.

Then for every F{\o}lner sequence $(M_i)_i$ of $M$ and every functional $\tau \in (\ell^\infty)^\ast$ which is associated to a free ultrafilter on $\IN$, we get an analytic index map
\[\ind_\tau\colon K_0(\IU(E)) \to \IR.\]

Furthermore, we have a natural map $K_\ast(\IU(E)) \to K_\ast(\IU(M))$ and the index maps are compatible with it, i.e., we have the following commutative diagram:

\begin{equation}\label{eq:natural_receptacle_index_maps}
\xymatrix{K_0(\IU(E)) \ar[rr] \ar[dr]_{\ind_\tau} & & K_0(\IU(M)) \ar[dl]^{\ind_\tau} \\ & \IR}
\end{equation}

We may analogously construct a natural map $K_\ast(\C(E)) \to K_\ast(\C(M))$. Since $\C(E) \subset \IU(E)$, we also have analytic index maps $\ind_\tau\colon K_0(\C(E)) \to \IR$, and they are compatible in the above sense (i.e., we have an analogous commutative diagram) with the natural map $K_\ast(\C(E)) \to K_\ast(\C(M))$.
\end{prop}

This result is analogous to the one in \cite[Section 7]{roe_index_1}, which establishes the group $K_0(\IU(M))$ as a universal receptacle for indices of elliptic operators. In our case we may also use the algebra $\C(M)$ instead.

\begin{proof}
The first part about the existence of the index maps was proved in the discussion above the proposition. So it remains for us to construct the natural map $K_0(\IU(E)) \to K_0(\IU(M))$ and show that it is compatible with these index maps.

We will use the terminology and results from Section \ref{sec:interpretation_uniform_k_theory}. Since $E$ has bounded geometry, it is $C_b^\infty$-complemented, i.e., there exists a bundle $E^\perp$ such that $E \oplus E^\perp$ is $C_b^\infty$-isomorphic to a trivial bundle $\IC^N$. Let us denote by $i \colon E \hookrightarrow \IC^N$ the inclusion and by $e \in \Idem_{N \times N}(C_b^\infty(M))$ the orthogonal projection matrix of $\IC^N$ onto its subspace $E$. Then we get an induced map $\IU(E) \to \IU(\IC^N)$ via $A \mapsto A \circ e$, and since this is a $^\ast$-homomorphism, an induced map $i_0 \colon K_0(\IU(E)) \to K_0(\IU(\IC^N))$.

Now let $F$ and $G$ be two $C_b^\infty$-complements of $E$ and of the same dimension\footnote{We can assume this w.l.o.g. since we may add a suitable trivial bundle to either $F$ or $G$.}. We will show that the induced maps $i_0^F, i_0^G \colon K_0(\IU(E)) \to K_0(\IU(\IC^N))$ are the same. Since $F$ and $G$ are as complements of $E$ unitarily $C_b^\infty$-isomorphic, we get a unitary $C_b^\infty$-isomorphism $U \colon \IC^N \to \IC^N$ intertwining the inclusions $i^F, i^G \colon E \hookrightarrow \IC^N$, i.e., $U \circ i^F = i^G$, via $\IC^N \cong E \oplus F \cong E \oplus G \cong \IC^N$. This gives an inner automorphism $\operatorname{Ad}_U \colon \IU(\IC^N) \to \IU(\IC^N)$ via $A \mapsto U A U^\ast$. Since inner automorphisms induce the identity map on the $K$-groups,\footnote{see, e.g., \cite[Lemma 4.6.1]{higson_roe}} we get from the commutative diagram
\[\xymatrix{& K_0(\IU(E))\ar[dl]_{i_0^F} \ar[dr]^{i_0^G} & \\ K_0(\IU(\IC^N)) \ar[rr]_{(\operatorname{Ad}_U)_0 = \id} & & K_0(\IU(\IC^N)) }\]
the claim $i_0^F = i_0^G$.

Now since $\IU(\IC^N) \cong \Mat_{N \times N}(\IU(M))$ we get on $K$-theory a natural isomorphism $K_0(\IU(\IC^N)) \cong K_0(\IU(M))$. So together with the above discussion we get for any vector bundle $E$ a natural map $K_0(\IU(E)) \to K_0(\IU(M))$. It is easily deduced from the construction that this map is compatible with the index maps on $K_0(\IU(E))$ and $K_0(\IU(M))$, i.e., the Diagram \eqref{eq:natural_receptacle_index_maps} is commutative.

Clearly all the above also works for the algebra $\C(E)$ instead of $\IU(E)$.
\end{proof}

\section{Uniform Roe algebras}\label{sec:uniform_roe_algebra}

Let us now introduce the uniform Roe algebras $C_u^\ast(Y)$ of a discete metric space $Y$ (the case of a non-discrete metric space will be discussed in Section \ref{sec:uniformity_PDOs}). Note that the first instance of the uniform Roe algebra was actually $\mathcal{U}_{-\infty}(M)$ defined by Roe in \cite{roe_index_1}. Later it was generalized to the algebra $C_u^\ast(Y)$, since $\mathcal{U}_{-\infty}(M)$ is only defined for manifolds. Also note that we will show that not $\IU(E)$, but $\C(E)$\footnote{see Definition \ref{defn:smooth_uniform_roe_algebra}} actually is a smooth subalgebra of $C_u^\ast(E)$.

We have already mentioned the importance of the uniform Roe algeba to coarse spaces in Section \ref{sec:coarsely_bounded_geometry}. Let us mention now another important property of it which states that the uniform Roe algebra contains all of the coarse information of a space (a rigidity result): \v{S}pakula and Willett showed in \cite[Theorem 1.4]{spakula_willett} that if $X$ and $Y$ are discrete metric spaces of coarsely bounded geometry and both have property $A$, then they are coarsely equivalent if $C_u^\ast(X)$ and $C_u^\ast(Y)$ are Morita equivalent.

\begin{defn}[Uniform Roe algebras]\label{defn:C_u^ast(Y)}
Let $Y$ be a uniformly discrete metric space with coarsely bounded geometry.

The \emph{uniform Roe algebra} $C_u^\ast(Y)$ of $Y$ is the norm closure of all finite propagation operators with uniformly bounded coefficients in $\IB(\ell^2(Y))$.

Representing an operator $T$ on $\ell^2(Y)$ as a matrix indexed by $Y$, the notion of ``uniformly bounded coefficients'' means that there exists a $C > 0$ such that we have $\|T(x,y)\| < C$ for all $x, y \in Y$. And ``finite propagation'' means that there is an $R > 0$ such that $T(x,y) = 0$ whenever $d(x,y) > R$.
\end{defn}

Let us also introduce another version of the uniform Roe algebra which is in some cases easier to work with.

\begin{defn}[{\cite[Definition 6.6 \& Remark 6.7]{spakula_uniform_k_homology}}]
\label{defn:C_k^ast(Y)}
Let $C_k^\ast(Y)$ be the norm closure of all locally compact, finite propagation operators $T \in \IB(\ell^2 (Y \times \IN))$ with uniformly bounded coefficients and the following additional property (the uniformity condition): for every $\varepsilon > 0$ there exists an $M > 0$ such that each entry $T(x,y)$ of $T$ is at most $\varepsilon$ away from a rank-$M$ operator.

Here ``locally compact'' means that each entry $T(x,y)$ is a compact operator on $\ell^2(\IN)$. But note that this is implied by the uniformity condition.
\end{defn}

Now we can state the following lemma which ensures that the $K$-theory of $C_u^\ast(Y)$ and the $K$-theory of $C_k^\ast(Y)$ are equal (i.e., it justifies calling both algebras as uniform Roe algebras).

\begin{lem}[{\cite[Lemma 6.10]{spakula_uniform_k_homology}}]
\label{lem:iso_discrete_versions_uniform_roe}
Let $Y$ be a uniformly discrete metric space with coarsely bounded geometry.

Then the obvious inclusion $C_u^\ast(Y) \otimes \IK(\ell^2(\IN)) \hookrightarrow C_k^\ast(Y)$ is an isomorphism.
\end{lem}

Let $M$ be a manifold of bounded geometry, $E$ be a vector bundle of bounded geometry over $M$ and $\C(E)$ the smooth uniform Roe algebra\footnote{see Definition \ref{defn:smooth_uniform_roe_algebra}} on $E$. The following lemma establishes $\C(E)$ as a ``smooth'' version of $C_k^\ast(Y)$, where $Y \subset M$ is a uniformly discrete quasi-lattice (recall Definition \ref{defn:coarsely_bounded_geometry}).

\begin{lem}\label{lem:quasi_local_smoothing_dense_uniform_roe}
We have a continuous\footnote{with respect to the operator norm on $\C(E) \subset \IB(L^2(E))$} inclusion $\C(E) \hookrightarrow C_k^\ast(Y)$ with dense image, where $Y \subset M$ is a uniformly discrete quasi-lattice.

But this inclusion is \emph{not} canonical.
\end{lem}

For the proof we first have to define the following notion:

\begin{defn}[Quasi-latticing partitions, {\cite[Definition 8.1]{spakula_uniform_k_homology}}]\label{defn:quasi-latticing_partitions}
Let $Y \subset M$ be a uniformly discrete quasi-lattice.

A collection $(V_y)_{y \in Y}$ of open, disjoint subsets of $M$ is called a \emph{quasi-latticing partition with diameters $\le$ d}, if $M = \bigcup_{y \in Y} \overline{V_y}$, $\sup_{y \in Y} \diam V_y \le d$ and for every $\varepsilon > 0$ we have $\sup_{y \in Y} \card \{z \in Y \ | \ V_z \cap B_\varepsilon(V_y) \not= \emptyset\} < \infty$.
\end{defn}

Note that on manifolds of bounded geometry such quasi-latticing partitions always exist (by letting $V_y$ consist of all points $x$ of $M$ such that $y$ is the closest point to $x$ of all the points in $Y$).

\begin{proof}
Let $(V_y)_{y \in Y}$ be a quasi-latticing partition and we denote by $H$ the Hilbert space $H := L^2(E)$. Setting $H_y := \chi_{V_y} \cdot H$, we get a bijective (since each $H_y$ is infinite dimensional) isometry $H = \bigoplus_{y \in Y} H_y \cong \ell^2(Y \times \IN)$ after fixing an orthonormal basis $(e_i^y)_{i \in \IN}$ of each $H_y$.

Since $C_k^\ast(Y)$ is defined as the norm closure of the locally compact, finite propagation operators on $\ell^2(Y \times \IN)$ with uniformly bounded coefficients and satisfying a uniformity condition, it corresponds under the above isometry $H = \bigoplus_{y \in Y} H_y \cong \ell^2(Y \times \IN)$ to the norm closure of the bounded operators $A \in \IB(H)$ with finite propagation and the property that for all points $x,y \in Y$ the operators $\chi_{V_y} \cdot A \cdot \chi_{V_x} \colon H_x \to H_y$ are compact and the collection of them satisfies the uniformity condition.

Let $A \in \C(E)$ have finite propagation. Since $A$ is by definition a continuous operator on $H = L^2(E)$, it remains to show that it is locally compact and satisfies the uniformity condition. We will do this at once: by \cite[Theorem 4.(3)]{stinespring} the operators $\chi_{V_y} \cdot A \chi_{V_x} \colon H_x \to H_y$ are trace class with traces uniformly bounded in $x,y \in Y$. From this the claim that $A$ defines a finite propagation operator in $C_k^\ast(Y)$ follows.

Since the finite propagation operators in $\C(E)$ are dense (by definition) and the identification $H = \bigoplus_{y \in Y} H_y \cong \ell^2(Y \times \IN)$ is an isometry, we get the claimed continuous inclusion $\C(E) \hookrightarrow C_k^\ast(Y)$. It is clear that it is not canonical since it depends on the chosen quasi-latticing partition $(V_y)_{y \in Y}$ and the chosen orthonormal bases $(e_i^y)_{i \in \IN}$ of the spaces $H_y$.

It remains to show that $\C(E) \subset C_k^\ast(Y)$ is dense. Let $T \in C_k^\ast(Y)$ be a locally compact, finite propagation operator with uniformly bounded coefficients and satisfying the uniformity condition, and let $\varepsilon > 0$ be given. Since $T(x,y) := \chi_{V_y} \cdot T \chi_{V_x} \colon H_x \to H_y$ is compact, there is a finite rank operator $k(x,y)\colon H_x \to H_y$ which is $\varepsilon$ away from $T(x,y)$. We can write $k(x,y) = \sum_{j=1}^N \lambda_j \langle f_j, \largecdot \rangle g_j$, where $\{f_j\} \subset H_x$ and $\{g_j\} \subset H_y$ are orthonormal families (but not necessarily complete ones) and the $\lambda_j$ are complex numbers bounded in absolute value from above by the common upper bound $C$ for the coefficients of $T$. We fix a $\delta > 0$ and choose functions $f_j^\prime \in H^\infty(E)$ with $\supp f_j^\prime \subset B_\delta(V_x)$ and $\|f_j - f_j^\prime\|_{L^2(E)} < \varepsilon / 2N$ and analogously we choose $g_j^\prime$. Then the operator $k^\prime(x,y) := \sum_{j=1}^N \lambda_j \langle f^\prime_j, \largecdot \rangle g^\prime_j$ is a smoothing operator which is $C\varepsilon$ away from $k(x,y)$, and therefore $2C\varepsilon$ from $T(x,y)$. Now we do this for all $x,y \in Y$. Since $T$ satisfies the uniformity condition, there is a common upper bound for the $N$s. This allows us to choose approximating functions from $H^\infty(E)$ such that they all have common upper bounds on all their derivatives (we use for all $x,y$ the same $\delta$), which leads to the fact that if we put all the smoothing operators $k^\prime(x,y)$ together into one operator (call it $T^\prime$), it will still be a bounded operator in all seminorms, i.e., a smoothing operator (here we also need the property of quasi-latticing partitions which is stated in the definition as the last one and denote its value for $\delta$ by $\delta_{(V_y)}$). Furthermore, $T^\prime$ is at most $2C\delta_{(V_y)}\varepsilon$ away from $T$ and $T^\prime$ has finite propagation. From this the claim follows.
\end{proof}

These inclusions $\C(E) \hookrightarrow C_k^\ast(Y)$ are $^\ast$-homomorphisms and therefore we get induced maps $K_\ast(\C(E)) \to K_\ast(C_k^\ast(Y))$. Since $\C(E)$ is a local $C^\ast$-algebra and densely included in $C_k^\ast(Y)$, we conclude with Lemma \ref{lem:loc_algebra_same_k_theory} that the induced maps $K_\ast(\C(E)) \to K_\ast(C_k^\ast(Y))$ are isomorphisms. Now we will show that though these inclusions $\C(E) \hookrightarrow C_k^\ast(Y)$ are not canonical, the induced isomorphisms $K_\ast(\C(E)) \to K_\ast(C_k^\ast(Y))$ are, because they are all equal.

\begin{prop}\label{prop:smooth_version_of_uniform_roe}
Let $M$ be a manifold of bounded geometry and $Y \subset M$ a uniformly discrete quasi-lattice. Furthermore, let $E \to M$ be a vector bundle of bounded geometry.

Then the groups $K_\ast(\C(E))$ and $K_\ast(C_k^\ast(Y))$ are naturally isomorphic.
\end{prop}

\begin{proof}
From the proof of the above Lemma \ref{lem:quasi_local_smoothing_dense_uniform_roe} we see that the constructed maps $\C(E) \hookrightarrow C_k^\ast(Y)$ depend on two things: on the chosen quasi-latticing partition $(V_y)_{y \in Y}$ and the chosen orthonormal bases $(e_i^y)_{i \in \IN}$ of the spaces $H_y$. The dependence on the first can be dropped by considering a canonical quasi-latticing partition, namely the one that picks the nearest points, i.e., $p \in V_y$ if and only if of all points from $Y$ the point $y$ is the nearest to $p$. The dependence of the inclusions on the second point can not be dropped, but given two different collections of bases for the spaces $H_y$, we may map the one to the other and get a unitary $H \to H$ intertwining the constructed inclusions of $\C(E)$ into $C_k^\ast(Y)$. Since inner automorphisms act trivially on $K$-theory, we get that the maps $K_\ast(\C(E)) \to K_\ast(C_k^\ast(Y))$ induced by the two different inclusions are equal.
\end{proof}

Let $E$ and $F$ be two vector bundles of bounded geometry over $M$. From the above result we conclude that $K_\ast(\C(E))$ and $K_\ast(\C(F))$ are isomorphic. Recall that for the trivial bundle $F = \IC$ we just write $K_\ast(\C(M))$ and that we constructed a natural map $K_\ast(\C(E)) \to K_\ast(\C(M))$ in Proposition \ref{prop:analytic_index_map_quasiloc_smoothing}. We will see now that this natural map is an isomorphism by comparing it with the isomorphisms with $K_\ast(C_k^\ast(Y))$.

\begin{cor}\label{cor:natural_receptacle_iso}
The map $K_\ast(\C(E)) \to K_\ast(\C(M))$ from Proposition \ref{prop:analytic_index_map_quasiloc_smoothing} is an isomorphism and we have a commutative diagram of isomorphisms
\[\xymatrix{ K_\ast(\C(E)) \ar[rr] \ar[dr] & & K_\ast(\C(M)) \ar[dl] \\ & K_\ast(C_k^\ast(Y)) &}\]
where the diagonal maps are the ones from the above proposition.
\end{cor}

\begin{proof}
That the map $K_\ast(\C(E)) \to K_\ast(\C(M))$ is an isomorphism will immediately follow if we can establish the commutativity of the diagram (since we already know that the diagonal maps are isomorphisms). Let $i\colon E \to \IC^N$ be an inclusion inducing the natural map $i_\ast\colon K_\ast(\C(E)) \to K_\ast(\C(\IC^N)) \cong K_\ast(\C(M))$. Then we have the following commutative diagram
\[\xymatrix{ K_\ast(\C(E)) \ar[r]^{i_\ast} \ar[d]^{\cong} & K_\ast(\C(\IC^N)) \ar[d]^{\cong} \\ K_\ast(C_k^\ast(Y)) \ar@{-->}[r]^{(\operatorname{Ad}_V)_\ast} & K_\ast(C_k^\ast(Y))}\]
where the map $(\operatorname{Ad}_V)_\ast$ is induced from the inner endomorphism $\operatorname{Ad}_V(T) := VTV^\ast$ on $C_k^\ast(Y)$, where the isometry $V \colon C_k^\ast(Y) \to C_k^\ast(Y)$ will be constructed in a moment, and the vertical maps in the diagram are the ones from the above Proposition \ref{prop:smooth_version_of_uniform_roe}.

Let us construct the isometry $V$: to construct the map $K_\ast(\C(E)) \to K_\ast(C_k^\ast(Y))$ we have chosen a collection of orthonormal bases of $L^2(E) = \bigoplus_{y \in Y} H_y$, and for the map $K_\ast(\C(\IC^N)) \to K_\ast(C_k^\ast(Y))$ we have chosen a collection of orthonormal bases of $L^2(\IC^N) = \bigoplus_{y \in Y} H_y^\prime$. Since $E \subset \IC^N$ as a subbundle, we have $H_y \subset H_y^\prime$ as subspaces. So we may choose the basis for each space $H_y^\prime$ such that it extends the chosen basis of $H_y$ to get an isometric inclusion $\bigoplus_{y \in Y} H_y \hookrightarrow \bigoplus_{y \in Y} H_y^\prime$. From this we get our isometry $V$. It is clear that with this construction the above diagram commutes.

Since inner endomorphisms induced by isometries act trivially on $K$-theory,\footnote{see, e.g., \cite[Lemma 4.6.2]{higson_roe}} i.e., $(\operatorname{Ad}_V)_\ast = \id$, we conclude from the diagram that $i_\ast\colon K_\ast(\C(E)) \to K_\ast(\C(\IC^N))$ is an isomorphism.
\end{proof}

Let us mention at least one example of the computation of the $K$-groups of a uniform Roe algebra:

\begin{example}
Mavra showed in \cite{mavra} that $K_1(C_u^\ast(\IR)) \cong \IZ$ and that $K_0(C_u^\ast(\IR))$ is an ininitely generated, torsion free abelian group. For the computation he used the Pimsner--Voiculescu exact sequence and the following fact proven by combining results of Yu (\cite[Proposition 3.2.4]{mavra}): if $M$ is a compact manifold with fundamental group $\Gamma$ and if $X$ denotes the universal cover of $M$, then $C_u^\ast(X)$ is Morita equivalent to the reduced crossed product algebra $\ell^\infty(\Gamma; \IC) \rtimes \Gamma$.
\end{example}

\section{Index maps for uniform Roe algebras}\label{sec:index_maps_uniform_roe}

Since $\C(E)$ is a ``smooth version'' of the uniform Roe algebra $C_u^\ast(Y)$ and since we do have analytic index maps on $K_0(\C(E))$, it is not surprising that we may define index maps on $K_0(C_u^\ast(Y))$ that are equal to the ones on $K_0(\C(E))$ under the isomorphism $K_0(\C(E)) \cong K_0(C_u^\ast(Y))$. To show this is the task of the first half of this section.

The second half is denoted to the index maps on $K_0(C_k^\ast(Y))$ and will be a bit more technically involved as the first half. The reason for this is that here we will have to approximate the compact operators $T(x,y) \in \IK(\ell^2(\IN))$ by operators of trace class, but we will have to do it uniformly (i.e., that the trace norms for different points $x, y \in Y$ stay uniformly bounded in the process of approximation). We will have the same problem when we define index maps on uniform $K$-homology $K_0^u(X)$ of a space $X$, but there it is even more involved (in fact, almost all of the Section \ref{sec:normalization} will be devoted to set up this approximation via ``uniformly traceable operators'' for $K_0^u(X)$). So the in comparison short solution here may be seen as a warm-up for Section \ref{sec:normalization}.

We may ask ourselves why we do need these much more complicated index maps on $K_0(C_k^\ast(Y))$. One explanation would be that the uniform coarse assembly map $\mu_u \colon K_\ast^u(X) \to K_\ast(C_k^\ast(Y))$ maps naturally into the $K$-theory of $C_k^\ast(Y)$ and not into the one of $C_u^\ast(Y)$ (though this $K$-groups are isomorphic). So we need the index maps on $K_0(C_k^\ast(Y))$ in order to relate these indices to the analytic indices that we define on $K_0^u(X)$.

As with the index maps on $K_0(\C(E))$, we will need here amenability so that the index maps that we define on $C_u^\ast(Y)$, resp. $C_k^\ast(Y)$, descend to $K$-theory. That we can not have similar index maps without amenability was shown by Elek in \cite{elek}.

But let us start with the index maps on $K_0(C_u^\ast(Y))$:

\begin{prop}[Index maps on $K_0(C_u^\ast(Y))$, {\cite[Section 2]{elek}}]\label{prop:index_maps_uniform_Roe_algebras}
Let $Y$ be a uniformly discrete metric space of coarsely bounded geometry and let it be amenable.

Then we get for any functional $\tau \in (\ell^\infty)^\ast$ associated to a free ultrafilter on $\IN$\footnote{That is, if we evaluate $\tau$ on a bounded sequence, we get the limit of some convergent subsequence.} and for any F{\o}lner sequence $(U_i)_i$ of $Y$ an index map
\[\ind_\tau \colon K_0(C_u^\ast(Y)) \to \IR.\]
\end{prop}

\begin{proof}
Let $T \in C_u^\ast(Y)$ and let $(U_i)_i$ be a sequence of finite subsets of $Y$ (at this stage of the proof not necessarily a F{\o}lner sequence). We define a sequence $(t_i)_i \in \ell^\infty$ by
\begin{equation*}
t_i := \frac{1}{\card U_i} \sum_{y \in U_i} T(y,y).
\end{equation*}
Since $T$ has uniformly bounded coefficients, this sequence is indeed a bounded sequence. Let $\tau \in (\ell^\infty)^\ast$ be a functional associated to a free ultrafilter of $\IN$. Then we can define $\ind_\tau (T) := \tau (t)$, where $t = (t_i)_i$ is the bounded sequence defined in the above display. Note that if $T$ is self-adjoint, then its index is real-valued.

Elek showed in \cite[Section 2]{elek} that if $T, S \in C_u^\ast(Y)$ both have finite propagation, then $\ind_\tau (TS) = \ind_\tau (ST)$ under the assumption that $(U_i)_i$ is a F{\o}lner sequence. So $\ind_\tau \colon C_u^\ast (Y) \to \IR$ is in the amenable case a trace on the dense subalgebra of all finite propagation operators with uniformly bounded coefficients in $\IB(\ell^2(Y))$. Now it remains to show that it is continuous against the norm on $\IB(\ell^2(Y))$ so that it extends to a trace on the whole algebra $C_u^\ast(Y)$.

It is clear that the index map is continuous against the total sup-norm of $T$ given by $\sup_{x,y \in Y} |T(x, y)|$. We can bound this norm from above against the operator norm on $\IB(\ell^2(Y))$ of $T$ in the following way: the supremum over the absolute values of all entries of $T$ is the same as the supremum over the sup-norms of the columns of $T$, which are regarded as elements of $\ell^\infty(Y)$. Now the $\ell^\infty$-norm is bounded from above by the $\ell^2$-norm, and the $\ell^2$-norm of the $y$th column of $T$ is $\|T e_y\|_2$, where $e_y$ is the sequence which is $1$ on the $y$th position and all other entries are $0$. Since $\|T e_y\|_2 \le \|T\|_{op} \cdot \|e_y\|_2 = \|T\|_{op}$, we get the estimate
\begin{equation*}
\sup_{x,y \in Y} |T(x, y)| \le \sup_{y \in Y} \|T e_y\|_\infty \le \sup_{y \in Y} \| T e_y \|_2 \le \|T\|_{op}.
\end{equation*}
From this we can conclude that the index map is continuous against the operator norm on $\IB(\ell^2(Y))$ since it is clearly continuous against the total sup-norm.

So it defines a trace on the whole algebra $C_u^\ast(Y)$ and descends therefore to a map on $K_0(C_u^\ast(Y))$.
\end{proof}

Now we get to the index maps on $K_0(C_k(Y))$:

\begin{prop}
Let $Y$ be an amenable, uniformly discrete metric space having coarsely bounded geometry.

Furthermore, let a functional $\tau \in (\ell^\infty)^\ast$ associated to a free ultrafilter on $\IN$ and a F{\o}lner sequence $(U_i)_i$ of $Y$ be given.

Then every $[T] \in K_0(C_k^\ast(Y))$ has a uniformly traceable representative $T$, i.e., one where each entry $T(x,y)$ is a trace class operator on $\ell^2(\IN)$ and the trace norms of these entries are bounded from above. Furthermore, the index defined via evaluating $\tau$ on the bounded sequence $t = (t_i)_i \in \ell^\infty$ given by
\begin{equation}\label{eq:defn_sequence_trace_entries}
t_i := \frac{1}{\card U_i} \sum_{y \in U_i} \trace T(y,y)
\end{equation}
is independent of the choice of such a representative.

Therefore we get index maps
\[\ind_\tau \colon K_0(C_k^\ast(Y)) \to \IR\]
which are compatible with the index maps on $K_0(C_u^\ast(Y))$, i.e., the diagram
\begin{equation*}
\xymatrix{K_0(C_u^\ast(Y)) \ar[rr]^\cong \ar[dr]_{\ind_\tau} & & K_0(C_k^\ast(Y)) \ar[dl]^{\ind_\tau} \\ & \IR}
\end{equation*}
commutes, where the horizontal map is induced by the isomorphism in Lemma \ref{lem:iso_discrete_versions_uniform_roe}.
\end{prop}

\begin{proof}
Suppose that we have an operator $T \in C_k^\ast(Y)$ such that each entry of $T$ is a trace class operator on $\ell^2(\IN)$ and such that their trace norms are bounded. Then we set $\ind_\tau (T) := \tau(t)$, where $t = (t_i)_i \in \ell^\infty$ is the sequence defined above by \eqref{eq:defn_sequence_trace_entries}. This is the basic idea on which we will now elaborate.

Let $C_{\mathrm{fin}}^\ast(Y)$ be the algebra of all finite propagation operators with uniformly bounded coefficients (i.e., $C_u^\ast(Y)$ is the operator norm completion of $C_{\mathrm{fin}}^\ast(Y)$) and let us denote by $N(\ell^2(\IN))$ the set of all trace class operators on $\ell^2(\IN)$. Then we get an inclusion $C_{\mathrm{fin}}^\ast(Y) \otimes_{\mathrm{alg}} N(\ell^2(\IN)) \hookrightarrow C_k^\ast(Y)$ such that the image consist only of operators for which the index map defined in the above paragraph makes sense. Note also that the image of this inclusion is dense, which can be proved analogously as Lemma \ref{lem:iso_discrete_versions_uniform_roe} (using that the trace class operators are dense in the compact operators).\footnote{Actually, this proof even shows that $C_u^\ast(Y) \otimes_{\mathrm{alg}} F(\ell^2(\IN))$ is dense in $C_k^\ast(Y)$, where $F(\ell^2(\IN))$ denotes the finite rank operators on $\ell^2(\IN)$.} It is clear that if we equip $C_{\mathrm{fin}}^\ast(Y)$ with the index map defined in the above Proposition \ref{prop:index_maps_uniform_Roe_algebras}, then the induced trace on the tensor product $C_{\mathrm{fin}}^\ast(Y) \otimes_{\mathrm{alg}} N(\ell^2(\IN))$ coincides with the index map defined in the first paragraph of this proof. From this the claimed commutativity of the diagram
\begin{equation*}
\xymatrix{K_0(C_u^\ast(Y)) \ar[rr]^\cong \ar[dr]_{\ind_\tau} & & K_0(C_k^\ast(Y)) \ar[dl]^{\ind_\tau} \\ & \IR}
\end{equation*}
follows, because the horizontal map in this diagram is induced from tensoring $C_u^\ast(Y)$ with the compacts $\IK(\ell^2(\IN))$ (see Lemma \ref{lem:iso_discrete_versions_uniform_roe}). Of course we now have to show that the index map defined on the dense subset $C_{\mathrm{fin}}^\ast(Y) \otimes_{\mathrm{alg}} N(\ell^2(\IN)) \subset C_k^\ast(Y)$ extends to the whole $C_k^\ast(Y)$ and passes to $K$-theory, in order that the diagram makes sense.

Regarding elements $T \in C_{\mathrm{fin}}^\ast(Y) \otimes_{\mathrm{alg}} N(\ell^2(\IN))$ as infinitely sized matrices with entries in $N(\ell^2(\IN))$, we introduce the following norm $\|\largecdot\|_{1,1}$ on them: we take the trace norm $\|T(x,y)\|_{tr} := \trace(|T(x,y)|)$ of every entry $T(x,y)$ of $T$ (where $|S| := \sqrt{S^\ast S}$ for bounded operators $S$) and then take the column sum norm of the resulting matrix. Recall that on any Hilbert space the trace class operators form a two-sided $^\ast$-ideal of the bounded operators and that we have the inequalities $\|PQ\|_{tr} \le \|P\|_{tr} \cdot \|Q\|_{op}$ and $\|P\|_{op} \le \|P\|_{tr}$ for a trace class operator $T$ and a bounded operator $Q$. It follows that the trace norm is submultiplicative and since this is also true for the column sum norm, we conclude that $\|\largecdot\|_{1,1}$ is submultiplicative. Just for the sake of symmetry we also define the norm $\|\largecdot\|_{\infty, \infty}$ by taking first the trace norm of every entry and then the row sum norm of the resulting matrix (note that the row sum norm of a matrix is the column sum norm of the adjoint matrix). So now we have on $C_{\mathrm{fin}}^\ast(Y) \otimes_{\mathrm{alg}} N(\ell^2(\IN))$ three submultiplicative norms (the usual operator norm as a subset of $\IB(\ell^2(Y \times \IN))$ and the two above defined trace norms $\|\largecdot\|_{1,1}$ and $\|\largecdot\|_{\infty,\infty}$) and taking the completion of it with respect to the sum of all three norms we get a Banach $^\ast$-algebra which we will denote by $C_{\mathrm{tr}}^\ast(Y)$. This is a dense, two-sided $^\ast$-ideal in $C_k^\ast(Y)$ (dense with respect to the operator norm on $C_k^\ast(Y)$). Furthermore, note that both norms $\|\largecdot\|_{1,1}$ and $\|\largecdot\|_{\infty, \infty}$ induce uniform convergence of the traces of the entries of $T$. So the index map $\ind_\tau (T) = \tau(t)$, where $t = (t_i)_i$ is the sequence from \eqref{eq:defn_sequence_trace_entries}, extends from $C_{\mathrm{fin}}^\ast(Y) \otimes_{\mathrm{alg}} N(\ell^2(\IN))$ to $C_{\mathrm{tr}}^\ast(Y)$.

Using the argument from \cite[Section 2]{elek}, we see that the index map on $C_{\mathrm{tr}}^\ast(Y)$ is a trace if $(U_i)_i$ is a F{\o}lner sequence and therefore descends to a map on $K_0(C_{\mathrm{tr}}^\ast(Y))$.

So we get a commutative diagram
\begin{equation*}
\xymatrix{K_0(C_u^\ast(Y)) \ar[rr] \ar[dr]_{\ind_\tau} & & K_0(C_{\mathrm{tr}}^\ast(Y)) \ar[dl]^{\ind_\tau} \\ & \IR}
\end{equation*}
Since $C_u^\ast(Y) \otimes \IK(\ell^2(\IN)) \cong C_k^\ast(Y)$ by Lemma \ref{lem:iso_discrete_versions_uniform_roe} and $C_{\mathrm{tr}}^\ast(Y) \subset C_k^\ast(Y)$ is dense, the conclusion follows.
\end{proof}

If $M$ is an amenable manifold of bounded geometry, $Y \subset M$ a uniformly discrete quasi-lattice and $E \to M$ a vector bundle of bounded geometry, we know from Proposition \ref{prop:smooth_version_of_uniform_roe} that $K_\ast(\C(E)) \cong K_\ast(C_k^\ast(Y))$. From the above proposition we know that we have an index map $\ind_\tau \colon K_0(C_k^\ast(Y)) \to \IR$ and from Proposition \ref{prop:analytic_index_map_quasiloc_smoothing} that we have an analytic index map $\ind_\tau \colon K_0(\C(E)) \to \IR$. We will now show that both index maps coincide under the isomorphism $K_\ast(\C(E)) \cong K_\ast(C_k^\ast(Y))$.

\begin{prop}
Let $M$ be an amenable manifold of bounded geometry, $Y \subset M$ a uniformly discrete quasi-lattice and $E \to M$ a vector bundle of bounded geometry.

Then the following diagram commutes:
\begin{equation*}
\xymatrix{K_0(\C(E)) \ar[rr] \ar[dr]_{\ind_\tau} & & K_0(C_k^\ast(Y)) \ar[dl]^{\ind_\tau} \\ & \IR}
\end{equation*}
For the definition of the index maps we have to use F{\o}lner sequences for $M$ and for $Y$ that are related as stated in the proof.
\end{prop}

\begin{proof}
Applying the inclusion $\C(E) \hookrightarrow C_k^\ast(Y)$ of Lemma \ref{lem:quasi_local_smoothing_dense_uniform_roe} to $T \in \C(E)$, we get a matrix $(T(x,y))_{x,y \in Y}$ of operators, which is defined by $T(x,y) := \chi_{V_x} \cdot T \cdot \chi_{V_y}$, where $(V_y)_{y \in Y}$ is a quasi-latticing partition with $y \in V_y$ (Definition \ref{defn:quasi-latticing_partitions}). Let $(U_i)_i$ be a F{\o}lner sequence for $Y$ and let $(M_i)_i$ be the corresponding F{\o}lner sequence for $M$ defined by $M_i := \bigcup_{y \in U_i} V_y$. Then the composition of the inclusion $\C(E) \hookrightarrow C_k^\ast(Y)$ with the index map on $C_k^\ast(Y)$ defines $\tau(t_i)$, where
\begin{equation*}
t_i = \frac{1}{\card U_i} \sum_{y \in U_i} \trace \big( \chi_{V_y} \cdot T \cdot \chi_{V_y}\big).
\end{equation*}
Since $T$ has a smooth kernel, we have $\trace \chi_{V_y} \cdot T \cdot \chi_{V_y} = \int_{V_y} \trace k_T(x,x) dM$.
So we have to show that $\tau$ evaluated on the sequence $(m_i)_i$ given by
\[m_i := \frac{1}{\vol M_i} \int_{M_i} \trace k_T(x,x) dM\]
(which is the analytic index of $T$) coincides with $\tau$ evaluated on the sequence $(t_i)_i$, where
\[t_i = \frac{1}{\card U_i} \sum_{y \in U_i} \int_{V_y} \trace k_T(x,x) dM\]
(which is the index of $T$ after inclusion into $C_k^\ast(Y)$). But the difference of $m_i$ and $t_i$ is concentrated on the boundary of $U_i$ of a suitable uniformly bounded radius, i.e., it goes to $0$ in the limit taken by $\tau$ since $(U_i)_i$ is amenable.
\end{proof}

\chapter{Pseudodifferential operators}\label{chap:PDOs}

Let us first define pseudodifferential operators on $\IR^n$, before we begin our discussion of these operators on manifolds: an operator $P \colon C_c^\infty(\IR^n) \to C_0(\IR^m)$ is a pseudodifferential operator of order $k \in \IZ$, if
\[(Pu)(x) = (2\pi)^{-n/2} \int e^{\langle x, \xi \rangle} p(x, \xi) \hat{u}(\xi) d\xi,\]
where the function $p(x, \xi)$ satisfies the estimates
\[\| D_x^\alpha D_\xi^\beta p(x, \xi) \| \le C^{\alpha \beta} (1 + |\xi|)^{k - |\beta|}\]
for all multi-indices $\alpha$ and $\beta$. Note that if we set $p(x, \xi) := \sum_{|\alpha| \le k} A^\alpha(x) \xi^\alpha$, then the associated pseudodifferential operator will be the differential operator $\sum A^\alpha D^\alpha$ of order $k$, i.e., pseudodifferential operators are a generalization of differential operators. The major reason for introducing them is that if $p(x,\xi)$ is an invertible matrix, then the pseudodifferential operators associated to $1/p(x, \xi)$\footnote{We are ignoring here the issue how to deal with the singularity at $\xi = 0$ where we usually have $p(x, \xi) = 0$.} will be an inverse to the former operator modulo smoothing operators, i.e., the class of pseudodifferential operators includes all parametrices to elliptic differential operators.

If we have a compact manifold $M$, then we may define an operator $P$ to be a pseudodifferential operator if it is locally (i.e., in charts) a pseudodifferential operator on $\IR^n$. But now generalizing this to non-compact manifolds is problematic. As an example, consider the operator $x^2 \cdot D_x$ on $\IR$. It is not a pseudodifferential operator since its symbol is $p(x, \xi) = x^2 \cdot \xi$ which does not satisfy the required boundedness condition (since it is unbounded in $x$). But if we look in the charts $(k, k+1) \subset \IR$ at this operator, then it \emph{does} satisfy in every chart the conditions, i.e., \emph{locally} this operator \emph{is} pseudodifferential. But we certainly do not want to admit this operator as a pseudodifferential operator: one reason is, e.g., that it does not extend to a bounded operator $H^1(\IR) \to L^2(\IR)$ (a major feature of pseudodifferential operators is that they extend uniquely to bounded operators $H^s \to H^{s-k}$ between the Sobolev spaces $H^s$ and $H^{s-k}$, where $k$ is the order of the operator).

We could try to solve this problem by requiering that the constants $C^{\alpha \beta}$ are uniformly bounded with respect to different charts. But choosing ``bad'' charts on a manifold may result in distorting the derivatives of the symbol of our operator arbitrarily high, i.e., there could be one covering of $M$ with charts such that the constant $C^{1 1}$ (i.e., $C^{\alpha \beta}$ for $\alpha, \beta = 1$) is uniformly bounded in the charts of this covering, and there may be another covering of $M$, such that now the constant $C^{1 1}$ is not bounded through all the charts of this second covering.

So the only solution to this problem is to fix particular charts of our manifold. But this is of course a very bad solution since we usually do \emph{not} want to fix the charts of a manifold. But now bounded geometry enters the game: on a manifold of bounded geometry, there are \emph{canonical} charts, namely the ones given by normal coordinates, and they are all compatible with each other in the sense that change of coordinates does never distort derivatives arbitrarily high (cf. Lemma \ref{lem:transition_functions_uniformly_bounded}).

We have seen that it is natural to define pseudodifferential operators on manifolds of bounded geometry as we will do it. But to the big surprise of the author, only few other authors have investigated them: the author could only find Kordyukov's and Shubin's papers \cite{kordyukov} and \cite{shubin} and Taylor's lecture notes \cite{taylor_pseudodifferential_operators_lectures}. But note that though all three have the same local definition of pseudodifferential operators on manifolds of bounded geometry, they use a different approach to the question how to control the integral kernels $k(x,y)$ of these operators at infinity (note that every pseudodifferential operator has an integral kernel that is smooth outside the diagonal): both Kordyukov and Shubin demand that their operators have finite propagation, i.e., that there is a $R > 0$ such that $k(x,y) = 0$ for $d(x,y) > R$, and Taylor demands that the kernels have an exponential decay at infinity. We will persue two slightly different approaches here: in one version we demand that our operators are quasilocal and in the other version we demand that they are limits of operators having finite propagation. Though one might conjecture that both versions coincide, it is an open question whether this is indeed the case. We consider both versions since the one using quasilocal operators does have good properties with respect to functional calculus (e.g., being closed under functional calculus with Schwartz functions), whereas the other version relates intimately to the uniform Roe algebra.

Let us explain our reasons for considering the version which relates to the uniform Roe algebra. Here we ensure that the algebra of pseudodifferential operators $\Psi \mathrm{DO}_u^{-\infty}(E)$ of order $-\infty$ (i.e., the smoothing ones) coincide exactly with the smooth uniform Roe algebra $\C(E)$. Now given an elliptic pseudodifferential operator, it is invertible modulo $\Psi \mathrm{DO}_u^{-\infty}(E)$, i.e., it will have an abstract index class in the $K$-theory of this algebra. If we would only work with finite propagation operators, then the corresponding algeba of smoothing operators of finite propagation would not be a local $C^\ast$-algebra and therefore we would have to resort to algebraic $K$-theory. But since $\C(E)$ \emph{is} a local $C^\ast$-algebra, we may use in our setting the machinery of operator $K$-theory. Furthermore, note that the $K$-groups of $\C(E)$ are the same as the $K$-groups of the uniform Roe algebra $C_u^\ast(E)$, i.e., we immediately get that the abstract index class of an elliptic pseudodifferential operator lives in the $K$-theory of the uniform Roe algebra, i.e., we do not have to map it there via a comparison map between the algebraic and operator $K$-theory.

Let us now get to the major advantage of using the other version of quasilocal operators. It lies in the computation of the analytic indices of elliptic pseudodifferential operators: Roe showed in \cite[Proposition 8.1]{roe_index_1} that if $f$ is a Schwartz function with $f(0) = 1$, then the smoothing operator $f(D)$ may be used to compute the analytic indices of $D$. But the operator $f(D)$ usually does not have finite propagation, but is only a quasilocal operator. So if we would work only with pseudodifferential operators of finite propagation as Kordyukov and Shubin do it, we would leave this class when computing the analytic indices of an elliptic operator. So the major advantage of considering quasilocal operators is that only then we get a class of operators which is closed under functional calculus with Schwartz functions, from which it follows that the smoothing operators used to compute analytic indices are still included in our class of operators. Note that it will be one of the major tasks of this chapter to show that if $P$ is a symmetric and elliptic pseudodifferential operator of positive order, then $f(P)$ will be a pseudodifferential operator of order $-\infty$ (Corollary \ref{cor:schwartz_function_of_PDO_quasilocal_smoothing}).

\section{Definition and basic properties}

So let us actually define pseudodifferential operators and deduce some of their basic properties. All of the results here hold true for compact manifolds, i.e., our task is to make sure that they remain true after passing to non-compact manifolds of bounded geometry. In most cases we will leave the proof out since it is the same as in the compact case.

We will define two slightly different algebras of pseudodifferential operators. The difference will lie in the smoothing term of the operator: either it is from $\IU(E)$ or from $\C(E)$. If it is from $\C(E)$, we will be able to relate these operators to the uniform Roe algebra of the manifold, which is quite a nice result. But our main technical results from Section \ref{sec:functions_of_PDOs} only hold for the class of pseudodifferential operators with smoothing term in $\IU(E)$. Fortunately this is the bigger class, i.e., the main result of the thesis, the index theorem, holds for this larger class.

Let always $M^m$ be an $m$-dimensional manifold of bounded geometry and let $E$ and $F$ be two vector bundles of bounded geometry over $M$.

\begin{defn}[Pseudodifferential operators]\label{defn:pseudodiff_operator}
An operator $P\colon C_c^\infty(E) \to C^\infty(F)$ is a \emph{pseudodifferential operator of order $k \in \IZ$}, if with respect to a uniformly locally finite covering $\{B_{2\varepsilon}(x_i)\}$ of $M$ with normal coordinate balls and corresponding subordinate partition of unity $\{\varphi_i\}$ as in Lemma \ref{lem:nice_coverings_partitions_of_unity} we can write
\begin{equation}
\label{eq:defn_pseudodiff_operator_sum}
P = P_{-\infty} + \sum_i P_i
\end{equation}
satisfying the following conditions:
\begin{itemize}
\item $P_{-\infty} \in \IU(E,F)$\footnote{see Definition \ref{defn:quasiloc_smoothing}}, i.e., it is a quasilocal smoothing operator,
\item for all $i$ the operator $P_i$ is with respect to synchronous framings of $E$ and $F$ in the ball $B_{2\varepsilon}(x_i)$ (cf. Definition \ref{defn:synchronous_framings}) a matrix of pseudodifferential operators on $\IR^m$ of order $k$ with support\footnote{An operator $P$ is \emph{supported in a subset $K$}, if $\supp Pu \subset K$ for all $u$ in the domain of $P$ and if $Pu = 0$ whenever we have $\supp u \cap K = \emptyset$.} in $B_{2\varepsilon}(0) \subset \IR^m$, and
\item the constants $C_i^{\alpha \beta}$ appearing in the bounds
\[\|D_x ^\alpha D_\xi^\beta p_i(x,\xi)\| \le C^{\alpha \beta}_i (1 + |\xi|)^{k - |\beta|}\]
of the symbols of the operators $P_i$ can be chosen to not depend on $i$, i.e., there are $C^{\alpha \beta} < \infty$ such that
\begin{equation}
\label{eq:uniformity_defn_PDOs}
C^{\alpha \beta}_i \le C^{\alpha \beta}
\end{equation}
for all multi-indices $\alpha, \beta$ and all $i$. We will call this the \emph{uniformity condition} for pseudodifferential operators on manifolds of bounded geometry.
\end{itemize}
We denote the set of all such operators by $\Psi \mathrm{DO}^k(E,F)$ and its subset consisting of all such operator with $P_{-\infty} \in \C(E,F)$ by $\Psi \mathrm{DO}_u^k(E,F)$.

If a statement holds for both version of pseudodifferential operators, i.e., with or without the subscript ``u'', we will write $\Psi \mathrm{DO}_?^k(E,F)$.
\end{defn}

From Lemma \ref{lem:transition_functions_uniformly_bounded} and Lemma \ref{lem:equiv_characterizations_bounded_geom_bundles} together with \cite[Theorem III.§3.12]{lawson_michelsohn} (and its proof which gives the concrete formula how the symbol of a pseudodifferential operator transforms under a coordinate change) we conclude that the above definition of pseudodifferential operators on manifolds of bounded geometry does neither depend on the chosen uniformly locally finite covering of $M$ by normal coordinate balls, nor on the subordinate partition of unity with uniformly bounded derivatives, nor on the synchronous framings of $E$ and $F$.

\begin{rem}
If the manifold $M$ happens to be compact, then the first bullet point of the above definition becomes vacuous, since every operator on a compact manifold is quasilocal. Furthermore, since on compact manifolds we only need finitely many charts to cover the manifold, the uniformity condition also becomes vacuous. We conclude that on compact manifolds our definition coincides with the usual one of pseudodifferential operators.

Furthermore, we are considering only operators that would correspond to H\"ormander's class $S_{1, 0}^k(\Omega)$, if we consider open subsets $\Omega$ of $\IR^m$ instead of an $m$-dimensional manifold $M$, i.e., we do not investigate operators corresponding to the more general classes $S_{\rho, \delta}^k(\Omega)$. The paper \cite[Definition 2.1]{hormander_ess_norm} is the one where H\"ormander introduced these classes.
\end{rem}

Recall that in the usual case of a compact manifold, a pseudodifferential operator $P$ of order $k$ has an extension to a continuous operator $H^s(E) \to H^{s-k}(F)$ for all $s \in \IZ$ (see, e.g., \cite[Theorem III.§3.17(i)]{lawson_michelsohn}). Due to the uniform local finiteness of the sum in \eqref{eq:defn_pseudodiff_operator_sum} and due to the Uniformity Condition \eqref{eq:uniformity_defn_PDOs}, this result does also hold in our case of a manifold of bounded geometry:

\begin{prop}\label{prop:pseudodiff_extension_sobolev}
Let $P \in \Psi \mathrm{DO}_?^k(E,F)$. Then $P$ has for all $s \in \IZ$ an extension to a continuous operator $P\colon H^s(E) \to H^{s-k}(F)$.
\end{prop}

\begin{rem}\label{rem:bound_operator_norm_PDO}
Later we will need the following fact: we can bound the operator norm of $P\colon H^s(E) \to H^{s-k}(F)$ from above by the maximum of the constants $C^{\alpha 0}$ with $|\alpha| \le K_s$ from the Uniformity Condition \eqref{eq:uniformity_defn_PDOs} for $P$ multiplied with a constant $C_s$, where $K_s \in \IN_0$ and $C_s$ only depend on $s \in \IZ$ and the dimension of the manifold $M$. This can be seen by carefully examining the proof of \cite[Proposition III.§3.2]{lawson_michelsohn} which is the above proposition for the compact case.\footnote{To be utterly concrete, we have to choose normal coordinate charts and a subordinate partition of unity as in Lemma \ref{lem:nice_coverings_partitions_of_unity} and also synchronous framings for $E$ and $F$ and then use Formula \eqref{eq:sobolev_norm_local} which gives Sobolev norms that can be computed locally and that are equivalent to the global norms \eqref{eq:sobolev_norm}.}
\end{rem}

Since we have $P = P_{-\infty} + \sum_i P_i$, where $P_{-\infty}$ is a quasilocal smoothing operator (resp., $P_{-\infty}$ is an element of the smooth uniform Roe algebra), the following corollary arises immediately:

\begin{cor}\label{cor:P_quasiloc}
Let $P \in \Psi \mathrm{DO}^k(E,F)$. Then $P$ is modulo quasilocal smoothing operators a pseudodifferential operator of finite propagation\footnote{see Definition \ref{defn:finite_prop_speed}}, and therefore $P$ is a quasilocal operator of order $k$\footnote{see Definition \ref{defn:quasiloc_ops}}.

If $P \in \Psi \mathrm{DO}_u^k(E,F)$, then $P$ is modulo operators from $\C(E,F)$ a pseudodifferential operator of finite propagation.
\end{cor}

With an analogous argumentation we can also extend the result for compact manifolds that the formal adjoint operator of a pseudodifferential operator of order $k$ is again a pseudodifferential operator of order $k$ to our case here:

\begin{lem}\label{lem:adjoint_pseudodiff_operator}
Let $P \in \Psi \mathrm{DO}_?^k(E,F)$. Then $P^\ast \in \Psi \mathrm{DO}_?^k(F,E)$.
\end{lem}

Let us define
\[\Psi \mathrm{DO}_?^{-\infty}(E,F) := \bigcap_k \Psi \mathrm{DO}_?^k(E,F).\]

We will show $\Psi \mathrm{DO}^{-\infty}(E,F) = \IU(E,F)$: from the previous Proposition \ref{prop:pseudodiff_extension_sobolev} we conclude that an operator $P \in \Psi \mathrm{DO}^{-\infty}(E,F)$ is a smoothing operator (using Lemma \ref{lem:smoothing_operator_iff_bounded}). Since we can write $P = P_{-\infty} + \sum_i P_i$, where $P_{-\infty} \in \IU(E,F)$ and the $P_i$ are supported in balls with uniformly bounded radii, the operator $\sum_i P_i$ is of finite propagation. So $P$ is the sum of a quasilocal smoothing operator $P_{-\infty}$ and a smoothing operator $\sum_i P_i$ of finite propagation, and therefore a quasilocal smoothing operator. Due to Lemma \ref{lem:adjoint_pseudodiff_operator} the same arguments also apply to the adjoint $P^\ast$ of $P$, so that in the end we can conclude $P \in \IU(E,F)$, i.e., we have shown $\Psi \mathrm{DO}^{-\infty}(E,F) \subset \IU(E,F)$.

Since the other inclusion does hold by definition, we get the claim. Furthermore, the same arguments apply to $\Psi \mathrm{DO}_u^{-\infty}(E,F)$ showing $\Psi \mathrm{DO}_u^{-\infty}(E,F) = \C(E,F)$.\footnote{Of course, our definition of pseudodifferential operators was arranged such that this lemma holds.}

\begin{lem}\label{lem:PDO_-infinity_equal_quasilocal_smoothing}
$\Psi \mathrm{DO}^{-\infty}(E,F) = \IU(E,F)$ and $\Psi \mathrm{DO}_u^{-\infty}(E,F) = \C(E,F)$.
\end{lem}

One of the important properties of pseudodifferential operators on compact manifolds is that the composition of an operator $P \in \Psi \mathrm{DO}_?^k(E,F)$ and $Q \in \Psi \mathrm{DO}_?^l(F,G)$ is again a pseudodifferential operator of order $k+l$: $PQ \in \Psi \mathrm{DO}_?^{k+l}(E,G)$. We can prove this also in our setting by writing
\begin{align*}
PQ & = \Big(P_{-\infty} + \sum_i P_i\Big) \Big(Q_{-\infty} + \sum_j Q_j\Big)\\
& = P_{-\infty} Q_{-\infty} + \sum_i P_i Q_{-\infty} + \sum_j P_{-\infty} Q_j + \sum_{i,j} P_i Q_j
\end{align*}
and then arguing as follows.

\begin{itemize}
\item The first summand is an element of $\IU(E,G)$, resp. of $\C(E,G)$: it was shown in \cite[Proposition 5.2]{roe_index_1} that the composition of two quasilocal operators is again quasilocal and it is clear that composing smoothing operators again gives smoothing operators, resp. it is easy to see that composing two operators which may be approximated by finite propagation operators again gives such an operator.

\item The second and third summands are also elements of $\IU(E,G)$, resp. of $\C(E,G)$ due to Proposition \ref{prop:pseudodiff_extension_sobolev} and since the sums are uniformly locally finite, the operators $P_i$ and $Q_j$ are supported in coordinate balls of uniform radii (i.e., have finite propagation which is uniformly bounded from above) and their operator norms are uniformly bounded due to the uniformity condition in the definition of pseudodifferential operators.

\item The last summand is a uniformly locally finite sum of pseudodifferential operators of order $k+l$ (here we use the corresponding result for compact manifolds) and to see the Uniformity Condition \eqref{eq:uniformity_defn_PDOs} we use \cite[Theorem III.§3.10]{lawson_michelsohn}: it states that the symbol of $P_i Q_j$ has formal development $\sum_\alpha \frac{i^{|\alpha|}}{\alpha !} (D_\xi^\alpha p_i)(D_x^\alpha q_j)$. So we may deduce the uniformity condition for $P_i Q_j$ from the one for $P_i$ and for $Q_j$.
\end{itemize}

Another properties that immediately generalize from the compact to the bounded geometry case is firstly, that the commutator of two pseudodifferential operators is of one order lower than it should a priori be, and secondly, that multiplication with a function $f \in C_b^\infty(M)$ defines a pseudodifferential operator of order $0$.

So together with Lemma \ref{lem:adjoint_pseudodiff_operator} we have the following important proposition:

\begin{prop}\label{prop:PsiDOs_filtered_algebra}
$\Psi \mathrm{DO}_?^\ast(E)$ is a filtered $^\ast$-algebra, i.e., for all $k, l \in\IZ $ we have
\[\Psi \mathrm{DO}_?^k(E) \circ \Psi \mathrm{DO}_?^l(E) \subset \Psi \mathrm{DO}_?^{k+l}(E),\]
and so $\Psi \mathrm{DO}_?^{-\infty}(E)$ is a two-sided $^\ast$-ideal in $\Psi \mathrm{DO}_?^\ast(E)$.

Furthermore, we have $[\Psi \mathrm{DO}_?^k(E), \Psi \mathrm{DO}_?^l(E)] \subset \Psi \mathrm{DO}_?^{k+l-1}(E)$ for all $k,l \in \IZ$, and multiplication with a function $f \in C_b^\infty(M)$ defines a pseudodifferential operator of order $0$.
\end{prop}

The last property that generalizes to our setting and that we want to mention is the following (the proof of \cite[Theorem III.§3.9]{lawson_michelsohn} generalizes directly):

\begin{prop}\label{prop:Pu_smooth_if_u_smooth}
Let $P \in \Psi \mathrm{DO}_?^k(E,F)$ be a pseudodifferential operator of arbitrary order and let $u \in H^s(E)$ for some $s \in \IZ$.

Then, if $u$ is smooth on some open subset $U \subset M$, $Pu$ is also smooth on $U$.
\end{prop}

\section{Principal symbols}

In this section we will discuss symbols of pseudodifferential operators and the corresponding symbol maps.

Let $\pi^\ast E$ and $\pi^\ast F$ denote the pull-back bundles of $E$ and $F$ to the cotangent bundle $\pi\colon T^\ast M \to M$ of the $m$-dimensional manifold $M$.

\begin{defn}[Symbols]
Let $p$ be a section of the bundle $\Hom(\pi^\ast E, \pi^\ast F)$ over $T^\ast M$. We call $p$ a \emph{symbol of order $k \in \IZ$}, if the following holds: choosing a uniformly locally finite covering $\{ B_{2 \varepsilon}(x_i) \}$ of $M$ through normal coordinate balls and corresponding subordinate partition of unity $\{ \varphi_i \}$ as in Lemma \ref{lem:nice_coverings_partitions_of_unity}, and choosing synchronous framings of $E$ and $F$ in these balls $B_{2\varepsilon}(x_i)$ (cf. Definition \ref{defn:synchronous_framings}), we can write $p$ as a uniformly locally finite sum $p = \sum_i p_i$, where $p_i(x,\xi) := p(x,\xi) \varphi(x)$ for $x \in M$ and $\xi \in T^\ast_x M$, and interpret each $p_i$ as a matrix-valued function on $B_{2 \varepsilon}(x_i) \times \IC^m$. Then for all multi-indices $\alpha$ and $\beta$ there must exist a constant $C^{\alpha \beta} < \infty$ such that for all $i$ and all $x, \xi$ we have
\begin{equation}\label{eq:symbol_uniformity}
\|D^\alpha_x D^\beta_\xi p_i(x,\xi) \| \le C^{\alpha \beta}(1 + |\xi|)^{k - |\beta|}.
\end{equation}
We denote the vector space all symbols of order $k \in \IZ$ by $\Symb^k(E,F)$.
\end{defn}

From Lemma \ref{lem:transition_functions_uniformly_bounded} and Lemma \ref{lem:equiv_characterizations_bounded_geom_bundles} we conclude that the above definition of symbols does neither depend on the chosen uniformly locally finite covering of $M$ through normal coordinate balls, nor on the subordinate partition of unity (as long as the functions $\{\varphi_i\}$ have uniformly bounded derivatives), nor on the synchronous framings of $E$ and $F$.

If all the choices made above are fixed, we immediately see from the definition of pseudodifferential operators that an operator $P \in \Psi \mathrm{DO}_?^k(E,F)$ has a symbol $p \in \Symb^k(E,F)$. Analogously as in the case of compact manifolds,\footnote{see, e.g., \cite[Theorem III.§3.19]{lawson_michelsohn}} we are able to show that if we make other choices for the coordinate charts, subordinate partition of unity and synchronous framings, the symbol $p$ of $P$ changes by an element of $\Symb^{k-1}(E,F)$. This means that the operator $P$ has a well-defined principal symbol class $[p] \in \Symb^k(E,F) / \Symb^{k-1}(E,F)$.

\begin{defn}[Principal symbol of an operator]
We define
\[\Symb^{k-[1]}(E,F) := \Symb^k(E,F) / \Symb^{k-1}(E,F).\footnote{This suggestive notation is taken from \cite[Chapter 6]{melrose_microlocal_lecture}.}\]

The symbol class $[p] \in \Symb^{k-[1]}(E,F)$ of an operator $P \in \Psi \mathrm{DO}_?^k(E,F)$ is called the \emph{principal symbol} of $P$, and we denote the symbol map associating an operator of order $k$ its principal symbol by
\[\sigma^k \colon \Psi \mathrm{DO}_?^k(E,F) \to \Symb^{k-[1]}(E,F).\]
\end{defn}

Analogously as in the case of compact manifolds one proves the following basic properties of symbol maps. As a reference for the proof in the compact case one can use, e.g., \cite[Proposition 6.3, Formula (6.44) and Chapter 6.5]{melrose_microlocal_lecture}.

\begin{prop}[Basic properties of the symbol maps]\label{prop:basics_symbol_maps}
For all $k, k^\prime \in \IZ$ the corresponding symbol maps enjoy the following properties:
\begin{itemize}
\item $\sigma^k(P + Q) = \sigma^k(P) + \sigma^k(Q)$,
\item $\sigma^{k + k^\prime}(P\circ Q^\prime) = \sigma^k(P) \circ \sigma^{k^\prime}(Q^\prime)$,
\item $\sigma^k(P^\ast) = \overline{\sigma^k(P)}$ and
\item $\sigma^k(f \cdot P) = f \cdot \sigma^k(P)$ for all $f \in C_b^\infty(M)$,
\end{itemize}
where $P, Q \in \Psi \mathrm{DO}_?^k(E,F)$ and $Q^\prime \in \Psi \mathrm{DO}_?^{k^\prime}(D, E)$.

Furthermore, we have for all $k \in \IZ$ the following short, exact sequence:
\[ 0 \to \Psi \mathrm{DO}_?^{k-1}(E,F) \hookrightarrow \Psi \mathrm{DO}_?^{k}(E,F) \stackrel{\sigma^k}\twoheadrightarrow \Symb^{k-[1]}(E,F) \to 0.\]
It follows that for all $k \in \IZ$ the linear map $\sigma^k$ induces an isomorphism of vector spaces
\begin{equation}\label{eq:linear_iso_symbol_maps}
\Psi \mathrm{DO}_?^{k-[1]}(E,F) \cong \Symb^{k-[1]}(E,F).
\end{equation}
\end{prop}

In the above proposition we have already used the fact that to every given symbol $p \in \Symb^k(E,F)$ we can construct an operator $P \in \Psi \mathrm{DO}_?^k(E,F)$ with $\sigma^k(P) = [p]$. An additional reference for such a construction is \cite[End of Chapter III.§3]{lawson_michelsohn}. Again, different choices of coordinates, partition of unity and framings used in the construction lead to another pseudodifferential operator $P^\prime$ with the same principal symbol $[p]$, but differing from $P$ only by an operator of order $k-1$. So we get for all $k \in \IZ$ well-defined maps
\[\Symb^k(E,F) \to \Psi \mathrm{DO}_?^{k - [1]}(E,F)\]
which have analogous properties as the symbol maps, i.e., the above Proposition \ref{prop:basics_symbol_maps} holds analogously for these maps (especially, the induced isomorphisms of vector spaces $\Symb^{k-[1]}(E,F) \cong \Psi \mathrm{DO}_?^{k-[1]}(E,F)$ are the inverses to \eqref{eq:linear_iso_symbol_maps}).

\section{Uniformity of operators of nonpositive order}\label{sec:uniformity_PDOs}

Now we get to an important section, namely the one where we show that the pseudodifferential operators that we have defined\footnote{The results of this section only apply to operators from $\Psi \mathrm{DO}_u^\ast(E)$. Note the subscript ``u''.} are uniform. Note that we have not yet defined what ``uniform'' shall mean, i.e., it is also the task of this section to make this notion precise. Concretely, recall that we have the notions of pseudolocality and of local compactness which state that $[f, P]$, resp. $fP$ and $Pf$, are compact operators for all functions $f \in C_c(M)$. Now the definition of pseudodifferential operators that we have given requieres the uniformity condition, i.e., that we can estimate the norms of the symbols uniformly with respect to the location in the manifold. We will see that this uniformity condition leads to the fact that we may estimate the degree of compactness of the operators $[f, P]$, resp. $fP$ and $Pf$, uniformly with respect to the location in the manifold where $f$ is supported. We will have to start with defining what we mean with the degree of compactness.

Let $T \in \IK(L^2(E))$ be a compact operator. We know that $T$ is the limit of finite rank operators, i.e., for every $\varepsilon > 0$ there is a finite rank operator $k$ such that $\|T - k\| < \varepsilon$. Now given a collection $\mathcal{A} \subset \IK(L^2(E))$ of compact operators, it may happen that for every $\varepsilon > 0$ the rank needed to approximate an operator from $\mathcal{A}$ may be bounded from above by a common bound for all operators. This is formalized in the following definition:

\begin{defn}[Uniformly approximable collections of operators]\label{defn:uniformly_approximable_collection}
A collection of operators $\mathcal{A} \subset \IK(L^2(E))$ is said to be \emph{uniformly approximable}, if for every $\varepsilon > 0$ there is an $N > 0$ such that for every $T \in \mathcal{A}$ there is a rank-$N$ operator $k$ with $\|T - k\| < \varepsilon$.
\end{defn}

\begin{examples}\label{ex:uniformly_approximable_collections}
Every collection of finite rank operators with uniformly bounded rank is uniformly approximable.

Furthermore, every finite collection of compact operators is uniformly approximable and so also every totally bounded subset of $\IK(L^2(E))$.

The converse is in general false since a uniformly approximable family need not be bounded (take infinitely many rank-$1$ operators with operator norms going to infinity).

Even if we assume that the uniformly approximable family is bounded we do not necessarily get a totally bounded set: let $(e_i)_{i \in \IN}$ be an orthonormal basis of $L^2(E)$ and $P_i$ the orthogonal projection onto the $1$-dimensional subspace spanned by the vector $e_i$. Then the collection $\{P_i\} \subset \IK(L^2(E))$ is uniformly approximable (since all operators are of rank $1$) but not totally bounded (since $\|P_i - P_j\| = 1$ for $i \not= j$)\footnote{Another way to see that the collection $\{P_i\}$ is not totally bounded is to use the characterization of totally bounded subsets of $\IK(H)$ for $H$ a Hilbert space from \cite[Theorem 3.5]{anselone_palmer}: $\mathcal{A} \subset \IK(H)$ is totally bounded if and only if both $\mathcal{A}$ and $\mathcal{A}^\ast$ are collectively compact, i.e., the sets $\{T v \ | \ T \in \mathcal{A}, v \in H \text{ with } \|v\| = 1\} \subset H$ and $\{T^\ast v \ | \ T \in \mathcal{A}, v \in H \text{ with } \|v\| = 1\} \subset H$ have compact closure. If $H$ is infinite-dimensional, then this is of course not the case for the collection $\{P_i\}$ in question.}.
\end{examples}

Let us define
\begin{equation*}\label{defn:L-Lip_R(M)}
\LLip_R(M) := \{ f \in C_c(M) \ | \ f \text{ is }L\text{-Lipschitz}, \diam(\supp f) \le R \text{ and } \|f\|_\infty \le 1\}.
\end{equation*}

\begin{defn}[{\cite[Definition 2.3]{spakula_uniform_k_homology}}]\label{defn:uniform_operators_manifold}
Let $T \in \IB(L^2(E))$. We say that $T$ is \emph{uniformly locally compact}, if for every $R, L > 0$ the collection
\[\{fT, Tf \ | \ f \in \LLip_R(M)\}\]
is uniformly approximable.

We say that $T$ is \emph{uniformly pseudolocal}, if for every $R, L > 0$ the collection
\[\{[T, f] \ | \ f \in \LLip_R(M)\}\]
is uniformly approximable.
\end{defn}

\begin{rem}\label{rem:renaming_l_dependence}
In \cite{spakula_uniform_k_homology} uniformly locally compact operators were called ``$l$-uniform'' and uniformly pseudolocal operators ``$l$-uniformly pseudolocal''. Moreover, \Spakula defined also versions of the notions without the $l$-dependence. The reason why we change the names is that the $l$-versions of these notions are the correct ones, i.e., the ones that we will use, and we will never need the versions without the $l$-dependence. So we have just discarded the version without $l$-dependence and renamed the one with $l$-dependence to make things easier to write down.
\end{rem}

\begin{defn}\label{defn_Du_Cu_manifolds}
Let $D^\ast_u(E) \subset \IB(L^2(E))$ denote the $C^\ast$-algebra generated by all uniformly pseudolocal operators with finite propagation, and $C^\ast_u(E) \subset \IB(L^2(E))$ the $C^\ast$-algebra generated by all uniformly locally compact operators having finite propagation.\footnote{We have a clash of notation with the uniform Roe $C^\ast$-algebra $C_u^\ast(Y)$ from Definition \ref{defn:C_u^ast(Y)} for a uniformly discrete metric space $Y$ with coarsely bounded geometry. But in Proposition \ref{prop:IU(E)_dense_Cu*(E)} we will show that if $Y \subset M$ is a uniformly discrete quasi-lattice, then
\[C_u^\ast(E) \cong C_k^\ast(Y) \cong C_u^\ast(Y) \otimes \IK(\ell^2(\IN)),\]
where the second isomorphism is from Lemma \ref{lem:iso_discrete_versions_uniform_roe}.}
\end{defn}

With similar arguments as in the proof of \cite[Lemma 4.2]{spakula_uniform_k_homology} we can show that $C^\ast_u(E) \subset D^\ast_u(E)$ is a closed, two-sided $^\ast$-ideal.

The rest of this section is devoted to showing that pseudodifferential operators of negative order are uniformly locally compact and that pseudodifferential operators of order $0$ are uniformly pseudolocal. We will start with the operators of negative order.

\begin{prop}\label{prop:quasilocal_negative_order_uniformly_locally_compact}
Let $A\in \IB(L^2(E))$ be a finite propagation operator of negative order $k < 0$\footnote{See Definition \ref{defn:quasiloc_ops}. Note that we do not assume that $A$ is a pseudodifferential operator.} such that its adjoint $A^\ast$ also has finite propagation and is of negative order $k^\prime < 0$. Then $A$ is uniformly locally compact. Even more, the collection
\[\{fT, Tf \ | \ f \in B_R(M)\}\]
is uniformly approximable for all $R, L > 0$, where $B_R(M)$ consists of all bounded Borel functions $h$ on $M$ with $\diam(\supp h) \le R$ and $\|h\|_\infty \le 1$.
\end{prop}

\begin{proof}
Let $f \in B_R(M)$, $K := \supp f \subset M$ and $r$ be the propagation of $A$. The operator $\chi A f = A f$, where $\chi$ is the characteristic function of $B_r(K)$, factores as
\[L^2(E) \stackrel{\cdot f}\longrightarrow L^2(E|_K) \stackrel{\chi \cdot A}\longrightarrow H^{-k}(E|_{B_r(K)}) \hookrightarrow L^2(E|_{B_r(K)}) \to L^2(E).\]
The following properties hold:
\begin{itemize}
\item multiplication with $f$ has operator norm $\le 1$, since $\|f\|_\infty \le 1$, and analogously for the multiplication with $\chi$,
\item the norm of $\chi \cdot A\colon L^2(E|_K) \to H^{-k}(E|_{B_r(K)})$ can be bounded from above by the norm of $A\colon L^2(E) \to H^{-k}(E)$ (i.e., the upper bound does not depend on $K$ nor $r$),
\item the inclusion $H^{-k}(E|_{B_r(K)}) \hookrightarrow L^2(E|_{B_r(K)})$ is compact (due to  the Theorem of Rellich--Kondrachov) and this compactness is uniform, i.e., its approximability by finite rank operators\footnote{Here we mean the existence of an upper bound on the rank needed to approximate the operator by finite rank operators, given an $\varepsilon > 0$.} depends only on $R$ (the upper bound for the diameter of $\supp f$) and $r$, but not on $K$ (this uniformity is due to the bounded geometry of $M$ and of the bundles $E$ and $F$), and
\item the inclusion $L^2(E|_{B_r(K)}) \to L^2(E)$ is of norm $\le 1$.
\end{itemize}
From this we conclude that the operator $\chi A f = A f$ is compact and this compactness is uniform, i.e., its approximability by finite rank operators depends only on $R$ and $r$. So we can conclude that $\{Af \ | \ f \in B_R(M)\}$ is uniformly approximable.

Applying the same reasoning to the adjoint operator,\footnote{By assumption the adjoint operator also has finite propagation and is of negative order. So we conclude that $\{A^\ast f \ | \ f \in B_R(M)\}$ is uniformly approximable. But a collection $\mathcal{A}$ of compact operators is uniformly approximable if and only if the adjoint family $\mathcal{A}^\ast$ is uniformly approximable. So we get that $\{(A^\ast f)^\ast = \overline{f} A \ | \ f \in B_R(M)\}$ is uniformly approximable.} we conclude that $A$ is uniformly locally compact.
\end{proof}

Using an approximation argument we may also show the following corollary:

\begin{cor}\label{cor:quasilocal_neg_order_uniformly_locally_compact}
Let $A$ be a quasilocal operator of negative order and let the same hold true for its adjoint $A^\ast$. Then $A$ is uniformly locally compact, resp., we even get the stronger statement as in the above Proposition \ref{prop:quasilocal_negative_order_uniformly_locally_compact}.
\end{cor}

\begin{proof}
We have to show that $\{Af \ | \ f \in B_R(M)\}$ is uniformly approximable. Let $\varepsilon > 0$ be given and let $r_\varepsilon$ be such that $\mu_A(r) < \varepsilon$ for all $r \ge r_\varepsilon$, where $\mu_A$ is the dominating function of $A$. Then $\chi_{B_{r_\varepsilon}(\supp f)} A f$ is $\varepsilon$-away from $Af$ and the same reasoning as in the proof of the above Proposition \ref{prop:quasilocal_negative_order_uniformly_locally_compact} shows that the approximability (up to an error of $\varepsilon$) of $\chi_{B_{r_\varepsilon}(\supp f)} A f$ does only depend on $R$ and $r_\varepsilon$. From this the claim that $\{Af \ | \ f \in B_R(M)\}$ is uniformly approximable follows.

Using the adjoint operator and the same arguments for it, we conclude that $A$ is uniformly locally compact.
\end{proof}

Due to Corollary \ref{cor:P_quasiloc} and Lemma \ref{lem:adjoint_pseudodiff_operator} we immediately get the following corollary:

\begin{cor}\label{prop:PsiDOs_negative_order_uniformly_locally_compact}
Let $P \in \Psi \mathrm{DO}_u^k(E)$ be a pseudodifferential operator of negative order $k < 0$. Then $P \in C_u^\ast(E)$.
\end{cor}

Recall from Lemma \ref{lem:quasi_local_smoothing_dense_uniform_roe} that we have a non-canonical, continuous inclusion $\C(E) \hookrightarrow C_k^\ast(Y)$ with dense image, where $Y \subset M$ is a uniformly discrete quasi-lattice. But examining its proof more closely, we see that we actually show there the statement that $\C(E) \subset C_u^\ast(E)$ is a dense subset and that we have a non-canonical isomorphism $C_u^\ast(E) \cong C_k^\ast(Y)$ coming from the discretization procedure. Now in Proposition \ref{prop:smooth_version_of_uniform_roe} we have shown that all these non-canonical dense inclusions $\C(E) \hookrightarrow C_k^\ast(Y)$ induce the same isomorphism on $K$-theory. It is of course the same for the isomorphisms $C_u^\ast(E) \cong C_k^\ast(Y)$, i.e., they induce the same map on $K$-theory. Together with the above corollary we therefore get the following proposition. Note that its last statement that $K_\ast(C_u^\ast(E)) \cong K_\ast(C_k^\ast(Y))$ was already proved by \Spakula in \cite[Remark 8.6]{spakula_uniform_k_homology}.

\begin{prop}\label{prop:IU(E)_dense_Cu*(E)}
We have dense inclusions $\C(E) \subset \Psi \mathrm{DO}_u^{-1}(E) \subset C_u^\ast(E)$.

Furthermore, we have non-canonical isomorphisms $C_u^\ast(E) \cong C_k^\ast(Y)$, where $Y$ is a uniformly discrete quasi-lattice in $M$, and all induce the same natural isomorphism on $K$-theory.
\end{prop}

\begin{rem}
Note that if $M$ is compact, then the statement of the above proposition is well-known. Concretely, in the compact case we have that $\C(E)$ is the algebra of all smoothing operators, and $C_u^\ast(E)$ the algebra of all compact operators. Now the fact that the smoothing operators are dense in the compact operators for a compact manifold is easily shown.
\end{rem}

Let us now get to the case of pseudodifferential operators of order $0$, where we want to show that such operators are uniformly pseudolocally compact.

Recall the following fact in the case that the manifold $M$ is compact: an operator $T$ is pseudolocal\footnote{That is to say, $[T, f]$ is a compact operator for all $f \in C(M)$.} if and only if $f T g$ is a compact operator for all $f, g \in C(M)$ with disjoint supports. This observation is due to Kasparov and a proof might be found in, e.g., \cite[Proposition 5.4.7]{higson_roe}. We can add another equivalent characterization which is basically also proved in the cited proposition: an operator $T$ is pseudolocal if and only if $f T g$ is a compact operator for all bounded Borel functions $f$ and $g$ on $M$ with disjoint supports.

We have analogous equivalent characterizations for uniformly pseudolocal operators, which we will state in the following lemma. The proof of it is analogous to the compact case (and uses the fact that the subset of all uniformly pseudolocal operators is closed in operator norm, which is proved in \cite[Lemma 4.2]{spakula_uniform_k_homology}). Furthermore, in order to prove that the Points 4 and 5 in the statement of the lemma are equivalent to the other points we need the bounded geometry of $M$ since we have to smooth functions and simultaneously control their derivatives. For the convenience of the reader we will give a full proof of the lemma.

Let us introduce the notions $B_b(M)$ for all bounded Borel functions on the manifold $M$ and $B_R(M)$ for its subset consisting of all Borel function $h$ with $\diam(\supp h) \le R$ and $\|h\|_\infty \le 1$.

\begin{lem}\label{lem:kasparov_lemma_uniform_approx_manifold}
The following are equivalent for an operator $T \in \IB(L^2(E))$:
\begin{enumerate}
\item $T$ is uniformly pseudolocal,
\item for all $R, L > 0$ the following collection is uniformly approximable:
\[\{f T g, g T f \ | \ f \in B_b(M), \ \! \|f\|_\infty \le 1, \ \! g \in \LLip_R(M), \ \! \supp f \cap \supp g = \emptyset\},\]
\item for all $R, L > 0$ the following collection is uniformly approximable:
\[\{f T g, g T f \ | \ f \in B_b(M), \ \! \|f\|_\infty \le 1, \ \! g \in B_R(M), \ \! d(\supp f, \supp g) \ge L\},\]
\item for every $L > 0$ there is a sequence $(L_j)_{j \in \IN}$ of positive numbers (not depending on the operator $T$) such that
\begin{align*}
\{ f T g, g T f \ | \ & f \in B_b(M)\text{ with }\|f\|_\infty \le 1,\\
& g \in B_R(M) \cap C_b^\infty(M)\text{ with }\|\nabla^j g\|_\infty \le L_j,\text{ and}\\
& \supp f \cap \supp g = \emptyset\}
\end{align*}
is uniformly approximable for all $R, L > 0$.
\item for every $L > 0$ there is a sequence $(L_j)_{j \in \IN}$ of positive numbers (not depending on the operator $T$) such that
\[\{ [T,g] \ | \ g \in B_R(M) \cap C_b^\infty(M)\text{ with }\|\nabla^j g\|_\infty \le L_j\}\]
is uniformly approximable for all $R, L > 0$.
\end{enumerate}
\end{lem}

\begin{proof}
\bm{$1 \Rightarrow 2$}\textbf{:} Let $f \in B_b(M)$ with $\|f \|_\infty \le 1$ and $g \in \LLip_R(M)$ have disjoint supports, i.e., $\supp f \cap \supp g = \emptyset$. From the latter we conclude $f T g = f [T,g]$, from which the claim follows (because $T$ is uniformly pseudolocal and because the operator norm of multiplication with $f$ is $\le 1$). Of course such an argument also works with the roles of $f$ and $g$ changed.

\bm{$2 \Rightarrow 3$}\textbf{:} Let $f \in B_b(M)$ with $\|f \|_\infty \le 1$ and $g \in B_R(M)$ with $d(\supp f, \supp g) \ge L$. We define $g^\prime(x) := \max\big( 0, 1 - \sfrac{1}{L} \cdot d(x, \supp g) \big) \in \sfrac{1}{L}\text{-}\operatorname{Lip}_{R+2L}(M)$. Since $g^\prime g = g$, the claim follows from writing $f T g = f T g^\prime g$ and because multiplication with $g$ has operator norm $\le 1$, and we of course also may change the roles of $f$ and $g$.

\bm{$3 \Rightarrow 1$}\textbf{:} Let $f \in \LLip_R(M)$. For given $\varepsilon > 0$ we partition the range of $f$ into a sequence of non-overlapping half-open intervals $U_1, \ldots, U_n$, each having diameter less than $\varepsilon$, such that $\overline{U_i}$ intersects $\overline{U_j}$ if and only if $|i - j| \le 1$. Denoting by $\chi_i$ the characteristic function of $f^{-1}(U_i)$, we get that $\chi_i \in B_R(M)$ if $0 \notin U_i$, since the support of $f$ has diameter less than or equal to $R$, and furthermore $d(\supp \chi_i, \supp \chi_j) \ge \tfrac{\varepsilon}{L}$ if $|i-j| > 1$, since $f$ is $L$-Lipschitz.

By Point 3 we have that the collections $\{\chi_i T \chi_j, \chi_j T \chi_i\}$ are uniformly approximable for all $i,j$ with $|i-j| > 1$. Choosing points $x_1, \ldots, x_n$ from $f^{-1}(U_1), \ldots, f^{-1}(U_n)$ and defining $f^\prime := f(x_1) \chi_1 + \cdots + f(x_n) \chi_n$, we get $\|f - f^\prime\|_\infty < \varepsilon$. The operator $[T,f]$ is $2\varepsilon \|T\|$-away from $[T, f^\prime]$, and since $\chi_1 + \cdots + \chi_n = 1$ we have
\[Tf^\prime - f^\prime T = \sum_{i,j} \chi_j T f(x_i)\chi_i - f(x_j) \chi_j T \chi_i.\]
Since we already know that $\{\chi_i T \chi_j, \chi_j T \chi_i\}$ are uniformly approximable for all $i,j$ with $|i-j| > 1$, it remains to treat the sum (note that the summand for $i=j$ is zero)
\[\sum_{|i-j|=1} \chi_j T f(x_i)\chi_i - f(x_j) \chi_j T \chi_i = \sum_{|i-j|=1} \big( f(x_i) - f(x_j) \big) \chi_j T \chi_i.\]
We split the sum into two parts, one where $i = j+1$ and the other one where $i=j-1$. The first part takes the form
\[\sum_j \big( f(x_{j+1}) - f(x_j) \big) \chi_j T \chi_{j+1},\]
i.e., is a direct sum of operators from $\chi_{j+1} \cdot L^2(E)$ to $\chi_j \cdot L^2(E)$. Therefore its norm is the maximum of the norms of its summands. But the latter are $\le 2 \varepsilon \|T\|$ since $|f(x_{j+1}) - f(x_j)| \le 2\varepsilon$. We treat the second part of the sum in the above display the same way and conclude that the sum in the above display is in norm $\le 4\varepsilon T$. Putting it all together it follows that $T$ is the operator norm limit of uniformly pseudolocal operators, from which it follows that $T$ itself is uniformly pseudolocal (it is proved in \cite[Lemma 4.2]{spakula_uniform_k_homology} that the uniformly pseudolocal operators are closed in operator norm, as are also the uniformly locally compact ones).

\bm{$2 \Rightarrow 4$}\textbf{:} Clear. We have to set $L_1 := L$ and the other values $L_{j \ge 2}$ do not matter (i.e., may be set to something arbitrary).

\bm{$4 \Rightarrow 3$}\textbf{:} This is similar to the proof of $2 \Rightarrow 3$, but we have to smooth the function $g^\prime$ constructed there. Let us make this concrete, i.e., let $f \in B_b(M)$ with $\|f \|_\infty \le 1$ and $g \in B_R(M)$ with $d(\supp f, \supp g) \ge L$ be given. We define
\[g^\prime(x) := \max\big( 0, 1 - \sfrac{2}{L} \cdot d(x, B_{\sfrac{1}{4} L}(\supp g)) \big) \in \sfrac{2}{L}\text{-}\operatorname{Lip}_{R+2\sfrac{3}{4}L}(M).\]
Note that $g^\prime \equiv 1$ on $B_{\sfrac{1}{4} L}(\supp g)$ and $g^\prime \equiv 0$ outside $B_{\sfrac{3}{4} L}(\supp g)$. We cover $M$ by normal coordinate charts and choose a ``nice'' subordinate partition of unity $\varphi_i$ as in Lemma \ref{lem:nice_coverings_partitions_of_unity}. If $\psi$ is now a mollifier on $\IR^m$ supported in $B_{\sfrac{1}{8}L}(0)$, we apply it in every normal coordinate chart to $\varphi_i g^\prime$ and reassemble then all the mollified parts of $g^\prime$ again to a (now smooth) function $g^\prime{}^\prime$ on $M$. This function $g^\prime{}^\prime$ is now supported in $B_{\sfrac{7}{8} L}(\supp g)$, and is constantly $1$ on $B_{\sfrac{1}{8} L}(\supp g)$. So $f T g = f T g^\prime{}^\prime g$ from which we may conclude the uniform approximability of the collection $\{f T g\}$ for $f$ and $g$ satisfying $f \in B_b(M)$ with $\|f \|_\infty \le 1$ and $g \in B_R(M)$ with $d(\supp f, \supp g) \ge L$. Note that the constants $L_j$ appearing in $\|\nabla^j g^\prime{}^\prime\|_\infty \le L_j$ depend on $L$, $\varphi_i$ and $\psi$, but not on $f$, $g$ or $R$. The dependence on $\varphi_i$ and $\psi$ is ok, since we may just fix a particular choice of them (note that the choice of $\psi$ also depends on $L$), and the dependence on $L$ is explicitly stated in the claim.

Of course we may also change the roles of $f$ and $g$ in this argument.

\bm{$5 \Rightarrow 4$}\textbf{:} Clear. We just have to write $fTg = f[T,g]$ and analogously for $gTf$.

\bm{$1 \Rightarrow 5$}\textbf{:} Clear.
\end{proof}

With the above lemma at our disposal we may now prove the following proposition.

\begin{prop}\label{prop:PDO_order_0_l-uniformly-pseudolocal}
Let $P \in \Psi \mathrm{DO}_u^0(E)$. Then $P \in D_u^\ast(E)$.
\end{prop}

\begin{proof}
Writing $P = P_{-\infty} + \sum_i P_i$ with $P_{-\infty} \in \C(E) \subset C_u^\ast(E) \subset D_u^\ast(E)$ (the first inclusion is due to the above Corollary \ref{prop:PsiDOs_negative_order_uniformly_locally_compact}), we may without loss of generality assume that $P$ has finite propagation $R^\prime$. So we have to show that $P$ is uniformly pseudolocal.

To show this we will use the equivalent characterization in Point 4 of the above lemma: let $R, L > 0$ and the corresponding sequence $(L_j)_{j \in \IN}$ be given. We have to show that
\begin{align*}
\{ f P g, g P f \ | \ & f \in B_b(M)\text{ with }\|f\|_\infty \le 1,\\
& g \in B_R(M) \cap C_b^\infty(M)\text{ with }\|\nabla^j g\|_\infty \le L_j,\text{ and}\\
& \supp f \cap \supp g = \emptyset\}
\end{align*}
is uniformly approximable for all $R, L > 0$.

We have
\[f P g = f \chi_{B_{R^\prime}(\supp g)} P g = f \chi_{B_{R^\prime}(\supp g)} [P, g]\]
since the supports of $f$ and $g$ are disjoint.

With Proposition \ref{prop:PsiDOs_filtered_algebra} we conclude that multiplication with $g$ is a pseudodifferential operator of order $0$ (since $g \in C_b^\infty(M)$) and furthermore, that the commutator $[P, g]$ is a pseudodifferential operator of order $-1$. Therefore, by the above Corollary \ref{prop:PsiDOs_negative_order_uniformly_locally_compact}, we know that the set $\{f \chi_{B_{R^\prime}(\supp g)} [P, g] \ | \ f \in B_R(M)\}$ is uniformly approximable. So we conclude that our operators $f[P,g]$ have the needed uniformity in the functions $f$.

It remains to show that we also have the needed uniformity in the functions $g$. Writing $P = \sum_i P_i$\footnote{Recall that we assumed without loss of generality that there is no $P_{-\infty}$.}, we get $[P,g] = \sum_i [P_i,g]$. Now each $[P_i, g]$ is a pseudodifferential operator of order $-1$, their supports\footnote{Recall that an operator $P$ is \emph{supported in a subset $K$}, if $\supp Pu \subset K$ for all $u$ in the domain of $P$ and if $Pu = 0$ whenever we have $\supp u \cap K = \emptyset$.} depend only on the propagation of $P$ and on the value of $R$ (but not on $i$ nor on the concrete choice of $g$) and their operator norms as maps $L^2(E) \to H^1(E)$ are bounded from above by a constant that only depends on $P$, on $R$ and on the values of all the $L_j$ (but again, neither on $i$ nor on $g$). The last fact follows from a combination of Remark \ref{rem:bound_operator_norm_PDO} together with the estimates on the symbols of the $[P_i,g]$ that we get from the proof that they are pseudodifferential operators of order $-1$. So examining the proof of Proposition \ref{prop:quasilocal_negative_order_uniformly_locally_compact} more closely, we see that these properties suffice to conclude the needed uniformity of $f[P,g]$ in the functions $g$.

The operators $g P f$ may be treated analogously.
\end{proof}

\section{Elliptic operators}

Now we get to elliptic operators since these are the ones for which we do have index theorems. The reason for this is that ellipticity is the condition that we need so that an operator is invertible modulo smoothing operators, i.e., possesses an abstract index class in the $K$-theory of these operators.

\begin{defn}[Elliptic symbols]
Let $p \in \Symb^k(E,F)$. Recall that $p$ is a section of the bundle $\Hom(\pi^\ast E, \pi^\ast F)$ over $T^\ast M$. We will call $p$ \emph{elliptic}, if there is an $R > 0$ such that $p|_{| \xi | > R}$\footnote{This notation means the following: we restrict $p$ to the bundle $\Hom(\pi^\ast E, \pi^\ast F)$ over the space $\{(x,\xi) \in T^\ast M \ | \ |\xi| > R\} \subset T^\ast M$.} is invertible and this inverse $p^{-1}$ satisfies the Inequality \eqref{eq:symbol_uniformity} for $\alpha, \beta = 0$ and order $-k$ (and of course only for $|\xi| > R$ since only there the inverse is defined). Note that analogously as in the compact case it follows that $p^{-1}$ satisfies the Inequality \eqref{eq:symbol_uniformity} for all multi-indices $\alpha$, $\beta$.
\end{defn}

\begin{lem}\label{lem:ellipticity_independent_of_representative}
If $p \in \Symb^k(E,F)$ is elliptic, then every other representative $p^\prime$ of the class $[p] \in \Symb^{k-[1]}(E,F)$ is also elliptic.
\end{lem}

\begin{proof}
The difference $p - p^\prime$ is a symbol of order $k-1$, i.e., it locally satisfies an estimate of the form $\|(p_i - p^\prime_i)(x, \xi)\| \le C_0 (1+|\xi|)^{k-1}$ for a constant $C_0 > 0$. The Inequality \eqref{eq:symbol_uniformity} for $p^{-1}$ means that $p_i(x,\xi)$ is bounded from below by $C(1+|\xi|)^k$ for $|\xi| > R$. Since the $k$th power grows faster than the $(k-1)$st power of $1+|\xi|$ (or, in the case $k < 0$, decreases more slowly), we can get by lowering $C$ to $C^\prime$ and enlarging $R$ to $R^\prime$ that $p_i^\prime(x,\xi)$ is bounded from below by $C^\prime (1+|\xi|)^k$ for $|\xi| > R^\prime$. Since this lowering of $C$ and enlarging of $R$ can be made independently of $i$, the claim follows.
\end{proof}

Due to the above lemma we are now able to define what it means for a pseudodifferential operator to be elliptic:

\begin{defn}[Elliptic $\Psi \mathrm{DO}$s]\label{defn:elliptic_operator}
Let $P \in \Psi \mathrm{DO}_?^k(E,F)$. We will call $P$ \emph{elliptic}, if its principal symbol $\sigma(P)$ is elliptic.
\end{defn}

The importance of elliptic operators lies in the fact that they admit an inverse modulo operators of order $-\infty$. We may prove this analogously as in the case of a compact manifold. See also \cite[Theorem 3.3]{kordyukov} where Kordyukov proves the existence of parametrices for his class of pseudodifferential operators (which is our class restricted to operators of finite propagation).

\begin{thm}[Existence of parametrices]
Let $P \in \Psi \mathrm{DO}_?^k(E,F)$ be elliptic. Then there exists an operator $Q \in\Psi \mathrm{DO}_?^{-k}(F, E)$ such that
\[PQ = \id - S_1 \text{ and }QP = \id - S_2,\]
where $S_1 \in \Psi \mathrm{DO}_?^{-\infty}(F)$ and $S_2 \in \Psi \mathrm{DO}_?^{-\infty}(E)$.
\end{thm}

Using parametrices, we can prove a lot of the important properties of elliptic operators, e.g., \emph{elliptic regularity} (which is a converse to Proposition \ref{prop:Pu_smooth_if_u_smooth} and a proof of it may be found in, e.g. \cite[Theorem III.§4.5]{lawson_michelsohn}):

\begin{thm}[Elliptic regularity]\label{thm:elliptic_regularity}
Let $P \in \Psi \mathrm{DO}_?^k(E,F)$ be elliptic and let furthermore $u \in H^s(E)$ for some $s \in \IZ$.

Then, if $Pu$ is smooth on an open subset $U \subset M$, $u$ is already smooth on $U$. Furthermore, for $k > 0$: if $Pu = \lambda u$ on $U$ for some $\lambda \in \IC$, then $u$ is smooth on $U$.
\end{thm}

Later we will also need the following \emph{fundamental elliptic estimate} (the proof from \cite[Theorem III.§5.2(iii)]{lawson_michelsohn} generalizes directly):

\begin{thm}[Fundamental elliptic estimate]\label{thm:elliptic_estimate}
Let $P \in \Psi \mathrm{DO}_?^k(E,F)$ be elliptic. Then for each $s \in \IZ$ there is a constant $C_s > 0$ such that
\[\|u\|_{H^s(E)} \le C_s\big(\|u\|_{H^{s-k}(E)} + \|Pu\|_{H^{s-k}(F)}\big)\]
for all $u \in H^s(E)$.
\end{thm}

Another important implication of ellipticity is that symmetric\footnote{This means that we have $\langle Pu, v\rangle_{L^2(E)} = \langle u, Pv\rangle_{L^2(E)}$ for all $u,v \in C_c^\infty(E)$.}, elliptic pseudodifferential operators of positive order are essentially self-adjoint\footnote{Recall that a symmetric, unbounded operator is called \emph{essentially self-adjoint}, if its closure is a self-adjoint operator.}. We need this since we will have to consider functions of pseudodifferential operators, cf. the next section. But first we will show that a symmetric and elliptic operator is also symmetric as an operator on Sobolev spaces.

\begin{lem}\label{lem:symmetric_on_Sobolev}
Let $P \in \Psi \mathrm{DO}_?^k(E)$ with $k \ge 1$ be symmetric on $L^2(E)$ and elliptic. Then $P$ is also symmetric on the Sobolev spaces $H^{lk}(E)$ for $l \in \IZ$, where we use on $H^{lk}(E)$ the scalar product as described in the proof.
\end{lem}

\begin{proof}
Due to the fundamental elliptic estimate the norm $\|u\|_{H^0} + \|Pu\|_{H^0}$ (note that $H^0(E) = L^2(E)$ by definition) on $H^k(E)$ is equivalent to the usual\footnote{We have of course possible choices here, e.g., the global norm \eqref{eq:sobolev_norm} or the local definition \eqref{eq:sobolev_norm_local}, but they are all equivalent to each other since $M$ and $E$ have bounded geometry.} norm $\|u\|_{H^k}$ on it. Now $\|u\|_{H^0} + \|Pu\|_{H^0}$ is equivalent to $\big( \|u\|^2_{H^0} + \|Pu\|^2_{H^0} \big)^{1/2}$ which is induced by the scalar product
\[\langle u, v \rangle_{H^k, P} := \langle u, v \rangle_{H^0} + \langle Pu, Pv \rangle_{H^0}.\]
Since $P$ is symmetric for the $H^0$-scalar product, we immediately see that it is also symmetric for this particular scalar product $\langle \largecdot, \largecdot \rangle_{H^k,P}$ on $H^k(E)$.

To extend to the Sobolev spaces $H^{lk}(E)$ for $l > 0$ we repeatedly invoke the above arguments, e.g., on $H^{2k}(E)$ we have the equivalent norm $\big( \|u\|^2_{H^k, P} + \|Pu\|^2_{H^k, P} \big)^{1/2}$ (again due to the fundamental elliptic estimate) which is induced by the scalar product $\langle u, v \rangle_{H^k,P} + \langle Pu, Pv \rangle_{H^k,P}$ and now we may use that we already know that $P$ is symmetric with respect to $\langle \largecdot, \largecdot \rangle_{H^k,P}$.

Finally, for $H^{lk}(E)$ for $l < 0$ we use the fact that they are the dual spaces to $H^{-lk}(E)$ where we know that $P$ is symmetric, i.e., we equip $H^{lk}(E)$ for $l < 0$ with the scalar product induced from the duality: $\langle u, v \rangle_{H^{lk},P} := \langle u^\prime, v^\prime\rangle_{H^{-lk},P}$, where $u^\prime, v^\prime \in H^{-lk}(E)$ are the dual vectors to $u,v \in H^{lk}(E)$ (note that the induced norm on $H^{lk}(E)$ is exactly the operator norm if we regard $H^{lk}(E)$ as the dual space of $H^{-lk}(E)$).
\end{proof}

Now we get to the proof that symmetric and elliptic operators are essentially self-adjoint. Note that if we work with differential operators $D$ of first order on open manifolds we do not need ellipticity for this result to hold, but weaker conditions suffice, e.g., that the symbol $\sigma_D$ of $D$ satisfies $\sup_{x \in M, \|\xi\| = 1} \|\sigma_D(x, \xi)\| < \infty$ (by the way, this condition is incorporated in our definition of pseudodifferential operators by the uniformity condition). But if we want essential self-adjointness of higher order operators, we have to assume stronger conditions (see the counterexample \cite{MO_elliptic_essentially_self_adjoint}).

\begin{prop}[Essential self-adjointness]\label{prop:elliptic_PDO_essentially_self-adjoint}
Let $P \in \Psi \mathrm{DO}_?^k(E)$ with $k \ge 1$ be symmetric and elliptic. Then the unbounded operator $P\colon H^{lk}(E) \to H^{lk}(E)$ is essentially self-adjoint for all $l \in \IZ$, where we equip these Sobolev spaces with the scalar products as described in the proof of the above Lemma \ref{lem:symmetric_on_Sobolev}.
\end{prop}

\begin{proof}
This proof is an adapted version of the proof of this statement for compact manifolds from \cite{MO_elliptic_essentially_self_adjoint}.

We will use the following sufficient condition for essential self-adjointness: if we have a symmetric and densely defined operator $T$ such that $\kernel (T^\ast \pm i) = \{0\}$, then the closure $\overline{T}$ of $T$ is self-adjoint and is the unique self-adjoint extension of $T$.

So let $u \in \kernel (P^\ast \pm i) \subset H^{lk}(E)$, i.e., $P^\ast u = \pm i u$. From elliptic regularity we get that $u$ is smooth and using the fundamental elliptic estimate for $P^\ast$\footnote{Note that $P^\ast$ is elliptic if and only if $P$ is.} we can then conclude $\|u\|_{H^{k+lk}} \le C_{k+lk}\big(\|u\|_{H^{lk}} + \|P^\ast u\|_{H^{lk}}\big) = 2 C_{k+lk} \|u\|_{H^{lk}} < \infty$, i.e., $u \in H^{k+lk}(E)$. Repeating this argument gives us $u \in H^\infty(E)$, i.e., $u$ lies in the domain of $P$ itself and is therefore an eigenvector of it to the eigenvalue $\pm i$. But since $P$ is symmetric we must have $u = 0$. This shows $\kernel (P^\ast \pm i) = \{0\}$ and therefore $P$ is essentially self-adjoint.
\end{proof}

\section{Functions of symmetric, elliptic operators}\label{sec:functions_of_PDOs}

Let $P \in \Psi \mathrm{DO}^k(E)$ be a symmetric and elliptic pseudodifferential operator of positive order $k \ge 1$. Then by Proposition \ref{prop:elliptic_PDO_essentially_self-adjoint} we know that $P\colon L^2(E) \to L^2(E)$ is essentially self-adjoint. So, if $f$ is a Borel function defined on the spectrum of $P$, the operator $f(P)$ is defined by the functional calculus. In this whole section $P$ will denote such an operator, i.e., a symmetric and elliptic one of positive order.

Given such a pseudodifferential operator $P$, we will show in Section \ref{sec:homology_classes_of_PDOs} that it defines naturally a class in uniform $K$-homology. For this we will have to consider $\chi(P)$, where $\chi$ is a so-called normalizing function (Definition \ref{defn:normalizing_function}), and we will have to show that $\chi(P)$ is uniformly pseudolocal and $\chi(P)^2 - 1$ is uniformly locally compact. For this we will need the analysis done in this section, i.e., this section is purely technical in nature.

If $f$ is a Schwartz function, we have the formula $f(P) = \frac{1}{\sqrt{2\pi}}\int_\IR \hat{f}(t) e^{itP} dt$, where $\hat{f}$ is the Fourier transform of $f$. In the case that $P = D$ is an elliptic, first-order differential operator and its symbol satisfies $\sup_{x \in M, \|\xi\| = 1} \|\sigma_D(x, \xi)\| < \infty$, the operator $e^{itD}$ has finite propagation (a proof of this may be found in, e.g., \cite[Proposition 10.3.1]{higson_roe}) from which (exploiting the above formula for $f(D)$) we may deduce the needed properties of $\chi(P)$ and $\chi(P)^2 - 1$. But this is no longer the case for a general elliptic pseudodifferential operator $P$ and therefore the analysis that we have to do here in this general case is much more sophisticated.

Note that the restriction to operators of order $k \ge 1$ in this section is no restriction on the fact that elliptic pseudodifferential operators define uniform $K$-homology classes. In fact, if $P$ has order $k \le 0$, then we already know from Proposition \ref{prop:PDO_order_0_l-uniformly-pseudolocal} that $P$ is uniformly pseudolocal, i.e., there is no need to form the expression $\chi(P)$ in order for $P$ to define a uniform $K$-homology class.

We start with the following crucial technical lemma which is a generalization of the fact that $e^{itD}$ has finite propagation to pseudodifferential operators. Note that we do not have to assume something like $\sup_{x \in M, \|\xi\| = 1} \|\sigma_D(x, \xi)\| < \infty$ that we had to for first-order differential operators, since such an assumption is subsumed in the uniformity condition that we have in the definition of pseudodifferential operators.

Note that the author does not know if the analogous result of the next lemma for operators from $\Psi \mathrm{DO}_u^k(E)$ does hold, i.e., if in that case $e^{itP}$ may be approximated by finite propagation operators. Since all the following results rely on this technical lemma, the results of this section only hold for operators from $\Psi \mathrm{DO}^k(E)$, i.e., without the subscript ``u''. Of course, since $\Psi \mathrm{DO}_u^k(E) \subset \Psi \mathrm{DO}^k(E)$, the next lemma still applies to $P \in \Psi \mathrm{DO}_u^k(E)$ in the sense that we may then conclude that $e^{itP}$ is quasilocal, though maybe not approximable by finite propagation operators.

\begin{lem}\label{lem:exp(itP)_quasilocal}
Let $P \in \Psi \mathrm{DO}^{k\ge 1}(E)$ be symmetric and elliptic. Then the operator $e^{itP}$ is a quasilocal operator $H^{lk}(E) \to H^{lk-k}(E)$ for all $l \in \IZ$ and $t \in \IR$.
\end{lem}

\begin{proof}
This proof is a watered down version of the proof of \cite[Theorem 3.1]{mcintosh_morris}.

We will need the following two facts:
\begin{enumerate}
\item $\|e^{itP}\|_{lk,lk} = 1$ for all $r \in \IZ$, where $\|\largecdot\|_{lk,lk}$ denotes the operator norm of operators $H^{lk}(E) \to H^{lk}(E)$ and
\item there is a $\kappa > 0$ such that $\|[\eta, P]\|_{s,s-(k-1)} \le \kappa \cdot \sum_{j=1}^N \|\nabla^j \eta\|_\infty$ for all smooth $\eta \in C_b^\infty(M)$, where $N$ depends on $s \in \IZ$ and the dimension of $M$.
\end{enumerate}

The first one holds since $e^{itP}$ is a unitary operator (using Proposition \ref{prop:elliptic_PDO_essentially_self-adjoint}) and the second is due to the facts that by Proposition \ref{prop:PsiDOs_filtered_algebra} the commutator $[\eta, P]$ is a pseudodifferential operator of order $k-1$ (recall that smooth functions with bounded derivatives are operators of order $0$) and due to Remark \ref{rem:bound_operator_norm_PDO} (where we have to recall the formula how to compute the symbol of the composition of two pseudodifferential operators from, e.g., \cite[Theorem III.§3.10]{lawson_michelsohn}).

Let $L \subset M$ and let $u \in H^{lk}(E)$ be supported within $L$. Furthermore, we choose an $R > 0$ and a smooth, real-valued function $\eta$ with $\eta \equiv 1$ on a neighbourhood of $\supp u$, $\eta \equiv 0$ on $M - B_{R+1}(L)$ and the first $N$ derivatives of $\eta$ bounded from above by $C/R$ for a fixed constant $C$. Then we have for all $v \in H^{lk-k}(E)$ that are supported in $M - B_{R+1}(L)$
\begin{align*}
\langle e^{itP} u, v\rangle_{H^{lk-k}} & = \langle e^{itP} \eta u, v\rangle_{H^{lk-k}} - \langle e^{itP} u, \eta v\rangle_{H^{lk-k}}\\
& = \langle [e^{itP},\eta] u, v\rangle_{H^{lk-k}},
\end{align*}
i.e., $|\langle e^{itP} u, v\rangle_{H^{lk-k}}| \le \|[e^{itP},\eta]\|_{lk,lk-k} \cdot \|u\|_{H^{lk}} \cdot \|v\|_{H^{lk-k}}$ and it remains to give an estimate for $\|[e^{itP},\eta]\|_{lk,lk-k}$:

We have
\begin{align*}
[e^{itP}, \eta] & = \int_0^1 \tfrac{d}{dx} \big( e^{ixtP} \eta e^{i(1-x)tP} \big) dx\\
& = -it \int_0^1 e^{ixtP} [\eta,P] e^{i(1-x)tP} dx
\end{align*}
which gives by factorizing the integrand as
\[H^{lk}(E) \stackrel{e^{i(1-x)tP}}\longrightarrow H^{lk}(E) \stackrel{[\eta, P]}\longrightarrow H^{lk-(k-1)}(E) \hookrightarrow H^{lk-k}(E) \stackrel{e^{ixtP}}\longrightarrow H^{lk-k}(E)\]
the estimate
\[\|[e^{itP},\eta]\|_{lk,lk-k} \le |t| \int_0^1 \|[\eta, P]\|_{lk,lk-(k-1)} dx \le |t| \cdot \kappa \cdot \sum_{j=1}^N \|\nabla^j \eta\|_\infty.\]
Since $\|\nabla^j \eta\|_\infty < C/R$ for all $1 \le j \le N$, we have shown
\begin{equation}
\label{eq:dominating_function_exp(itP)}
|\langle e^{itP} u, v\rangle_{H^{lk-k}}| < \frac{|t| \kappa N C}{R} \cdot \|u\|_{H^{lk}} \cdot \|v\|_{H^{lk-k}}
\end{equation}
for all $u$ supported in $L$ and all $v$ in $M-B_{R+1}(L)$. Because $R > 0$ and $l \in \IZ$, $t \in\IR$ were arbitrary, the claim that $e^{itP}$ is a quasilocal operator $H^{lk}(E) \to H^{lk-k}(E)$ for all $l \in \IZ$ and $t \in \IR$ follows.
\end{proof}

An important technical corollary is the following one:

\begin{cor}[cf. {\cite[Lemma 1.1 in Chapter XII.§1]{taylor_pseudodifferential_operators}}]
\label{cor:lth_derivative_integrable_defines_quasilocal_operator}
Let $q(t)$ be a function on $\IR$ such that for an $n \in \IN_0$ the functions $q(t)|t|$, $q^\prime(t)|t|$,  $\ldots$, $q^{(n)}(t)|t|$ are integrable, i.e., belong to $L^1(\IR)$.

Then the operator defined by $\int_\IR q(t) e^{itP} dt$ is for all $l \in \IZ$ a quasilocal operator $H^{lk-nk+k}(E) \to H^{lk}(E)$, i.e., is of order $-nk + k$.
\end{cor}

\begin{proof}
Let $Q \in \Psi \mathrm{DO}^{-k}(E)$ be a parametrix for $P$, i.e., $PQ = \id - S_1$ and also $QP = \id - S_2$, where $S_1, S_2 \in \Psi \mathrm{DO}^{-\infty}(E)$. Integration by parts $n$ times yields:
\begin{equation}
\label{eq:formula_integration_by_parts_quasilocal}
(i Q)^n \int_\IR q^{(n)}(t) e^{itP} dt = (i Q)^n (-i P)^n \int_\IR q(t) e^{itP} dt = (\id - S_2)^n \int_\IR q(t)e^{itP} dt.
\end{equation}
Since $q(t)|t|$ and $q^{(n)}(t)|t|$ are integrable and due to the Estimate \eqref{eq:dominating_function_exp(itP)}, we conclude from the above Lemma \ref{lem:exp(itP)_quasilocal} that both integrals $\int_\IR q(t)e^{itP} dt$ and $\int_\IR q^{(n)}(t)e^{itP} dt$ define quasilocal operators of order $k$ on $H^{lk}(E)$. Note that in the case of $\int_\IR q(t)e^{itP} dt$ this is just a first result which we will need now in order to show that the order of this operator is in fact lower (as claimed by this corollary).

Now $(\id - S_2)^n = \id + \sum_{j=1}^n \binom{n}{j}(-S_2)^j$ and the sum is a quasilocal smoothing operator because $S_2$ is one. Since the composition of quasilocal operators is again a quasilocal operator (see \cite[Proposition 5.2]{roe_index_1}), we conclude that the second summand $R$ of
\begin{equation}
\label{eq:formula_integration_by_parts_quasilocal_2}
(\id - S_2)^n \int_\IR q(t)e^{itP} dt = \int_\IR q(t)e^{itP} dt + \underbrace{\sum_{j=1}^n \binom{n}{j}(-S_2)^j \int_\IR q(t)e^{itP} dt}_{=: R}
\end{equation}
is also a quasilocal smoothing operator. Now Equations \eqref{eq:formula_integration_by_parts_quasilocal} and \eqref{eq:formula_integration_by_parts_quasilocal_2} together yield
\[\int_\IR q(t)e^{itP} dt = (i Q)^n \int_\IR q^{(n)}(t) e^{itP} dt - R,\]
from which the claim follows.
\end{proof}

Recall that if $f$ is a Schwartz function, then the operator $f(P)$ is given by
\begin{equation}\label{eq:schwartz_function_of_PDO}
f(P) = \frac{1}{\sqrt{2\pi}}\int_\IR \hat{f}(t) e^{itP} dt,
\end{equation}
where $\hat{f}$ is the Fourier transform of $f$. Since $\hat{f}$ is also a Schwartz function, it satisfies the assumption in Corollary \ref{cor:lth_derivative_integrable_defines_quasilocal_operator} for all $n \in \IN_0$, i.e., $f(P)$ is a quasilocal smoothing operator. Applying this argument to the adjoint operator $f(P)^\ast = \overline{f}(P)$, we get with Lemma \ref{lem:PDO_-infinity_equal_quasilocal_smoothing} our next corollary:

\begin{cor}\label{cor:schwartz_function_of_PDO_quasilocal_smoothing}
If $f$ is a Schwartz function, then $f(P) \in \Psi \mathrm{DO}^{-\infty}(E)$.
\end{cor}

We recall from \cite[Lemma 4.2]{spakula_uniform_k_homology} that the uniformly pseudolocal operators are a $C^\ast$-algebra and the uniformly locally compact operators a closed, two-sided $^\ast$-ideal in it. Now since the Schwartz functions are dense in $C_0(\IR)$ and quasilocal smoothing operators are uniformly locally compact (see Corollary \ref{cor:quasilocal_neg_order_uniformly_locally_compact}), we conclude from the above Corollary \ref{cor:schwartz_function_of_PDO_quasilocal_smoothing} that $g(P)$ is uniformly locally compact if $g \in C_0(\IR)$:

\begin{cor}\label{cor:g(P)_uniformly_locally_compact_g_vanishing_at_infinity}
Let $g \in C_0(\IR)$. Then $g(P)$ is uniformly locally compact.
\end{cor}

Now we turn our attention to functions which are more general than Schwartz functions. To be concrete, we consider functions of the following type:

\begin{defn}[Symbols on $\IR$]\label{defn:symbols_on_R}
For arbitrary $m \in \IZ$ we define
\[\mathcal{S}^m(\IR) := \{f \in C^\infty(\IR) \ | \ |f^{(n)}(x)| < C_n(1 + |x|)^{m-n} \text{ for all } n \in \IN_0\}.\]

Note that we have $\mathcal{S}(\IR) = \bigcap_m \mathcal{S}^m(\IR)$, where $\mathcal{S}(\IR)$ denotes the Schwartz space.
\end{defn}

Let us state now the generalization of \cite[Theorem 5.5]{roe_index_1} from operators of Dirac type to pseudodifferential operators:

\begin{prop}[cf. {\cite[Theorem 5.5]{roe_index_1}}]\label{prop:f(P)_quasilocal_of_symbol_order}
Let $f \in \mathcal{S}^m(\IR)$. Then for all $l \in \IZ$ the operator $f(P)$ is a quasilocal operator of order $mk$ on the spaces $H^{lk}(E)$, i.e., $f(P) \colon H^{lk}(E) \to H^{lk-mk}(E)$.
\end{prop}

The proof of it is analogous but more technical since the operators $e^{itP}$ are only quasilocal (Lemma \ref{lem:exp(itP)_quasilocal}) contrary to the operators $e^{itD}$ which have finite propagation (see, e.g., \cite[Theorem 1.3]{roe_index_1}). Moreover, we need Corollary \ref{cor:lth_derivative_integrable_defines_quasilocal_operator} and the techniques developed in its proof for the adaption of \cite[Theorem 5.5]{roe_index_1} to our case.

At last, let us turn our attention to a result regarding differences $\psi(P) - \psi(P^\prime)$ of operators defined via functional calculus. We will need the following proposition in the proof of Proposition \ref{prop:same_symbol_same_k_hom_class} where we show that elliptic pseudodifferential operators with the same symbol define the same uniform $K$-homology class.

\begin{prop}[{\cite[Proposition 10.3.7]{higson_roe}}\footnote{The cited proposition requires additionally a common invariant domain for $P$ and $P^\prime$. In our case here this domain is given by, e.g., $H^\infty(E)$.}]\label{prop:norm_estimate_difference_func_calc}
Let $\psi$ be a bounded Borel function whose distributional Fourier transform $\hat{\psi}$ is such that the product $s\hat{\psi}(s)$ is in $L^1(\IR)$.

If $P$ and $P^\prime$ are symmetric and elliptic pseudodifferential operators of positive order $k \ge 1$ such that their difference $P - P^\prime$ has order $qk$, then we have for all $l \in \IZ$
\[\|\psi(P) - \psi(P^\prime)\|_{lk,lk-qk} \le C_\psi \cdot \|P - P^\prime\|_{lk,lk-qk},\]
where the constant $C_\psi = \frac{1}{2\pi} \int |s \hat{\psi}(s)| ds$ does not depend on the operators.
\end{prop}

\begin{proof}
We first assume that $\hat{\psi}$ is compactly supported and that $s\hat{\psi}(s)$ is a smooth function. Then we use the result \cite[Proposition 10.3.5]{higson_roe}\footnote{Though stated there only for differential operators, its proof also works word-for-word for pseudodifferential ones.}, which is a generalization of Equation \ref{eq:schwartz_function_of_PDO} to more general functions than Schwartz functions, and get
\[\Big\langle \big( \psi(P) - \psi(P^\prime) \big) u, v \Big\rangle_{H^{lk-qk}} = \frac{1}{2 \pi} \int \left\langle \big( e^{isP} - e^{isP^\prime} \big) u, v \right\rangle_{H^{lk-qk}} \cdot \hat{\psi}(s) ds,\]
for all $u,v \in C_c^\infty(E)$. From the Fundamental Theorem of Calculus we get
\[\left\langle \big( e^{isP} - e^{isP^\prime} \big) u, v \right\rangle_{H^{lk-qk}} = i \cdot \int_0^s \left\langle \big( e^{itP} (P - P^\prime) e^{i(s-t)P^\prime} \big) u, v \right\rangle_{H^{lk-qk}} dt\]
and therefore
\[\left| \left\langle \big( e^{isP} - e^{isP^\prime} \big) u, v \right\rangle_{H^{lk-qk}}\right| \le s \cdot \|P - P^\prime\|_{lk,lk-qk} \cdot \|u\|_{lk} \cdot \|v\|_{lk-qk}.\]
Putting it all together, we get
\[\left| \Big\langle \big( \psi(P) - \psi(P^\prime) \big) u, v \Big\rangle_{H^{lk-qk}} \right| \le C_\psi \cdot \|P - P^\prime\|_{lk,lk-qk} \cdot \|u\|_{lk} \cdot \|v\|_{lk-qk}.\]

Now the general claim follows from an approximation argument analogous as at the end of the proof of \cite[Proposition 10.3.5]{higson_roe}.
\end{proof}

\chapter{Uniform \texorpdfstring{$K$}{K}-homology}\label{chapter:uniform_k_homology}

In the last chapter we have defined pseudodifferential operators on manifolds of bounded geometry and this thesis is about their indices. Up to now $K$-homology was the preferred theory to contain the classes of elliptic operators and various index theorem use $K$-homology. So we could show now that elliptic pseudodifferential operators define naturally classes in $K$-homology and then use all the theory that is developed for it to investigate their indices. But we will do something else: we will first recall the definition of \emph{uniform} $K$-homology from \cite{spakula_uniform_k_homology} and then show that elliptic pseudodifferential operators define naturally classes \emph{there}.

Our reasons for doing this are the following: in the Definition \ref{defn:pseudodiff_operator} of pseudodifferential operators we have the Uniformity Condition \eqref{eq:uniformity_defn_PDOs} which states that the bounds on the symbols of these operators are uniform with respect to the location in the manifold where we compute the symbol locally. This means that our pseudodifferential operators behave equally well on every local coordinate chart of the manifold. But the definition of usual $K$-homology takes this not into account, but uniform $K$-homology does. So it seems more appropriate to work with uniform $K$-homology when working with pseudodifferential operators.

Recall that we have said that if the manifold is compact, we do have an index map $K_\ast(M) \to K_\ast(\pt) \cong \IZ$ giving us the analytic index, and that we have said that we do not have something similar in the case that $M$ is not compact. But recall furthermore that in Chapter \ref{chap:quasilocal_smoothing_operators} we have defined analytic index maps for amenable manifolds via an averaging procedure. Now the uniformity in the definition of uniform $K$-homology enables us to do the same for it, i.e., to define analytic index maps $K_0^u(M) \to \IR$ via an analogous averaging procedure. So uniform $K$-homology is exactly what we need for an index theory similar to the one developed in Chapter \ref{chap:quasilocal_smoothing_operators}. And this is something that would not be possible with the usual $K$-homology.

\Spakula introduced uniform $K$-homology in \cite{spakula_uniform_k_homology} and developed some of its basic properties there (like the existence of a Mayer--Vietoris sequence or proving a version of Paschke duality for it). But there is a crucial properties that he did not treat: the existence of an external product. Recall that the external product is the major ingredient in the theory of usual $K$-homology since from its existence we may deduce Bott periodicity and homotopy invariance of $K$-homology and combining these we may prove the Atiyah--Singer index theorem. So our major task in this chapter will be to construct an external product for uniform $K$-homology and then use it to deduce homotopy invariance of it, which will be a major ingredient in our later discussion of indices of pseudodifferential operators.

We will of course start this chapter with recalling the definition and basic properties of uniform $K$-homology. We will use this opportunity to change slightly the definition of uniform $K$-homology in order to firstly encompass locally compact, separable metric spaces (\Spakula gave his definition only for proper spaces) and secondly correct an error in one of \v{S}pakulas proofs. We will elaborate more on this directly in the section where we discuss the definition of uniform $K$-homology. But we may remark now that all the results that \Spakula proved do also hold with our changed definition, i.e., we propose our definition as the one that should have been given from the start.

After discussing the basics we will show that elliptic pseudodifferential operators define classes in it and then we will construct analytic index maps for amenable manifolds. And only at the end we will treat the external product.

\section{Gradings and multigradings}\label{sec:gradings_multigradings}

Before we discuss the definition of uniform $K$-homology, we need to discuss graded and multigraded vector spaces. Briefly speaking, a graded vector space $V$ is one which is provided with a decomposition $V = V^+ \oplus V^-$, and a multigraded one is a graded vector space together with a bunch of multigrading operators anti-commuting with each other and with the grading operator from the decomposition $V = V^+ \oplus V^-$. Though this seems at first glance as a quite unusual constellation, it happens regularly in index theory: Dirac bundles naturally come with the action of a Clifford algebra and this Clifford action gives us exactly the multigrading. And since Clifford algebras are naturally assigned to Dirac bundles, it just makes sense to take their actions into account when defining $K$-homology classes.

All the conventions that we use for the gradings, multigrading and corresponding (multi-)grading operators are from \cite[Appendix A]{higson_roe}.

\subsection*{Gradings}

In this section we revisit the notion of $\IZ_2$-graded vector spaces, algebras and tensor products. For our convenience, we will drop the ``$\IZ_2$'' from our notation.

\begin{defn}[Graded vector spaces]
A \emph{grading} of a vector space $V$ is a decomposition into a direct sum $V = V^+ \oplus V^-$, called the \emph{positive} and \emph{negative} parts of the graded vector space $V$.
\end{defn}

Given a graded vector space $V$, its \emph{grading operator $\epsilon$} is the unique involution whose $\pm  1$-eigenspaces are exactly $V^\pm$.

A \emph{graded vector bundle} is a vector bundle $E$ which is fiberwise a graded vector space and such that we have a decomposition $E = E^+ \oplus E^-$ as vector bundle.

For a Hilbert space $H$ we require the decomposition $H = H^+ \oplus H^-$ to be one into closed, orthogonal subspaces. This is equivalent to the grading operator $\epsilon$ being a selfadjoint unitary.

\begin{defn}[Opposite grading $V^\op$]\label{defn_opposite_grading}
If $V$ is a graded vector space, then its \emph{opposite} is the graded vector space $V^\op$ whose underlying vector space is $V$, but with the reversed grading, i.e., $(V^\op)^+ = V^-$ and $(V^\op)^- = V^+$. This is equivalent to $\epsilon_{V^\op} = -\epsilon_V$.
\end{defn}

An operator on a graded vector space $V$ is called \emph{even} if it maps $V^\pm$ again to $V^\pm$, and it is called \emph{odd} if it maps $V^\pm$ to $V^\mp$. Equivalently, an operator is even if it commutes with the grading operator of $V$, and it is odd if it anti-commutes with the grading operator.

\begin{defn}[Graded algebras]
A \emph{graded algebra} is an algebra $A$ provided with a direct sum decomposition $A = A^+ \oplus A^-$ such that $A^\pm \cdot A^+ \subset A^\pm$ and $A^\pm \cdot A^- \subset A^\mp$.

If $A$ is a $^\ast$-algebra, we require the involution to be an even endomorphism.
\end{defn}

\begin{example}
If $V$ is a graded vector space, its endomorphism algebra $\End(V)$ becomes in a natural way a graded algebra.
\end{example}

An element $a \in A$ is called \emph{homogeneous}, if it belongs either to $A^+$ or to $A^-$. In that case its \emph{degree $\deg(a)$} is defined as $\deg(a) = 0$ if $a \in A^+$, and $\deg(a) = 1$ if $a \in A^-$. It is convenient to define $\deg(0) = 0$.

\begin{defn}[Graded commutator]
If $a$ and $b$ are two homogeneous elements of a graded algebra $A$, then their \emph{graded commutator} is defined as
\begin{equation*}
\label{eq:defn_graded_commutator}
[a,b]_\epsilon := ab - (-1)^{\deg(a)\deg(b)} ba
\end{equation*}
and this is extended to the whole algebra $A$ by linearity.
\end{defn}

\begin{defn}[Graded trace]\label{defn:graded_trace}
If $A$ is an operator of trace class on a graded Hilbert space $H$, then the \emph{graded trace} of $A$ is defined as
\begin{equation*}
\label{eq:defn_graded_trace}
\trace_\epsilon(A) := \trace(\epsilon A).
\end{equation*}
\end{defn}

If $A$ and $B$ are both trace class operators, then one can show that the graded trace vanishes on their graded commutator, i.e.,
\begin{equation*}
\label{eq:graded_trace_vanishes_graded_commutator}
\trace_\epsilon([A,B]_\epsilon) = 0.
\end{equation*}

\begin{defn}[Graded tensor products]
Let $V_1$ and $V_2$ be two graded vector spaces (resp. Hilbert spaces). Then the \emph{graded tensor product $V = V_1 \hatotimes V_2$} of $V_1$ and $V_2$ is the algebraic tensor product (resp. the Hilbert space tensor product) of $V_1$ and $V_2$ equipped with the grading operator $\epsilon_V := \epsilon_{V_1} \otimes \epsilon_{V_2}$, i.e.,
\[V^+ = (V_1^+ \otimes V_2^+) \oplus (V_1^- \otimes V_2^-) \text{ and } V^- = (V_1^+ \otimes V_2^-) \oplus (V_1^- \otimes V_2^+).\]

If we have two graded algebras $A_1$ and $A_2$, then their graded vector space tensor product $A = A_1 \hatotimes A_2$ becomes a graded algebra by setting
\[(a_1 \hatotimes a_2)(b_1 \hatotimes b_2) := (-1)^{\deg(a_2)\deg(b_1)} (a_1 b_1) \hatotimes (a_2 b_2)\]
for homogeneous elements and extending it by linearity.
\end{defn}

So if $a \in A_1$ and $b \in A_2$ are both odd, we get
\[(a \hatotimes 1 + 1 \hatotimes b)^2 = a^2 \hatotimes 1 + 1 \hatotimes b^2.\]

\begin{lem}
For graded and finite-dimensional vector spaces $V_1$ and $V_2$ we have a canonical isomorphism
\[\End(V_1) \hatotimes \End(V_2) \cong \End(V_1 \hatotimes V_2).\]
\end{lem}

\subsection*{Multigradings}

\begin{defn}[Multigraded Hilbert spaces]
Let $p \in \IN_0$. We define a \emph{$p$-multigraded Hilbert space} to be a graded Hilbert space which is equipped with $p$ odd unitary operators $\epsilon_1, \ldots, \epsilon_p$ such that $\epsilon_i \epsilon_j + \epsilon_j \epsilon_i = 0$ for $i \not= j$, and $\epsilon_j^2 = -1$ for all $j$.
\end{defn}

As we have already said, multigraded Hilbert spaces arise naturally via an action of a Clifford algebra. In fact, a $p$-multigraded Hilbert space is the same as a graded module over the Clifford algebra $\IC_p$. It will be useful to regard this module structure as a \emph{right} action of $\IC_p$ on $H$.

Note that a $0$-multigraded Hilbert space is just a graded Hilbert space in the sense of the above subsection. We make the convention that a $(-1)$-multigraded Hilbert space is an ungraded one.

\begin{lem}\label{lem:tensor_prod_multigraded}
Let $p_1, p_2 \ge 0$. Then the graded tensor product $H_1 \hatotimes H_2$ of a $p_1$-multigraded Hilbert space $H_1$ and a $p_2$-multigraded Hilbert space $H_2$ is $(p_1 + p_2)$-multigraded, if we let $\epsilon_j$ act on the tensor product as
\[\epsilon_j(v_1 \otimes v_2) := (-1)^{\deg(v_2)}\epsilon_j(v_1) \otimes v_2\]
for $1 \le j \le p_1$, and for $1 \le j \le p_2$ we let $\epsilon_{p_1 + j}$ act as
\[\epsilon_{p_1 + j}(v_1 \otimes v_2) := v_1 \otimes \epsilon_{p_1 + j}(v_2).\]
\end{lem}

Note that the above sign convention is compatible with the one for graded algebras, if we think of the multigrading operators as acting from the right.

\begin{defn}[Multigraded operators]
Let $H$ be a $p$-multigraded Hilbert space. Then an operator on $H$ will be called \emph{multigraded}, if it commutes with the multigrading operators $\epsilon_1, \ldots, \epsilon_p$ of $H$.
\end{defn}

\begin{prop}\label{prop:multigraded_categories_equiv}
The categories of $p$-multigraded and $(p+2)$-multigraded Hilbert spaces are equivalent for $p \ge -1$.
\end{prop}

For $p \ge 0$ we can see this equivalence in the following way: if $H_1$ is $p$-multigraded, then we define $H := H_1 \oplus H_1^{\op}$. $H$ is $p$-multigraded and we may make it $(p+2)$-multigraded by introducing
\[\epsilon_{p + 1} := \begin{pmatrix}0 & 1 \\ -1 & 0\end{pmatrix} \text{ and } \epsilon_{p + 2} := \begin{pmatrix}0 & i \\ i & 0\end{pmatrix}.\]
Conversely, if $H$ is $(p+2)$-multigraded, we first define $e := i \epsilon_{p+1} \epsilon_{p+2}$. Then $e$ is even, self-adjoint and squares to $1$, i.e., $H$ decomposes as an orthogonal direct sum of graded eigenspaces $H^{+1} \oplus H^{-1}$ with respect to $e$. Since $e$ commutes with $\epsilon_1, \ldots, \epsilon_p$, $H^{+1}$ is a $p$-multigraded Hilbert space. This constructions are inverse to each other up to multigraded unitary isomorphism and the case $p = -1$ is similar.

\section{Definition and basic properties}

Let us now get to the definition of uniform $K$-homology. Note that we have already encountered the uniformity condition that we will impose on the Fredholm modules in Definition \ref{defn:uniform_operators_manifold}.

Our definition will be based on multigraded Fredholm modules, since we have already said this is more suited for the study of index theory since it encompasses the natural actions of Clifford algebras that we have for Dirac bundles. But \Spakula who was the first to define uniform $K$-homology did not use multigrading for his definition, but he only defined $K_0^u$ and $K_1^u$. This is not a real restriction since usual $K$-homology also has, due to Bott periodicity, only two essentially different groups: $K_0$ and $K_1$. We mention this since if the reader wants to look up the original reference, he has to keep in mind that we work with multigraded modules, but \Spakula not.

At last, let us mention that \Spakula gives the definition of uniform $K$-homology only for proper\footnote{That means that all closed balls are compact.} metric spaces. The reason for this is that certain results of him (concretely, Sections 8-9 in \cite{spakula_uniform_k_homology}) only work for proper spaces. This results are all connected to the uniform coarse assembly map $\mu_u\colon K_\ast^u(X) \to K_\ast(C_u^\ast(Y))$, where $Y \subset X$ is a uniformly discrete quasi-lattice, and this is nor surprising: the (uniform) Roe algebra only has on proper spaces nice properties (like its $K$-theory being a coarse invariant) and therefore we expect that results of uniform $K$-homology that connect to the uniform Roe algebra also should need properness assumptions on the space. But we can see by looking into the proofs of \Spakula in all the other sections of \cite{spakula_uniform_k_homology} that all results except the ones in Sections 8-9 also hold only for locally compact, separable metric spaces (without assumptions on completeness). Note that this is a crucial fact for us since in the proof of \Poincare duality we will have to consider the uniform $K$-homology of open balls in $\IR^n$.

So let $X$ be a locally compact, separable metric space and let us first recall the usual definition of multigraded Fredholm modules:

\begin{defn}[Multigraded Fredholm modules]
Let $p \in \IZ_{\ge -1}$. A \emph{$p$-multigraded Fredholm module $(H, \rho, T)$ over $X$} is given by the following data:
\begin{itemize}
\item a separable $p$-multigraded Hilbert space $H$,
\item a representation $\rho\colon C_0(X) \to \IB(H)$ by even, multigraded operators and
\item an odd multigraded operator $T \in \IB(H)$ such that
\begin{itemize}
\item the operators $T^2 - 1$ and $T - T^\ast$ are locally compact and
\item the operator $T$ itself is pseudolocal.
\end{itemize}
Here an operator $S$ is called \emph{locally compact}, if for all $f \in C_0(X)$ the operators $\rho(f) S$ and $S \rho(f)$ are compact, and $S$ is called \emph{pseudolocal}, if for all $f \in C_0(X)$ the operator $[S, \rho(f)]$ is compact.
\end{itemize}
\end{defn}

Now we start working towards the definition of uniform Fredholm modules. For the convenience of the reader, let us restate the Definitions \ref{defn:uniformly_approximable_collection} and \ref{defn:uniform_operators_manifold} since we will need them now.

\begin{defn}[Uniformly approximable collections of operators]\label{defn:uniformly_approx_collections}
A collection of operators $\mathcal{A} \subset \IK(H)$ is said to be \emph{uniformly approximable}, if for every $\varepsilon > 0$ there is an $N > 0$ such that for every $T \in \mathcal{A}$ there is a rank-$N$ operator $k$ with $\|T - k\| < \varepsilon$.
\end{defn}

Consider also the basic Examples \ref{ex:uniformly_approximable_collections} to get again a feeling for this notion. Let us furthermore define
\begin{equation*}
\LLip_R(X) := \{ f \in C_c(X) \ | \ f \text{ is }L\text{-Lipschitz}, \diam(\supp f) \le R \text{ and } \|f\|_\infty \le 1\}.
\end{equation*}

\begin{defn}[{\cite[Definition 2.3]{spakula_uniform_k_homology}}]\label{defn:uniform_operators}
Let $T \in \IB(H)$ be an operator on a Hilbert space $H$ and $\rho\colon C_0(X) \to \IB(H)$ a representation.

We say that $T$ is \emph{uniformly locally compact}, if for every $R, L > 0$ the collection
\[\{\rho(f)T, T\rho(f) \ | \ f \in \LLip_R(X)\}\]
is uniformly approximable.

We say that $T$ is \emph{uniformly pseudolocal}, if for every $R, L > 0$ the collection
\[\{[T, \rho(f)] \ | \ f \in \LLip_R(X)\}\]
is uniformly approximable.
\end{defn}

\begin{rem}
In \cite{spakula_uniform_k_homology} uniformly locally compact operators were called ``$l$-uniform'' and uniformly pseudolocal operators ``$l$-uniformly pseudolocal''. We have already explained in Remark \ref{rem:renaming_l_dependence} our reasons for changing this names.
\end{rem}

The following lemma states that on proper spaces we may drop the $l$-dependence for uniformly locally compact operators.

\begin{lem}[{\cite[Remark 2.5]{spakula_uniform_k_homology}}]\label{lem:l_uniformly_loc_compact_without_l}
Let $X$ be a proper space. If $T$ is uniformly locally compact, then for every $R > 0$ the collection
\[\{\rho(f) T, T \rho(f) \ | \ f \in C_c(X), \diam(\supp f) \le R \text{ and } \|f\|_\infty \le 1\}\]
is also uniformly approximable (i.e., we can drop the $L$-dependence).
\end{lem}

Note that an analogous lemma for uniformly pseudolocal operators does not hold. We may see this via the following example: if we have an operator $D$ of Dirac type on a manifold $M$ and if $g$ is a smooth function on $M$, then we have the equation $([D,g]u)(x) = \sigma_D(x, dg) u(x)$, where $u$ is a section into the Dirac bundle $S$ on which $D$ acts, $\sigma_D(x, \xi)$ is the symbol of $D$ regarded as an endomorphism of $S_x$ and $\xi \in T^\ast_x M$. So we see that the norm of $[D,g]$ does depend on the first derivative of the function $g$.

\begin{defn}[Uniform Fredholm modules, cf. {\cite[Definition 2.6]{spakula_uniform_k_homology}}]\label{defn:uniform_fredholm_modules}
A Fredholm module $(H, \rho, T)$ is called \emph{uniform}, if $T$ is uniformly pseudolocal and the operators $T^2-1$ and $T - T^\ast$ are uniformly locally compact.
\end{defn}

\begin{rem}\label{rem:changed_L_cont}
Note that \v{S}pakula, who was the first to define uniform $K$-homology in his PhD thesis \cite{spakula_thesis}, uses the notion ``$L$-continuous'' instead of ``$L$-Lipschitz'' for the definition of $\LLip_R(X)$ (which he also denotes by $C_{R,L}(X)$, i.e., we have also changed the notation), so that he gets slightly differently defined uniform Fredholm modules. But the author was not able to deduce Proposition \ref{prop:compact_space_every_module_uniform} with \v{S}pakula's definition, which is why we have changed it to ``$L$-Lipschitz'' (since the statement of Proposition \ref{prop:compact_space_every_module_uniform} is a very desirable one and, in fact, we will need it crucially in the proof of \Poincare duality).

\Spakula noted that for a geodesic metric space both notions ($L$-continuous and $L$-Lipschitz) coincide. We will prove this assertion in the next Section \ref{sec:local_lipschitz_conditions}.

Note that all the results of \Spakula do also hold with our definition of uniform Fredholm modules, i.e., changing the definition to ours does not affect the validity of his results.
\end{rem}

For a totally bounded metric space uniform Fredholm modules are the same as usual Fredholm modules. Since \v{S}pakula does not give a proof of (and in fact, the author was not able to give it under the definition of $\LLip_R(X)$ that \v{S}pakula uses), we will do it now for our definition:

\begin{prop}\label{prop:compact_space_every_module_uniform}
Let $X$ be a totally bounded metric space. Then every Fredholm module over $X$ is uniform.
\end{prop}

\begin{proof}
Let $(H, \rho, T)$ be a Fredholm module.

First we will show that $T$ is uniformly pseudolocal. We will use the fact that the set $\LLip_R(X) \subset C(X)$ is relatively compact (i.e., its closure is compact) by the Theorem of Arzel\`{a}--Ascoli.\footnote{Since Lipschitz functions are uniformly continuous they have a unique extension to the completion $\overline{X}$ of $X$. Since $\overline{X}$ is compact, Arzel\`{a}--Ascoli applies.} Assume that $T$ is not uniformly pseudolocal. Then there would be $R, L > 0$ and $\varepsilon > 0$, so that for all $N > 0$ we would have an $f_N \in \LLip_R(X)$ such that for all rank-$N$ operators $k$ we have $\|[T, \rho(f_N)] - k\| \ge \varepsilon$. Since $\LLip_R(X)$ is relatively compact, the sequence $f_N$ has an accumulation point $f_\infty \in \LLip_R(X)$. Then we have $\|[T, \rho(f_\infty)] - k\| \ge \varepsilon / 2$ for all finite rank operators $k$, which contradicts the pseudolocality of $T$.

The proofs that $T^2 - 1$ and $T - T^\ast$ are uniformly locally compact are analogous.
\end{proof}

Now we come to the relations that we have to impose on uniform Fredholm modules in order to get uniform $K$-homology.

A collection $(H, \rho, T_t)$ of uniform Fredholm modules is called an \emph{operator homotopy} if $t \mapsto T_t \in \IB(H)$ is norm continuous. As in the non-uniform case, we have an analogous lemma about compact perturbations:

\begin{lem}[Compact perturbations, cf. {\cite[Lemma 2.16]{spakula_uniform_k_homology}}]
\label{lem:compact_perturbations}
Let $(H, \rho, T)$ be a uniform Fredholm module and $K \in \IB(H)$ a uniformly locally compact operator.

Then $(H, \rho, T)$ and $(H, \rho, T + K)$ are operator homotopic.
\end{lem}

\begin{defn}[Uniform $K$-homology, cf. {\cite[Definition 2.13]{spakula_uniform_k_homology}}]
We define the \emph{uniform $K$-homology group $K_{p}^u(X)$} of a locally compact and separable metric space $X$ to be the abelian group generated by unitary equivalence classes of $p$-multigraded uniform Fredholm modules with the relations:
\begin{itemize}
\item if $x$ and $y$ are operator homotopic, then $[x] = [y]$, and
\item $[x] + [y] = [x \oplus y]$,
\end{itemize}
where $x$ and $y$ are $p$-multigraded uniform Fredholm modules.
\end{defn}

\begin{rem}
Again, we have changed a definition of \v{S}pakula. To define uniform $K$-homology he does not use operator homotopy as a relation but a weaker form of homotopy (\cite[Definition 2.11]{spakula_uniform_k_homology}). The reasons why we changed this are the following: firstly, the definition of usual $K$-homology uses operator homotopy and it seems desirable to have uniform $K$-homology to be similarly defined, i.e., just imposing an additional condition on the Fredholm modules. Secondly, \v{S}pakula's proof of \cite[Proposition 4.9]{spakula_uniform_k_homology} is not correct under his notion of homotopy, but it becomes correct if we use operator homotopy as a relation. So by changing the definition we ensure that \cite[Proposition 4.9]{spakula_uniform_k_homology} does hold. And thirdly, we prove in Section \ref{sec:homotopy_invariance} that we would get the same uniform $K$-homology groups if we impose weak homotopy (Definition \ref{defn:weak_homotopy}) as a relation instead of operator homotopy. Though our notion of weak homotopies is different from \v{S}pakula's notion of homotopies, all the homotopies that he constructs in his paper \cite{spakula_uniform_k_homology} are weak homotopies, i.e., all the results of him that rely on his notion of homotopy are also true with our definition.
\end{rem}

All the basic properties of the non-uniform case do also hold in the uniform case. Let us quickly summarize some of them without proofs since the proofs are analogous as the ones for usual $K$-homology.

\begin{itemize}
\item A uniform Fredholm module is called \emph{degenerate}, if the relations defining it hold exactly and not only up to compact operators, i.e.,
\[T^2 - 1 = T - T^\ast = [T, \rho(f)] = 0\]
for all $f \in C_0(X)$. The uniform $K$-homology class of a degenerate module is $[0] \in K_{p}^u(X)$.
\item The additive inverse of $[(H, \rho, T)] \in K^u_{p}(X)$ is given by $[(H^{op}, \rho^{op}, -T^{op})]$, where $\largecdot^{op}$ denotes the change of the grading (Definition \ref{defn_opposite_grading}).
\item Every element of $K^u_{p}(X)$ can be represented as the class of a single uniform $p$-multigraded Fredholm module.
\item $[x] = [y]$ in $K_{p}^u(X)$ if and only if there is a degenerate module $z$ such that $x \oplus z$ and $y \oplus z$ are unitarily equivalent to a pair of operator homotopic modules.
\item We have a formal $2$-periodicity $K_{p}^u(X) \cong K_{p+2}^u(X)$ for all $p \ge -1$.\footnote{This comes basically from Proposition \ref{prop:multigraded_categories_equiv}, cf. \cite[Proposition 8.2.13]{higson_roe}.}
\end{itemize}

At last, let us discuss functoriality, before we will compute a basic, but important example. For the functoriality we first need the following definition:

\begin{defn}[Uniformly proper maps]
Let us call a map $g\colon X \to Y$ with the property
\[\sup_{y \in Y} \diam (g^{-1}(B_r(y))) < \infty \text{ for all }r > 0\]
\emph{uniformly proper}\footnote{\v{S}pakula calls this property \emph{uniformly cobounded} in \cite[Definition 2.15]{spakula_uniform_k_homology} and Block and Weinberger call it \emph{effectively proper} in \cite{block_weinberger_1}.}.

Note that if $X$ is proper, then every uniformly proper map is proper (i.e., preimages of compact subsets are compact).
\end{defn}

The following lemma about functoriality of uniform $K$-homology was proved by \Spakula (see the paragraph directly after \cite[Definition 2.15]{spakula_uniform_k_homology}).

\begin{lem}
Uniform $K$-homology is functorial with respect to uniformly proper, proper Lipschitz maps, i.e., if $g\colon X \to Y$ is uniformly proper, proper and Lipschitz, then it induces maps $g_\ast\colon K_\ast^u(X) \to K_\ast^u(Y)$ on uniform $K$-homology via
\[g_\ast [(H, \rho, T)] := [(H, \rho \circ g^\ast, T)],\]
where $g^\ast\colon C_0(Y) \to C_0(X)$, $f \mapsto f \circ g$ is the by $g$ induced map on functions.
\end{lem}

Now we will compute the uniform $K$-homology of a uniformly discrete metric space of coarsely bounded geometry (Definition \ref{defn:coarsely_bounded_geometry}). This computation is important for us since it is crucially needed in the proof of \Poincare duality in Section \ref{sec:poincare_duality}.

\begin{lem}\label{lem:uniform_k_hom_discrete_space}
Let $Y$ be a uniformly discrete, proper metric space of coarsely bounded geometry. Then $K_0^u(Y)$ is isomorphic to the group $\ell_\IZ^\infty(Y)$ of all bounded, integer-valued sequences indexed by $Y$, and $K_1^u(Y) = 0$.
\end{lem}

\begin{proof}
For the proof we will need Proposition \ref{prop:normalization_finite_prop_speed} that states that we may normalize uniform $K$-homology to operators of finite propagation, i.e., there is an $R > 0$ such that every uniform Fredholm module over $Y$ may be represented by a module $(H, \rho, T)$ where $T$ has propagation no more than $R$\footnote{This means $\rho(f) T \rho(g) = 0$ if $d(\supp f, \supp g) > R$.}, and all homotopies may be also represented by homotopies where the operators have propagation at most $R$.

Going into the proof of Proposition \ref{prop:normalization_finite_prop_speed}, we see that in our case of a uniformly discrete metric space $Y$ we may choose $R$ less than the least distance between two points of $Y$, i.e., $0 < R < \inf_{x \not= y \in Y} d(x,y)$. So given a module $(H, \rho, T)$ where $T$ has propagation at most $R$, the operator $T$ decomposes as a direct sum $T = \bigoplus_{y \in Y} T_y$ with $T_y \colon H_y \to H_y$. The Hilbert space $H_y$ is defined as $H_y := \rho(\chi_y) H$, where $\chi_y$ is the characteristic function of the single point $y \in Y$. Note that $\chi_y$ is a continuous function since the space $Y$ is discrete. Hence $(H, \rho, T) = \bigoplus (H_y, \rho_y, T_y)$ with $\rho_y\colon C_0(Y) \to \IB(H_y)$, $f \mapsto \rho(\chi_y) \rho(f) \rho(\chi_y)$. Now each $(H_y, \rho_y, T_y)$ is a Fredholm module over the point $y$ and so we get a map
\[K_\ast^u(Y) \to \prod_{y \in Y} K_\ast^u(y).\]
Note that we need that the homotopies also all have propagation at most $R$ so that the above defined decomposition of a uniform Fredholm module descends to the level of uniform $K$-homology.

Since a point $y$ is for itself a compact space, we have $K_\ast^u(y) = K_\ast(y)$, and the latter group is isomorphic to $\IZ$ for $\ast = 0$ and it is $0$ for $\ast = 1$. Since the above map $K_\ast^u(Y) \to \prod_{y \in Y} K_\ast^u(y)$ is injective, we immediately conclude $K_1^u(Y) = 0$.

So it remains to show that the image of this map in the case $\ast = 0$ consists of the \emph{bounded} integer-valued sequences indexed by $Y$. But this follows from the uniformity condition in the definition of uniform $K$-homology: the isomorphism $K_0(y) \cong \IZ$ is given by assigning a module $(H_y, \rho_y, T_y)$ the Fredholm index of $T$ (note that $T_y$ is a Fredholm operator since $(H_y, \rho_y, T_y)$ is a module over a single point). Now since $(H, \rho, T) = \bigoplus (H_y, \rho_y, T_y)$ is a \emph{uniform} Fredholm module, we may conclude that the Fredholm indices of the single operators $T_y$ are bounded with respect to $y$.
\end{proof}

\section{\texorpdfstring{$L$}{L}-continuous functions}\label{sec:local_lipschitz_conditions}

We have mentioned in Remark \ref{rem:changed_L_cont} that we have changed \v{S}pakula's definition of $\LLip_R(X)$ (which he denotes by $C_{R,L}(X)$, i.e., we have also changed the notation), from $L$-continuous to $L$-Lipschitz. But \Spakula also writes that for a geodesic metric space both notions coincides, but without any proof. The goal of this section is to give this proof and therefore to show that for such spaces $X$ our definition of uniform Fredholm modules coincides with \v{S}pakula's.

\begin{defn}[{\cite[Section 2]{spakula_uniform_k_homology}}]
Let $(X, d_X)$ and $(Y, d_Y)$ be metric spaces and $f\colon X \to Y$ a function. We will say that $f$ is \emph{$L$-continuous}, if there is a continuously differentiable, non-decreasing function $\alpha\colon [0, \infty] \to [0, \infty)$ with $\alpha^\prime(0) \ge \tfrac{1}{L}$ so that $d_Y(f(x), f(x^\prime)) \le s$ for all $x, x^\prime \in X$ with $d_X(x, x^\prime) \le \alpha(s)$.
\end{defn}

We will compare this notion to the following one:

\begin{defn}[{\cite[Section 1]{garrido_jaramillo}}]
Let $(X, d_X)$ and $(Y, d_Y)$ be metric spaces and $f\colon X \to Y$ a function. We will say that $f$ is \emph{$L$-Lipschitz in the small}, if there is an $r > 0$ such that $d_Y(f(x), f(x^\prime)) \le L\cdot d_X(x, x^\prime)$ for all $x, x^\prime \in X$ with $d_X(x, x^\prime) \le r$.
\end{defn}

A function which is $L$-Lipschitz in the small is $L$-continuous (set $\alpha(s) := \tfrac{1}{L} s$ in a small enough interval around $0$). Conversely, an $L$-continuous function is $K$-Lipschitz in the small for all $K > L$. To prove this, use that $\alpha^\prime(s) \ge \tfrac{1}{K}$ in a small interval around $0$, i.e., $\alpha(s) \ge \tfrac{1}{K} s$ near $0$. The expense for getting $K$ arbitrarily close to $L$ is that the $r$ appearing in the definition of ``Lipschitz in the small'' may become arbitrarily small. In the end we get the following chain of implication:
\begin{align*}
& \ L\text{-Lipschitz}\\
\Rightarrow & \ L\text{-Lipschitz in the small}\\
\Rightarrow & \ L\text{-continuous}\\
\Rightarrow & \ K\text{-Lipschitz in the small for all } K > L\\
\Rightarrow & \ \text{locally Lipschitz and uniformly continuous}.
\end{align*}

Our goal is now to show that for a geodesic metric space every $L$-continuous map is $L$-Lipschitz (an observation due to \v{S}pakula). As a corollary we get that for such spaces the definition of uniform Fredholm modules that \v{S}pakula gives is the same as ours.

A metric space $X$ where every two points can be joined by a minimizing geodesic is called a \emph{geodesic metric space}. Recall that a continuous curve $\gamma\colon [a,b] \to X$ is called a \emph{minimizing geodesic} if there exists some $v > 0$ (the speed of the geodesic) such that $d_X(\gamma(t_1), \gamma(t_2)) = v |t_1 - t_2|$ for all $t_1, t_2 \in [a,b]$.

\begin{lem}
Let $X$ be a geodesic metric space.

Then every $L$-continuous map $f\colon X \to Y$ is $L$-Lipschitz.
\end{lem}

\begin{proof}
Since $f$ is $L$-continuous, it is $K$-Lipschitz in the small for all $K > L$. Denote the corresponding $r$ in the definition of ``Lipschitz in the small'' by $r_K$.

Let $x, x^\prime \in X$ and let $\gamma\colon [a,b] \to X$ be a minimizing geodesic in $X$ joining $x$ and $x^\prime$. Now for every partition $a = t_0 < t_1 < \ldots < t_N = b$ of the interval $[a,b]$ we have
\[d_X(\gamma(a), \gamma(b)) = \sum_{i=0}^{N-1} d_X(\gamma(t_i), \gamma(t_{i+1})).\]
Now we take an equidistant partition with $|t_i - t_{i+1}| = \tfrac{|a-b|}{N} < \tfrac{r_K}{v}$. Then we get, since $d_X(\gamma(t_i), \gamma(t_{i+1})) = v |t_i - t_{i+1}| < r_K$,
\begin{align*}
d_Y(f(x), f(x^\prime)) & \le \sum_{i=0}^{N-1} d_Y(f(\gamma(t_i)), f(\gamma(t_{i+1})))\\
& \le \sum_{i=0}^{N-1} K \cdot d_X(\gamma(t_i), \gamma(t_{i+1}))\\
& = \sum_{i=0}^{N-1} K \cdot \tfrac{v |a-b|}{N} = K \cdot v |a-b|\\
& = K \cdot d_X(x, x^\prime).
\end{align*}
So $f$ is $K$-Lipschitz. But since we can get the $K$ arbitrarily close to $L$ (that we then may have $r_K \to 0$ does not matter), we conclude that $f$ is in fact $L$-Lipschitz.
\end{proof}

\section{Homology classes of elliptic operators}\label{sec:homology_classes_of_PDOs}

In this section we will show that symmetric and elliptic pseudodifferential operators of positive order naturally define classes in uniform $K$-homology. This result is a crucial generalization of \cite[Theorem 3.1]{spakula_uniform_k_homology}, where it is only proved for operators of Dirac type.

First we need a definition and then we will plunge right into the main result of this section:

\begin{defn}[Normalizing functions]\label{defn:normalizing_function}
A smooth function $\chi\colon \IR \to [-1, 1]$ is called a \emph{normalizing function}, if
\begin{itemize}
\item $\chi$ is odd, i.e., $\chi(x) = -\chi(-x)$ for all $x \in \IR$,
\item $\chi(x) > 0$ for all $x > 0$ and
\item $\chi(x) \to \pm 1$ for $x \to \pm \infty$.
\end{itemize}

\begin{figure}[htbp]
\centering
\includegraphics[scale=0.6]{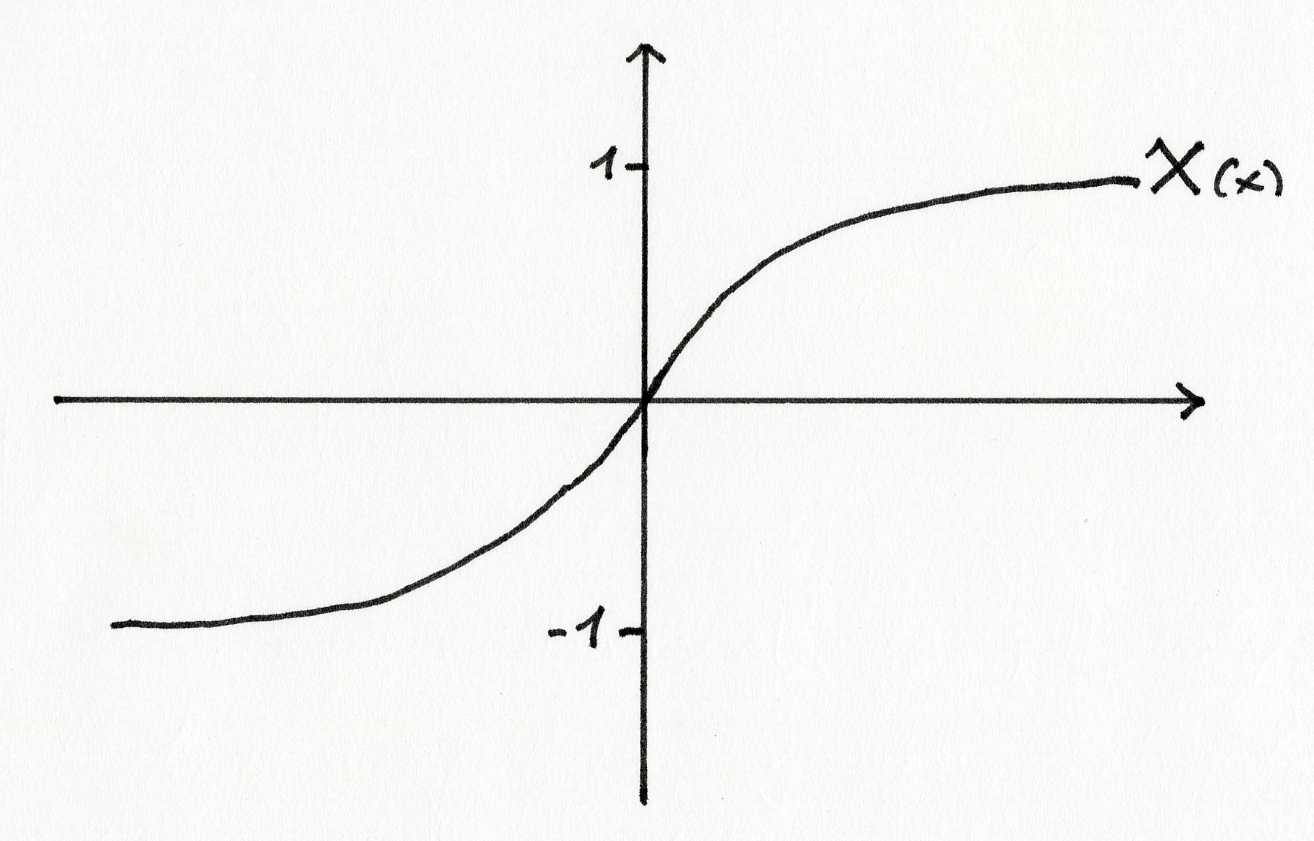}
\caption{A normalizing function.}
\end{figure}

\end{defn}

\begin{thm}\label{thm:elliptic_symmetric_PDO_defines_uniform_Fredholm_module}
Let $M$ be a manifold of bounded geometry, $E \to M$ a $p$-multigraded vector bundle of bounded geometry and let $P \in \Psi \mathrm{DO}_?^k(E)$ be a symmetric and elliptic pseudodifferential operator on $E$ of positive order $k \ge 1$, and let $P$ be odd and multigraded.

Then $(H, \rho, \chi(P))$ is a $p$-multigraded uniform Fredholm module over $M$, where the Hilbert space is $H := L^2(E)$, the representation $\rho\colon C_0(M) \to \IB(H)$ is the one via multiplication operators and $\chi$ is a normalizing function. Furthermore, the uniform $K$-homology class $[(H, \rho, \chi(P))] \in K_{p}^u(M)$ does not depend on the choice of $\chi$.
\end{thm}

\begin{proof}
Recall from Definition \ref{defn:uniform_fredholm_modules} that for the first statement that $(H, \rho, \chi(P))$ defines an ungraded uniform Fredholm module over $M$ we have to show that $\chi(P)$ is uniformly pseudolocal and that $\chi(P)^2 - 1$ and $\chi(P) - \chi(P)^\ast$ are uniformly locally compact.

Since $\chi$ is real-valued and $P$ essentially self-adjoint by Proposition \ref{prop:elliptic_PDO_essentially_self-adjoint}, we have $\chi(P) - \chi(P)^\ast = 0$, i.e., the operator $\chi(P) - \chi(P)^\ast$ is trivially uniformly locally compact. Moreover, since we have $\chi(P)^2 - 1 = (\chi^2 - 1)(P)$ and $\chi^2-1 \in C_0(\IR)$, we conclude with Corollary \ref{cor:g(P)_uniformly_locally_compact_g_vanishing_at_infinity} that $\chi(P)^2 - 1$ is uniformly locally compact.

Because the difference of two normalizing functions is a function from $C_0(\IR)$, we conclude from the same corollary that in order to show that $\chi(P)$ is uniformly pseudolocal, it suffices to show this for one particular normalizing function (and secondly, we get that the class $[(H, \rho, \chi(P))]$ is independent of the concrete choice of $\chi$ due to Lemma \ref{lem:compact_perturbations}).

From now on we proceed as in the proof of \cite[Theorem 3.1]{spakula_uniform_k_homology} using the same formulas: we choose the particular normalizing function $\chi(x) := \frac{x}{\sqrt{1+x^2}}$ to prove that $\chi(P)$ is uniformly pseudolocal. We have $\chi(P) = \frac{2}{\pi} \int_0^\infty \frac{P}{1 + \lambda^2 + P^2} d\lambda$ with convergence of the integral in the strong operator topology\footnote{This follows from the equality $\frac{x}{\sqrt{1 + x^2}} = \frac{2}{\pi} \int_0^\infty \frac{x}{1 + \lambda^2 + x^2} d\lambda$ for all $x \in \IR$.} and get then for $f \in \LLip_R(M)$
\[[\rho(f), \chi(P)] = \frac{2}{\pi} \int_0^\infty \frac{1}{1 + \lambda^2 + P^2} \big( (1+\lambda^2)[\rho(f), P] + P[\rho(f),P]P \big) \frac{1}{1 + \lambda^2 + P^2} d\lambda.\]

Suppose $f \in \LLip_R(M) \cap C_b^\infty(M)$. Then the integral converges in operator norm\footnote{To see this, we have to find upper bounds for the operator norms $\| \largecdot \|_{0,k-1}$ of $\frac{1+\lambda^2}{1+\lambda^2 + P^2} \frac{1}{1+\lambda^2 + P^2}$ and $\frac{P}{1+\lambda^2 + P^2} \frac{P}{1+\lambda^2 + P^2}$, that are integrable with respect to $\lambda$. This can be done by, e.g., using the estimates derived in the proof of Proposition \ref{prop:f(P)_quasilocal_of_symbol_order}. But note that we need the generalization of this proposition to all $m \in \IR$. For the definition of the corresponding Sobolev norms we have to use \eqref{eq:sobolev_norm_local} with fixed coordinate charts, corresponding partition of unity and chosen synchronous framing. Since different choices lead to equivalent norms, our needed result that the integrand is integrable with respect to $\lambda$ is independent of these choices.} and there exists an $N > 0$ which depends only on an $\varepsilon > 0$, $R = \diam (\supp f)$ and the norms of the derivatives of $f$,\footnote{The dependence on $R$ and on the derivatives of $f$ comes from the operator norm estimate of $[\rho(f), P]$.} such that there are $\lambda_1, \ldots, \lambda_N$ and the integral is at most $\varepsilon$ away from the sum of the integrands for $\lambda_1, \ldots, \lambda_N$.

Now we recall Definition \ref{defn:symbols_on_R} of the symbol classes on $\IR$:
\[\mathcal{S}^m(\IR) := \{g \in C^\infty(\IR) \ | \ |g^{(n)}(x)| < C_l(1 + |x|)^{m-n} \text{ for all } n \in \IN_0\}.\]
Since both $\frac{1}{1+\lambda^2 + x^2} \in \mathcal{S}^{-2}(\IR)$ and $\frac{1 + \lambda^2}{1+\lambda^2 + x^2} \in \mathcal{S}^{-2}(\IR)$ (with respect to the variable $x$, i.e., for fixed $\lambda$), the operators $\frac{1}{1+\lambda^2 + P^2}$ and $\frac{1 + \lambda^2}{1+\lambda^2 + P^2}$ are quasilocal operators of order $-2k$ by Proposition \ref{prop:f(P)_quasilocal_of_symbol_order}. This also holds for their adjoints and so, by Corollary \ref{cor:quasilocal_neg_order_uniformly_locally_compact}, they are uniformly locally compact. The same conclusion applies to the operators $\frac{P}{1+\lambda^2+P^2}$ and $\frac{(1+\lambda^2)P}{1+\lambda^2+P^2}$ which are quasilocal of order $-k$ and hence also uniformly locally compact.

So the first summand
\[\frac{1+\lambda^2}{1+\lambda^2 + P^2} [\rho(f), P] \frac{1}{1+\lambda^2 + P^2}\]
of the integrand is the difference of two compact operators and their approximability by finite rank operators depends only on $R = \diam (\supp f)$ and the Lipschitz constant $L$ of $f$. The same also applies to the second summand
\[\frac{1}{1+\lambda^2 + P^2}P [\rho(f), P] P \frac{1}{1+\lambda^2 + P^2}\]
of the integrand (note that $\frac{P^2}{1+\lambda^2+P^2}$ is a bounded operator).

So the operator $[\rho(f), \chi(P)]$ is for $f \in \LLip_R(M) \cap C_b^\infty(M)$ compact and its approximability by finite rank operators depends only on $R$, $L$ and the norms of the derivatives of $f$. That this suffices to conclude that the operator is uniformly pseudolocal is exactly Point 5 in Lemma \ref{lem:kasparov_lemma_uniform_approx_manifold}.

To conclude the proof we have to show that $\chi(P)$ is odd and multigraded. But this was already shown in full generality in \cite[Lemma 10.6.2]{higson_roe}.
\end{proof}

We have shown in the above Theorem that an elliptic pseudodifferential operator naturally defines a class in uniform $K$-homology. Now we will show that this class does only depend on the principal symbol of the pseudodifferential operator. Note that ellipticity of an operator does only depend on its symbol (since it is actually defined that way, see Definition \ref{defn:elliptic_operator}, which is possible due to Lemma \ref{lem:ellipticity_independent_of_representative}), i.e., another pseudodifferential operator with the same symbol is automatically also elliptic.

\begin{prop}\label{prop:same_symbol_same_k_hom_class}
The uniform $K$-homology class of a symmetric and elliptic pseudodifferential operator $P \in \Psi \mathrm{DO}_?^{k \ge 1}(E)$ does only depend on its principal symbol $\sigma(P)$, i.e., any other such operator $P^\prime$ with the same principal symbol defines the same uniform $K$-homology class.
\end{prop}

\begin{proof}
Consider in $\Psi \mathrm{DO}_?^k(E)$ the linear path $P_t := (1-t)P + t P^\prime$ of operators. They are all symmetric and, since $\sigma(P) = \sigma(P^\prime)$, they all have the same principal symbol. So they are all elliptic and therefore we get a family of uniform Fredholm modules $(H, \rho, \chi(P_t))$, where we use a fixed normalizing function $\chi$.

Now if the family $\chi(P_t)$ of bounded operators would be norm-continuous, the claim that we get the same uniform $K$-homology classes would follow directly from the relations defining uniform $K$-homology. But it seems that in general it is only possible to conclude the norm continuity of $\chi(P_t)$ if the difference $P - P^\prime$ is a bounded operator,\footnote{see, e.g., \cite[Proposition 10.3.7]{higson_roe}} i.e., if the order $k$ of $P$ is $1$ (since then the order of the difference $P - P^\prime$ would be $0$, i.e., it would define a bounded operator on $L^2(E)$).

In the case $k > 1$ we get continuity of $\chi(P_t)$ only in the strong-$^\ast$ operator topology\footnote{Recall that if $H$ is a Hilbert space, then the \emph{strong-$^\ast$ operator topology} on $\IB(H)$ is generated by the family of semi-norms $p_x(T) := \|Tx\| + \|T^\ast x\|$ for all $x \in H$, where $T \in \IB(H)$.} on $\IB(L^2(E))$. This is easily seen with Proposition \ref{prop:norm_estimate_difference_func_calc}.\footnote{An example of a normalizing function $\chi$ fulfilling the prerequisites of Proposition \ref{prop:norm_estimate_difference_func_calc} may be found in, e.g., \cite[Exercise 10.9.3]{higson_roe}.} To conclude in this case that $(H, \rho, \chi(P_0))$ and $(H, \rho, \chi(P_1))$ define the same class in uniform $K$-homology, we will use Theorem \ref{thm:weak_homotopy_equivalence_K_hom} (which states that weakly homotopic\footnote{see Definition \ref{defn:weak_homotopy}} uniform Fredholm modules give the same class in uniform $K$-homology), i.e., we will show now that the family $(H, \rho, \chi(P_t))$ is a weak homotopy.

The first bullet point of the definition of a weak homotopy is clearly satisfied since our representation $\rho$ is fixed, i.e., does not depend on the time $t$. Moreover, we have already incidentally discussed the second bullet point in the paragraph above, so it remains to varify that the third point is satisfied. We start with investigating the operators $[\rho(f), \chi(P_t)]$. Let $\chi$ be the normalizing function $\chi(x) = \frac{x}{\sqrt{1+x^2}}$ (this is the one used in the proof of the above Theorem \ref{thm:elliptic_symmetric_PDO_defines_uniform_Fredholm_module}). It satisfies the assumptions of Proposition \ref{prop:norm_estimate_difference_func_calc} since the integral $\int |s \hat{\chi}(s)| ds$ has a finite value (we will use this at the end of this paragraph). From the end of the proof of the above Theorem \ref{thm:elliptic_symmetric_PDO_defines_uniform_Fredholm_module} we get that the approximation of $[\rho(f), \chi(P_t)]$ up to an $\varepsilon$ via finite rank operators is done by approximating finitely many quasilocal operators of negative order times the operator $\rho(f)$. But from the proof of Proposition \ref{prop:quasilocal_negative_order_uniformly_locally_compact} (where we do this approximation), we see that we actually approximate the compact inclusions of Sobolev spaces into the $L^2$-space. So the images of these finite rank operators consist of functions from a Sobolev space of appropriate order and, this is the most important, the Sobolev norms of $L^2$-orthonormal basis of these images may be bounded from above independently of the time $t$, i.e., by the same bound for all operators $[\rho(f), \chi(P_t)]$. But this together with the norm estimate from Proposition \ref{prop:norm_estimate_difference_func_calc} shows that the third bullet point in the Definition \ref{defn:weak_homotopy} of weak homotopies is fulfilled.

The arguments for $\rho(f)(\chi(P_t)^2-1)$ are similar and the case of $\rho(f)(\chi(P_t) - \chi(P_t)^\ast)$ is clear since $\chi(P_t) - \chi(P_t)^\ast = 0$, because $P_t$ is essentially self-adjoint.
\end{proof}

\section{Paschke duality}\label{sec:paschke_duality}

Usual $K$-homology of a $C^\ast$-algebra $A$ may be equivalently defined as the $K$-theory of a certain dual algebra $\frakD(A)$ of $A$. This result is originally due to Paschke (\cite{paschke}) and its reformulation to the nowadays common definition of $K$-homology was done by Higson in \cite{higson_paschke_duality}. \Spakula proved an analogous result for uniform $K$-homology and the goal of this section is to summarize this since we will need this uniform version of Paschke duality later.

The other main point of this section is to introduce the notion of ``jointly bounded geometry'' (Definition \ref{defn:jointly_bounded_geometry}). This notion will be crucial since only if a space $X$ has it we will be able to construct the exterior product on uniform $K$-homology of $X$.

\begin{defn}[{\cite[Definition 4.1]{spakula_uniform_k_homology}}]
Let $H$ be some separable Hilbert space and $\rho \colon C_0(X) \to \IB(H)$ a representation of a locally compact and separable metric space $X$.

We denote by $\frakD^u_{\rho \oplus 0}(X) \subset \IB(H \oplus H)$ the set of all uniformly pseudolocal operators with respect to the representation $\rho \oplus 0$ of $C_0(X)$ on $H \oplus H$.

Analogously we denote by $\frakC^u_{\rho \oplus 0}(X) \subset \IB(H \oplus H)$ the set of all uniformly locally compact operators with respect to $\rho \oplus 0$.
\end{defn}

In \cite[Lemma 4.2]{spakula_uniform_k_homology} it was shown that $\frakD^u_{\rho \oplus 0}(X)$ is a $C^\ast$-algebra and that we have $\frakC^u_{\rho \oplus 0}(X) \subset \frakD^u_{\rho \oplus 0}(X)$ as a closed, two-sided $^\ast$-ideal in it.

\begin{defn}
The groups $K_{-1}^u(X; {\rho \oplus 0})$ are analogously defined as $K_{-1}^u(X)$, except that we consider only uniform Fredholm modules whose Hilbert spaces and representations are (finite or countably infinite) direct sums of $H \oplus H$ and $\rho \oplus 0$.

For $K_0^u(X; {\rho \oplus 0})$ we consider only uniform Fredholm modules modeled on $H^\prime \oplus H^\prime$ with the representation $\rho^\prime \oplus \rho^\prime$, where $H^\prime$ is a finite or countably infinite direct sum of $H \oplus H$ and $\rho^\prime$ analogously a direct sum of finitely or infinitely many $\rho \oplus 0$, and the grading is given by interchanging the two summands in $H^\prime \oplus H^\prime$.
\end{defn}

\begin{rem}\label{rem:defn_balanced_module}
Such Fredholm modules as we use for $K_0^u(X; {\rho \oplus 0})$ are called \emph{balanced} in \cite[Definition 8.3.10]{higson_roe}. Note that since we have defined $K_0^u(X)$ using graded uniform Fredholm modules, but \Spakula uses a slightly other presentation which does not use gradings at all, his definition of $K_0^u(X; {\rho \oplus 0})$ is therefore also a bit different from ours here.
\end{rem}

We define a group homomorphism
\[\varphi_0 \colon K_1(\frakD^u_{\rho \oplus 0}(X)) \to K_0^u(X; {\rho \oplus 0})\]
in the following way: let $[U] \in \Mat_{n \times n}(\frakD^u_{\rho \oplus 0}(X))$ with $U$ unitary be given. Then
\[\varphi_0([U]) := \left[\left((H \oplus H)^n \oplus (H \oplus H)^n, (\rho \oplus 0)^n \oplus (\rho \oplus 0)^n, \begin{pmatrix}0 & U^\ast \\ U & 0\end{pmatrix}\right)\right].\]
Analogously, we define a group homomorphism
\[\varphi_1 \colon K_0(\frakD^u_{\rho \oplus 0}(X)) \to K_{-1}^u(X; {\rho \oplus 0})\]
in the following way: for $[P] \in \Mat_{n \times n}(\frakD^u_{\rho \oplus 0}(X))$ with $P$ a projection, we set
\[\varphi_0([P]) := [((H \oplus H)^n, (\rho \oplus 0)^n, 2P-1)].\]

\begin{prop}[{\cite[Proposition 4.3]{spakula_uniform_k_homology}}]\label{prop:paschke_duality}
The maps
\[\varphi_\ast \colon K_{1+\ast}(\frakD^u_{\rho \oplus 0}(X)) \to K_\ast^u(X; {\rho \oplus 0})\]
for $\ast = -1, 0$ are isomorphisms.
\end{prop}

We can define on the set $\mathcal{R}$ of all unitary equivalence classes of representations of $C_0(X)$ on separable Hilbert spaces a reflexive and transitive relation $\prec$ which turns $(\mathcal{R}, \prec)$ into a directed system. Since we do not need the concrete definition of this relation, we just refer to \cite[Definition 4.7]{spakula_uniform_k_homology}. For $(H, \rho) \prec (H^\prime, \rho^\prime)$ \Spakula defines a homomorphism
\begin{equation*}
K_\ast^u(X; {\rho \oplus 0}) \to K_\ast^u(X; {\rho^\prime \oplus 0}).
\end{equation*}
such that the set $\mathcal{K}$ of all the groups $K_\ast^u(X; {\rho \oplus 0})$ together with the above maps becomes a directed system. Furthermore, the relation $\prec$ now becomes antisymmetric when it descends to $\mathcal{K}$.

Now for each $(H, \rho)$ there is an obvious homomorphisms $K_\ast^u(X; {\rho \oplus 0}) \to K_\ast^u(X)$ which is compatible with the maps from the directed system $\mathcal{K}$, i.e., we get a limit homomorphism
\begin{equation*}
\label{eq:limit_homomorphism_directed_system_paschke}
j_\ast \colon \underrightarrow{\lim} \ K_\ast^u(X; {\rho \oplus 0}) \to K_\ast^u(X).
\end{equation*}

\begin{prop}[{\cite[Proposition 4.9]{spakula_uniform_k_homology}}]\label{prop:direct_limit_version}
The map $j_\ast$ is an isomorphism.
\end{prop}

In the case of usual $K$-homology, using Voiculescu's theorem one can show that the above directed system has a maximal element, i.e., it is possible to represent every Fredholm module over $X$ using a fixed, so-called universal representation. In order to show a similar result for uniform $K$-homology, we need a uniform version of Voiculescu's theorem. Such a uniform version was proved by \Spakula in \cite{spakula_universal_rep}, but under the assumption that the space $X$ has a certain property which we are going to state now.

\begin{defn}[Locally bounded geometry, {\cite[Definition 3.1]{spakula_universal_rep}}]\label{defn:locally_bounded_geometry}
A metric space $X$ has \emph{locally bounded geometry}, if it admits a countable Borel decomposition $X = \cup X_i$ such that
\begin{itemize}
\item each $X_i$ has non-empty interior,
\item each $X_i$ is totally bounded, and
\item for all $\varepsilon > 0$ there is an $N > 0$ such that for every $X_i$ there exists an $\varepsilon$-net in $X_i$ of cardinality at most $N$.
\end{itemize}

Note that \Spakula demands that the closure of each $X_i$ is compact instead of the total boundedness of them. The reason for this is that he considers only proper spaces, whereas we need a more general notion to encompass also non-complete spaces.
\end{defn}

Now \Spakula went on and proved that if $X$ has locally bounded geometry and coarsely bounded geometry (recall Definition \ref{defn:coarsely_bounded_geometry}), then the universal representation constitutes a maximal element for the above directed system (note that his proof also works with our more general notion which applies to probably non-complete spaces). But at the end of his proof of \cite[Corollary 3.3]{spakula_universal_rep} he assumes an additional property which he does not state in the formulation of his main result of that paper: he assumes that the locally bounded geometry and the coarsely bounded geometry of a space $X$ are compatible with each other. So let us state this needed compatibility condition:

\begin{defn}[Jointly bounded geometry]\label{defn:jointly_bounded_geometry}
A metric space $X$ has \emph{jointly coarsely and locally bounded geometry}, if
\begin{itemize}
\item it admits a countable Borel decomposition $X = \cup X_i$ satisfying all the properties of the above Definition \ref{defn:locally_bounded_geometry} of locally bounded geometry,
\item it admits a quasi-lattice $\Gamma \subset X$ (i.e., $X$ has coarsely bounded geometry), and
\item for all $r > 0$ we have $\sup_{y \in \Gamma} \card \{i \ | \ B_r(y) \cap X_i \not= \emptyset\} < \infty$.
\end{itemize}
The last property ensures that there is an upper bound on the number of subsets $X_i$ that intersect any ball of radius $r > 0$ in $X$.
\end{defn}

\begin{examples}
Recall from Examples \ref{ex:coarsely_bounded_geometry} that manifolds of bounded geometry and simplicial complexes of bounded geometry (i.e., the number of simplices in the link of each vertex is uniformly bounded) equipped with the metric derived from barycentric coordinates have coarsely bounded geometry. Now a moment of reflection reveals that they even have jointly bounded geometry.

In the next Figure \ref{fig:not_jointly_but_others} we give an example of a space $X$ having coarsely and locally bounded geometry, but where the quasi-lattice $\Gamma$ and the Borel decomposition $X = \cup X_i$ are not compatible with each other:

\begin{figure}[htbp]
\centering
\includegraphics[scale=0.5]{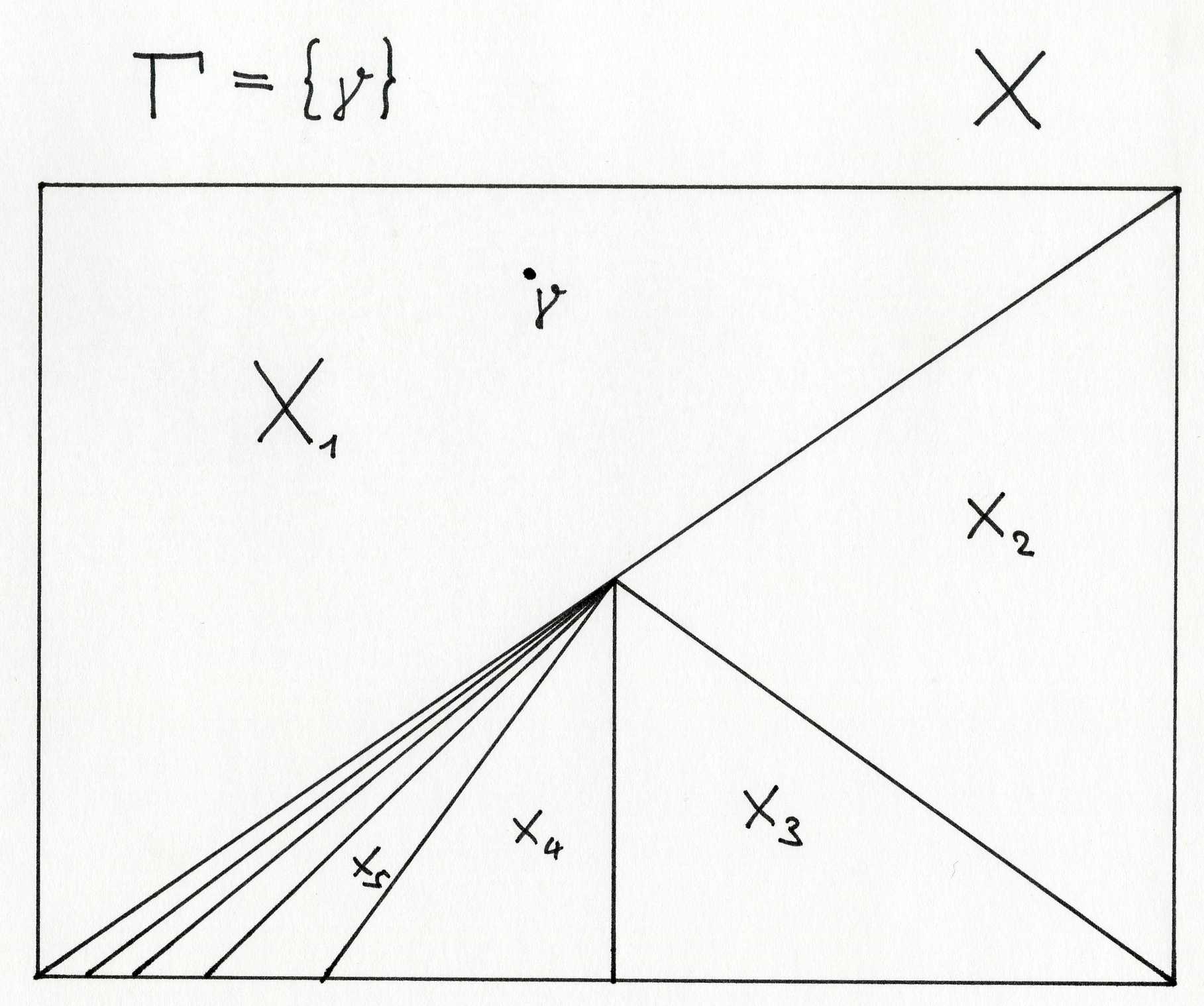}
\caption{Coarsely and locally bounded geometry, but they are not compatible.}
\label{fig:not_jointly_but_others}
\end{figure}

\end{examples}

\begin{thm}[Paschke duality for uniform $K$-homology, {\cite[Corollary 3.6]{spakula_universal_rep}}]\label{thm:paschke_universal}
Let $X$ be a locally compact and separable metric space of jointly bounded geometry and $\rho\colon C_0(X) \to \IB(H)$ an ample representation, i.e., $\rho$ is non-degenerate and $\rho(f) \in \IK(H)$ implies $f \equiv 0$.

Then $\rho$ constitutes a maximal element for the above directed system $\mathcal{K}$, i.e.,
\[K_\ast^u(X; \rho \oplus 0) \cong K_\ast^u(X)\]
for $\ast = -1,0$.
\end{thm}

\section{Normalizations}\label{sec:normalization}

In this section we will first revisit some possible normalizations that we can have for uniform $K$-homology and that were proved by \v{S}pakula. And secondly, we will prove the normalization to ``uniformly traceable operators'', which is crucial for the definition of the index maps on $K_0^u(X)$ in the next Section \ref{sec:index_maps_K_hom}.

Recall that ``normalization'' for $K$-homology means that we may assume that the Fredholm modules have a certain form or a certain property and that this holds also for all homotopies. So plainly speaking, it means that we may change the whole definition of $K$-homology without changing the resulting groups.

\begin{lem}[{\cite[Lemmas 4.5 \& 4.6]{spakula_uniform_k_homology}}, cf. {\cite[Lemma 8.3.5]{higson_roe}}]
For any representation $\rho\colon C_0(X) \to \IB(H)$ we can normalize $K_\ast^u(X; {\rho \oplus 0})$ to uniform modules $((H \oplus H)^n, (\rho \oplus 0)^n, T)$ which are \emph{involutive}, i.e., $T = T^\ast$ and $T^2 = 1$.
\end{lem}

Together with Proposition \ref{prop:direct_limit_version} from the last section, we conclude that we can normalize uniform $K$-homology $K_\ast^u(X)$ to involutive modules for $\ast = -1,0$.

The proof of the following lemma in the non-uniform case may be found in, e.g., \cite[Lemma 8.3.8]{higson_roe}. The proof in the uniform case is analogous and the arguments similar to the ones in the proofs of \cite[Lemmas 4.5 \& 4.6]{spakula_uniform_k_homology}.

\begin{lem}\label{lem:normalization_non-degenerate}
Uniform $K$-homology $K_\ast^u(X)$ may be normalized to \emph{non-degenerate} Fredholm modules, i.e., such that all occuring representations $\rho$ are non-degenerate\footnote{This means that $\rho(C_0(X)) H$ is dense in $H$.}.
\end{lem}

Note that in general we can not normalize uniform $K$-homology to be simultaneously involutive and non-degenerate, just as usual $K$-homology.

If $X$ has coarsely bounded geometry (see Definition \ref{defn:coarsely_bounded_geometry}) it will be crucial that we can normalize uniform $K$-homology to uniform finite propagation, i.e., there is an $R > 0$ depending only on $X$ such that every Fredholm module has propagation at most $R$\footnote{This means $\rho(f) T \rho(g) = 0$ if $d(\supp f, \supp g) > R$.}. This was proved by \Spakula in \cite[Proposition 7.4]{spakula_uniform_k_homology}.

Note that it is in general not possible to make this common propagation $R$ arbitrarily small. A sufficient condition for this being possible is that the space $X$ not only admits \emph{some} quasi-lattice, but \emph{arbitrarily fine} quasi-lattices, i.e., for all $c > 0$ there should be a quasi-lattice $\Gamma_c \subset X$ with $B_c(\Gamma_c) = X$ (see Definition \ref{defn:arbitrarily_fine_coarsely_bounded_geometry}).

Furthermore, we can combine the normalization to finite propagation with the other normalizations, i.e., the procedures are compatible: normalizing first to self-adjoint ($T = T^\ast$) and contractive ($\|T\| \le 1$) modules, then to modules of finite propagation and last to involutive ones, we get modules that are both involutive and of finite propagation. Also, normalizing first to finite propagation modules and then to non-degenerate ones, we get modules that are both non-degenerate and of finite propagation:

\begin{prop}[{\cite[Section 7]{spakula_uniform_k_homology}}]\label{prop:normalization_finite_prop_speed}
If $X$ has coarsely bounded geometry, then there is an $R > 0$ depending only on $X$ such that uniform $K$-homology may be normalized to Fredholm modules that have propagation at most $R$.

Furthermore, we can additionally normalize it to either involutive modules or to non-degenerate ones.
\end{prop}

\subsection*{Uniformly pseudolocally traceable operators}

The last normalization that we will prove is normalization to \emph{uniformly pseudolocally traceable} operators (Definition \ref{defn:uniform_pseudolocal_traceability}). We need this normalization so that we may define index maps on $K_0^u(X)$ in the next section. Since such index maps are tied to the index maps on the $K$-theory of the uniform Roe algebra and since the latter is only well-behaved if the space is proper, it is of no surprise that the results in this subsection only apply to proper spaces.

To prove this normalization, we will need Paschke duality for uniform $K$-homology. Let us recall it: for every representation $\rho$ we have $K_\ast^u(X; {\rho \oplus 0}) \cong K_{1+\ast}(\frakD_{\rho \oplus 0}^u(X))$, where $\frakD_{\rho \oplus 0}^u(X) \subset \IB(H \oplus H)$ is the $C^\ast$-algebra of all uniformly pseudolocal operators with respect to the representation $\rho \oplus 0$. Our goal will be to show that there is a dense, local $C^\ast$-algebra $\frakD_{\rho \oplus 0}^{tr}(X) \subset \frakD_{\rho \oplus 0}^u(X)$, where $\frakD_{\rho \oplus 0}^{tr}(X)$ are the uniformly pseudolocally traceable operators. Then if follows with Lemma \ref{lem:loc_algebra_same_k_theory} that $K_{1+\ast}(\frakD_{\rho \oplus 0}^{tr}(X)) = K_{1+\ast}(\frakD_{\rho \oplus 0}^u(X)) \cong K_\ast^u(X; {\rho \oplus 0})$ which is the formulation of the desired normalization.

We will also have to recall Lemma \ref{lem:kasparov_lemma_uniform_approx_manifold}. Since it was stated only for manifolds, we restate it now in the more general setting of metric spaces. Let us denote by $B_b(X)$ all bounded Borel functions on $X$ and by $B_R(X)$ the subset consisting of all Borel function $h$ with $\diam(\supp h) \le R$ and $\|h\|_\infty \le 1$. Note that we can always extend a representation of $C_0(X)$ canonically to one of $B_b(X)$.

\begin{lem}\label{lem:kasparov_lemma_uniform_approx}
Let $\rho\colon C_0(X) \to \IB(H)$ be a non-degenerate representation and let $X$ be proper. Then the following are equivalent for an operator $T \in \IB(H)$:
\begin{enumerate}
\item $T$ is uniformly pseudolocal,
\item for all $R, L > 0$ the following collection is uniformly approximable:
\begin{align*}
\{\rho(f) T \rho(g), \rho(g) T \rho(f) \ | \ & f \in B_b(X)\text{ with }\|f\|_\infty \le 1,\text{and}\\
& g \in \LLip_R(X)\text{ with }\supp f \cap \supp g = \emptyset\},
\end{align*}
\item for all $R, L > 0$ the following collection is uniformly approximable:
\begin{align*}
\{\rho(f) T \rho(g), \rho(g) T \rho(f) \ | \ & f \in B_b(X)\text{ with }\|f\|_\infty \le 1,\text{and}\\
& g \in B_R(X)\text{ with }d(\supp f, \supp g) \ge L\}.
\end{align*}
\end{enumerate}
\end{lem}

The following Lemma \ref{lem:kasparov_lemma_uniform_traceable} is the analogue of Lemma \ref{lem:kasparov_lemma_uniform_approx} for uniform traceability and its proof is completely analogous. But note that now we do not have equivalence of the three different statements but only a chain of implications. The reason why the other implications are missing is that their proofs use approximation arguments in the operator norm, which may of course destroy every traceability we had. Indeed, we will show in Example \ref{ex:reverse_implications_l-univ-pseudoloc-traceable} that one of the missing implications is in general false. Furthermore, since the non-degeneracy of the representation $\rho$ in Lemma \ref{lem:kasparov_lemma_uniform_approx} was only needed for the implication $3 \Rightarrow 1$ which does not hold here, we can drop there the requirement that $\rho$ has to be non-degenerate.

\begin{lem}\label{lem:kasparov_lemma_uniform_traceable}
For $X$ a proper space, let $T \in \IB(H)$ and consider the following three statements:
\begin{enumerate}
\item $\{\|[T, \rho(f)]\|_{tr} \ | \ f \in \LLip_R(X)\}$ is a bounded subset of $\IR$ for all $R, L > 0$,
\item for all $R, L > 0$ the following subset of $\IR$ is bounded:
\[\{\|\rho(f) T \rho(g)\|_{tr} \ | \ f \in B_b(X), \|f\|_\infty \le 1, \ \! g \in \LLip_R(X), \ \! \supp f \cap \supp g = \emptyset\},\]
\item for all $R, L > 0$ the following subset of $\IR$ is bounded:
\[\{\|\rho(f) T \rho(g)\|_{tr} \ | \ f \in B_b(X), \|f\|_\infty \le 1, \ \! g \in B_R(X), \ \! d(\supp f, \supp g) \ge L\},\]
\end{enumerate}
Then we have the implications $1 \Rightarrow 2 \Rightarrow 3$.

Furthermore, we can change the roles of $f$ and $g$ in Points 2 and 3 (number these new statements 2' and 3') and then we have the analogous chain of implications $1 \Rightarrow 2^\prime \Rightarrow 3^\prime$.
\end{lem}

It is clear that if an operator fulfills 1, 2+2' or 3+3' of the above Lemma \ref{lem:kasparov_lemma_uniform_traceable}, then it fulfills correspondingly points 1, 2 or 3 of Lemma \ref{lem:kasparov_lemma_uniform_approx}.

\begin{defn}[Uniform pseudolocal traceability]\label{defn:uniform_pseudolocal_traceability}
Let $\rho$ be a non-degenerate representation of a proper space. An operator $T \in \IB(H)$ is called \emph{uniformly pseudolocally traceable}, if it fulfills Points 3 and 3' of the above Lemma \ref{lem:kasparov_lemma_uniform_traceable}.

We denote the set of all such operators by $\frakD_\rho^{tr}(X)$.
\end{defn}

This definition is only given for non-degenerate representations since for degenerate ones it is not the right one. This may be seen by comparing Definition \ref{defn:Dtr_rho0} of $\frakD_{\rho \oplus 0}^{tr}(X)$ with the algebra that we would get if we used for the representation $\rho \oplus 0$ the above definition.

\begin{example}\label{ex:reverse_implications_l-univ-pseudoloc-traceable}
We will now construct an operator which is uniformly pseudolocally traceable, but which does not have the Properties 2 and 2' of Lemma \ref{lem:kasparov_lemma_uniform_traceable} (and therefore it does also not have the Properties 1 and 1'). Even worse, this operator will not even be approximable (in operator norm) by operators fulfilling 2 and 2' (this shows that the corresponding Lemma \ref{lem:D_tr_dense} for operators having Properties 2 and 2' does not hold).

Let $H := L^2(\IR)$ and let $\rho\colon C_0(\IR) \to L^2(\IR)$ be the representation via multiplication operators. We denote by $I_n$ the intervals $I_n := (\tfrac{1}{2^{n+1}}, \tfrac{1}{2^n}]$ and choose for every $n \in \IN$ an orthonormal basis $\{\varphi_k^n\}_{k \in \IN}$ of $L^2(I_n)$. Furthermore, we have the intervals $-I_n = [-\tfrac{1}{2^n}, -\tfrac{1}{2^{n+1}})$ with corresponding orthonormal basis $\{\psi_k^n\}_{k \in \IN}$, i.e., $\psi_k^n(x) := \varphi_k^n(-x)$. At last, we choose for every $n \in \IN$ some number $K_n \in \IN$ and then we define an operator $T\colon L^2(\IR) \to L^2(\IR)$ via
\[T(f) := \sum_{n=1}^\infty \pi_n(f),\text{where }\pi_n(f) := \sum_{k=1}^{K_n} \langle f, \varphi_k^n \rangle \psi_k^n\colon L^2(\IR) \to L^2(-I_n),\]
i.e., $\pi_n$ projects onto the subspace of $L^2(I_n)$ spanned by the first $K_n$ basis vectors of it and then reflects at $0 \in \IR$; see the next Illustration \ref{fig:illustration_counterex}.

\begin{figure}[htbp]
\centering
\includegraphics[scale=0.8]{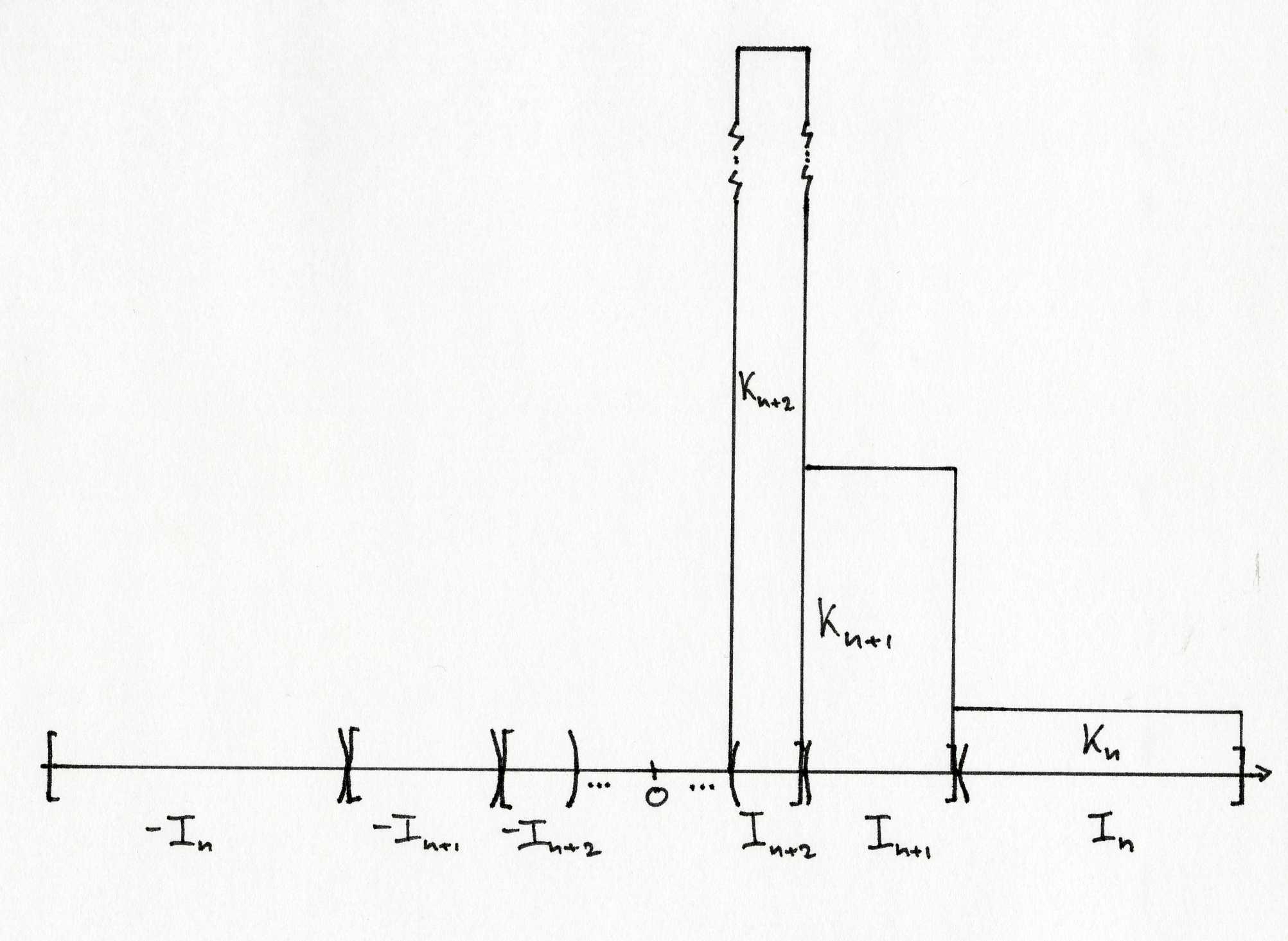}
\caption{Illustration for the operator $T\colon L^2(\IR) \to L^2(\IR)$.}
\label{fig:illustration_counterex}
\end{figure}

It is immediately clear that $T$ is uniformly pseudolocally traceable. But if the sequence $K_n$ grows too fast (e.g., $K_n \sim 2^n$ as $n \to \infty$ would suffice) then $T$ does neither have Property 2 of Lemma \ref{lem:kasparov_lemma_uniform_traceable} nor 2'. In fact, in this case $T$ can not even be approximated (in operator norm) by such operators.

\end{example}

Recall Definition \ref{defn:local_Cstar_algebra} of local $C^\ast$-algebras for the next lemma:

\begin{lem}\label{lem:D_tr_local}
Let $\rho$ be non-degenerate. Then $\frakD_\rho^{tr}(X)$ is a local $C^\ast$-algebra.
\end{lem}

\begin{proof}
Let us first show that $\frakD_\rho^{tr}(X)$ is an algebra, i.e., that the product $TS$ of two uniformly pseudolocally traceable operators is again uniformly pseudolocally traceable: let $f \in B_b(X)$, $g \in B_R(X)$ with $\|f\|_\infty \le 1$ and $d(\supp f, \supp g) =: L > 0$ be given. Let $A := B_{L/2}(\supp g)$, $B := X - A$ and denote by $\chi_A$ the characteristic function of $A$ and by $\chi_B$ the one of $B$; see the above Figure \ref{fig:subsets_A_B}.

\begin{figure}[htbp]
\centering
\includegraphics[scale=0.6]{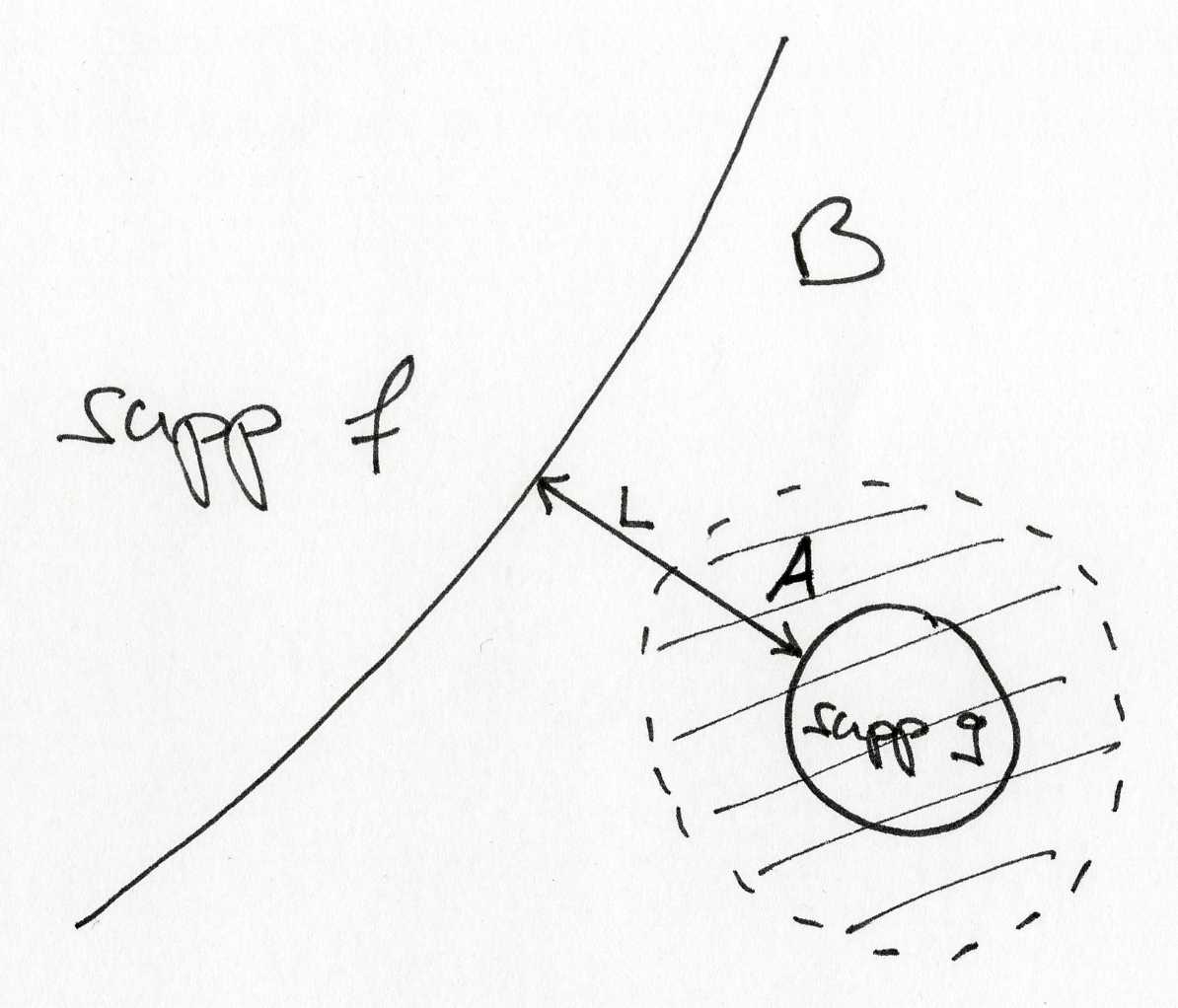}
\caption{The subsets $A$ and $B$ corresponding to $f \in B_b(X)$ and $g \in B_R(X)$.}
\label{fig:subsets_A_B}
\end{figure}

Then we have (since $\rho$ is non-degenerate)
\begin{equation}
\label{eq:D_tr_algbra}
\rho(f) TS \rho(g) =  \rho(f) T \rho(\chi_A) S \rho(g) + \rho(f) T \rho(\chi_B) S \rho(g).
\end{equation}

Denoting by $C_T(R,L)$ the least possible constant such that $\|\rho(f) T \rho(g)\|_{tr} \le C$ for all $g \in B_R(X)$, $f \in B_b(X)$ with $\|f\|_\infty \le 1$ and $d(\supp f, \supp g) \ge L$, we get from Equation \eqref{eq:D_tr_algbra}
\begin{equation}
\label{eq:D_tr_seminorms_submult}
C_{TS}(R, L) \le C_T(R+L,L/2) \cdot \|S\|_{op} + \|T\|_{op} \cdot C_S(R, L/2).
\end{equation}
Denoting by $C^\prime_T(R,L)$ the same as $C_T(R,L)$ but with the roles of $f$ and $g$ changed, we get a similar inequality for $C^\prime_{TS}(R,L)$. This shows that $TS$ is again uniformly pseudolocally traceable.

It is clear that both $C_{\largecdot}(R,L)$ and $C^\prime_{\largecdot}(R,L)$ are semi-norms on $\frakD_\rho^{tr}(X)$ for all $R, L > 0$. Furthermore, the collection $\{C_{\largecdot}(R,L), C^\prime_{\largecdot}(R,L)\}$ for all positive rational numbers $R,L \in \IQ_{> 0}$ together with the operator norm $\|\largecdot\|_{op}$ let $\frakD_\rho^{tr}(X)$ be a \Frechet algebra (Definition \ref{defn:Frechet_subalgebras}). That the multiplication is jointly continuous follows from the Estimate \eqref{eq:D_tr_seminorms_submult} together with the corresponding one for $C^\prime_{TS}(R,L)$.

From the Lemmas \ref{lem:matrix_algebras_holomorphically_closed} and \ref{lem:local_algebra_equivalent} we conclude that in order to show that $\frakD_\rho^{tr}(X)$ is a local $C^\ast$-algebra, it suffices to show that it is inverse closed, i.e., if $T \in \frakD_\rho^{tr}(X)$ is invertible in $\frakD_\rho^u(X)$\footnote{Chosing here $\frakD_\rho^u(X)$ is more or less arbitrary. We just need some $C^\ast$-superalgebra of $\frakD_\rho^{tr}(X)$, i.e., we could have also chosen $\IB(H)$.}, then $T^{-1} \in \frakD_\rho^{tr}(X)$. So let $f \in B_b(X)$ with $\|f\|_\infty \le 1$ and $g \in B_R(X)$ with $d(\supp f, \supp g) =: L > 0$ be given and w have to show that $\|\rho(f) T^{-1} \rho(g)\|_{tr} < \infty$ and that this bound does not depend on the concrete choices of $f$ or $g$.

First of all note that
\[\|\rho(f) T^{-1} \rho(g)\|_{tr} \le \|\chi_f T^{-1} \chi_g\|_{tr}\]
since $\rho(f) T^{-1} \rho(g) = \rho(f) \chi_f T^{-1} \chi_g \rho(g)$, where $\chi_f$ is $\rho(\chi_{\supp f})$ for the characteristic function $\chi_{\supp f}$ of $\supp f$ and analogously for $\chi_g$. Furthermore,
\begin{align*}
\|\chi_f T^{-1} \chi_g\|_{tr} & = \sum \big\langle \big( (\chi_f T^{-1} \chi_g)^\ast (\chi_f T^{-1} \chi_g) \big)^{1/2} e_k, e_k \big\rangle\\
& \le \| \chi_f T^{-1} \chi_g + \chi_g (T^\ast)^{-1} \chi_f \|_{tr},
\end{align*}
because we have
\begin{align*}
(\chi_f T^{-1} & \chi_g)^\ast (\chi_f T^{-1} \rho(g))\\
& \le (\chi_f T^{-1} \chi_g + \chi_g (T^\ast)^{-1} \chi_f)^\ast (\chi_f T^{-1} \chi_g + \chi_g (T^\ast)^{-1} \chi_f)
\end{align*}
in the sense of positive operators (to prove the above inequality multiply out and use that $\chi_f \chi_g = 0$ since they have disjoint supports).

Now we consider the operator $\chi_f T^\ast \chi_g + \chi_g T \chi_f$. It is self-adjoint and traceable (the latter since $T \in \frakD_\rho^{tr}(X)$), i.e., the eigenvalues of it are $1$-summable. Let $v \in H$ be an eigenvector of this operator to the eigenvalue $\lambda \not= 0$. Applying the operator to $\chi_g v$ we get $\chi_f T^\ast \chi_g v$ which must equal $\lambda \chi_f v$ (we use here multiple times the fact that the supports of $\chi_f$ and $\chi_g$ are disjoint). On the other side, applying the operator to $\chi_f v$ we get $\chi_g T \chi_f v$ which must equal $\lambda \chi_g v$. Putting this together we conclude that $\chi_f T^{-1} \chi_g + \chi_g (T^\ast)^{-1} \chi_f$ is the inverse of $\chi_f T^\ast \chi_g + \chi_g T \chi_f$ for the vector $v$. Since $v$ was arbitrary, we conclude that the first operator is the inverse of the latter on all non-zero eigenspaces of the latter. With a similar argument we get that the first operator is $0$ whenever the latter is so, i.e., their kernels coincide.

We conclude that the non-zero eigenvalues of $\chi_f T^{-1} \chi_g + \chi_g (T^\ast)^{-1} \chi_f$ are exactly the reciprocals of the eigenvalues of $\chi_f T^\ast \chi_g + \chi_g T \chi_f$, and we already know that their kernels coincide. Since $\chi_f T^{-1} \chi_g + \chi_g (T^\ast)^{-1} \chi_f$ is a bounded operator (with norm bounded by, say, $C$), all its eigenvalues are bounded from above by $C$. Therefore the eigenvalues of $\chi_f T^\ast \chi_g + \chi_g T \chi_f$ bounded from below by $1/C$. But we also know that the eigenvalues of this operator are $1$-summable, i.e., it can therefore have only finitely many non-zeroes eigenvalues (and the number of them is bounded independently of the concrete choices of $f$ and $g$, since $C$ may be chosen to not depend on $f$ or $g$ and the same also holds for the trace norm of $\chi_f T^\ast \chi_g + \chi_g T \chi_f$). So we conclude that the trace norm of $\chi_f T^{-1} \chi_g + \chi_g (T^\ast)^{-1} \chi_f$ is bounded from above and this bound may be chosen to not depend on $f$ or $g$. From this our claim that $\|\rho(f) T^{-1} \rho(g)\|_{tr} < \infty$ and that this bound does not depend on the concrete choices of $f$ or $g$ follows.

Applying the same reasoning to the adjoint operator and also with the roles of $f$ and $g$ changed, we finally conclude that $T^{-1} \in \frakD_\rho^{tr}(X)$.
\end{proof}

Now we would really like to prove that $\frakD_\rho^{tr}(X)$ is dense in $\frakD_\rho^{u}(X)$ so that we can conclude that their $K$-theory groups coincide. But unfortunately, we get here the same problems that we would get if we tried to prove that $\C(E)$ is dense in $\IU(E)$: the approximation ``at infinity'' is problematic. We will circumvent this problem by passing to representatives of finite propagation. But before we may do this, we have to define the algebra $\frakD^{tr}_{\rho \oplus 0}(X)$, i.e., we can not use the above Definition \ref{defn:uniform_pseudolocal_traceability} since $\rho \oplus 0$ is not non-degenerate. We have to define this since by Paschke duality we know that $K_\ast^u(X; {\rho \oplus 0}) \cong K_{1+\ast}(\frakD_{\rho \oplus 0}^u(X))$ and here we now want to pass to uniformly pseudolocally traceable operators.

\begin{defn}\label{defn:Dtr_rho0}
Let $\rho\colon C_0(X) \to \IB(H)$ be a non-degenerate representation. We define $\frakD^{tr}_{\rho \oplus 0}(X) \subset \IB(H \oplus H)$ to be
\[\frakD^{tr}_{\rho \oplus 0}(X) := \begin{pmatrix} \frakD^{tr}_\rho(X) & \fraklC^{tr}_\rho(X) \\ \frakCr_\rho^{tr}(X) & \IB(H)\end{pmatrix} \subset \frakD_{\rho \oplus 0}^u(X),\]
where $\fraklC^{tr}_\rho(X)$ are the operators that are uniformly locally traceable \emph{from the left}, i.e., operators $T \in \IB(H)$ for which the subset $\{\|\rho(f) T\|_{tr} \ | \ f \in \LLip_R(X)\} \subset \IR$ is bounded for all $R, L > 0$. Analogously we define $\frakCr^{tr}_\rho(X)$ to consist of the operators that are uniformly locally traceable from the right.
\end{defn}

Similar as in the proof of Lemma \ref{lem:D_tr_local} we can prove the fact that $\frakD^{tr}_{\rho \oplus 0}(X)$ is a local $C^\ast$-algebra. Note that the non-degeneracy assumption in that lemma is only used in the proof that such operators compose correctly. But in our case here for $\frakD^{tr}_{\rho \oplus 0}(X)$ this follows from the concrete form that we give the algebra by definition together with Lemma \ref{lem:D_tr_local}.

Now we define the finite propagation versions of uniformly pseudolocal and uniformly locally compact operators and then show that we may pass to them on the level of uniform $K$-homology:

\begin{defn}[cf. Definition \ref{defn_Du_Cu_manifolds}]\label{defn:_Du_Cu_spaces}
Let $\rho\colon C_0(X) \to \IB(H)$ be a non-degenerate representation. We will denote by $D_u^\ast(X) \subset \IB(H)$ the $C^\ast$-algebra generated by all uniformly pseudolocal operators having finite propagation, and by $C_u^\ast(X) \subset D_u^\ast(X)$ the closed, two-sided $^\ast$-ideal generated by all uniformly locally compact operators with finite propagation.
\end{defn}

The next lemma is similar to the results of \cite[Lemma 12.3.2]{higson_roe}.

\begin{lem}\label{lem:uniform_k_hom_via_dual_quot_2}
We have $K_\ast^u(X; \rho \oplus 0) \cong K_{1 + \ast}(D_u^\ast(X) / C_u^\ast(X))$ for $\ast = -1,0$.
\end{lem}

\begin{proof}
By Paschke duality we know
\[K_\ast^u(X; \rho \oplus 0) \cong K_{1+\ast}(\frakD^u_{\rho \oplus 0}(X)).\]
We will now show that
\begin{equation}\label{eq:uniform_K_as_quotient_D*u/C*u_2}
K_\ast(\frakD_{\rho \oplus 0}^u(X)) \cong K_\ast(\frakD^u_\rho(X) / \frakC^u_\rho(X)),
\end{equation}
and that the inclusion of $D_u^\ast(X)$ into $\frakD^u_\rho(X)$ induces an isomorphism
\begin{equation}\label{eq:quotients_frakD_D*_equal_2}
D_u^\ast(X) / C_u^\ast(X) \cong \frakD^u_\rho(X) / \frakC^u_\rho(X)
\end{equation}
of $C^\ast$-algebras. These two statements together prove the claim.

To prove the first Isomorphism \eqref{eq:uniform_K_as_quotient_D*u/C*u_2} first note that
\[\frakD^u_{\rho \oplus 0}(X) = \begin{pmatrix}\frakD_\rho^u(X) & \fraklC^u_\rho(X) \\ \frakCr_\rho^u(X) & \IB(H)\end{pmatrix} \subset \IB(H \oplus H),\]
where $\fraklC^u_\rho(X)$ contains the operators that are uniformly locally compact \emph{from the left}, i.e., operators $T \in \IB(H)$ for which the collection $\{\rho(f) T \ | \ f \in \LLip_R(X)\}$ is uniformly approximable for all $R, L > 0$ (cf. Definitions \ref{defn:uniformly_approx_collections} and \ref{defn:uniform_operators}). Analogously, $\frakCr_\rho^u(X)$ is defined as the algebra containing all operators that are uniformly locally compact from the right. Note that we have $\fraklC_\rho^u(X) \cdot \frakCr_\rho^u(X) \subset \frakC_\rho^u(X)$ and also $\IB(H) \cdot \frakCr_\rho^u(X) \subset \frakCr_\rho^u(X)$ and $\fraklC_\rho^u(X) \cdot \IB(H) \subset \fraklC_\rho^u(X)$. Furthermore, we have $\frakD_\rho^u(X) \cdot \fraklC_\rho^u(X) \subset \fraklC_\rho^u(X)$ and analogously $\frakCr_\rho^u(X) \cdot \frakD_\rho^u(X) \subset \frakCr_\rho^u(X)$.

All the above inclusions are needed to show that $J \subset \frakD^u_{\rho \oplus 0}(X)$ defined as
\[J := \begin{pmatrix}\frakC_\rho^u(X) & \fraklC^u_\rho(X) \\ \frakCr_\rho^u(X) & \IB(H)\end{pmatrix} \subset \frakD^u_{\rho \oplus 0}(X)\]
is a closed, two-sided $^\ast$-ideal in $\frakD^u_{\rho \oplus 0}(X)$. So we get a short exact sequence
\[ 0 \to J \to \frakD^u_{\rho \oplus 0}(X) \to \frakD^u_{\rho \oplus 0}(X) / J \to 0.\]
Now we may identify the quotient $\frakD^u_{\rho \oplus 0}(X) / J$ with $\frakD_\rho^u(X) / \frakC_\rho^u(X)$ and the quotient map $\frakD^u_{\rho \oplus 0}(X) \to \frakD^u_{\rho \oplus 0}(X) / J$ in the short exact sequence becomes the map
\[\frakD^u_{\rho \oplus 0}(X) \ni \begin{pmatrix}T_{11} & T_{12} \\ T_{21} & T_{22}\end{pmatrix} \mapsto [T_{11}] \in \frakD_\rho^u(X) / \frakC_\rho^u(X).\]
Furthermore, we have $J = \frakC_{\rho \oplus 0}^u(X)$. Hence the above short exact sequence becomes
\[0 \to \frakC_{\rho \oplus 0}^u(X) \to \frakD_{\rho \oplus 0}^u(X) \to \frakD_\rho^u(X) / \frakC_\rho^u(X) \to 0\]
and the claim that $K_\ast(\frakD_{\rho \oplus 0}^u(X)) \cong K_\ast(\frakD^u_\rho(X) / \frakC^u_\rho(X))$ now follows from the $6$-term exact sequence for $K$-theory and the fact that all the $K$-groups of $\frakC_{\rho \oplus 0}^u(X)$ vanish. This is a uniform analogue of the corresponding non-uniform statement which is proved in, e.g., \cite[Lemma 5.4.1]{higson_roe}, and this uniform analogue was essentially proven by \Spakula in \cite[Lemma 5.3]{spakula_uniform_k_homology} (by ``setting $Z := \emptyset$'' in that lemma).

The proof of the second Isomorphism \eqref{eq:quotients_frakD_D*_equal_2} is analogous to the proof of the corresponding non-uniform statement $D^\ast(X) / C^\ast(X) \cong \frakD_\rho(X) / \frakC_\rho(X)$ which may be found in, e.g., \cite[Lemma 12.3.2]{higson_roe}. This uniform version \eqref{eq:quotients_frakD_D*_equal_2} was also basically already shown by \Spakula in \cite[Section 7]{spakula_uniform_k_homology}.
\end{proof}

Let us now define the finite propagation versions of uniformly pseudolocally traceable and uniformly locally traceable operators:

\begin{defn}
Let $\rho\colon C_0(X) \to \IB(H)$ be a non-degenerate representation. We will denote by $D_{\mathrm{tr}}^\ast(X) \subset \IB(H)$ the \Frechet algebra\footnote{As semi-norms we use the collection $\{C_{\largecdot}(R,L), C^\prime_{\largecdot}(R,L)\}$, as defined in the proof of Lemma \ref{lem:D_tr_local}, for all positive rational numbers $R,L \in \IQ_{> 0}$ together with the operator norm $\|\largecdot\|_{op}$.} generated by all uniformly pseudolocally traceable operators having finite propagation, and by $C_{\mathrm{tr}}^\ast(X) \subset D_{\mathrm{tr}}^\ast(X)$ the closed, two-sided $^\ast$-ideal generated by all uniformly locally traceable operators with finite propagation.
\end{defn}

Both algebras $D^\ast_{\mathrm{tr}}(X)$ and $C^\ast_{\mathrm{tr}}(X)$ are local $C^\ast$-algebras. The proof for $D^\ast_{\mathrm{tr}}(X)$ is similar to the one of Lemma \ref{lem:D_tr_local}: we have additionally to argue why the inverse operator $T^{-1}$ occuring in the proof of inverse closedness is again approximable by finite propagation operators, but this follows from the fact that $T$ is so. The case of $C^\ast_{\mathrm{tr}}(X)$ may also be handled.

Now we will finally prove that $D^\ast_{\mathrm{tr}}(X)$ is dense in $D^\ast_u(X)$ and that $C^\ast_{\mathrm{tr}}(X)$ is dense in $C^\ast_u(X)$. For this we will use the notion of a \emph{quasi-latticing partition}, which we have already defined in Definition \ref{defn:quasi-latticing_partitions}. But for the convenience of the reader we will restate it:

\begin{defn}[Quasi-latticing partitions, {\cite[Definition 8.1]{spakula_uniform_k_homology}}]
Let $Y \subset X$ be a uniformly discrete quasi-lattice (see Definition \ref{defn:coarsely_bounded_geometry}).

A collection $(V_y)_{y \in Y}$ of open, disjoint subsets of $X$ is called a \emph{quasi-latticing partition with diameters $\le$ d}, if $X = \bigcup_{y \in Y} \overline{V_y}$, $\sup_{y \in Y} \diam V_y \le d$ and for every $\varepsilon > 0$ we have
\begin{equation}\label{eq:constant_in_quasi_latticing_def}
R_\varepsilon := \sup_{y \in Y} \card \{z \in Y \ | \ V_z \cap B_\varepsilon(V_y) \not= \emptyset\} < \infty.
\end{equation}
\end{defn}

Given a uniformly discrete quasi-lattice $Y \subset X$, we can immediately produce a quasi-latticing partition from it by ``putting'' a point $x \in X$ into that $V_y$ for which $y$ is the closest point of $Y$ to $x$. But for the next Lemma \ref{lem:D_tr_dense} we need quasi-latticing partitions with arbitrarily small diameters and this can in general not be achieved if $X$ is only of coarsely bounded geometry. A sufficient condition for this would be, e.g., that $X$ has arbitrarily fine coarsely bounded geometry:

\begin{defn}[Arbitrarily fine coarsely bounded geometry]\label{defn:arbitrarily_fine_coarsely_bounded_geometry}
A metric space $X$ has \emph{arbitrarily fine coarsely bounded geometry}, if for all $c > 0$ it admits quasi-lattices $\Gamma_c \subset X$ with $B_c(\Gamma_c) = X$.
\end{defn}

Recalling the notion of jointly bounded geometry from Definition \ref{defn:jointly_bounded_geometry}, we see that a space $X$ having jointly bounded geometry has arbitrarily fine coarsely bounded geometry. But the converse is also true (at least, if $X$ admits some countable quasi-lattice), which is easily seen.

\begin{lem}\label{lem:equiv_jointly_arbitrarily_fine_bounded_geometry}
Let $X$ be a metric space.

If $X$ has jointly bounded geometry, then it follows that $X$ has arbitrarily fine coarsely bounded geometry.

Conversely, if $X$ admits at least one countable quasi-lattice, then having arbitrarily fine coarsely bounded geometry implies that $X$ has jointly bounded geometry.
\end{lem}

We can now prove the main technical result of this section:

\begin{lem}\label{lem:D_tr_dense}
Let $\rho\colon C_0(X) \to \IB(H)$ be a non-degenerate representation and let $X$ admit quasi-latticing partitions with arbitrarily small diameters.

Then $D^\ast_{\mathrm{tr}}(X)$ is dense in $D^\ast_u(X)$ and $C^\ast_{\mathrm{tr}}(X)$ is dense in $C^\ast_u(X)$.
\end{lem}

\begin{proof}
We first treat the case of $D^\ast_{\mathrm{tr}}(X) \subset D^\ast_u(X)$.

We choose for all $n \in \IN$ quasi-latticing partitions $(V_y^n)_{y \in Y_n}$ with diameters $\le 1/n$. Denoting by $\chi_y^n$ the characteristic function of the subset $V_y^n \subset X$, we get for all $n \in \IN$ direct sum decompositions $H = \bigoplus_{y \in Y_n} H_y^n$, where $H_y^n := \rho(\chi_y^n)H$. With respect to these decompositions we may write an operator $T \in \IB(H)$ as a direct sum $T = \bigoplus_{y,z \in Y_n} T_{y,z}^n$, where $T_{y,z}^n = \rho(\chi_z^n) T \rho(\chi_y^n) \colon H_y^n \to H_z^n$.

Now let $T \in D^\ast_u(X)$. By Lemma \ref{lem:kasparov_lemma_uniform_approx} this means that if we have for some $y, z \in Y_n$ that $d(\supp \chi^n_y, \supp \chi^n_z) \ge L > 0$, then $T_{y,z}^n = \rho(\chi_z^n) T \rho(\chi_y^n)$ is up to $\varepsilon$ a finite rank operator and this rank is bounded from above by a constant that depends only on $R = 1/n$ and $L$. In our construction (which we will describe in a moment) we will have to choose concrete approximations of the operators $T_{y,z}^n$ by finite rank ones: since $T_{y,z}^n$ is a compact operator, we may write it as a sum $T_{y,z}^n = \sum_{i=1}^\infty \lambda_i \langle \largecdot, e_i\rangle f_i$, where $(\lambda_i)_i$ is a null sequence. Then, given an $\varepsilon > 0$, we will choose $\sum_{i=1}^{N_\varepsilon} \lambda_i \langle \largecdot, e_i\rangle f_i$ as the approximating finite rank operator, where $N_\varepsilon$ is such that $\lambda_i \le \varepsilon$ for all $i > N_\varepsilon$ and $\lambda_i > \varepsilon$ for all $i \le N_\varepsilon$. Note that it is no problem for us that the orthonormal (but not necessarily complete) sets $\{e_i\}$, $\{f_i\}$ are not necessarily unique in this representation (but the values $\lambda_i$ are).

Let $\varepsilon > 0$ be given. We replace in the decomposition $T = \bigoplus_{y,z \in Y_1} T_{y,z}^1$ every $T_{y,z}^1$ with $d(\supp \chi^1_y, \supp \chi^1_z) \ge 1/1$ by the corresponding finite rank operator which is $\varepsilon$ away and denote the resulting operator by $T^1$. See Figure \ref{fig:approx_pseudoloc} for an illustration of this first and also the next step described in the next paragraph.

\begin{figure}[htbp]
\centering
\includegraphics[scale=0.55]{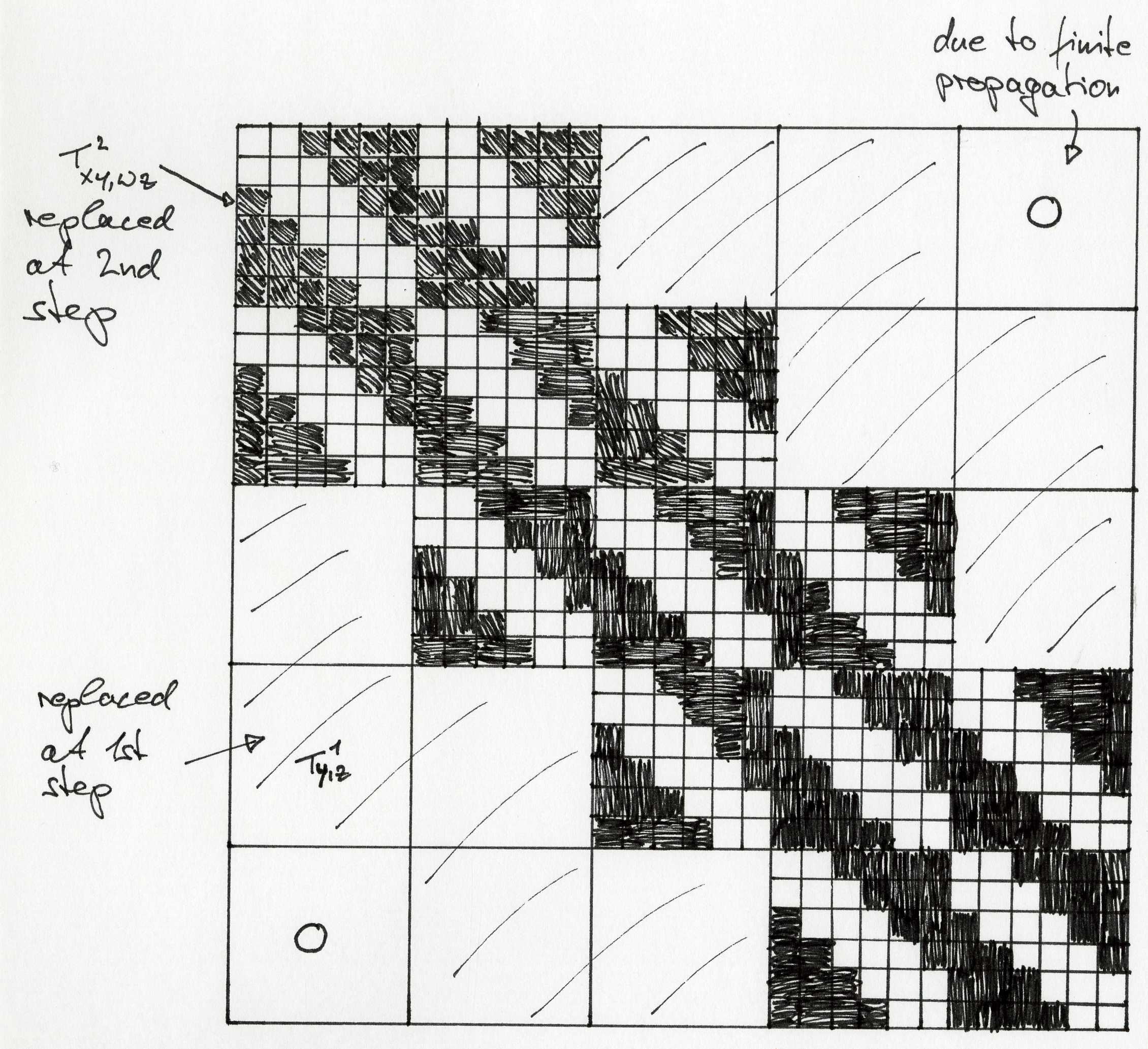}
\caption{The first and second step of the approximation of $T$.}
\label{fig:approx_pseudoloc}
\end{figure}

Let $T^1_{y,z}$ be one of the operators which were not replaced and consider the family of open sets $(V_x^2 \cap V_y^1)_{x \in Y_2}$. This is a quasi-latticing partition of $V_y^1$ with diameters $\le 1/2$. We denote the intersection $V_x^2 \cap V_y^1$ by $V_{xy}^2$ and analogously we get a quasi-latticing partition $(V_{xz}^2)_{x \in Y_2}$ of $V_z^1$ with diameters $\le 1/2$. With respect to this we may write $T_{y,z}^1 = \bigoplus_{x,w \in Y_2} T^2_{xy,wz}$ and again we replace every $T^2_{xy,wz}$ with $d(\supp \chi_{xy}^2, \supp \chi_{wz}^2) \ge 1/2$ by the corresponding finite rank operator which is $\varepsilon$ away. Now we do this for every $T^1_{y,z}$ which was not replaced in the first step and get some operator $T^2$.

So at the $n$th step we have some ``parts of $T$'' left which are not replaced, because the two characteristic functions used to define such a part have supports which are not at least $1/n$ apart from each other. In the $(n+1)$st step we then take each of this left over parts, break it down into finer parts (by passing to the induced, finer quasi-latticing partition) and replace some of these finer parts (the ones which characteristic functions have supports at least $1/(n+1)$ apart from each other). This produces a sequence of operators $T^n$ with the following properties:
\begin{enumerate}
\item $\|T^n - T\|_{op} \le R_{n,\sfrac{1}{n}} \cdot \varepsilon$ for all $n$, where $R_{n,\varepsilon}$ is the constant \eqref{eq:constant_in_quasi_latticing_def} in the definition of quasi-latticing partitions of unity corresponding to the partition $(V_y^n)_{y \in Y_n}$ used at the $n$th step,
\item for every $L > 0$ there exists an $N_L \in \IN$ such that $\rho(f) T^n \rho(g) = \rho(f) T^m \rho(g)$ for all $n, m \ge N_L$ and all $f, g \in B_b(X)$ with $d(\supp f, \supp g) \ge L$ and
\item for all $n$ there is an $L_n > 0$ such that $T^n$ is uniformly pseudolocally traceable if we restrict this notion only to all $L > L_n$, and furthermore $L_n \to 0$ as $n \to \infty$.
\end{enumerate}
Furthermore, the operators $T^n$ convergence in the weak operator topology to an operator $T^\infty$. Using the first of the above three properties of the operators $T^n$ we may force that $\|T^\infty - T\|_{op} \le \varepsilon$ by approximating at the $n$-th step not up to $\varepsilon$, but up to $\varepsilon / R_{n,\sfrac{1}{n}}$ (this makes the uniform pseudolocal traceability of the limit operator $T^\infty$ worse in $L$, but this is ok for us). From the second and third point we conclude that $T^\infty$ is uniformly pseudolocally traceable. Since $\varepsilon$ was arbitrary, the proof that $D^\ast_{\mathrm{tr}}(X)$ is dense in $D^\ast_u(X)$ is completed.

Similarly we may prove that $C^\ast_{\mathrm{tr}}(X)$ is dense in $C^\ast_u(X)$.
\end{proof}

From the above lemma we conclude $D_u^\ast(X) / C_u^\ast(X) \cong D_{\mathrm{tr}}^\ast(X) / C_{\mathrm{tr}}^\ast(X)$. Moreover, in Lemma \ref{lem:uniform_k_hom_via_dual_quot_2} we showed $D_u^\ast(X) / C_u^\ast(X) \cong \frakD^u_\rho(X) / \frakC^u_\rho(X)$ and an analogous proof shows us the corresponding statement
\[D_{\mathrm{tr}}^\ast(X) / C_{\mathrm{tr}}^\ast(X) \cong \frakD^{tr}_\rho(X) / \frakC^{tr}_\rho(X).\]
Furthermore, it was also shown in Lemma \ref{lem:uniform_k_hom_via_dual_quot_2} that we have an isomorphism of $K$-groups $K_\ast(\frakD_{\rho \oplus 0}^u(X)) \cong K_\ast(\frakD^u_\rho(X) / \frakC^u_\rho(X))$ and again we may analogously show the corresponding result
\[K_\ast(\frakD_{\rho \oplus 0}^{tr}(X)) \cong K_\ast(\frakD^{tr}_\rho(X) / \frakC^{tr}_\rho(X)).\]
Note that the proof of the last statement also requires us to show that the $K$-theory groups of $\frakC^{tr}_{\rho \oplus 0}(X)$ vanish, which may be shown analogously as the fact that the $K$-groups of $\frakC_{\rho \oplus 0}^u(X)$ vanish (on which we elaborated in the proof of Lemma \ref{lem:uniform_k_hom_via_dual_quot_2}). The reason for this is that the proof does not use any approximation in operator norm, i.e., any traceability we had is still preserved in all constructions.

Combining all the above we get
\[K_{\ast}(\frakD^u_{\rho \oplus 0}(X)) \cong K_{\ast}(\frakD_{\rho \oplus 0}^{tr}(X)).\]
So we have proved our final lemma of this section:

\begin{lem}[Normalization to uniformly pseudolocally traceable operators]\label{lem:K_groups_D_tr_D_u_equal_degenerate_rho}
Let $\rho$ be a non-degenerate representation and let $X$ admit quasi-latticing partitions with arbitrarily small diameters.

Then $\frakD_{\rho \oplus 0}^{tr}(X)$ is a local $C^\ast$-algebra and its operator $K$-groups coincide with the ones of $\frakD^u_{\rho \oplus 0}(X)$, i.e.,
\[K_\ast(\frakD_{\rho \oplus 0}^{tr}(X)) \cong K_\ast(\frakD_{\rho \oplus 0}^{u}(X)).\]
\end{lem}

\section{Index maps on even uniform \texorpdfstring{$K$}{K}-homology}
\label{sec:index_maps_K_hom}

Let $X$ be a proper metric space having coarsely bounded geometry and let $Y \subset X$ be a uniformly discrete quasi-lattice. We have already mentioned in Section \ref{sec:amenability} that \Spakula constructed in \cite[Section 9]{spakula_uniform_k_homology} a uniform coarse assembly map
\[\mu_u\colon K_\ast^u(X) \to K_\ast(C_u^\ast(Y))\]
for $\ast = -1,0$. Recall from Proposition \ref{prop:index_maps_uniform_Roe_algebras} that if $Y$ is amenable, then for every \Folner sequence $(U_i)_i$ of $Y$ and any functional $\tau \in (\ell^\infty)^\ast$ associated to a free ultrafilter on $\IN$\footnote{That is, if we evaluate $\tau$ on a bounded sequence, we get the limit of some convergent subsequence.} there is a corresponding index map $\ind_\tau \colon K_0(C_u^\ast(Y)) \to \IR$. In this section we will now construct compatible index maps on $K_0^u(X)$.

\begin{prop}
Let $X$ be a proper metric space of jointly bounded geometry\footnote{see Definition \ref{defn:jointly_bounded_geometry}} and $Y \subset X$ a uniformly discrete quasi-lattice. Then for every \Folner sequence $(U_i)_i$ of $Y$ and any functional $\tau \in (\ell^\infty)^\ast$ associated to a free ultrafilter on $\IN$ there is an index map
\[\ind_\tau \colon K_0^u(X) \to \IR\]
(defined via Formula \eqref{eq:formula_ind_T_supertrace} in the proof) such that the following diagram commutes:
\begin{equation*}\label{eq:diag_index_maps_uniform_assembly}
\xymatrix{K_0^u(X) \ar[rr]^{\mu_u} \ar[dr]_{\ind_\tau} & & K_0(C_u^\ast(Y)) \ar[dl]^{\ind_\tau} \\ & \IR}
\end{equation*}
\end{prop}

\begin{proof}
Let $\tilde{\rho}$ be an ample representation of $C_0(X)$ on the Hilbert space $\tilde{H}$. Then by Theorem \ref{thm:paschke_universal} we get $K_0^u(X) \cong K_0^u(X; \tilde{\rho} \oplus 0)$, i.e., every graded uniform Fredholm module in $K_0^u(X)$ may be forced to be of the form $(H^\prime \oplus H^\prime, \rho^\prime \oplus \rho^\prime, T)$, where $H^\prime$ is a finite or countably infinite direct sum of $\tilde{H} \oplus \tilde{H}$ and $\rho^\prime$ analogously a direct sum of finitely or infinitely many $\tilde{\rho} \oplus 0$, and the grading is given by interchanging the two summands in $H^\prime \oplus H^\prime$.

For simplicity of our arguments, we will assume $H^\prime = \tilde{H} \oplus \tilde{H}$ and $\rho^\prime = \tilde{\rho} \oplus 0$. Furthermore, let us rearrange the space $H^\prime \oplus H^\prime = (H_1 \oplus H_2) \oplus (H_3 \oplus H_4)$ with $H_{1,2,3,4} = \tilde{H}$ so that we have $H^\prime \oplus H^\prime = H \oplus H$, where $H := H_1 \oplus H_3$, resp. $H := H_2 \oplus H_4$, is now a graded Hilbert space and the representation $\rho^\prime \oplus \rho^\prime$ is now under this rearranging of the form $(\tilde{\rho} \oplus \tilde{\rho}) \oplus (0 \oplus 0) =: \rho \oplus 0$.

Summarizing this, we assume that every element of $K_0^u(X)$ may be written as $(H \oplus H, \rho \oplus 0, T)$, where $H$ is a graded Hilbert space and $\rho$ is ample. Furthermore, we will do some normalization from Proposition \ref{prop:normalization_finite_prop_speed}: we assume that $T$ is involutive, i.e., $T = T^\ast$ and $T^2 = 1$, and that $T$ has finite propagation.

Now in order to define an index of $T$, we have to pass to a non-degenerate module (not doing this would result in a not well-defined index). Examining the procedure of passing to non-degenerate modules (Lemma \ref{lem:normalization_non-degenerate}) reveals that in our case here we arrive at the uniform module $(H, \rho, T_{11})$, where $T_{11} = \pi_1 T \pi_1$ for the projection $\pi_1 \colon H \oplus H \to H$ onto the first summand.

Let $(V_y)_{y \in Y}$ be a quasi-latticing partition with diameters $\le d$ (Definition \ref{defn:quasi-latticing_partitions}) and such that $y \in V_y$ for all $y \in Y$. Let $(U_i)_i$ be a \Folner sequence of $Y$ and we denote by $\chi_i$ the characteristic function of $\bigcup_{y \in U_i} \overline{V_y}$. Now if $\tau \in (\ell^\infty)^\ast$ is a functional associated to a free ultrafilter on $\IN$, we define
\begin{equation}
\label{eq:formula_ind_T_supertrace}
\ind_\tau (T_{11}) = \tau\big(\tfrac{1}{\card U_i} \trace_\epsilon \rho(\chi_i) (1_H - T_{11}^2) \rho(\chi_i)\big),
\end{equation}
where $\epsilon$ is the grading automorphism of the space $H$ and $\trace_\epsilon$ is the graded trace from Definition \ref{defn:graded_trace}.

First we have to show that this index is well-defined, i.e., that $\rho(\chi_i) (1-T_{11}^2) \rho(\chi_i)$ is of trace class and that the so defined sequence (indexed by $i$) is bounded from above by a constant times $\card U_i$ (so that we may evaluate $\tau$ on it). Now the problem is that this is generally false. But that is the reason why we proved Lemma \ref{lem:K_groups_D_tr_D_u_equal_degenerate_rho}. Since $X$ has jointly bounded geometry, it admits quasi-latticing partitions with arbitrarily small diameters (Lemma \ref{lem:equiv_jointly_arbitrarily_fine_bounded_geometry}). This means that Lemma \ref{lem:K_groups_D_tr_D_u_equal_degenerate_rho} is applicable, and combining it with Paschke duality, we may assume that $T \in \frakD_{\rho \oplus 0}^{tr}(X)$.

Now $\pi_1 = \rho(\chi_X)$, where $\chi_X$ is the characteristic function of the whole space $X$, which leads to $\pi_1 \rho(\chi_i) = \rho(\chi_i) = \rho(\chi_i) \pi_1$. Denoting furthermore the projection of $H \oplus H$ onto the second summand by $\pi_2$ and using $T^2=1$, we get
\[\rho(\chi_i) (1_H - T_{11}^2) \rho(\chi_i) = \rho(\chi_i) (T \pi_2 T) \rho(\chi_i).\]
Writing $T = \begin{pmatrix} T_{11} & T_{12} \\ T_{21} & T_{22}\end{pmatrix} \in \IB(H \oplus H)$, we see that $T \pi_2 T = \begin{pmatrix} T_{12} T_{21} & T_{12} T_{22} \\ T_{22} T_{21} & T^2_{22}\end{pmatrix}$, and therefore $\rho(\chi_i) (T \pi_2 T) \rho(\chi_i) = \rho(\chi_i) (T_{12} T_{21}) \rho(\chi_i)$. Since $T \in \frakD_{\rho \oplus 0}^{tr}(X)$, we have $T_{12} \in \fraklC_{\rho \oplus 0}^{tr}(X)$ and $T_{21} \in \frakCr_{\rho \oplus 0}^{tr}(X)$ (Definition \ref{defn:Dtr_rho0}), i.e., $T_{12} T_{21} \in \frakC_{\rho \oplus 0}^{tr}(X)$. From this the claim follows that $\rho(\chi_i) (1-T_{11}^2) \rho(\chi_i)$ is of trace class and that the sequence $\tfrac{1}{\card U_i} \trace_\epsilon \rho(\chi_i) (1_H - T_{11}^2) \rho(\chi_i)$ is bounded. So Formula \eqref{eq:formula_ind_T_supertrace} is well-defined.

We will now show that the index does not depend on the chosen quasi-latticing partition. So let $(V_y^\prime)_{y \in Y}$ be another one with diameters $\le d$ and $y \in V_y^\prime$ for all $y \in Y$. Then the difference between $\chi_i$ (the characteristic function of $\bigcup_{y \in U_i} \overline{V_y}$) and $\chi_i^\prime$ (the one of $\bigcup_{y \in U_i} \overline{V^\prime_y}$) is supported on the boundary of $U_i$ since in the interior both sum up to $1$. Concretely, we have $\supp(\chi_i - \chi^\prime_i) \subset B_d(U_i - \partial_{2d} U_i)$, where $\partial_r U_i$ denotes the $r$-boundary of $U_i$ in $Y$ (Definition \ref{defn:amenability_metric_spaces}). Since $(U_i)_i$ is a \Folner sequence, the resulting difference in the indices vanishes in the limit under $\tau$. Analogously we can show that the index does not depend on the value of $d$ (the upper bound on the diameters of the quasi-latticing partitions).

In order to show that $\ind_\tau$ (which is defined up to now only on modules) descends to $K_0^u(X)$, we have to show that it is invariant under operator homotopies (that it respects direct sums of modules and unitary equivalence classes, is clear). Suppose that we have an operator homotopy $T_t$. Since all the normalizations that we did prior to defining the index respect operator homotopies, we have to show that $\ind_\tau(T_{t, 11})$ is constant in $t$. Due to Paschke duality, the operator homotopy $T_t$ translates to a homotopy in the $C^\ast$-algebra $\frakD^{tr}_{\rho \oplus 0}(X)$. But here we may now replace it by a \emph{smooth} homotopy, i.e., we may assume without loss of generality that $T_t$ is differentiable. Then we get
\[\tfrac{d}{dt} \ind_\tau (T_{t, 11}) = \tau\big(\tfrac{1}{\card U_i} \trace_\epsilon \rho(\chi_i) [\tfrac{d}{dt} T_{t, 11},T_{t, 11}]_\epsilon \rho(\chi_i)\big),\]
where we already used that the operators $T_t$, and therefore also $\tfrac{d}{dt} T_t$, are odd to derive this formula.

Now if the projections $\rho(\chi_i)$ weren't there, we would say that the graded trace vanishes on the graded commutator and conclude that the time derivative of the index is $0$, i.e., the index is constant along our homotopy. But the projections are there and so we argue instead in the following way: due to the finite propagation $R$ of the operators $T_t$\footnote{Here we need the crucial fact from Proposition \ref{prop:normalization_finite_prop_speed} that we may have a common bound for the propagations of the operators $T_t$.} the quantity $\trace_\epsilon \rho(\chi_i) [\tfrac{d}{dt} T_{t, 11},T_{t, 11}]_\epsilon \rho(\chi_i)$ vanishes on $\bigcup_{y \in U_i - \partial_{2d+R}U_i} \overline{V_y}$, i.e., the non-zero part is concentrated on the boundary $\partial_{2d+R}U_i \cap U_i$. Due to the amenability of $(U_i)_i$ this becomes $0$ in the limit under $\tau$.

To complete the proof, it remains to show that the diagram
\[\xymatrix{K_0^u(X) \ar[rr]^{\mu_u} \ar[dr]_{\ind_\tau} & & K_0(C_u^\ast(Y)) \ar[dl]^{\ind_\tau} \\ & \IR}\]
commutes. Recalling the definition of the uniform assembly map $\mu_u$ (\cite[Proposition 9.1]{spakula_uniform_k_homology}) we see that we first have to normalize a given Fredholm module to being non-degenerate, i.e., $\mu_u([T])$ is defined using the operator $T_{11}$. So showing the commutativity of the above diagram boils down to showing $\ind_\tau (T_{11}) = \ind_\tau (\mu_u(T_{11}))$. Writing $T_{11} := \begin{pmatrix}0&U^\ast\\U&0\end{pmatrix}$ with respect to the grading decomposition $H = H^+ \oplus H^-$, we get
\[\ind_\tau (T_{11}) = \tau\big(\tfrac{1}{\card U_i} \trace_\epsilon \rho(\chi_i) \begin{pmatrix}1^+ - U^\ast U & 0 \\ 0 & 1^- - UU^\ast \end{pmatrix} \rho(\chi_i)\big)\]
and for the index of $\mu_u(T_{11})$ we have\footnote{Here we need again amenability of $(U_i)_i$ since we get an error on the boundary $\partial_{2d}U_i$ of $U_i$ due to the discretization which $\mu_u$ does (it passes to $Y \subset X$)}
\[\ind_\tau (\mu_u(T_{11})) = \tau\big(\tfrac{1}{\card U_i}\trace_\epsilon \rho(\chi_i) \begin{pmatrix}(1^+ - U^\ast U)^2 & 0 \\ 0 & (1^- - UU^\ast)^2 \end{pmatrix} \rho(\chi_i)\big).\]
Now we can compute their difference:
\begin{align*}
\ind_\tau (\mu_u(T_{11} & )) - \ind_\tau (T_{11}) = \\
& = \tau\left(\tfrac{1}{\card U_i} \trace_\epsilon \rho(\chi_i) \begin{pmatrix}U^\ast U (U^\ast U - 1^+) & 0 \\ 0 & U U^\ast (U U^\ast - 1^-)\end{pmatrix} \rho(\chi_i) \right)\\
& = \tau\big(\tfrac{1}{\card U_i} \trace_\epsilon \rho(\chi_i) T_{11}^2 (T_{11}^2 - 1) \rho(\chi_i) \big)\\
& = \tau\big(\tfrac{1}{\card U_i} \trace_\epsilon \rho(\chi_i) \tfrac{1}{2}[T_{11},T_{11}(T_{11}^2 - 1)]_\epsilon \rho(\chi_i) \big)
\end{align*}
which equals $0$ due to the same argument that we used to show the homotopy invariance (the fact that the graded trace vanishes on graded commutators combined with the amenability of $(U_i)_i$).
\end{proof}

Note that if the space $X$ is not amenable, then we do not have such index maps $\ind_\tau \colon K_0^u(X) \to \IR$. This follows from \cite[Theorem 11.2]{spakula_uniform_k_homology}: if $X$ is a connected graph with vertex set $Y$, then $X$ is amenable if and only if $[Y] \not= [0] \in K_0^u(X)$. From the definition of the index map $\ind_\tau$ on the level of Fredholm modules we get $\ind_\tau(Y) = 1$ (independent of the choice of functional $\tau$ or the exhaustion of $Y$ via finite subsets $U_i$). So if $X$ is not amenable, we conclude that $\ind_\tau$ can not descend to the level of uniform $K$-homology.

\section{External product}

Now we get to probably the most important technical part in this thesis: the construction of the external product for uniform $K$-homology. Its main application will be to deduce homotopy invariance of uniform $K$-homology in the next section.

Note that we can construct the product only if the involved metric spaces have jointly bounded geometry\footnote{see Definition \ref{defn:jointly_bounded_geometry}}, and since this property is crucially used, the author does not see any way to overcome this requirement. But fortunately, both major classes of spaces on which we want to apply our theory, namely manifolds and simplicial complexes of bounded geometry, do have jointly bounded geometry.

In our construction of the external product for uniform $K$-homology we follow the presentation of \cite[Section 9.2]{higson_roe}, where the product is constructed for the usual $K$-homology.

Let $X_1$ and $X_2$ be locally compact and separable metric spaces and both having jointly bounded geometry, $(H_1, \rho_1, T_1)$ a $p_1$-multigraded uniform Fredholm module over the space $X_1$ and $(H_2, \rho_2, T_2)$ a $p_2$-multigraded module over $X_2$, and both modules will be assumed to have finite propagation (see Proposition \ref{prop:normalization_finite_prop_speed}).

\begin{defn}[cf. {\cite[Definition 9.2.2]{higson_roe}}]
Define $\rho$ to be the tensor product representation of $C_0(X_1 \times X_2) \cong C_0(X_1) \otimes C_0(X_2)$ on $H := H_1 \hatotimes H_2$, i.e.,
\[\rho(f_1 \otimes f_2) = \rho_1(f_1) \hatotimes \rho_2(f_2) \in \IB(H_1) \hatotimes \IB(H_2)\]
and equip $H_1 \hatotimes H_2$ with the induced $(p_1 + p_2)$-multigrading (see Lemma \ref{lem:tensor_prod_multigraded}).

We say that a $(p_1 + p_2)$-multigraded uniform Fredholm module $(H, \rho, T)$ is \emph{aligned} with the modules $(H_1, \rho_1, T_1)$ and $(H_2, \rho_2, T_2)$, if
\begin{itemize}
\item $T$ has finite propagation,
\item for all $f \in C_0(X_1 \times X_2)$ the operators
\[\rho(f) \big( T (T_1 \hatotimes 1) + (T_1 \hatotimes 1) T \big) \rho(\bar f) \text{ and } \rho(f) \big( T (1 \hatotimes T_2) + (1 \hatotimes T_2) T \big) \rho(\bar f)\]
are positive modulo compact operators,\footnote{That is to say, they are positive in the Calkin algebra $\IB(H) / \IK(H)$.} and
\item for all $f\in C_0(X_1 \times X_2)$ the operator $\rho(f) T$ derives $\IK(H_1) \hatotimes \IB(H_2)$, i.e.,
\[[\rho(f) T, \IK(H_1) \hatotimes \IB(H_2)] \subset \IK(H_1) \hatotimes \IB(H_2).\]
\end{itemize}
Since both $H$ and $\rho$ are uniquely determined from $H_1$, $\rho_1$, $H_2$ and $\rho_2$, we will often just say that \emph{$T$ is aligned with $T_1$ and $T_2$}.
\end{defn}

Our major technical lemma is the following one. It is a uniform version of Kasparov's Technical Lemma, which is suitable for our needs.

\begin{lem}\label{lem:construction_partition_unity}
Let $X_1$ and $X_2$ be locally compact and separable metric spaces that have jointly coarsely and locally bounded geometry.

Then there exist commuting, even, multigraded, positive operators $N_1$, $N_2$ of finite propagation on $H := H_1 \hatotimes H_2$ with $N_1^2 + N_2^2 = 1$ and the following properties:

\begin{enumerate}
\item $N_1 \cdot \big\{ (T_1^2 - 1) \rho_1(f) \hatotimes 1 \ | \ f \in \LLip_{R^\prime}(X_1) \big\} \subset \IK(H_1 \hatotimes H_2)$ is uniformly approximable for all $R^\prime, L > 0$ and analogously for $(T^\ast_1 - T_1)\rho_1(f)$ and for $[T_1, \rho_1(f)]$ instead of $(T_1^2 - 1) \rho_1(f)$,
\item $N_2 \cdot \big\{ 1 \hatotimes (T_2^2 - 1) \rho_2(f) \ | \ f \in \LLip_{R^\prime}(X_2) \big\} \subset \IK(H_1 \hatotimes H_2)$ is uniformly approximable for all $R^\prime, L > 0$ and analogously for $(T_2^\ast - T_2) \rho_2(f)$ and for $[T_2, \rho_2(f)]$ instead of $(T_2^2 - 1) \rho_2(f)$,
\item $\{[N_i, T_1 \hatotimes 1]\rho(f), [N_i, 1 \hatotimes T_2]\rho(f) \ | \ f \in \LLip_{R^\prime}(X_1 \times X_2)\}$ is uniformly approximable for all $R^\prime, L > 0$ and both $i=1,2$,
\item $\big\{ [N_i, \rho(f \otimes 1)], [N_i, \rho(1 \otimes g)] \ | \ f \in \LLip_{R^\prime}(X_1), g \in \LLip_{R^\prime}(X_2) \big\}$ is uniformly approximable for all $R^\prime, L > 0$ and both $i = 1,2$, and
\item both $N_1$ and $N_2$ derive $\IK(H_1) \hatotimes \IB(H_2)$.
\end{enumerate}
\end{lem}

\begin{proof}
Due to the jointly bounded geometry there is a countable Borel decomposition $\{X_{1,i}\}$ of $X_1$ such that each $X_{1,i}$ has non-empty interior, the completions $\{\overline{X_{1,i}}\}$ form an admissible class\footnote{This means that for every $\varepsilon > 0$ there is an $N > 0$ such that in every $\overline{X_{1,i}}$ exists an $\varepsilon$-net of cardinality at most $N$.} of compact metric spaces and for each $R > 0$ we have
\begin{equation}
\label{eq:bound_jointly_bounded_geom}
\sup_i \card \{j \ | \ B_R(X_{1,i}) \cap X_{1,j} \not= \emptyset\} < \infty.
\end{equation}

The completions of the $1$-balls $B_1(X_{1,i})$ are also an admissible class of compact metric spaces and the collection of these open balls forms a uniformly locally finite open cover of $X_1$. We may find a partition of unity $\varphi_{1,i}$ subordinate to the cover $\{B_1(X_{1,i})\}$ such that every function $\varphi_{1,i}$ is $L_0$-Lipschitz for a fixed $L_0 > 0$ (but we will probably have to enlarge the value of $L_0$ a bit in a moment). The same holds also for a countable Borel decomposition $\{X_{2,i}\}$ of $X_2$ and we choose a partition of unity $\varphi_{2,i}$ subordinate to the cover $\{B_1(X_{2,i})\}$ such that every function $\varphi_{2,i}$ is also $L_0$-Lipschitz (by possibly enlargening $L_0$ so that we have the same Lipschitz constant for both partitions of unity).

Since $\{\overline{B_1(X_{1,i})}\}$ is an admissible class of compact metric spaces, we have for each $\varepsilon > 0$ and $L > 0$ a bound independent of $i$ on the number of functions from
\[\varphi_{1,i} \cdot \LLip_c(X_1) := \{ \varphi_{1,i} \cdot f \ | \ f\text{ is }L\text{-Lipschitz, compactly supported and }\|f\|_\infty \le 1\}\]
to form an $\varepsilon$-net in $\varphi_{1,i} \cdot \LLip_c(X_1)$, and analogously for $X_2$ (this can be proved by a similar construction as the one from \cite[Lemma 2.4]{spakula_universal_rep}). We denote this upper bound by $C_{\varepsilon, L}$.

Now for each $N \in \IN$ and $i \in \IN$ we choose $C_{1/N,N}$ functions $\{f_k^{i,N}\}_{k=1, \ldots, C_{1/N, N}}$ from $\varphi_{1,i} \cdot N\text{-}\operatorname{Lip}_c(X_{1,i})$ constituting an $1/N$-net.\footnote{If we need less functions to get an $1/N$-net, we still choose $C_{1/N,N}$ of them. This makes things easier for us to write down.} Analogously we choose $C_{1/N, N}$ functions $\{g_k^{i,N}\}_{k=1, \ldots, C_{1/N, N}}$ from $\varphi_{2,i} \cdot N\text{-}\operatorname{Lip}_c(X_{2,i})$ that are $1/N$-nets.

We choose a sequence $\{u_n \hatotimes 1\} \subset \IB(H_1) \hatotimes \IB(H_2)$ of operators in the following way: $u_n$ will be a projection operator onto a subspace $U_n$ of $H_1$. To define this subspace, we first consider the operators
\begin{equation}\label{eq:operators_X_1_to_approximate}
(T_1^2 - 1)\rho_1(f), \ (T_1 - T_1^\ast)\rho_1(f), \text{ and } [T_1, \rho_1(f)]
\end{equation}
for suitable functions $f \in C_0(X_1)$ that we will choose in a moment. These operators are elements of $\IK(H_1)$ since $(H_1, \rho_1, T_1)$ is a Fredholm module. So up to an error of $2^{-n}$ they are of finite rank and the span $V_n$ of the images of these finite rank operators will be the building block for the subspace $U_n$ on which the operator $u_n$ projects\footnote{This finite rank operators are of course not unique. Recall that every compact operator on a Hilbert space $H$ may be represented in the form $\sum_{n \ge 1} \lambda_n \langle f_n, \largecdot \rangle g_n$, where the values $\lambda_n$ are the singular values of the operator and $\{f_n\}$, $\{g_n\}$ are orthonormal (though not necessarily complete) families in $H$ (but contrary to the $\lambda_n$ they are not unique). Now we choose our finite rank operator to be the operator given by the same sum, but only with the $\lambda_n$ satisfying $\lambda_n \ge 2^{-n}$.} (i.e., we will say in a moment how to enlarge $V_n$ in order to get $U_n$). We choose the functions $f \in C_0(X_1)$ as all the functions from the set $\bigcup \{f_k^{i,N}\}_{k=1, \ldots, C_{1/N,N}}$, where the union ranges over all $i \in \IN$ and $1 \le N \le n$. Note that since the Fredholm module $(H_1, \rho_1, T_1)$ is uniform, the rank of the finite rank operators approximating \eqref{eq:operators_X_1_to_approximate} up to an error of $2^{-n}$ is bounded from above with a bound that depends only on $N$ and $n$, but not on $i$ nor $k$. Since we will have $V_n \subset U_n$, we can already give the first estimate that we will need later:
\begin{equation}\label{eq:Kasparov_estimate_383}
\|(u_n \hatotimes 1)(x \hatotimes 1) - (x \hatotimes 1)\| < 2^{-n},
\end{equation}
where $x$ is one of the operators from \eqref{eq:operators_X_1_to_approximate} for all $f_k^{i,N}$ with $1 \le N \le n$.\footnote{Actually, to have this estimate we would need that $x$ is self-adjoint. We can pass from $x$ to $\tfrac{1}{2}(x + x^\ast)$ and $\tfrac{1}{2i}(x - x^\ast)$, do all the constructions with these self-adjoint operators and get the needed estimates for them, and then we get the same estimates for $x$ but with an additional factor of $2$.} Moreover, denoting by $\chi_{1,i}$ the characteristic function of $B_1(X_{1,i})$, then $\rho_1(\chi_{1,i}) \cdot V_n$ is a subspace of $H_1$ of finite dimension that is bounded independently of $i$.\footnote{We have used here the fact that we may uniquely extend any representation of $C_0(Z)$ to one of the bounded Borel functions $B_b(Z)$ on a space $Z$.} The reason for this is because $T_1$ has finite propagation and the number of functions $f_k^{i,N}$ for fixed $N$ is bounded independently of $i$. For all $n$ we also have $V_n \subset V_{n+1}$ and that the projection operator onto $V_n$ has finite propagation which is bounded independently of $n$.

For each $n \in \IN$ we partition $\chi_{1,i}$ for all $i \in \IN$ into disjoint characteristic functions $\chi_{1,i} = \sum_{j=1}^{J_n} \chi_{1,i}^{j,n}$ such that we may write each function $f_k^{i,N}$ for all $i \in \IN$, $1 \le N \le n$ and $k = 1, \ldots, C_{1/N, N}$ up to an error of $2^{-n-1}$ as a sum $f_k^{i,N} = \sum_{j=1}^{J_n} \alpha_k^{i,N}(j,n) \cdot \chi_{1,i}^{j,n}$ for suitable constants $\alpha_k^{i,N}(j,n)$. Note that since $X_1$ has jointly coarsely and locally bounded geometry, we can choose the upper bounds $J_n$ such that they do not depend on $i$. Now we can finally set $U_n$ as the linear span of $V_n$ and $\rho_1(\chi_{1,i}^{j,n}) \cdot V_n$ for all $i \in \IN$ and $1 \le j \le J_n$. Note that $\rho_1(\chi_{1,i}) \cdot U_n$ is a subspace of $H_1$ of finite dimension that is bounded independently of $i$, that we may choose the characteristic functions $\chi_{1,i}^{j,n}$ such that we have $U_n \subset U_{n+1}$ (by possibly enlargening each $J_n$), and that the projection operator $u_n$ onto $U_n$ has finite propagation which is bounded independently of $n$. Since we have $[u_n, \rho_1(\chi_{1,i}^{j,n})] = 0$ for all $i \in \IN$, $1 \le j \le J_n$ and all $n \in \IN$, we get our second crucial estimate:
\begin{equation}\label{eq:Kasparov_estimate_384}
\|[u_n \hatotimes 1, \rho_1(f_k^{i,N}) \hatotimes 1]\| < 2^{-n}
\end{equation}
for all $i \in \IN$, $k = 1, \ldots, C_{1/N, N}$, $1 \le N \le n$ and all $n \in \IN$.

By an argument similar to the proof of the existence of quasicentral approximate units, we may conclude that for each $n \in \IN$ there exists a finite convex combination $\nu_n$ of the elements $\{u_n, u_{n+1}, \ldots\}$ such that
\begin{equation}\label{eq:Kasparov_estimate_384_2}
\|[\nu_n \hatotimes 1, T_1 \hatotimes 1]\| < 2^{-n} \text{, } \|[\nu_n \hatotimes 1, \epsilon_1 \hatotimes \epsilon_2]\| < 2^{-n} \text{ and } \|[\nu_n \hatotimes 1, \epsilon^j]\| < 2^{-n}
\end{equation}
for all $n \in \IN$, where $\epsilon_1 \hatotimes \epsilon_2$ is the grading operator of $H_1 \hatotimes H_2$ and $\epsilon^j$, $1 \le j \le p_1 + p_2$, are the multigrading operators of $H_1 \hatotimes H_2$. Note that the Estimates \eqref{eq:Kasparov_estimate_383} and \eqref{eq:Kasparov_estimate_384} also hold for $\nu_n$. Note furthermore that we can arrange that the maximal index occuring in the finite convex combination for $\nu_n$ is increasing in $n$.

Now we will construct a sequence $w_n \in \IB(H_1) \hatotimes \IB(H_2)$ with suitable properties. We have that $\nu_n$ is a finite convex combination of the elements $\{u_n, u_{n+1}, \ldots\}$. So for $n \in \IN$ we let $m_n$ denote the maximal occuring index in that combination. Furthermore, we let the projections $p_n \in \IB(H_2)$ be analogously defined as $u_n$, where we consider now the operators
\begin{equation}\label{eq:operators_X_2_to_approximate}
(T_2^2 - 1)\rho_2(g), \ (T_2 - T_2^\ast)\rho_2(g), \text{ and } [T_2, \rho_2(g)]
\end{equation}
for the analogous sets of functions $\bigcup \{g_k^{i,N}\}_{k=1, \ldots, C_{1/N,N}}$ depending on $n \in \IN$. Then we define $w_{n-1} := u_{m_n} \hatotimes  p_n$\footnote{The index is shifted by one so that we get the Estimates \eqref{eq:Kasparov_estimate_386}--\eqref{eq:Kasparov_estimate_388} with $2^{-n}$ and not with $2^{-n+1}$; though this is not necessary for the argument.} and get for all $n \in \IN$ the following:
\begin{equation}
\label{eq:Kasparov_estimate_385_-1}
w_n (\nu_n \hatotimes 1) (1 \hatotimes p_n) = (\nu_n \hatotimes 1) (1 \hatotimes p_n)
\end{equation}
and
\begin{align}
\label{eq:Kasparov_estimate_386}
\| [ w_n, x \hatotimes 1 ] \| & < 2^{-n}\\
\label{eq:Kasparov_estimate_387}
\| [ w_n, 1 \hatotimes y ] \| & < 2^{-n}\\
\label{eq:Kasparov_estimate_388}
\| [ w_n, \rho(f_k^{i,N} \otimes g_k^{i,N}) ] \| & < 2^{-n}
\end{align}
for all $i \in \IN$, $1 \le N \le n$ and $k = 1, \ldots, C_{1/N,N}$, where $x$ is one of the operators from \eqref{eq:operators_X_1_to_approximate} for all $f_k^{i,N}$ and $y$ is one of the operators from \eqref{eq:operators_X_2_to_approximate} for all $g_k^{i,N}$.

Let now $d_n := (w_n - w_{n-1})^{1/2}$. With a suitable index shift we can arrange that firstly, the Estimates \eqref{eq:Kasparov_estimate_386}--\eqref{eq:Kasparov_estimate_388} also hold for $d_n$ instead of $w_n$,\footnote{see \cite[Exercise 3.9.6]{higson_roe}} and that secondly, using Equation \eqref{eq:Kasparov_estimate_385_-1},
\begin{equation}
\label{eq:Kasparov_estimate_385}
\| d_n (\nu_n \hatotimes 1) y \| < 2^{-n},
\end{equation}
where $y$ is again one of the operators from \eqref{eq:operators_X_2_to_approximate} for all $g_k^{i,N}$ and $1 \le N \le n$.

Now as in the same way as we constructed $\nu_n$ out of the $u_n$s, we construct $\delta_n$ as a finite convex combination of the elements $\{d_n, d_{n+1}, \ldots\}$ such that
\begin{equation*}
\|[\delta_n, T_1 \hatotimes 1]\| < 2^{-n} \text{, } \|[\delta_n, 1 \hatotimes T_2]\| < 2^{-n} \text{, } \|[\delta_n, \epsilon_1 \hatotimes \epsilon_2]\| < 2^{-n} \text{ and } \|[\delta_n, \epsilon^j]\| < 2^{-n},\notag
\end{equation*}
where $\epsilon_1 \hatotimes \epsilon_2$ is the grading operator of $H_1 \hatotimes H_2$ and $\epsilon^j$ for $1 \le j \le p_1 + p_2$ are the multigrading operators of $H_1 \hatotimes H_2$. Clearly, all the Estimates \eqref{eq:Kasparov_estimate_386}--\eqref{eq:Kasparov_estimate_385} also hold 
for $\delta_n$.

Define $X := \sum \delta_n \nu_n \delta_n$. It is a positive operator of finite propagation and fulfills the Points 2--4 that $N_2$ should have. The arguments for this are analogous to the ones given at the end of the proof of \cite[Kasparov's Technical Theorem 3.8.1]{higson_roe}, but we have to use all the uniform approximations that we additionally have (to use them, we have to cut functions $f \in \LLip_{R^\prime}(X_1)$ down to the single ``parts'' $X_{1,i}$ of $X_1$ by using the partition of unity $\{\varphi_{1,i}\}$ that we have chosen at the beginning of this proof, and analogously for $X_2$). Furthermore, the operator $1-X$ fulfills the desired Points 1, 3 and 4 that $N_1$ should fulfill. That both $X$ and $1-X$ derive $\IK(H_1) \hatotimes \IB(H_2)$ is clear via construction. Since $X$ commutes modulo compact operators with the grading and multigrading operators, we can average it over them so that it becomes an even and multigraded operator and $X$ and $1-X$ still have all the above mentioned properties.

Finally, we set $N_1 := (1-X)^{1/2}$ and $N_2 := X^{1/2}$.
\end{proof}

Now we will use this technical lemma to construct the external product and to show that it is well-defined on the level of uniform $K$-homology.

\begin{prop}\label{prop:external_prod_exists}
Let $X_1$ and $X_2$ be locally compact and separable metric spaces that have jointly coarsely and locally bounded geometry.

Then there exists a $(p_1 + p_2)$-multigraded uniform Fredholm module $(H, \rho, T)$ which is aligned with the modules $(H_1, \rho_1, T_1)$ and $(H_2, \rho_2, T_2)$.

Furthermore, any two such aligned Fredholm modules are operator homotopic and this operator homotopy class is uniquely determined by the operator homotopy classes of $(H_1, \rho_1, T_1)$ and $(H_2, \rho_2, T_2)$.
\end{prop}

\begin{proof}
We invoke the above Lemma \ref{lem:construction_partition_unity} to get operators $N_1$ and $N_2$ and then set
\[T := N_1(T_1 \hatotimes 1) + N_2(1 \hatotimes T_2).\]

To deduce that $(H, \rho, T)$ is a uniform Fredholm module, we have to use the following facts (additionally to the ones that $N_1$ and $N_2$ have): that $T_1$ and $T_2$ have finite propagation and are odd (we need that $(T_1 \hatotimes 1)(1 \hatotimes T_2) + (1 \hatotimes T_2)(T_1 \hatotimes 1) = 0$). To deduce that it is a multigraded module, we need that we constructed $N_1$ and $N_2$ as even and multigraded operators on $H$.

It is easily seen that for all $f \in C_0(X_1 \times X_2)$
\[\rho(f) \big( T (T_1 \hatotimes 1) + (T_1 \hatotimes 1) T \big) \rho(\bar f) \text{ and } \rho(f) \big( T (1 \hatotimes T_2) + (1 \hatotimes T_2) T \big) \rho(\bar f)\]
are positive modulo compact operators and that $\rho(f)T$ derives $\IK(H_1) \hatotimes \IB(H_2)$, i.e., we conclude that $T$ is aligned with $T_1$ and $T_2$.

Since all four operators $T_1$, $T_2$, $N_1$ and $N_2$ have finite propagation, $T$ has also finite propagation.

Suppose that $T^\prime$ is another operator aligned with $T_1$ and $T_2$. We construct again operators $N_1$ and $N_2$ using the above Lemma \ref{lem:construction_partition_unity}, but we additionally enforce
\[\|[w_n, \rho(f^{i,N}_k \otimes g_k^{i,N}) T^\prime]\| < 2^{-n}\]
analogously as we did it there to get Equation \eqref{eq:Kasparov_estimate_388}. So $N_1$ and $N_2$ will commute modulo compacts with $\rho(f)T^\prime$ for all functions $f \in C_0(X_1 \times X_2)$. Again, we set $T := N_1(T_1 \hatotimes 1) + N_2(1 \hatotimes T_2)$. Since $N_1$ and $N_2$ commute modulo compacts with $\rho(f)T^\prime$ for all $f \in C_0(X_1 \times X_2)$ and since $T^\prime$ is aligned with $T_1$ and $T_2$, we conclude
\[\rho(f)(T T^\prime + T^\prime T) \rho(\bar f) \ge 0\]
modulo compact operators for all $f \in C_0(X_1 \times X_2)$. Using a uniform version of \cite[Proposition 8.3.16]{higson_roe} we conclude that $T$ and $T^\prime$ are operator homotopic via multigraded, uniform Fredholm modules. We conclude that every aligned module is operator homotopic to one of the form that we constructed above, i.e., to one of the form $N_1(T_1 \hatotimes 1) + N_2(1 \hatotimes T_2)$. But all such operators are homotopic to one another: they are determined by the operator $Y = N_2^2$ used in the proof of the above lemma and the set of all operators with the same properties as $Y$ is convex.

At last, suppose that one of the operators is varied by an operator homotopy, e.g., $T_1$ by $T_1(t)$. Then, in order to construct $N_1$ and $N_2$, we enforce in Equation \eqref{eq:Kasparov_estimate_384_2} instead of $\|[\nu_n \hatotimes 1, T_1 \hatotimes 1]\| < 2^{-n}$ the following one:
\[\|[\nu_n \hatotimes 1, T_1(j/n) \hatotimes 1]\| < 2^{-n}\]
for $0 \le j \le n$. Now we may define
\[T(t) := N_1(T_1(t) \hatotimes 1) + N_2(1 \hatotimes T_2),\]
i.e., we got operators $N_1$ and $N_2$ which are independent of $t$ but still have all the needed properties. This gives us the desired operator homotopy.
\end{proof}

\begin{defn}[External product]
The \emph{external product} of the multigraded uniform Fredholm modules $(H_1, \rho_1, T_1)$ and $(H_2, \rho_2, T_2)$ is a multigraded uniform Fredholm module $(H, \rho, T)$ which is aligned with $T_1$ and $T_2$. We will use the notation $T := T_1 \times T_2$.

By the above Proposition \ref{prop:external_prod_exists} we know that if the locally compact and separable metric spaces $X_1$ and $X_2$ both have jointly coarsely and locally bounded geometry, then the external product always exists, that it is well-defined up to operator homotopy and that it descends to a well-defined product on the level of uniform $K$-homology:
\[K_{p_1}^u(X_1) \times K_{p_2}^u(X_2) \to K_{p_1+p_2}^u(X_1 \times X_2)\]
for $p_1, p_2 \ge 0$. Furthermore, this product is bilinear.\footnote{To see this, suppose that, e.g., $T_1 = T_1^\prime \oplus T_1^{\prime \prime}$. Then it suffices to show that $T_1^\prime \times T_2 \oplus T_1^{\prime \prime} \times T_2$ is aligned with $T_1$ and $T_2$, which is not hard to do.}
\end{defn}

For the remaining products (i.e., the product of an ungraded and a multigraded module, resp., the product of two ungraded modules) we can appeal to the formal $2$-periodicity. But for the convenience of the reader, we will now explain how to construct the external products, where one module is ungraded and the other one is ordinary graded (i.e., $0$-multigraded), and where both are ungraded.

Suppose that $(H_1, \rho_1, T_1)$ is ungraded and that $(H_2, \rho_2, T_2)$ is graded. Then we form the usual, ungraded tensor products $H := H_1 \otimes H_2$ and $\rho := \rho_1 \otimes \rho_2$ and set
\[T := N_1(T_1 \otimes \epsilon) + N_2(1 \otimes T_2),\]
where $\epsilon$ is the grading operator of $H_2$. This defines an ungraded uniform Fredholm module\footnote{In the proof of Proposition \ref{prop:external_prod_exists} we needed that $T_1$ and $T_2$ are both odd to deduce $(T_1 \hatotimes 1)(1 \hatotimes T_2)+(1\hatotimes T_2)(T_1 \hatotimes 1) = 0$. We get the analogous conclusion in this case here (where $T_1$ is ungraded) since $T_2$ anti-commutes with $\epsilon$.} and we get the external product $K_{-1}^u(X_1) \times K_0^u(X_2) \to K_{-1}^u(X_1 \times X_2)$. Analogously we get the product $K_0^u(X_1) \times K_{-1}^u(X_2) \to K_{-1}^u(X_1 \times X_2)$.

Now suppose that both modules are ungraded. Again we set $H := H_1 \otimes H_2$ and $\rho := \rho_1 \otimes \rho_2$, but this time we form the graded module $(H \oplus H, \rho \oplus \rho, T)$, where $T$ is defined as
\[T := \begin{pmatrix}0 & N_1(T_1 \otimes 1) - N_2(1 \otimes i T_2) \\ N_1(T_1 \otimes 1) + N_2(1 \otimes iT_2) & 0\end{pmatrix}.\]
This defines a product $K_{-1}^u(X_1) \times K_{-1}^u(X_2) \to K_0^u(X_1 \times X_2)$.

\begin{defn}
The above constructions give us a bilinear\footnote{Again, we have left out the proof of bilinearity since it is easily deduced.} product
\[K_{p_1}^u(X_1) \times K_{p_2}^u(X_2) \to K^u_{p_1+p_2}(X_1 \times X_2)\]
for all $p_1, p_2 \ge -1$.
\end{defn}

Now we will describe basic, but crucial properties of the external product. The first is its \emph{associative} and the proof of this is analogous to the one in the non-uniform case (see, e.g., \cite[Proposition 9.2.12]{higson_roe}).

Let us denote by $\tau\colon X_1 \times X_2 \to X_2 \times X_1$ the flip map and consider the induced maps $\tau_\ast\colon K_\ast^u(X_1 \times X_2) \to K_\ast^u(X_2 \times X_1)$ on uniform $K$-homology. Then, if $[T_1] \in K_{p_1}^u(X_1)$ and $[T_2] \in K_{p_2}^u(X_2)$, we get the relation $\tau_\ast[T_1 \times T_2] = (-1)^{p_1 p_2} [T_2 \times T_1]$. The proof is straightforward and word-for-word the same as in the non-uniform case.

The external product is also natural with respect to induced maps: let $g\colon Y \to Z$ be a uniformly proper\footnote{Recall that this means $\sup_{z \in Z} \diam (g^{-1}(B_r(z))) < \infty \text{ for all }r > 0$.}, proper and Lipschitz map and denote by $g_\ast\colon K_\ast^u(Y) \to K_\ast^u(Z)$ the induced map on uniform $K$-homology, and let $[T] \in K_{p}^u(X)$. Then the following diagram commutes:
\[\xymatrix{K_\ast^u(Y) \ar[rr]^{g_\ast} \ar[d]_{[T] \times \largecdot} & & K_\ast^u(Z) \ar[d]_{[T] \times \largecdot} \\ K_{\ast + p}^u(X \times Y) \ar[rr]^{(\id_X \operatorname{\times} g)_{\ast}} & & K_{\ast + p}^u(X \times Z)}\]

At last, recall that we have $K_0^u(\pt) \cong \IZ$ and it is generated by the class of any graded Fredholm operator\footnote{Recall the fact that since $\pt$ is compact, uniform $K$-homology of $\pt$ is the same as usual $K$-homology of $\pt$ (Proposition \ref{prop:compact_space_every_module_uniform}).} of index one which acts on a graded Hilbert space equipped with the non-degenerate representation of $C_0(\pt) = \IC$. So denoting the generator of $K_0^u(\pt)$ by $[1]$, we get with a proof which is word-for-word the same as the corresponding one in the non-uniform case (see, e.g., \cite[Proposition 9.3.1]{higson_roe}) that $[1]$ is the identity for the external product, i.e.,
\[[T] \times [1] = [T] = [1] \times [T] \in K_\ast^u(X)\]
for all $[T] \in K_\ast^u(X)$ (where we of course identify $X \times \pt = X = X \times \pt$).

Let us summarize all results of this section in the following theorem:

\begin{thm}[External product for uniform $K$-homology]\label{thm:external_prod_homology}
Let $X_1$ and $X_2$ be locally compact and separable metric spaces of jointly coarsely and locally bounded geometry.\footnote{see Definition \ref{defn:jointly_bounded_geometry}} Then there exists an associative product
\[\times \colon K_{p_1}^u(X_1) \otimes K_{p_2}^u(X_2) \to K^u_{p_1 + p_2}(X_1 \times X_2)\]
for $p_1 , p_2 \ge -1$ with the following properties:
\begin{itemize}
\item for the flip map $\tau\colon X_1 \times X_2 \to X_2 \times X_1$ and all elements $[T_1] \in K_{p_1}^u(X_1)$ and $[T_2] \in K_{p_2}^u(X_2)$ we have
\[\tau_{\ast}[T_1 \times T_2] = (-1)^{p_1 p_2} [T_2 \times T_1],\]
\item we have for $g\colon Y \to Z$ a uniformly proper, proper Lipschitz map and elements $[T] \in K_{p_1}^u(X)$ and $[S] \in K_{p_2}^u(Y)$
\[(\id_X \operatorname{\times} g)_\ast [T \times S] = [T] \times g_\ast[S] \in K^u_{p_1 + p_2}(X \times Z),\]
and
\item denoting the generator of $K_0^u(\pt) \cong \IZ$ by $[1]$, we have
\[[T] \times [1] = [T] = [1] \times [T] \in K_\ast^u(X)\]
for all $[T] \in K_\ast^u(X)$.
\end{itemize}
\end{thm}

\section{Homotopy invariance}\label{sec:homotopy_invariance}

In this section we use the external product that we have constructed in the last section to prove homotopy invariance of uniform $K$-homology. Note that homotopoy invariance turns out to be the most important property that we derive from the existence of the external product: using homotopy invariance, we will be able to prove \Poincare duality, which will give us our desired index theorem for pseudodifferential operators. 

Furthermore, we will derive that weakly homotopic uniform Fredholm modules define the same uniform $K$-homology class. A fact that is needed for showing that all the results of \Spakula from \cite{spakula_uniform_k_homology} do also hold with our definition of uniform $K$-homology (recall that we use operator homotopy as a relation, whereas \Spakula uses more general homotopies). And secondly, we have already used the fact that weakly homotopic modules define the same class in uniform $K$-homology to prove that pseudodifferential operators with the same symbol define the same class in uniform $K$-homology.

And last, we will derive from homotopy invariance some neat results regarding the uniform coarse Baum--Connes conjecture in Chapter \ref{chap:uniform_coarse_indices}.

Let $X$ and $Y$ be locally compact and separable metric spaces having jointly bounded geometry and let $g_0, g_1 \colon X \to Y$ be uniformly proper, proper and Lipschitz maps which are homotopic in the following sense: there exists a uniformly proper, proper and Lipschitz map $G\colon X \times [0,1] \to Y$ with $G(x,0) = g_0(x)$ and $G(x,1) = g_1(x)$ for all $x \in X$.

\begin{thm}\label{thm:homotopy_equivalence_k_hom}
If $g_0, g_1\colon X \to Y$ are homotopic in the above sense, then they induce the same maps $(g_0)_\ast = (g_1)_\ast \colon K_\ast^u(X) \to K_\ast^u(Y)$ on uniform $K$-homology.
\end{thm}

\begin{proof}
Let $\varepsilon_0, \varepsilon_1\colon \pt \to [0,1]$ be the inclusions defined by $\varepsilon_0(\pt) := 0$ and $\varepsilon_1(\pt) := 1$. The homotopy $G\colon X \times [0,1] \to Y$ between $g_0$ and $g_1$ induces maps on uniform $K$-homology $G_\ast\colon K_\ast^u(X \times [0,1]) \to K_\ast^u(Y)$ and since
\[G \circ (\id_X \operatorname{\times} \varepsilon_i) = g_i \circ \pi_X \colon X \times \pt \to Y,\]
where $\pi_X\colon X \times \pt \to X$ is the projection onto $X$, we have
\[G_\ast \circ (\id_X \operatorname{\times} \varepsilon_i)_\ast = \big( G \circ (\id_X \operatorname{\times} \varepsilon_i) \big)_\ast = (g_i \circ \pi_X)_\ast \colon K_\ast^u(X \times \pt) \to K_\ast^u(Y).\]
Since of course $(\pi_X)_\ast\colon K_\ast^u(X \times \pt) \to K_\ast^u(X)$ is an isomorphism, we see that it suffices to show
\[(\id_X \operatorname{\times} \varepsilon_0)_\ast = (\id_X \operatorname{\times} \varepsilon_1)_\ast \colon K_\ast^u(X \times \pt) \to K_\ast^u(X \times [0,1])\]
to conclude $(g_0)_\ast = (g_1)_\ast$.

We know from the third bullet point of Theorem \ref{thm:external_prod_homology} that taking the product with $[1] \in K_0^u(\pt)$ gives an isomorphism from $K_\ast^u(X)$ to $K_\ast^u(X \times \pt)$, so we have to check that for all $[T] \in K_\ast^u(X)$ we have
\[(\id_X \operatorname{\times} \varepsilon_0)_\ast([T] \times [1]) = (\id_X \operatorname{\times} \varepsilon_1)_\ast([T] \times [1]).\]
But this is equivalent to $[T] \times (\varepsilon_0)_\ast [1] = [T] \times (\varepsilon_1)_\ast [1]$ by the second bullet point of Theorem \ref{thm:external_prod_homology}, i.e., it suffices to show $(\varepsilon_0)_\ast [1] = (\varepsilon_1)_\ast [1]$. But this is known, since we are now in the case of ordinary $K$-homology, because $\varepsilon_0, \varepsilon_1 \colon \pt \to [0,1]$ are maps between compact spaces. So now we can use to the homotopy invariance of ordinary $K$-homology.
\end{proof}

We come to the definition of weak homotopies of uniform Fredholm modules then to the proof that such modules define the same class in uniform $K$-homology. Though our definition of weak homotopies is probably not the most general one, it suffices for all our applications. Furthermore, note that the above theorem is a special case of the one that follows. Indeed, given a uniform Fredholm module $(H, \rho, T)$ over $X$, the push-forward of it under $g_i$ is defined as $(H, \rho \circ g_i^\ast, T)$ and it is easily seen that these modules are weakly homotopic via the map $G$.

\begin{defn}[Weak homotopies]\label{defn:weak_homotopy}
Let $(H, \rho_t, T_t)$ for $t \in [0,1]$ be a family of uniform Fredholm modules over $X$ with the following three properties:
\begin{itemize}
\item the family $\rho_t$ is pointwise strong-$^\ast$ operator continuous, i.e., for all $f \in C_0(X)$ we get a path $\rho_t(f)$ in $\IB(H)$ that is continuous in the strong-$^\ast$ operator topology\footnote{Recall that if $H$ is a Hilbert space, then the \emph{strong-$^\ast$ operator topology} on $\IB(H)$ is generated by the family of seminorms $p_v(T) := \|Tv\| + \|T^\ast v\|$ for all $v \in H$, where $T \in \IB(H)$.} on $\IB(H)$,
\item the family $T_t$ is continuous in the strong-$^\ast$ operator topology on $\IB(H)$, i.e., for all $v \in H$ we get norm continuous paths $T_t(v)$ and $T_t^\ast(v)$ in $H$, and
\item for every $\varepsilon > 0$ and $f \in C_0(X)$ we have the following:

For all $t \in [0, 1]$ the operator $[T_t, \rho_t(f)]$ is a compact operator since $(H, \rho_t, T_t)$ is a Fredholm module and can therefore be approximated up to $\varepsilon$ by some finite rank operator $k_t$.\footnote{This finite rank operator $k_t$ is of course not unique. Recall that every compact operator on a Hilbert space $H$ may be represented in the form $\sum_{n \ge 1} \lambda_n \langle f_n, \largecdot \rangle g_n$, where the values $\lambda_n$ are the singular values of the operator and $\{f_n\}$, $\{g_n\}$ are orthonormal (though not necessarily complete) families in $H$ (but contrary to the $\lambda_n$ they are not unique). Now we choose $k_t$ to be the operator given by the same sum, but only with the $\lambda_n$ satisfying $\lambda_n \ge \varepsilon$.} So let $\{v_i\}_{i = 1, \ldots, I}$ be an orthonormal basis of the image of $k_t$ and consider the strong-$^\ast$ operator neighbourhood $U(\varepsilon; v_1, \ldots, v_I)$\footnote{For an operator $A \in \IB(H)$ we define
\[U(\varepsilon; v_1, \ldots, v_I) := \{B \in \IB(H) \ | \ \|(B-A)v_i\| + \|(B-A)^\ast v_i)\| < \varepsilon\text{ for all }i=1, \ldots, I\}.\]
Note that the collection of all such sets $U(\varepsilon; \mathcal{V})$ for all $\varepsilon > 0$ and all finite collections $\mathcal{V} \subset H$ forms a neighbourhood basis of the strong-$^\ast$ operator topology at $A \in \IB(H)$.} of $[T_t, \rho_t(f)]$ in $\IB(H)$. Now for every $[T_s, \rho_s(f)]$ in that neighbourhood we also consider its finite rank approximation $k_s$ up to $\varepsilon$ and an orthonormal basis $\{w_j\}_{j = 1, \ldots, J}$ of its image. Then we require that $[T_t, \rho_t(f)]$ lies in the strong-$^\ast$ operator neighbourhood $U(\varepsilon; w_1, \ldots, w_J)$ of $[T_s, \rho_s(f)]$ in $\IB(H)$.\footnote{It follows that $\|k_t - k_s\|_{op} < 2\varepsilon$, which is the crucial thing that we need.}

Additionally, we require the analogous property for the operators $(T_t^2 - 1)\rho_t(f)$ and $(T_t - T_t^\ast)\rho_t(f)$.
\end{itemize}
Then we call $(H, \rho_t, T_t)$ a \emph{weak homotopy} between $(H, \rho_0, T_0)$ and $(H, \rho_1, T_1)$.
\end{defn}

Note that if $\rho_t$ is pointwise norm continuous and $T_t$ is also norm continuous, then the modules are clearly weakly homotopic. Especially, weak homotopy generalizes operator homotopy.

\begin{rem}
Since the family $T_t$ is continuous in the strong-$^\ast$ operator topology and since it is defined on the compact interval $[0,1]$, we conclude with the uniform boundedness principle $\sup_t \|T_t\|_{op} < \infty$. Furthermore, we have $\|\rho_t(f)\|_{op} \le \|f\|_\infty$ for all $t \in [0,1]$ since $\rho_t$ are representations of $C^\ast$-algebras. Now though multiplication is not continuous as a map $\IB(H) \times \IB(H) \to \IB(H)$, where $\IB(H)$ is equipped with the strong-$^\ast$ operator topology, it is continuous if restricted to norm bounded subsets of $\IB(H)$. So all three families $[T_t, \rho_t(f)]$, $(T_t^2 - 1)\rho_t(f)$ and $(T_t - T_t^\ast)\rho_t(f)$ are also continuous in the strong-$^\ast$ operator topology.
\end{rem}

\begin{thm}\label{thm:weak_homotopy_equivalence_K_hom}
Let $(H, \rho_0, T_0)$ and $(H, \rho_1, T_1)$ be weakly homotopic uniform Fredholm modules over a locally compact and separable metric space $X$ of jointly bounded geometry. Then they define the same uniform $K$-homology class.
\end{thm}

\begin{proof}
Let our weakly homotopic family $(H, \rho_t, T_t)$ be parametrized on the interval $t \in [0, 2\pi]$, so that our notation here will coincide with the one in the proof of \cite[Theorem 1 in §6]{kasparov_KK} that we mimic. Furthermore, we assume that $\rho_t$ and $T_t$ are constant in the intervals $[0, 2\pi/3]$ and $[4\pi / 3, 2\pi]$

We consider the graded Hilbert space $\IH := H \hatotimes (L^2[0,2\pi] \oplus L^2[0, 2\pi])$ (where the space $L^2[0,2\pi] \oplus L^2[0, 2\pi]$ is graded by interchanging the summands).

The family $T_t$ maps continuous paths $v_t$ in $H$ again to continuous paths $T_t(v_t)$: indeed, if $t_n \to t$ is a convergent sequence, we get
\begin{align*}
\|T_{t_n} (v_{t_n}) - T_t (v_t)\| & \le \|T_{t_n} (v_{t_n}) - T_{t_n} (v_t)\| + \|T_{t_n} (v_t) - T_t (v_t)\|\\
& \le \underbrace{\|T_{t_n}\|_{op}}_{< \infty} \cdot \underbrace{\|v_{t_n} - v_t\|}_{\to 0} + \underbrace{\|(T_{t_n} - T_t)(v_t)\|}_{\to 0},
\end{align*}
where the second limit to $0$ holds due to the continuity of $T_t$ in the strong-$^\ast$ operator topology. So the family $T_t$ maps the dense subspace $H \otimes C[0,2\pi]$ of $H \otimes L^2[0,2\pi]$ into itself, and since it is norm bounded from above by $\sup_t \|T_t\|_{op} < \infty$, it defines a bounded operator on $H \otimes L^2[0,2\pi]$. We define an odd operator $\begin{pmatrix} 0 & T_t^\ast \\ T_t & 0\end{pmatrix}$ on $\IH$, which we also denote by $T_t$ (there should arise no confusion by using the same notation here).

Since $\rho_t(f)$ is strong-$^\ast$ continuous in $t$, we can analogously show that it maps continuous paths $v_t$ in $H$ again to continuous paths $\rho_t(f) (v_t)$, and it is norm bounded from above by $\|f\|_\infty$. So $\rho_t(f)$ defines a bounded operator on $H \otimes L^2[0,2\pi]$ and we can get a representation $\rho_t \oplus \rho_t$ of $C_0(X)$ on $\IH$ by even operators, that we denote by the symbol $\rho_t$ (again, no confusion should arise by using the same notation).

We consider now the uniform Fredholm module
\[(\IH, \rho_t, N_1(T_t) + N_2(1 \hatotimes T(f)),\]
where $T(f)$ is defined in the proof of \cite[Theorem 1 in §6]{kasparov_KK} (unfortunately, the overloading of the symbol ``$T$'' is unavoidable here). For the convenience of the reader, we will recall the definition of the operator $T(f)$ in a moment. That we may find a suitable partition of unity $N_1, N_2$ is due to the last bullet point in the definition of weak homotopies, and the construction of $N_1, N_2$ proceeds as in the end of the proof of our Proposition \ref{prop:external_prod_exists}.

To define $T(f)$, we first define an operator $d\colon L^2[0, 2\pi] \to L^2[0, 2\pi]$ using the basis $1, \ldots, \cos nx, \ldots, \sin nx, \ldots$ by the formulas
\[d(1) := 0 \text{, } d(\sin nx) := \cos nx \text{ and } d(\cos nx) := - \sin nx.\]
This operator $d$ is anti-selfadjoint, $d^2 + 1 \in \IK(L^2[0, 2\pi])$, and $d$ commutes modulo compact operators with multiplication by functions from $C[0, 2\pi]$. Let $f \in C[0, 2\pi]$ be a continuous, real-valued function with $|f(x)| \le 1$ for all $x \in [0, 2\pi]$, $f(0) = 1$ and $f(2\pi) = -1$. Then we set $T_1(f) := f - \sqrt{1 - f^2}\cdot d \in \IB(L^2[0, 2\pi])$. This operator $T_1(f)$ is Fredholm and with $1$, both $1 - T_1(f) \cdot T_1(f)^\ast$ and $1 - T_1(f)^\ast \cdot T_1(f)$ are compact, and $T_1(f)$ commutes modulo compacts with multiplication by functions from $C[0, 2\pi]$. Furthermore, any two operators of the form $T_1(f)$ (for different $f$) are connected by a norm continuous homotopy consisting of operators having the same form. Finally, we define $T(f) := \begin{pmatrix}0 & T_1(f)^\ast \\ T_1(f) & 0\end{pmatrix} \in \IB(L^2[0,2\pi] \oplus L^2[0, 2\pi])$.

We assume the our homotopies $\rho_t$ and $T_t$ are constant in the intervals $[0, 2\pi/3]$ and $[4\pi / 3, 2\pi]$. Furthermore, we set
\[f(t) := \begin{cases} \cos 3t, & 0 \le t \le \pi / 3,\\ -1, & \pi/3 \le t \le 2 \pi.\end{cases}\]
Then $T_1(f)$ commutes with the projection $P$ onto $L^2[0, 2\pi / 3]$, $P \cdot T_1(f)$ is an operator of index $1$ on $L^2[0, 2\pi / 3]$, and $(1-P) T_1(f) \equiv -1$ on $L^2[2\pi/3, 2\pi]$. We choose $\alpha(t) \in C[0, 2\pi]$ with $0 \le \alpha(t) \le 1$, $\alpha(t) = 0$ for $t \le \pi / 3$, and $\alpha(t) = 1$ for $t \ge 2\pi / 3$. Using a norm continuous homotopy, we replace $N_1$ and $N_2$ by
\[\widetilde{N_1} := \sqrt{1 \hatotimes (1 - \alpha)} \cdot N_1 \cdot \sqrt{1 \hatotimes (1 - \alpha)}\]
and
\[\widetilde{N_2} := 1 \hatotimes \alpha + \sqrt{1 \hatotimes (1 - \alpha)} \cdot N_2 \cdot \sqrt{1 \hatotimes (1 - \alpha)}.\]
The operator $\widetilde{N_1}(T_t) + \widetilde{N_2}(1 \hatotimes T(f))$ commutes with $1 \hatotimes (P \oplus P)$ and we obtain for the decomposition $L^2[0,2\pi] \oplus L^2[0, 2\pi] = \image(P \oplus P) \oplus \image(1 - P \oplus P)$
\[\big( \IH, \rho_t, \widetilde{N_1}(T_t) + \widetilde{N_2}(1 \hatotimes T(f) \big) = \big( (H, \rho_0, T_0) \times [1] \big) \oplus \big(\text{degenerate}\big),\]
where $[1] \in K_0^u(\pt)$ is the multiplicative identity from the third bullet point of Theorem \ref{thm:external_prod_homology} (recall that we assumed that $\rho_t$ and $T_t$ are constant in the intervals $[0, 2\pi/3]$ and $[4\pi / 3, 2\pi]$).

Setting
\[f(t) := \begin{cases} 1, & 0 \le t \le 5\pi / 3,\\ -\cos 3t, & 5\pi/3 \le t \le 2 \pi,\end{cases}\]
we get analogously
\[\big( \IH, \rho_t, \overline{N_1}(T_t) + \overline{N_2}(1 \hatotimes T(f) \big) = \big(\text{degenerate}\big) \oplus \big( (H, \rho_1, T_1) \times [1] \big),\]
for suitably defined operators $\overline{N_1}$ and $\overline{N_2}$ (their definition is similar to the one of $\widetilde{N_1}$ and $\widetilde{N_2}$). Putting all the homotopies of this proof together, we get that the modules $\big( (H, \rho_0, T_0) \times [1] \big) \oplus \big(\text{degenerate}\big)$ and $\big( (H, \rho_1, T_1) \times [1] \big) \oplus \big(\text{degenerate}\big)$ are operator homotopic, from which the claim follows.
\end{proof}

\chapter{Uniform \texorpdfstring{$K$}{K}-theory}\label{chap:uniform_k_theory}

In this chapter we will identify the dual theory of uniform $K$-homology: the uniform $K$-theory $K^\ast_u(M)$ of a manifold $M$ of bounded geometry. Here we mean ``dual'' in the strongest possible sense: we will show in Section \ref{sec:poincare_duality} that we have \Poincare duality $K^\ast_u(M) \cong K_{m - \ast}^u(M)$ if $M$ is an $m$-dimensional \spinc manifold. Moreover, if $M$ is amenable, we will discuss in Section \ref{sec:props_uniform_k_hom} the possibility of defining index pairings $\langle \largecdot, \largecdot \rangle_\tau \colon K_u^i(M) \otimes K^u_i(M) \to \IR$ which will have all the properties that we have in the usual case of a compact manifold.

The construction of uniform $K$-theory emerges as follows: for a compact manifold we have $K^\ast(M) \cong K_{m-\ast}(M)$ and $K^0(M)$ may be described as consisting of formal differences $[E] - [F]$ of isomorphism classes of vector bundles over $M$. More generally, for a possibly non-compact manifold $M$, we have $K^\ast_{\mathrm{cpt}}(M) \cong K_{m-\ast}^c(M)$, where for $K^0_{\mathrm{cpt}}(M)$ we now use vector bundles that are trivial outside a compact subset of $M$ and $K_r^c(M)$ is defined as the direct limit $\underrightarrow{\lim} \ K_r(L)$ where $L \subset M$ runs over all compact subsets of $M$ (see \cite[Exercise 11.8.11]{higson_roe}).

So if we want to find the dual theory of uniform $K$-theory, it is naturally to search for a version of $K$-theory that firstly, consists of vector bundles that may be non-trivial at infinity (since the Fredholm modules of $K^u_\ast(M)$ are in general not trivial outside compact subsets), secondly, it must somehow depend on the Riemannian metric of $M$ (since uniform $K$-homology does depend on it), and thirdly, it must have some sort of uniformity build into it. The natural conclusion is then the following: to fulfill the first requirement, we consider general vector bundles over $M$, i.e., vector bundles that are not necessarily trivial outside a compact subset. Furthermore, we equip the vector bundles with metrics and compatible connections such that they become vector bundles of bounded geometry. This gets the metric of $M$ into the game and at the same time provides a uniformity condition. We will see that in fact this first idea of a definition of a uniform $K$-theory is fruitful.

Another quite useful property of usual $K$-theory for compact spaces is that the Chern character induces an isomorphism $K^\ast(X) \otimes \IQ \cong H^\ast(X; \IQ)$. So it is natural to ask if there is some cohomology theory for non-compact manifolds with the same property with respect to uniform $K$-theory. Again, we have to search for a theory that consists of cohomology classes that may be non-trivial at infinity, that depends on the Riemannian metric of $M$, and where we have a uniformity condition somehow. But such a cohomology theory is easily found: in Definition \ref{defn:bounded_de_rham} we have defined bounded de Rham cohomology and we see that it fulfills all three requirements. And indeed, in Section \ref{sec:proof_chern_iso} we will a proof of the corresponding Chern character isomorphism in our uniform case.

\section{Definition and basic properties}

Our definition of uniform $K$-theory will be based on the following fact: the topological $K$-theory $K_{\mathrm{cpt}}^\ast(X)$ of a space $X$ equals the operator $K$-theory $K_\ast(C_0(X))$ of the $C^\ast$-algebra $C_0(X)$. Applying the proof of this fact to vector bundles of bounded geometry over a manifold $M$ of bounded geometry, we get that the uniform $K$-theory as we have explained it in the introduction to this chapter should coincide with the operator $K$-theory of the local $C^\ast$-algebra $C_b^\infty(M)$. Since the completion of $C_b^\infty(M)$ is the $C^\ast$-algebra $C_u(M)$ of all bounded, uniformly continuous functions on $M$, we therefore obtain a definition of uniform $K$-theory for all metric spaces. So our presentation of uniform $K$-theory will be reversed: we will define it as the operator $K$-theory of $C_u(X)$ and then, in the next Section \ref{sec:interpretation_uniform_k_theory}, we will prove that if $X$ is a manifold of bounded geometry, this coincides with the description via vector bundles of bounded geometry.

So let us start with the definition:

\begin{defn}[Uniform $K$-theory]
Let $X$ be a metric space. The \emph{uniform $K$-theory groups of $X$} are defined as
\[K^p_u(X) := K_{-p}(C_u(X)),\]
where $C_u(X)$ is the $C^\ast$-algebra of bounded, uniformly continuous functions on $X$.
\end{defn}

The introduction of the minus sign in the index $-p$ in the above definition is just a convention which ensures that the indices in formulas, like the one for the cap product between uniform $K$-theory and uniform $K$-homology, coincide with the indices from the corresponding formulas for (co-)homology. Since complex $K$-theory is $2$-periodic, the minus sign does not change anything in the formulas.

Denoting by $\overline{X}$ the completion of the metric space $X$, we have $K^\ast_u(\overline{X}) = K^\ast_u(X)$ because every uniformly continuous function on $X$ has a unique extension to $\overline{X}$, i.e., $C_u(\overline{X}) = C_u(X)$. This means that, e.g., the uniform $K$-theories of the spaces $[0,1]$, $[0,1)$ and $(0,1)$ are all equal. Furthermore, since on a compact space $X$ we have $C_u(X) = C(X)$, uniform $K$-theory coincides for compact spaces with usual $K$-theory. Let us state this as a small lemma:

\begin{lem}
If $X$ is compact, then $K^\ast_u(X) = K^\ast(X)$.
\end{lem}

\begin{rem}
Note some subtle differences between uniform $K$-theory and uniform $K$-homology. Whereas uniform $K$-theory of $X$ coincides with the uniform $K$-theory of the completion $\overline{X}$, this is in general not true for uniform $K$-homology.

Recall that in Proposition \ref{prop:compact_space_every_module_uniform} we have shown that if $X$ is totally bounded, then the uniform $K$-homology of $X$ coincides with the usual $K$-homology of $X$. So for, e.g., the open unit ball in $\IR^n$ uniform and usual $K$-homology coincide, and also for the closed ball. But of course the usual $K$-homologies of the open and closed balls are not always equal.

Contrary to this the uniform $K$-theory of the open ball equals the uniform $K$-theory of the closed ball, as we have seen in the discussion above. But we generally do not have that, as for uniform $K$-homology, uniform $K$-theory of a totally bounded space which is not compact equal usual $K$-theory.
\end{rem}

The second-to-last thing that we will do in this section is to compute the uniform $K$-theory groups of uniformly discrete spaces. Recall that in Lemma \ref{lem:uniform_k_hom_discrete_space} we have shown that the uniform $K$-homology group $K_0^u(Y)$ of such a space is isomorphic to the group $\ell^\infty_\IZ(Y)$ of all bounded, integer-valued sequences indexed by $Y$, and that $K_1^u(Y) = 0$. Since we want uniform $K$-theory to be dual to uniform $K$-homology, we need the corresponding result for uniform $K$-theory:

\begin{lem}\label{lem:uniform_k_th_discrete_space}
Let $Y$ be a uniformly discrete metric space. Then $K^0_u(Y)$ is isomorphic to $\ell^\infty_\IZ(Y)$ and $K^1_u(Y) = 0$.
\end{lem}

The proof is an easy consequence of the fact that $C_u(Y) \cong \prod_{y \in Y} C(y) \cong \prod_{y \in Y} \IC$ for a uniformly discrete space $Y$, where the direct product of $C^\ast$-algebras is equipped with the pointwise algebraic operations and the sup-norm. The computation of the operator $K$-theory of $\prod_{y \in Y} \IC$ is now easily done (cf. \cite[Exercise 7.7.3]{higson_roe}).

And last, we will give a relation of uniform $K$-theory with amenability. Note that an analogous relation for bounded de Rham cohomology is already well-known, and also for other, similar (co-)homology theories (see, e.g., \cite[Section 8]{block_weinberger_large_scale}).

\begin{lem}
Let $M$ be a metric space with amenable fundamental group\footnote{see Definition \ref{defn:amenable_group}}.

We let $X$ be the universal cover of $M$ and we denote the covering projection by $\pi\colon X \to M$. Then the pull-back map $K^\ast_u(M) \to K^\ast_u(X)$ is injective.
\end{lem}

\begin{proof}
The projection $\pi$ induces a map $\pi^\ast \colon C_u(M) \to C_u(X)$ which then induces the pull-back map $K^\ast_u(M) \to K^\ast_u(X)$. We will prove the lemma by constructing a left inverse to the above map $\pi^\ast$, i.e., we will construct $p \colon C_u(X) \to C_u(M)$ with $p \circ \pi^\ast = \id \colon C_u(M) \to C_u(M)$.

Let $F \subset X$ be a fundamental domain for the action of the deck transformation group on $X$. Since $\pi_1(M)$ is amenable, we choose a \Folner sequence $(E_i)_i \subset \pi_1(M)$ in it. Now given a function $f \in C_u(X)$, we set
\[f_i(y) := \frac{1}{\card E_i} \sum_{x \in \pi^{-1}(y) \cap E_i \cdot F} f(x)\]
for $y \in M$. This gives us a sequence of functions $f_i$ on $M$, but they are in general not even continuous.

Now choosing a functional $\tau \in (\ell^\infty)^\ast$ associated to a free ultrafilter on $\IN$, we define $p(f)(y) := \tau(f_i(y))$. Due to the \Folner condition on $(E_i)_i$ all discontinuities that the functions $f_i$ may have vanish in the limit under $\tau$, and we get a bounded, uniformly continuous function $p(f)$ on $M$.

It is clear that $p$ is a left inverse to $\pi^\ast$.
\end{proof}

\section{Interpretation via vector bundles}\label{sec:interpretation_uniform_k_theory}

As we have announced in the last section, we will show now that if $M$ is a manifold of bounded geometry, we have a description of the uniform $K$-theory of $M$ via vector bundles of bounded geometry.

To show this, we first need to show that the operator $K$-theory of $C_u(M)$ coincides with the operator $K$-theory of $C_b^\infty(M)$. This is established via the following two lemmas:

\begin{lem}\label{lem:C_b_infty_local}
Let $M$ be a manifold of bounded geometry.

Then $C_b^\infty(M)$ is a local $C^\ast$-algebra.
\end{lem}

\begin{proof}
Since $C_b^\infty(M)$ is a $^\ast$-subalgebra of the $C^\ast$-algebra $C_b(M)$ of bounded continuous functions on $M$, then norm completion of $C_b^\infty(M)$, i.e., its closure in $C_b(M)$, is surely a $C^\ast$-algebra.

So we have to show that $C_b^\infty(M)$ and all matrix algebras over it are closed under holomorphic functional calculus. Since $C_b^\infty(M)$ is naturally a \Frechet algebra with a \Frechet topology which is finer than the sup-norm topology, it remains to show that $C_b^\infty(M)$ itself is closed under holomorphic functional calculus (Lemma \ref{lem:matrix_algebras_holomorphically_closed}).

But that $C_b^\infty(M)$ is closed under holomorphic functional calculus is easily seen using Lemma \ref{lem:local_algebra_equivalent}.
\end{proof}

\begin{lem}\label{lem:norm_completion_C_b_infty}
Let $M$ be a manifold of bounded geometry.

Then the sup-norm completion of $C_b^\infty(M)$ is the $C^\ast$-algebra $C_u(M)$ of bounded, uniformly continuous functions on $M$.
\end{lem}

\begin{proof}
We surely have $\overline{C_b^\infty(M)} \subset C_u(M)$. To show the converse inclusion, we have to approximate a bounded, uniformly continuous function by a smooth one with bounded derivatives. This can be done by choosing a nice cover of $M$ with subordinate partitions of unity via Lemma \ref{lem:nice_coverings_partitions_of_unity} and then apply in every coordinate chart the same mollifier to the uniformly continous function.

Let us elaborate a bit more on the last sentence of the above paragraph: after choosing the nice cover and cutting a function $f \in C_u(M)$ with the subordinate partition of unity $\{\varphi_i\}$, we have transported the problem to Euclidean space $\IR^n$ and our family of functions $\varphi_i f$ is uniformly equicontinuous (this is due to the uniform continuity of $f$ and will be crucially important at the end of this proof). Now let $\psi$ be a mollifier on $\IR^n$, i.e., a smooth function with $\psi \ge 0$, $\supp \psi \subset B_1(0)$, $\int_{\IR^n} \psi d\lambda = 1$ and $\psi_\varepsilon := \varepsilon^{-n} \psi(\largecdot / \varepsilon) \stackrel{\varepsilon \to 0}\longrightarrow \delta_0$. Since convolution satisfies $D^\alpha (\varphi_i f \ast \psi_\varepsilon) = \varphi_i f \ast D^\alpha \psi_\varepsilon$, where $D^\alpha$ is a directional derivative on $\IR^n$ in the directions of the multi-index $\alpha$ and of order $|\alpha|$, we conclude that every mollified function $\varphi_i f \ast \psi_\varepsilon$ is smooth with bounded derivatives. Furthermore, we know $\| \varphi_i f \ast D^\alpha \psi_\epsilon \|_\infty \le \| \varphi_i f \|_\infty \cdot \| D^\alpha \psi_\varepsilon \|_1$ from which we conclude that the bounds on the derivatives of $\varphi_i f \ast \psi_\varepsilon$ are uniform in $i$, i.e., if we glue the functions $\varphi_i f \ast \psi_\epsilon$ together to a function on the manifold $M$ (note that the functions $\varphi_i f \ast \psi_\epsilon$ are supported in our chosen nice cover since convolution with $\psi_\varepsilon$ enlarges the support at most by $\varepsilon$), we get a function $f_\varepsilon \in C_b^\infty(M)$. It remains to show that $f_\varepsilon$ converges to $f$ in sup-norm, which is equivalent to the statement that $\varphi_i f \ast \psi_\varepsilon$ converges to $\varphi_i f$ in sup-norm and uniformly in $i$. But we know that
\[ \big| (\varphi_i f \ast \psi_\epsilon) (x) - (\varphi_i f) (x) \big| \le \sup_{\substack{x \in \supp \varphi_i f\\y \in B_\varepsilon(0)}} \big| (\varphi_i f) (x - y) - (\varphi_i f) (x) \big| \]
from which the claim follows since the family of functions $\varphi_i f$ is uniformly equicontinuous (recall that this followed from the uniform continuity of $f$ and this here is actually the only point in this proof where we need that property of $f$).
\end{proof}

Since $C_b^\infty(M)$ is an $m$-convex \Frechet algebra\footnote{That is to say, a \Frechet algebra such that its topology is given by a countable family of submultiplicative seminorms (see Definition \ref{defn:Frechet_subalgebras}).}, we can also use the $K$-theory for $m$-convex \Frechet algebras as developed by Phillips in \cite{phillips} to define the $K$-theory groups of $C_b^\infty(M)$. But this produces the same groups as operator $K$-theory, since $C_b^\infty(M)$ is an $m$-convex \Frechet algebra with a finer topology than the norm topology and therefore its $K$-theory for $m$-convex \Frechet algebras coincides with its operator $K$-theory (see Lemma \ref{lem:k_theory_frechet_algebras_coincide}).

We summarize this observations in the following lemma:

\begin{lem}\label{lem:equivalent_defns_uniform_k_theory}
Let $M$ be a manifold of bounded geometry.

Then the operator $K$-theory of $C_u(M)$, the operator $K$-theory of $C_b^\infty(M)$ and Phillips $K$-theory for $m$-convex \Frechet algebras of $C_b^\infty(M)$ are all pairwise naturally isomorphic.
\end{lem}

So we have shown $K^\ast_u(M) \cong K_{-\ast}(C_b^\infty(M))$. In order to conclude the description via vector bundles of bounded geometry, we will need to establish the correspondence between vector bundles of bounded geometry and idempotent matrices with entries in $C_b^\infty(M)$. This will be done in the next two subsections.

\subsection*{Isomorphism classes and complements}

Let $M$ be a manifold of bounded geometry and $E$ and $F$ two complex vector bundles equipped with Hermitian metrics and compatible connections.

\begin{defn}[$C^\infty$-boundedness / $C_b^\infty$-isomorphy of vector bundle homomorphisms]\label{defn:C_infty_bounded}
We will call a vector bundle homomorphism $\varphi\colon E \to F$ \emph{$C^\infty$-bounded}, if with respect to synchronous framings of $E$ and $F$ the matrix entries of $\varphi$ are bounded, as are all their derivatives, and these bounds do not depend on the chosen base points for the framings or the synchronous framings themself.

$E$ and $F$ will be called \emph{$C_b^\infty$-isomorphic}, if there exists an isomorphism $\varphi\colon E \to F$ such that both $\varphi$ and $\varphi^{-1}$ are $C^\infty$-bounded. In that case we will call the map $\varphi$ a $C_b^\infty$-isomorphism. Often we will write $E \cong F$ when no confusion can arise with mistaking it with algebraic isomorphy.
\end{defn}

Using the characterization of bounded geometry via the matrix transition functions from Lemma \ref{lem:equiv_characterizations_bounded_geom_bundles}, we immediately see that if $E$ and $F$ are $C_b^\infty$-isomorphic, than $E$ is of bounded geometry if and only if $F$ is.

It is clear that $C_b^\infty$-isomorphy is compatible with direct sums and tensor products, i.e., if $E \cong E^\prime$ and $F \cong F^\prime$ then $E \oplus F \cong E^\prime \oplus F^\prime$ and $E \otimes F \cong E^\prime \otimes F^\prime$.

We will now give a useful global characterization of $C_b^\infty$-isomorphisms if the vector bundles have bounded geometry:

\begin{lem}\label{lem:C_b_infty_Iso_equivalent}
Let $E$ and $F$ have bounded geometry and let $\varphi\colon E \to F$ be an isomorphism. Then $\varphi$ is a $C_b^\infty$-isomorphism if and only if
\begin{itemize}
\item $\varphi$ and $\varphi^{-1}$ are bounded, i.e., $\|\varphi(v)\| \le C \cdot \|v\|$ for all $v \in E$ and a fixed $C > 0$ and analogously for $\varphi^{-1}$, and
\item $\nabla^E - \varphi^\ast \nabla^F$ is bounded and also all its covariant derivatives.
\end{itemize}
\end{lem}

\begin{proof}
For a point $p \in M$ let $B \subset M$ be a geodesic ball centered at $p$, $\{ x_i \}$ the corresponding  normal coordinates of $B$, and let $\{ E_\alpha(y) \}$, $y \in B$, be a framing for $E$. Then we may write every vector field $X$ on $B$ as $X = X^i \frac{\partial}{\partial x_i} = (X^1 , \ldots, X^n)^T$ and every section $e$ of $E$ as $e = e^\alpha E_\alpha = (e^1, \ldots, e^k)^T$, where we assume the Einstein summation convention and where $\largecdot^T$ stands for the transpose of the vector (i.e., the vectors are actually column vectors). Furthermore, after also choosing a framing for $F$, $\varphi$ becomes a matrix for every $y \in B$ and $\varphi(e)$ is then just the matrix multiplication $\varphi(e) = \varphi \cdot e$. Finally, $\nabla^E_X e$ is locally given by
\[\nabla^E_X e = X(e) + \Gamma^E(X)\cdot e,\]
where $X(e)$ is the column vector that we get after taking the derivative of every entry $e^j$ of $e$ in the direction of $X$ and $\Gamma^E$ is a matrix of $1$-forms (i.e., $\Gamma^E(X)$ is then a usual matrix that we multiply with the vector $e$). The entries of $\Gamma^E$ are called the connection $1$-forms.

Since $\varphi$ is an isomorphism, the pull-back connection $\varphi^\ast \nabla^F$ is given by
\[(\varphi^\ast \nabla^F)_X e = \varphi^\ast (\nabla^F_X (\varphi^{-1})^\ast e),\]
so that locally we get
\[(\varphi^\ast \nabla^F)_X e = \varphi^{-1}\cdot \big( X(\varphi \cdot e) + \Gamma^F(X) \cdot \varphi \cdot e\big).\]
Using the product rule we may rewrite $X(\varphi \cdot e) = X(\varphi) \cdot e + \varphi \cdot X(e)$, where $X(\varphi)$ is the application of $X$ to every entry of $\varphi$. So at the end we get for the difference $\nabla^E - \varphi^\ast \nabla^F$ in local coordinates and with respect to framings of $E$ and $F$
\begin{equation}\label{eq:difference_connections_local}
(\nabla^E - \varphi^\ast \nabla^F)_X e = \Gamma^E(X) \cdot e - \varphi^{-1} \cdot X(\varphi) \cdot e - \varphi^{-1} \cdot \Gamma^F(X) \cdot \varphi \cdot e.
\end{equation}

Since $E$ and $F$ have bounded geometry, by Lemma \ref{lem:equiv_characterizations_bounded_geom_bundles} the Christoffel symbols of them with respect to synchronous framings are bounded and also all their derivatives, and these bounds are independent of the point $p \in M$ around that we choose the normal coordinates and the framings. Assuming that $\varphi$ is a $C_b^\infty$-isomorphism, the same holds for the matrix entries of $\varphi$ and $\varphi^{-1}$ and we conclude with the above Equation \eqref{eq:difference_connections_local} that the difference $\nabla^E - \varphi^\ast \nabla^F$ is bounded and also all its covariant derivatives (here we also need to consult the local formula for covariant derivatives of tensor fields).

Conversely, assume that $\varphi$ and $\varphi^{-1}$ are bounded and that the difference $\nabla^E - \varphi^\ast \nabla^F$ is bounded and also all its covariant derivatives. If we denote by $\Gamma^{\text{diff}}$ the matrix of $1$-forms given by
\[\Gamma^{\text{diff}}(X) = \Gamma^E(X) - \varphi^{-1} \cdot X(\varphi) - \varphi^{-1} \cdot \Gamma^F(X) \cdot \varphi,\]
we get from Equation \eqref{eq:difference_connections_local}
\[X(\varphi) = \varphi \cdot (\Gamma^E(X) - \Gamma^{\text{diff}}(X)) - \Gamma^F(X) \cdot \varphi.\]
Since we assumed that $\varphi$ is bounded, its matrix entries must be bounded. From the above equation we then conclude that also the first derivatives of these matrix entries are bounded. But now that we know that the entries and also their first derivatives are bounded, we can differentiate the above equation once more to conclude that also the second derivatives of the matrix entries of $\varphi$ are bounded, on so on. This shows that $\varphi$ is $C^\infty$-bounded. At last, it remains to see that the matrix entries of $\varphi^{-1}$ and also all their derivatives are bounded. But since locally $\varphi^{-1}$ is the inverse matrix of $\varphi$, we just have to use Cramer's rule.
\end{proof}

An important property of vector bundles over compact spaces is that they are always complemented, i.e., for every bundle $E$ there is a bundle $F$ such that $E \oplus F$ is isomorphic to the trivial bundle. Note that this fails in general for non-compact spaces, which prevents to define $K$-theory the same way for non-compact spaces as for compact one (i.e., this is one of the reasons why for non-compact spaces $X$ we consider the compactly supported $K$-theory $K_{\mathrm{cpt}}^\ast(X)$). So our important task is now to show that we have an analogous proposition for vector bundles of bounded geometry, i.e., that they are always complemented (in a suitable way).

\begin{defn}[$C_b^\infty$-complemented vector bundles]
A vector bundle $E$ will be called \emph{$C_b^\infty$-complemented}, if there is some vector bundle $E^\perp$ such that $E \oplus E^\perp$ is $C_b^\infty$-isomorphic to a trivial bundle with the flat connection.
\end{defn}

Since a bundle with a flat connection is trivially of bounded geometry, we get that $E \oplus E^\perp$ is of bounded geometry. And since a direct sum $E \oplus E^\perp$ of vector bundles is of bounded geometry if and only if both vector bundles $E$ and $E^\perp$ are of bounded geometry, we conclude that if $E$ is $C_b^\infty$-complemented, then both $E$ and its complement $E^\perp$ are of bounded geometry. It is also clear that if $E$ is $C_b^\infty$-complemented and $F \cong E$, then $F$ is also $C_b^\infty$-complemented.

We will now prove the crucial fact that every vector bundle of bounded geometry is $C_b^\infty$-complemented. The proof is just the usual one for vector bundles over compact Hausdorff spaces, but we additionally have to take care of the needed uniform estimates. As a source for this usual proof the author used \cite[Proposition 1.4]{hatcher_VB}.

\begin{prop}\label{prop:every_bundle_complemented}
Let $M$ be a manifold of bounded geometry and let $E \to M$ be a vector bundle of bounded geometry.

Then $E$ is $C_b^\infty$-complemented.
\end{prop}

\begin{proof}
Since $M$ and $E$ have bounded geometry, we can find a uniformly locally finite cover of $M$ by normal coordinate balls of a fixed radius together with a subordinate partition of unity whose derivatives are all uniformly bounded and such that over each coordinate ball $E$ is trivialized via a synchronous framing. This follows basically from Lemma \ref{lem:nice_coverings_partitions_of_unity}.

Now we color the coordinate balls with finitely many colors so that no two balls with the same color do intersect.\footnote{\label{footnote:coloring}Construct a graph whose vertices are the coordinate balls and two vertices are connected by an edge if the corresponding coordinate balls do intersect. We have to find a coloring of this graph with only finitely many colors (where of course connected vertices do have different colors). To do this, we firstly use the theorem of de Bruijin--Erd\"{o}s stating that an infinite graph may be colored by $k$ colors if and only if every of its finite subgraphs may be colored by $k$ colors. Secondly, since the coordinate balls have a fixed radius and since our manifold has bounded geometry, the number of balls intersecting a fixed one is uniformly bounded from above. It follows that the number of edges attached to each vertex in our graph is uniformly bounded from above, i.e., the maximum vertex degree of our graph is finite. But this also holds for every subgraph of our graph, with the maximum vertex degree possibly only decreasing by passing to a subgraph. Now a simple greedy algorithm shows that every finite graph may be colored with one more color than its maximum vertex degree.} This gives a partition of the coordinate balls into $N$ families $U_1, \ldots, U_N$ such that every $U_i$ is a collection of disjoint balls, and we get a corresponding subordinate partition of unity $1 = \varphi_1 + \ldots + \varphi_N$ with uniformly bounded derivatives (each $\varphi_i$ is the sum of all the partition of unity functions of the coordinate balls of $U_i$). Furthermore, $E$ is trivial over each $U_i$ and we denote these trivializations coming from the synchronous framings by $h_i \colon p^{-1}(U_i) \to U_i \times \IC^k$, where $p\colon E \to M$ is the projection.

Now we set
\[g_i\colon E \to \IC^k, \ g_i(v) := \varphi_i(p(v)) \cdot \pi_i (h_i (v)),\]
where $\pi_i \colon U_i \times \IC^k \to \IC^k$ is the projection. Each $g_i$ is a linear injection on each fiber over $\varphi_i^{-1}(0,1]$ and so, if we define
\[g \colon E \to \IC^{Nk}, \ g(v) := (g_1(v), \ldots, g_N(v)),\]
we get a map $g$ that is a linear injection on each fiber of $E$. Finally, we define a map
\[G\colon E \to M \times \IC^{Nk}, \ G(v) := (p(v), g(v)).\]
This establishes $E$ as a subbundle of a trivial bundle.

If we equip $M \times \IC^{Nk}$ with a constant metric and the flat connection, we get that the induced metric and connection on $E$ is $C_b^\infty$-isomorphic to the original metric and connection on $E$ (this is due to our choice of $G$). Now let us denote by $e$ the projection matrix of the trivial bundle $\IC^{Nk}$ onto the subbundle $G(E)$ of it, i.e., $e$ is an $Nk \times Nk$-matrix with functions on $M$ as entries and $\image e = E$. Now, again due to our choice of $G$, we can conclude that these entries of $e$ are bounded functions with all derivatives of them also bounded, i.e., $e \in \Idem_{Nk \times Nk}(C_b^\infty(M))$. Now the claim follows with the Proposition \ref{prop:image_proj_matrix_complemented} which establishes the orthogonal complement $E^\perp$ of $E$ in $\IC^{Nk}$ with the induced metric and connection as a $C_b^\infty$-complement to $E$.
\end{proof}

We have seen in the above proposition that every vector bundle of bounded geometry is $C_b^\infty$-complemented. Now if we have a manifold of bounded geometry $M$, then its tangent bundle $TM$ is of bounded geometry and so we know that it is $C_b^\infty$-complemented (although $TM$ is real and not a complex bundle, the above proof of course also holds for real vector bundles). But in this case we usually want the complement bundle to be given by the normal bundle $NM$ coming from an embedding $M \hookrightarrow \IR^N$. We will prove this now under the assumption that the embedding of $M$ into $\IR^N$ is ``nice'':

\begin{cor}\label{cor:tangent_bundle_complemented}
Let $M$ be a manifold of bounded geometry and let it be isometrically embedded into $\IR^N$ such that the second fundamental form is $C^\infty$-bounded.

Then its tangent bundle $TM$ is $C_b^\infty$-complemented by the normal bundle $NM$ corresponding to this embedding $M \hookrightarrow \IR^N$, equipped with the induced metric and connection.
\end{cor}

\begin{proof}
Let $M$ be isometrically embedded in $\IR^N$. Then its tangent bundle $TM$ is a subbundle of $T\IR^N$ and we denote the projection onto it by $\pi\colon T\IR^N \to TM$. Because of Point 1 of the following Proposition \ref{prop:image_proj_matrix_complemented} it suffices to show that the entries of $\pi$ are $C^\infty$-bounded functions.

Let $\{v_i\}$ be the standard basis of $\IR^N$ and let $\{E_\alpha(y)\}$ be the orthonormal frame of $TM$ arising out of normal coordinates $\{\partial_k\}$ of $M$ via the Gram-Schmidt process. Then the entries of the projection matrix $\pi$ with respect to the basis $\{v_i\}$ are given by
\[\pi_{ij}(y) = \sum_\alpha \langle E_\alpha(y), v_j\rangle \langle E_\alpha(y), v_i\rangle.\]

Let $\widetilde{\nabla}$ denote the flat connection on $\IR^N$. Since $\widetilde{\nabla}_{\partial_k} v_i = 0$ we get
\[\partial_k \pi_{ij} (y) = \sum_\alpha \langle \widetilde{\nabla}_{\partial_k} E_\alpha (y), v_j\rangle \langle E_\alpha(y), v_i\rangle + \langle E_\alpha(y), v_j\rangle \langle \widetilde{\nabla}_{\partial_k} E_\alpha (y), v_i\rangle.\]
Now if we denote by $\nabla^M$ the connection on $M$, we get
\[\widetilde{\nabla}_{\partial_k} E_\alpha(y) = \nabla^{M}_{\partial_k} E_\alpha(y) + \operatorname{I\!\!\;I}(\partial_k, E_\alpha),\]
where $\operatorname{I\!\!\;I}$ is the second fundamental form. So to show that $\pi_{ij}$ is $C^\infty$-bounded, we must show that $E_\alpha(y)$ are $C^\infty$-bounded sections of $TM$ (since by assumption the second fundamental form is a $C^\infty$-bounded tensor field). But that these $E_\alpha(y)$ are $C^\infty$-bounded sections of $TM$ follows from their construction (i.e., applying Gram-Schmidt to the normal coordinate fields $\partial_k$) and because $M$ has bounded geometry (we need here the characterization of bounded geometry via the metric coefficients from Lemma \ref{lem:bounded_geometry_christoffel_symbols}).
\end{proof}

\subsection*{Interpretation of \texorpdfstring{$K^0_u(M)$}{even uniform K(M)}}

Recall for the understanding of the following proposition that if a vector bundle is $C_b^\infty$-complemented, then it is of bounded geometry. Furthermore, this proposition is the crucial one that gives us the description of uniform $K$-theory via vector bundles of bounded geometry.

\begin{prop}\label{prop:image_proj_matrix_complemented}
Let $M$ be a manifold of bounded geometry.
\begin{enumerate}
\item Let $e \in \Idem_{N \times N}(C_b^\infty(M))$ be an idempotent matrix.

Then the vector bundle $E := \image e$, equipped with the induced metric and connection, is $C_b^\infty$-complemented.

\item Let $E$ be a $C_b^\infty$-complemented vector bundle, i.e., there is a vector bundle $E^\perp$ such that $E \oplus E^\perp$ is $C_b^\infty$-isomorphic to the trivial $N$-dimensional bundle $\IC^N \to M$.

Then all entries of the projection matrix $e$ onto the subspace $E \oplus 0 \subset \IC^N$ with respect to a global synchronous framing of $\IC^N$ are $C^\infty$-bounded, i.e., we have $e \in \Idem_{N \times N}(C_b^\infty(M))$.
\end{enumerate}
\end{prop}

\begin{proof}[Proof of point 1]
We denote by $E$ the vector bundle $E := \image e$ and by $E^\perp$ its complement $E^\perp := \image (1-e)$ and equip them with the induced metric and connection. So we have to show that $E \oplus E^\perp$ is $C_b^\infty$-isomorphic to the trivial bundle $\IC^N \to M$.

Let $\varphi\colon E \oplus E^\perp \to \IC^N$ be the canonical algebraic isomorphism $\varphi(v,w) := v + w$. We have to show that both $\varphi$ and $\varphi^{-1}$ are $C^\infty$-bounded. 

Let $p \in M$. Let $\{ E_\alpha \}$ be an orthonormal basis of the vector space $E_p$ and $\{ E^\perp_\beta \}$ an orthonormal basis of $E^\perp_p$. Then the set $\{ E_\alpha, E^\perp_\beta \}$ is an orthonormal basis for $\IC_p^N$. We extend $\{E_\alpha\}$ to a synchronous framing $\{E_\alpha(y)\}$ of $E$ and $\{E^\perp_\beta\}$ to a synchronous framing $\{E^\perp_\beta(y)\}$ of $E^\perp$. Since $\IC^N$ is equipped with the flat connection, the set $\{ E_\alpha, E^\perp_\beta \}$ forms a synchronous framing for $\IC^N$ at all points of the normal coordinate chart. Then $\varphi(y)$ is the change-of-basis matrix from the basis $\{E_\alpha(y), E_\beta^\perp(y)\}$ to the basis $\{ E_\alpha, E^\perp_\beta \}$ and vice versa for $\varphi^{-1}(y)$; see the Illustration \ref{fig:frames} on the next page.

\begin{figure}[htbp]
\centering
\includegraphics[scale=0.7]{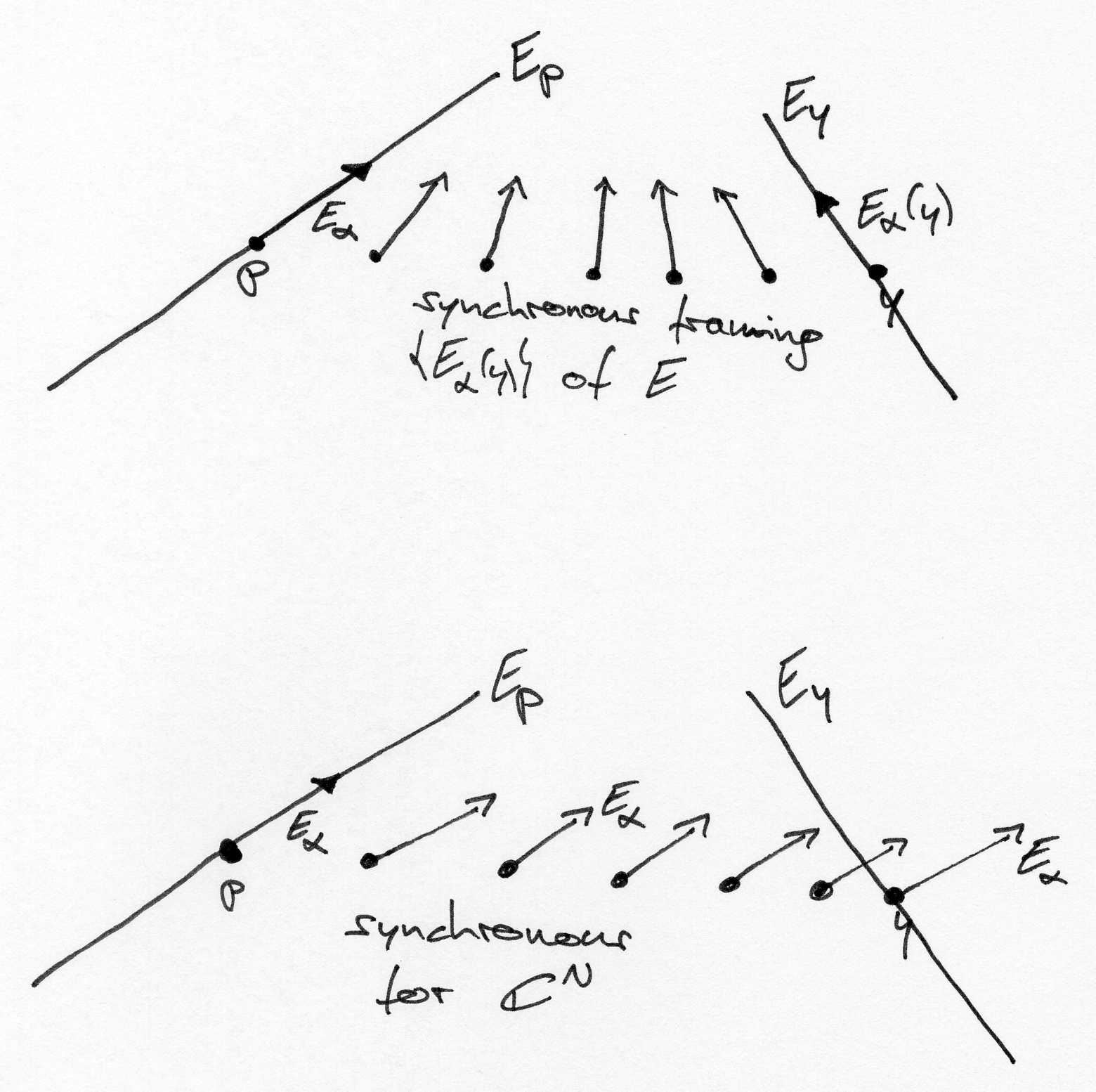}
\caption{The frames $\{E_\alpha(y)\}$ and $\{E_\alpha\}$.}
\label{fig:frames}
\end{figure}

We have $e(p)(E_\alpha) = E_\alpha$. Since the entries of $e$ are $C^\infty$-bounded and the rank of a matrix is a lower semi-continuous function of the entries, there is some geodesic ball $B$ around $p$ such that $\{ e(y)(E_\alpha) \}$ forms a basis of $E_y$ for all $y \in B$ and the diameter of the ball $B$ is bounded from below independently of $p \in M$. We denote by $\Gamma_{i \nu}^\mu(y)$ the Christoffel symbols of $E$ with respect to the frame $\{ e(y)(E_\alpha) \}$. Let $\gamma(t)$ be a radial geodesic in $M$ with $\gamma(0) = p$. If we now let $E_\alpha(\gamma(t))^\mu$ denote the $\mu$th entry of the vector $E_\alpha(\gamma(t))$ represented in the basis $\{ e(\gamma(t))(E_\alpha) \}$, then (since it is a synchronous frame) it satisfies the ODE
\[ \tfrac{d}{dt} E_\alpha(\gamma(t))^\mu = -\sum_{i,\nu} E_\alpha(\gamma(t))^\nu \cdot \tfrac{d}{dt}{\gamma_i}(t) \cdot \Gamma_{i \nu}^\mu(\gamma(t)),\]
where $\{\gamma_i\}$ is the coordinate representation of $\gamma$ in normal coordinates $\{x_i\}$. Since $\gamma$ is a radial geodesic, its representation in normal coordinates is $\gamma_i(t) = t \cdot \gamma_i(0)$ and so the above formula simplifies to
\begin{equation}\label{eq:ODE_synchronous_frame}
\tfrac{d}{dt} E_\alpha(\gamma(t))^\mu = -\sum_{i,\nu} E_\alpha(\gamma(t))^\nu \cdot \gamma_i(0) \cdot \Gamma_{i \nu}^\mu(\gamma(t)).
\end{equation}

Since $\Gamma_{i \nu}^\mu(y)$ are the Christoffel symbols with respect to the frame $\{ e(y)(E_\alpha) \}$, we get the equation
\begin{equation}\label{eq:christoffel_symbols_representation}
\sum_{\mu} \Gamma_{i \nu}^\mu(y) \cdot e(y)(E_\mu) = \nabla^E_{\partial_i} e(y)(E_\nu).
\end{equation}
Now using that $\nabla^E$ is induced by the flat connection, we get
\[\nabla^E_{\partial_i} e(y)(E_\nu) = e (\partial_i(e(y)(E_\nu))) = e((\partial_i e)(y)(E_\nu)),\]
i.e., $e((\partial_i e)(y)(E_\nu))$ is the representation of $\nabla^E_{\partial_i} e(y)(E_\nu)$ with respect to the frame $\{E_\alpha, E_\beta^\perp\}$. Since the entries of $e$ are $C^\infty$-bounded, the entries of this representation $e((\partial_i e)(y)(E_\nu))$ are also $C^\infty$-bounded. From Equation \eqref{eq:christoffel_symbols_representation} we see that $\Gamma_{i \nu}^\mu(y)$ is the representation of $\nabla^E_{\partial_i} e(y)(E_\nu)$ in the frame $\{e(y)(E_\mu)\}$. So we conclude that the Christoffel symbols $\Gamma_{i \nu}^\mu(y)$ are $C^\infty$-bounded functions.

Equation \eqref{eq:ODE_synchronous_frame} and the theory of ODEs now tell us that the functions $E_\alpha(y)^\mu$ are $C^\infty$-bounded. Since these are the representations of the vectors $E_\alpha(y)$ in the basis $\{ e(y)(E_\alpha) \}$, we can conclude that the entries of the representations of the vectors $E_\alpha(y)$ in the basis $\{E_\alpha, E_\beta^\perp\}$ are $C^\infty$-bounded. But now these entries are exactly the first $(\dim E)$ columns of the change-of-basis matrix $\varphi(y)$.

Arguing analogously for the complement $E^\perp$, we get that the other columns of $\varphi(y)$ are also $C^\infty$-bounded, i.e., $\varphi$ itself is $C^\infty$-bounded.

It remains to show that the inverse homomorphism $\varphi^{-1}$ is $C^\infty$-bounded. But since pointwise it is given by the inverse matrix, i.e., $\varphi^{-1}(y) = \varphi(y)^{-1}$, this claim follows immediately from Cramer's rule, because we already know that $\varphi$ is $C^\infty$-bounded.
\end{proof}

\begin{proof}[Proof of point 2]
Let $\{E_\alpha(y)\}$ be a synchronous framing for $E$ and $\{E_\beta^\perp(y)\}$ one for $E^\perp$. Then $\{E_\alpha(y), E_\beta^\perp(y)\}$ is one for $E \oplus E^\perp$. Furthermore, let $\{v_i(y)\}$ be a synchronous framing for the trivial bundle $\IC^N$ and let $\varphi\colon E \oplus E^\perp \to \IC^N$ be the $C_b^\infty$-isomorphism.

Then projection matrix $e \in \Idem_{N \times N}(C^\infty(M))$ onto the subspace $E \oplus 0$ is given with respect to the basis $\{E_\alpha(y), E_\beta^\perp(y)\}$ of $E \oplus E^\perp$ and of $\IC^N$ by the usual projection matrix onto the first $(\dim E)$ vectors, i.e., its entries are clearly $C^\infty$-bounded since they are constant. Now changing the basis to $\{v_i(y)\}$, the representation of $e(y)$ with respect to this new basis is given by $\varphi^{-1}(y) \cdot e \cdot \varphi(y)$, i.e., $e \in \Idem_{N \times N}(C_b^\infty(M))$.
\end{proof}

If we have a $C_b^\infty$-complemented vector bundle $E$, then different choices of complements and different choices of isomorphisms with the trivial bundle lead to similar projection matrices. The proof of this is analogous to the corresponding proof in the usual case of vector bundles over compact Hausdorff spaces. We also get that $C_b^\infty$-isomorphic vector bundles produce similar projection matrices. Of course this also works the other way round, i.e., similar idempotent matrices give us $C_b^\infty$-isomorphic vector bundles. Again, the proof of this is the same as the one in the topological category.

\begin{defn}
Let $M$ be a manifold of bounded geometry. We define
\begin{itemize}
\item $\Vect_u(M)/_\sim$ as the abelian monoid of all complex vector bundles of bounded geometry over $M$ modulo $C_b^\infty$-isomorphism (the addition is given by the direct sum $[E] + [F] := [E \oplus F]$) and
\item $\Idem(C_b^\infty(M))/_\sim$ as the abelian monoid of idempotent matrizes of arbitrary size over the \Frechet algebra $C_b^\infty(M)$ modulo similarity (with addition defined as $[e] + [f] := \left[\begin{pmatrix}e&0\\0&f\end{pmatrix}\right]$).
\end{itemize}
\end{defn}

Let $f\colon M \to N$ be a $C^\infty$-bounded map\footnote{We use covers of $M$ and of $N$ via normal coordinate charts of a fixed radius and demand that locally in this charts the derivatives of $f$ are all bounded and these bounds are independent of the chart used.} and $E$ a vector bundle of bounded geometry over $N$. Then it is clear that the pullback bundle $f^\ast E$ equipped with the pullback metric and connection is a vector bundle of bounded geometry over $M$.

The above discussion together with Proposition \ref{prop:image_proj_matrix_complemented} prove the following corollary:

\begin{cor}\label{cor:two_monoids_isomorphic}
The monoids $\Vect_u(M)/_\sim$ and $\Idem(C_b^\infty(M))/_\sim$ are isomorphic and this isomorphism is natural with respect to $C^\infty$-bounded maps between manifolds.
\end{cor}

From this Corollary \ref{cor:two_monoids_isomorphic}, Lemma \ref{lem:C_b_infty_local} and Proposition \ref{prop:every_bundle_complemented} we immediately get the following interpretation of the $0$th uniform $K$-theory group $K^0_u(M)$ of a manifold of bounded geometry:

\begin{thm}[Interpretation of $K^0_u(M)$]\label{thm:interpretation_K0u}
Let $M$ be a manifold of bounded geometry. Then every element of $K^0_u(M)$ is of the form $[E] - [F]$, where $[E]$ and $[F]$ are $C_b^\infty$-isomorphism classes of complex vector bundles of bounded geometry over $M$.

Furthermore, every such vector bundle defines an element in $K^0_u(M)$.
\end{thm}

\subsection*{Interpretation of \texorpdfstring{$K^1_u(M)$}{odd uniform K(M)}}

For the interpretation of $K^1_u(M)$ we will make use of suspensions of algebras. The suspension isomorphism theorem for operator $K$-theory states that we have an isomorphism $K_1(C_u(M)) \cong K_0(S C_u(M))$, where $S C_u(M)$ is the suspension of $C_u(M)$:
\begin{align*}
S C_u(M) & := \{ f \colon S^1 \to C_u(M) \ | \ f \text{ continuous and } f(1) = 0\}\\
& \cong \{ f \in C_u(S^1 \times M) \ | \ f(1, x) = 0 \text{ for all }x \in M\}.
\end{align*}
Equipped with the sup-norm this is again a $C^\ast$-algebra. Since functions $f \in S C_u(M)$ are uniformly continuous, the condition $f(1, x) = 0$ for all $x \in M$ is equivalent to $\lim_{t \to 1} f(t, x) = 0 \text{ uniformly in } x$.

Now in order to interpret $K_0(S C_u(M))$ via vector bundles of bounded geometry over $S^1 \times M$, we will need to find a suitable \Frechet subalgebra of $S C_u(M)$ so that we can again use Proposition \ref{prop:image_proj_matrix_complemented}. Luckily, this was already done by Phillips in his paper \cite{phillips}:

\begin{defn}[Smooth suspension of a \Frechet algebras, {\cite[Definition 4.7]{phillips}}]\label{defn:smooth_suspension_algebra}
Let $A$ be a \Frechet algebra. Then the \emph{smooth suspension $S_\infty A$ of $A$} is defined as the \Frechet algebra
\[ S_\infty A := \{ f\colon S^1 \to A \ | \ f \text{ smooth and } f(1) = 0\}\]
equipped with the topology of uniform convergence of every derivative in every seminorm of $A$.
\end{defn}

For a manifold $M$ we have
\begin{align*}
S_\infty C_b^\infty(M) & \cong \{ f \in C_b^\infty(S^1 \times M) \ | \ f(1, x) = 0 \text{ for all }x \in M\} \\
& = \{ f \in C_b^\infty(S^1 \times M) \ | \ \forall k \in \IN_0 \colon \lim_{t \to 1} \nabla^k_x f(t, x) = 0 \text{ uniformly in } x\}.
\end{align*}

The proof of the following lemma is analogous to the proof of the Lemma \ref{lem:C_b_infty_local}:

\begin{lem}
Let $M$ be a manifold of bounded geometry.

Then the sup-norm completion of $S_\infty C_b^\infty(M)$ is $S C_u(M)$ and $S_\infty C_b^\infty(M)$ is a local $C^\ast$-algebra.
\end{lem}

Putting it all together, we get $K^1_u(M) = K_0(S_\infty C_b^\infty(M))$, and Proposition \ref{prop:image_proj_matrix_complemented}, adapted to our case here, gives us the following interpretation of the $1$st uniform $K$-theory group $K^1_u(M)$ of a manifold of bounded geometry:

\begin{thm}[Interpretation of $K^1_u(M)$]\label{thm:interpretation_K1u}
Let $M$ be a manifold of bounded geometry. Then every elements of $K^1_u(M)$ is of the form $[E] - [F]$, where $[E]$ and $[F]$ are $C_b^\infty$-isomorphism classes of complex vector bundles of bounded geometry over $S^1 \times M$ with the following property: there is some neighbourhood $U \subset S^1$ of $1$ such that $[E|_{U \times M}]$ and $[F|_{U \times M}]$ are $C_b^\infty$-isomorphic to a trivial vector bundle with the flat connection (the dimension of the trivial bundle is the same for both $[E|_{U \times M}]$ and $[F|_{U \times M}]$).

Furthermore, every such vector bundle defines an element of $K^1_u(M)$.
\end{thm}

\section{External and internal product}

In this section we will briefly recall how the external and internal products for the operator $K$-theory of $C_u(M)$ translate to the level of vector bundles.

Recall that we have an associative external product
\[K_{-p}(C_u(X)) \otimes K_{-q}(C_u(Y)) \to K_{-p-q}(C_u(X) \otimes C_u(Y))\]
on $K$-theory and, since $C_u(X)$ is commutative, also an internal product
\begin{equation}\label{eq:ext_prod_uniform_K}
K_{-p}(C_u(X)) \otimes K_{-q}(C_u(X)) \to K_{-p-q}(C_u(X))
\end{equation}
induced from the external one via composing with the map on $K$-theory induced from $C_u(X) \otimes C_u(X) \to C_u(X)$, $f \otimes g \mapsto fg$.

To get a corresponding external product on the uniform $K$-theories of $X$ and $Y$, we compose the above external product with the from $C_u(X) \otimes C_u(Y) \to C_u(X \times Y)$\footnote{Note that this is in general (i.e., if $X$ and $Y$ are not totally bounded) not an isomorphism, contrary to the case $C_0(X) \otimes C_0(Y) \cong C_0(X \times Y)$.} induced map on $K$-theory. This leads to an associative external product
\[\times \colon K_u^p(X) \otimes K_u^q(Y) \to K_u^{p+q}(X \times Y).\]
To get a corresponding internal product on uniform $K$-theory we compose with the map on uniform $K$-theory induced from the embedding $X \hookrightarrow X \times X$ and get
\[\otimes \colon K_u^p(X) \otimes K_u^q(X) \to K_u^{p+q}(X).\]
This internal product coincides with the above \eqref{eq:ext_prod_uniform_K} since $C_u(X) \otimes C_u(X) \to C_u(X)$ factors as $C_u(X) \otimes C_u(X) \to C_u(X \times X) \to C_u(X)$.

As in the case of compact manifolds and topological $K$-theory, the interpretation of these products for vector bundles of bounded geometry is the following one: if $E$ and $F$ are two vector bundles of bounded geometry over $M$ and $N$, then $[E] \times [F] \in K_u^0(M \times N)$ is given by $[E \boxtimes F]$. If both vector bundles are defined over the same manifold $M$, then their internal product $[E] \otimes [F] \in K_u^0(M)$ is given by $[E \otimes F]$.

Let us summarize the properties of these product in the following proposition:

\begin{prop}
Their exists an associative external product
\[\times \colon K_u^p(X) \otimes K_u^q(Y) \to K_u^{p+q}(X \times Y)\]
and an associative internal product
\[\otimes \colon K_u^p(X) \otimes K_u^q(X) \to K_u^{p+q}(X).\]

For manifolds $M$ and $N$ of bounded geometry the external product of two vector bundles of bounded geometry $E$ over $M$ and $F$ over $N$ is given by their external tensor product
\[[E] \times [F] = [E \boxtimes F] \in K_u^0(M \times N)\]
and for $M = N$ their internal product is given by the tensor product
\[[E] \otimes [F] = [E \otimes F] \in K_u^0(M).\]
\end{prop}

\section{Bounded de Rham cohomology}\label{sec:bounded_de_rham_cohomology}

The rest of this chapter is devoted to the discussion and proof of the Chern character isomorphism theorem in our uniform setting. For this we first have to discuss the target of the Chern character, i.e., the corresponding cohomology theory. We have already encountered it in our discussion of amenability, but for the convenience of the reader we will recall its definition:

\begin{defn}[Bounded de Rham cohomology]
Let $\Omega_b^p(M)$ denote the vector space of $p$-forms on $M$, which are bounded in the norm
\[\| \alpha \| := \sup_{x \in M} \{\|\alpha(x)\| + \|d \alpha(x)\|\}.\]
We define the \emph{bounded de Rham cohomology of $M$} as
\[ H_{b, \mathrm{dR}}^p(M) := \kernel d_p / \image d_{p-1}.\]
\end{defn}

\begin{rem}\label{rem:reduced_bounded_de_Rham_cohomology}
Since in general the subspace $\image d_{p-1} \subset \kernel d_p$ is not closed, the induced norm on the bounded de Rham cohomology vector space is in general just a seminorm, i.e., in general there are elements with induced norm $0$ in $H_{b, \mathrm{dR}}^\ast(M)$. The bounded de Rham cohomology as we have defined it is sometimes called \emph{unreduced}. The \emph{reduced bounded de Rham cohomology} is then defined as
\[\bar{H}_{b, \mathrm{dR}}^p(M) := \kernel d_p / \closure(\image d_{p-1}).\]
\end{rem}

We will now give an example that reduced and unreduced bounded de Rham cohomology differ on the manifold $\IR$ of bounded geometry.

\begin{example}\label{ex:difference_(un)reduced_de_rham}
Consider the $1$-form $\tfrac{1}{x} dx$ on $\IR_{> 1}$. It is clearly a bounded form and its exterior derivative vanishes, since there are no $2$-forms on $\IR_{> 1}$, so it defines a class $[\tfrac{1}{x} dx] \in H^1_{b, \mathrm{dR}}(\IR_{>1})$. This class is not $0$, because the antiderivative of it is $\log x$ which is not a bounded $0$-form.

But in reduced bounded de Rham cohomology we have $[\tfrac{1}{x} dx] = [0] \in \bar{H}^1_{b, \mathrm{dR}}(\IR_{>1})$, because we can cut off $\tfrac{1}{x}$ sufficiently far to the right, which reveals that $[\tfrac{1}{x} dx]$ is a norm limit of derivatives of bounded $0$-forms (those are given by cutting of $\log x$).

Note that if we extend this $1$-form $\tfrac{1}{x} dx$ constantly $0$ to the left (so that it is defined on $\IR$), we get the same counterexample on the manifold $\IR$ of bounded geometry.
\end{example}

Let $f\colon M \to N$ be a smooth map such that $\|f_\ast X\| \le C \|X\|$ for all $X \in TM$ and a fixed $C > 0$. In this case we say that $f$ has \emph{bounded dilatation}. Such a map induces a map on the bounded de Rham cohomology. We say that the map $f$ is a \emph{smooth quasi-isometry}, if it is a diffeomorphism and both $f$ and $f^{-1}$ are of bounded dilatation. We conclude that smoothly quasi-isometric Riemannian manifolds have isomorphic bounded de Rham cohomology.

Let $f, g\colon M \to N$ be two maps of bounded dilatation. We say that they are \emph{boundedly homotopic}, if there is a homotopy $H\colon M \times [0,1] \to N$ from $f$ to $g$, which itself is of bounded dilatation. The same argument as for usual de Rham cohomology (see, e.g., \cite[15.Homotopy Invariance]{lee_smooth}) gives us the homotopy invariance of bounded de Rham cohomology. Note that this is one of the major ingredients in the proof of the uniform Chern character isomorphism theorem.

\begin{prop}[Homotopy invariance]\label{prop:homotopy_invariance_de_rham}
Let $f, g\colon M \to N$ be smooth maps of bounded dilatation.

If $f$ and $g$ are boundedly homotopic, then the induced maps $f^\ast$ and $g^\ast$ on bounded de Rham cohomology are equal.
\end{prop}

We call two manifolds \emph{boundedly homotopy equivalent}, if there are maps of bounded dilatation $f\colon M \to N$ and $g\colon N \to M$ such that both $f \circ g$ and $g\circ f$ are boundedly homotopic to the corresponding identity maps. From the above Proposition \ref{prop:homotopy_invariance_de_rham} we immediately get the following corollary:

\begin{cor}\label{cor:homotopy_equivalence_de_rham}
Let $M$ and $N$ be two boundedly homotopy equivalent manifolds.

Then their bounded de Rham cohomology groups are isomorphic.
\end{cor}

Since especially $[0,1] \times M$ and $M$ are boundedly homotopy equivalent, we conclude
\begin{equation}\label{eq:bdR_M_M01}
H_{b, \mathrm{dR}}^p([0,1] \times M) \cong H_{b, \mathrm{dR}}^p(M).
\end{equation}
We will need this fact in the following discussion.

Recall that for topological spaces the cohomology groups of the suspension of the space do coincide (up to a degree shift) with the cohomology groups of the space itself. We will need this fact for bounded de Rham cohomology (to interpret the image of $K^1_u(M)$ under the Chern character), but since the suspension of a smooth manifold is in general not a smooth manifold, we have to reinterpret things a bit. So let us define for $p \ge 1$ and $i=0,1$ the following spaces of forms, which emulate the cohomology of a suspension and of a cone of a topological space:
\begin{align}
S \Omega_b^p(M) & := \{ \alpha \in \Omega_b^p(S^1 \times M) \ | \ \alpha|_{\{1\} \times M} = d\alpha|_{\{1\} \times M} = 0\}\label{eq:defn_suspension_de_Rham}\\
C_i \Omega_b^p(M) & := \{ \alpha \in \Omega_b^p([0,1] \times M) \ | \ \alpha|_{\{i\} \times M} = d\alpha|_{\{i\} \times M} = 0\}\notag
\end{align}
and for $p = 0$ we only require $d\alpha|_{\{1\} \times M} = 0$, resp. $d\alpha|_{\{i\} \times M} = 0$.

The same argument as in case of usual de Rham cohomology (see, e.g., \cite[Theorem 15.9]{lee_smooth}) gives us the exactness of the short sequence
\[0 \to S \Omega_{b, \mathrm{dR}}^p(M) \to C_0 \Omega_{b, \mathrm{dR}}^p(M) \oplus C_1 \Omega_{b, \mathrm{dR}}^p(M) \to \Omega^p_{b, \mathrm{dR}}([0,1] \times M) \to 0\]
from which we get for the corresponding cohomology groups\footnote{We denote by $S H^{p+1}_{b, \mathrm{dR}}(M)$ the group $\kernel d_{p+1} / \image d_p$, where $d_\ast$ is defined on $S \Omega_{b, \mathrm{dR}}^\ast(M)$. Analogously for $C_0$ and $C_1$ instead of $S$.} the long, exact sequence
\begin{align}\label{eq:MV_suspension}
\cdots & \to C_0 H_{b, \mathrm{dR}}^p(M) \oplus C_1 H_{b, \mathrm{dR}}^p(M) \to H_{b, \mathrm{dR}}^p([0,1] \times M) \to S H^{p+1}_{b, \mathrm{dR}}(M) \to\notag\\
& \to C_0 H_{b, \mathrm{dR}}^{p+1}(M) \oplus C_1 H_{b, \mathrm{dR}}^{p+1}(M) \to \cdots
\end{align}
which is an analogue of the Mayer--Vietoris sequence for the suspension of a space.

The same homotopy which we used to show \eqref{eq:bdR_M_M01} gives us $C_i H^p_{b, \mathrm{dR}}(M) = 0$ for all $p \ge 1$ and $i=0,1$. So putting all together we conclude the following lemma:

\begin{lem}[Suspension isomorphism]\label{lem:bounded_deRham_suspension_equal}
For all $p \ge 1$ we have
\[H^{p}_{b, \mathrm{dR}}(M) \cong SH^{p+1}_{b, \mathrm{dR}}(M).\]
\end{lem}

\begin{rem}[Mayer--Vietoris sequences]\label{rem:mayer_vietoris_de_rham}
The above Sequence \eqref{eq:MV_suspension} resembles a Mayer--Vietoris sequence. In general, given an open cover $\{U, V\}$ of a manifold $M$, we do have a Mayer--Vietoris sequence for unreduced bounded de Rham cohomology associated to this cover if there exists a partition of unity $\{\varphi, \psi\}$ subordinate to this cover such that both $1$-forms $d\varphi$ and $d\psi$ are bounded forms (i.e., the first derivatives of $\varphi$ and $\psi$ with respect to unit tangent vectors are bounded).

Note that the existence of such subordinate partitions of unity is very unlikely (just consider, e.g., $U$ and $V$ being balls in $\IR^2$ which intersect). But at least it exists for balls in $\IR$ which gives us the above Sequence \eqref{eq:MV_suspension} and enables us to compute $H^1_{b, \mathrm{dR}}(\IR)$ in Example \ref{ex:computation_de_rham_real_line}.
\end{rem}

Note that a Mayer--Vietoris sequence exists in general only for unreduced bounded de Rham cohomology. To show that a Mayer--Vietoris sequence for the reduced version does in general not exist, we can use the $1$-form $\tfrac{1}{x}|_{> 1} dx \in \Omega_b^1(\IR)$ from Example \ref{ex:difference_(un)reduced_de_rham} (we use the notation $\tfrac{1}{x}|_{> 1}$ to denote a smooth function which is $\tfrac{1}{x}$ for $x > 1$ and which is constantly $0$ to the left). We will do this in the following Example \ref{ex:computation_de_rham_real_line}, where we first show that $H^1_{b, \mathrm{dR}}(\IR)$ is a vector space of uncountable algebraic dimension and then use parts of this computation to show that we do not have a Mayer--Vietoris sequence for the reduced bounded de Rham cohomology.

\begin{example}[Computation of $H^1_{b, \mathrm{dR}}(\IR)$]\label{ex:computation_de_rham_real_line}
We define the following open cover of $\IR$:
\begin{align*}
U & := \ldots \cup (-3, 0) \cup (1, 4) \cup (5, 8) \cup (9, 12) \cup \ldots\\
V & := \ldots \cup (-1, 2) \cup (3, 6) \cup (7, 10) \cup (11, 14) \cup \ldots
\end{align*}
Then we have $U \cap V = \ldots \cup (-1, 0) \cup (1, 2) \cup (3, 4) \cup (5, 6) \cup (7, 8) \cup \ldots$ Since all the occuring intervals are of length at most $3$, we see that all three spaces $U$, $V$ and $U \cap V$ are boundedly homotopy equivalent to a countable infinite number of points with the discrete metric (it is the discrete metric, since the intervals are all at least the distance $1$ apart). So the reduced and unreduced bounded de Rham cohomologies of $U$, $V$ and $U \cap V$ coincide (and are given by $\ell^\infty(\IZ)$ in degree $0$ and the trivial space in all other degrees, where $\ell^\infty(\IZ)$ denotes the vector space of all bounded sequences indexed by $\IZ$).

The needed part of the Mayer--Vietoris sequence is
\[\to \underbrace{H^0_{b, \mathrm{dR}}(U) \oplus H^0_{b, \mathrm{dR}}(V)}_{\cong \ell^\infty(\IZ) \oplus \ell^\infty(\IZ)} \stackrel{\Delta}\to \underbrace{H^0_{b, \mathrm{dR}}(U \cap V)}_{\cong \ell^\infty(\IZ)} \stackrel{\delta}\twoheadrightarrow H^1_{b, \mathrm{dR}}(\IR) \to \underbrace{H^1_{b, \mathrm{dR}}(U) \oplus H^1_{b, \mathrm{dR}}(V)}_{\cong 0} \to\]
The map $\Delta$ is given by $\Delta((a_n), (b_n)) = (c_n)$ with $c_{2k} = a_{k+1} - b_k$ and $c_{2k+1} = a_k - b_k$. The other way round, given a sequence $(c_n)$ and a value for $a_0$ we can compute inductively the sequences $(a_n)$ and $(b_n)$ with $\Delta((a_n), (b_n)) = c_n$: for $k \ge 0$ their values are given as $a_k = a_0 - c_1 + c_2 - c_3 \pm \ldots \pm c_{2k}$ and $b_k = a_0 - c_1 + c_2 - c_3 \pm \ldots \pm c_{2k+1}$. The formulas for $k \le 0$ are analogous. But even if the sequence $(c_n)$ is bounded, the corresponding sequences $(a_n)$ and $(b_n)$ may not be bounded.

We know from the above exact sequence
\[H^1_{b, \mathrm{dR}}(\IR) \cong \ell^\infty(\IZ) / \image \Delta.\]
Let $Z \subset \ell^\infty(\IN)$ be the subspace consisting of all bounded sequences such that their partial sums form again a bounded sequence, i.e., $(z_n) \in Z :\Leftrightarrow \left(\sum_{i = 1}^n z_i\right) \in \ell^\infty(\IN)$. Now we have an injective inclusion $\ell^\infty(\IN) / Z \hookrightarrow \ell^\infty(\IZ) / \image \Delta$ given by $[(x_n)] \mapsto [(c_n)]$ with $c_{2k} := x_k$ for $k \ge 1$ and all other entries of $(c_n)$ are zero. Since the algebraic vector space dimension of $\ell^\infty(\IN) / Z$ is uncountably infinite (see, e.g., \cite{MathSE_codim_uncountable}) we conclude that $H^1_{b, \mathrm{dR}}(\IR)$ is a vector space of uncountable algebraic dimension.

Now to the reduced bounded de Rham cohomology. Note that in the above part of the Mayer--Vietoris sequence the space $H^1_{b, \mathrm{dR}}(\IR)$ is the only one for which the reduced and unreduced versions differ. Let us define a sequence $(c_n) \in \ell^\infty(\IZ)$ as $c_{2k} := \tfrac{1}{k}$ for $k \ge 1$ and all other entries are $0$. Then its image in $H^1_{b, \mathrm{dR}}(\IR)$ is some non-zero multiple of the form $[\tfrac{1}{x}|_{> 1} dx]$. Now if we had a Mayer--Vietoris sequence for the reduced bounded de Rham cohomology, then that would mean that this sequence $(c_n)$ is mapped to zero, i.e., it would have a lift to $\ell^\infty(\IZ) \oplus \ell^\infty(\IZ)$. But we know that this is not the case since the corresponding sequences $(a_n)$ and $(b_n)$ with $\Delta((a_n), (b_n)) = (c_n)$ are not bounded.
\end{example}

\section{Chern character on uniform \texorpdfstring{$K$}{K}-theory}

Let $(E, h^E, \nabla^E)$ and $(F, h^F, \nabla^F)$ be two complex vector bundles over the manifold $M$ equipped with metric connections ($M$ is not necessarily assumed to be of bounded geometry). We will call a homomorphism $f\colon E \to F$ \emph{bounded}, if $\|f(v)\| \le C \cdot \|v\|$ for all $v \in E$ and a fixed constant $C > 0$ (so a map $f\colon M \to N$ is of bounded dilatation, if and only if the map $f_\ast \colon TM \to TN$ is bounded). Analogously to the manifold case, we say that $f$ is a \emph{smooth quasi-isomorphism}, if it is a bundle isomorphism and both $f$ and $f^{-1}$ are bounded. Finally, we call $f$ a \emph{smooth strict quasi-isomorphism}, if it is a smooth quasi-isomorphism and additionally the form $\nabla^E - f^\ast \nabla^F$ is bounded.

Now we can state the bounded Chern--Weil theorem:

\begin{thm}[Bounded Chern--Weil theorem, {\cite[Theorem 3.8]{roe_index_1}}]\label{thm:bounded_chern_weil}
Let $E$ be a complex vector bundle of bounded geometry over $M$ ($M$ itself not required to have bounded geometry) and denote the curvature tensor of $E$ by $R^E$.

If $f$ is a power series with real coefficients, then
\[ \left[ \trace f \left( \tfrac{i}{2 \pi}R^E \right) \right] \in H^\ast_{b, \mathrm{dR}}(M)\]
and this class does only depend on the smooth strict quasi-isomorphism class of $E$.
\end{thm}

Now if the manifold $M$ has bounded geometry and the two bundles $E$ and $F$ are of bounded geometry and $C^\infty_b$-isomorphic, then they are smoothly strictly quasi-isomorphic (Lemma \ref{lem:C_b_infty_Iso_equivalent}). So using the interpretation of $K^0_u(M)$ from Theorem \ref{thm:interpretation_K0u}, we get a bounded Chern character homomorphism $\ch \colon K^0_u(M) \to H_{b, \mathrm{dR}}^{\mathrm{ev}}(M)$ by using the function $f := \exp$.

Now we also want a map on $K^1_u(M) \to H^{\mathrm{odd}}_{b, \mathrm{dR}}(M)$. Because of the interpretation of $K^1_u(M)$ from Theorem \ref{thm:interpretation_K1u}, we first get a map to $H_{b, \mathrm{dR}}^{\mathrm{ev}}(S^1 \times M)$. But because the vector bundles from $K^1_u(M)$ are trivial on a neighbourhood of $\{1\} \times M$, we actually have a map to $SH^{\mathrm{ev}}_{b, \mathrm{dR}}(M)$ (recall \eqref{eq:defn_suspension_de_Rham} for the definition),
and its image does furthermore not contain any forms of degree $0$. Now all that remains is to invoke the suspension isomorphism from Lemma \ref{lem:bounded_deRham_suspension_equal}.

Putting everything together, we have now proved the existence of a Chern character map on $K^\ast_u(M)$. That it is a ring homomorphism is proved analogously as in the case of a compact manifold.

\begin{prop}[Existence of the Chern character]
Let $M$ be a manifold of bounded geometry.

Then we have a ring homomorphism $\ch\colon K_u^\ast(M) \to H^\ast_{b, \mathrm{dR}}(M)$ with
\[\ch(K^0_u(M)) \subset H_{b, \mathrm{dR}}^{\mathrm{ev}}(M) \text{ and }\ch(K^1_u(M)) \subset H_{b, \mathrm{dR}}^{\mathrm{odd}}(M).\]
\end{prop}

In \cite{atiyah_hirzebruch} Atiyah and Hirzebruch constructed for a finite CW complex $X$ (i.e., not only for manifolds) a ring homomorphism $K^\ast(X) \to H^\ast(X; \IQ)$ which induces an isomorphism $K^\ast(X) \otimes \IQ \cong H^\ast(X; \IQ)$. This was later generalized to the statement $K^\ast(X) \otimes \IQ \cong \check{H}^\ast(X; \IQ)$ for any compact Hausdorff space $X$, where $\check{H}^\ast(X; \IQ)$ denotes \v{C}ech cohomology (see, e.g., \cite{karoubi_cartan_16}\footnote{Though there was only $K^0(X) \otimes \IQ \cong \check{H}^{\mathrm{ev}}(X; \IQ)$ shown, Th\'{e}or\`{e}me 2 is also applicable to $K^\ast(X) \otimes \IQ$ and $\check{H}^\ast(X; \IQ)$.}). Note that for CW complexes \v{C}ech and singular cohomology coincide.

Now we will state an analogous result for uniform $K$-theory:

\begin{chernthm}\label{thm:chern_character_iso}
Let $M$ be a manifold of bounded geometry. Then the Chern character induces a linear isomorphism
\[K_u^\ast(M)\barotimes\IR \cong H^\ast_{b, \mathrm{dR}}(M).\]
\end{chernthm}

The proof of this theorem will be given in the next section and the tensor product $K_u^\ast(M)\barotimes\IR$ will be discussed now.

\subsection*{Completed topological tensor product with \texorpdfstring{$\IR$}{R}}

We will need the notion of the \emph{free (abelian) topological group}: if $X$ is any completely regular\footnote{That is to say, every closed set $K$ can be separated with a continuous function from every point $x \notin K$. Note that this does not necessarily imply that $X$ is Hausdorff.} topological space, then the free topological group $F(X)$ on $X$ is a topological group such that we have
\begin{itemize}
\item a topological embedding $X \hookrightarrow F(X)$ of $X$ as a closed subset, so that $X$ generates $F(X)$ algebraically as a free group (i.e., the algebraic group underlying the free topological group on $X$ is the free group on $X$), and we have
\item the following universal property: for every continuous map $\phi\colon X \to G$, where $G$ is any topological group, we have a unique extension $\Phi\colon F(X) \to G$ of $\phi$ to a continuous group homomorphism on $F(X)$:
\[\xymatrix{X \ar@{^{(}->}[r] \ar[d]_{\phi} & F(X) \ar@{-->}[dl]^{\exists ! \Phi}\\ G}\]
\end{itemize}
The free abelian topological group $A(X)$ has the corresponding analogous properties. Furthermore, the commutator subgroup $[F(X), F(X)]$ of $F(X)$ is closed and the quotient $F(X) / [F(X), F(X)]$ is both algebraically and topologically $A(X)$.

As an easy example consider $X$ equipped with the discrete topology. Then $F(X)$ and $A(X)$ also have the discrete topology.

It seems that free (abelian) topological groups were apparently introduced by Markov in \cite{markoff_original}. But unfortunately, the author could not obtain any (neither russian nor english) copy of this article. A complete proof of the existence of such groups was given by Markov in \cite{markoff}. Since his proof was long and complicated, several other authors gave other proofs, e.g., Nakayama in \cite{nakayama}, Kakutani in \cite{kakutani} and Graev in \cite{graev}.

Now let us construct for any abelian topological group $G$ the complete topological vector space $G \barotimes \IR$. We form the topological tensor product $G \otimes \IR$ of abelian topological groups in the usual way: we start with the free abelian topological group $A(G \times \IR)$ over the topological space $G \times \IR$ equipped with the product topology\footnote{Note that every topological group is automatically completely regular and therefore the product $G \times \IR$ is also completely regular.} and then take the quotient $A(G \times \IR) / \mathcal{N}$ of it,\footnote{Since $A(X)$ is both algebraically and topologically the quotient of $F(X)$ by its commutator subgroup, we could also have started with $F(G \times \IR)$ and additionally put the commutator relations into $\mathcal{N}$.} where $\mathcal{N}$ is the closure of the normal subgroup generated by the usual relations for the tensor product.\footnote{That is to say, $\mathcal{N}$ contains $(g_1 + g_2) \times r - g_1 \times r - g_2 \times r$, $g \times (r_1 + r_2) - g \times r_1 - g \times r_2$ and $zg \times r - z(g \times r)$, $g \times zr - z(g \times r)$, where $g, g_1, g_2 \in G$, $r, r_1, r_2 \in \IR$ and $z \in \IZ$.} Now we may put on $G \otimes \IR$ the structure of a topological vector space by defining the scalar multiplication to be $\lambda (g \otimes r) := g \otimes \lambda r$.

What we now got is a topological vector space $G \otimes \IR$ together with a continuous map $G \times \IR \to G \otimes \IR$ with the following universal property: for every continuous map $\phi\colon G \times \IR \to V$ into any topological vector space $V$ and such that $\phi$ is bilinear\footnote{That is to say, $\phi(\largecdot, r)$ is a group homomorphism for all $r \in \IR$ and $\phi(g, \largecdot)$ is a linear map for all $g \in G$. Note that we then also have $\phi(zg,r) = z\phi(g,r) = \phi(g,zr)$ for all $z \in \IZ$, $g \in G$ and $r \in \IR$.\label{footnote:univ_prop_TVS}}, there exists a unique, continuous linear map $\Phi\colon G \otimes \IR \to V$ such that the following diagram commutes:
\[\xymatrix{G \times \IR \ar[r] \ar[d]_{\phi} & G \otimes \IR \ar@{-->}[dl]^{\exists ! \Phi} \\ V}\]

Since every topological vector space may be completed we do this with $G \otimes \IR$ to finally arrive at $G \barotimes \IR$. Since every continuous linear map of topological vector spaces is automatically uniformly continuous, i.e., may be extended to the completion of the topological vector space, $G \barotimes \IR$ enjoys the following universal property which we will raise to a definition:

\begin{defn}[Completed topological tensor product with $\IR$]
Let $G$ be an abelian topological group. Then $G \barotimes \IR$ is a complete topological vector space over $\IR$ together with a continuous map $G \times \IR \to G \barotimes \IR$ that enjoy the following universal property: for every continuous map $\phi\colon G \times \IR \to V$ into any complete topological vector space $V$ and such that $\phi$ is bilinear\footnote{See the above Footnote \ref{footnote:univ_prop_TVS}.}, there exists a unique, continuous linear map $\Phi\colon G \barotimes \IR \to V$ such that the following diagram commutes:
\[\xymatrix{G \times \IR \ar[r] \ar[d]_{\phi} & G \barotimes \IR \ar@{-->}[dl]^{\exists ! \Phi} \\ V}\]
\end{defn}

We equip the group $K^\ast_u(M)$ with the discrete topology and then form the complete topological vector space $K^\ast_u(M) \barotimes \IR$. The map
\[K_u^\ast(M) \times \IR \to H^\ast_{b, \mathrm{dR}}(M), \ ([E], r) \mapsto r \cdot [\ch(E)]\]
fulfills the requirements of the universal property of $K_u^\ast(M) \barotimes \IR$ (for $V:= H^\ast_{b, \mathrm{dR}}(M)$) and therefore induces a map on it into $H^\ast_{b, \mathrm{dR}}(M)$, which is the isomorphism from the above Chern Character Isomorphism Theorem \ref{thm:chern_character_iso}. But note that $H^\ast_{b, \mathrm{dR}}(M)$ is usually non-Hausdorff, i.e., it seems that the inverse of the Chern Character isomorphism may not be continuous.

We will give now two examples for the computation of $G \barotimes \IR$. The first one is easy and just a warm-up for the second which we will need in the proof of the Chern Character Isomorphism Theorem in the next section. Both examples are proved by checking the universal property.

\begin{examples}\label{ex:completed_tensor_prod}
The first one is $\IZ \barotimes \IR \cong \IR$.

For the second example consider the group $\ell^\infty_\IZ$ consisting of bounded, integer-valued sequences. Then $\ell^\infty_\IZ \barotimes \IR \cong \ell^\infty$.
\end{examples}

For the proof of the Chern Character Isomorphism Theorem we will also need that the functor $G \mapsto G \barotimes \IR$ is exact, i.e., maps exact sequences again to exact sequences. But we have to be careful here: though taking the tensor product with $\IR$ is exact, passing to completions is usually not---at least if the exact sequence we started with was only algebraically exact. Let us explain this a bit more thoroughly: if we have a sequence of topological vector spaces
\[\ldots \longrightarrow V_i \stackrel{\varphi_i}\longrightarrow V_{i+1} \stackrel{\varphi_{i+1}}\longrightarrow V_{i+2} \longrightarrow \ldots\]
which is exact in the algebraic sense (i.e., $\image \varphi_i = \kernel \varphi_{i+1}$), and if the maps $\varphi_i$ are continuous such that they extend to maps on the completions $\overline{V_i}$, we do not necessarily get that
\[\ldots \longrightarrow \overline{V_i} \stackrel{\overline{\varphi_i}}\longrightarrow \overline{V_{i+1}} \stackrel{\overline{\varphi_{i+1}}}\longrightarrow \overline{V_{i+2}} \longrightarrow \ldots\]
is again algebraically exact. The problem is that though we always have $\overline{\kernel \varphi_i} = \kernel \overline{\varphi_i}$, we generally only get $\overline{\image \varphi_i} \supset \image \overline{\varphi_i}$. To correct this problem we have to start with an exact sequence which is also topologically exact, i.e., we need that not only $\image \varphi_i = \kernel \varphi_{i+1}$, but we also need that $\varphi_i$ induces a \emph{topological} isomorphism $V_i / \kernel \varphi_i \cong \image \varphi_i$.

To prove that in this case we get $\overline{\image \varphi_i} = \image \overline{\varphi_i}$ we consider the inverse map
\[\psi_i := \varphi_i^{-1}\colon \image \varphi_i \to V_i / \kernel \varphi_i.\]
Since $\psi_i$ is continuous (this is the point which breaks down without the additional assumption that $\varphi_i$ induces a topological isomorphism $V_i / \kernel \varphi_i \cong \image \varphi_i$), we may extend it to a map
\[\overline{\psi_i} \colon \overline{\image \varphi_i} \to \overline{V_i / \kernel \varphi_i} = \overline{V_i} / \overline{\kernel \varphi_i},\]
which obviously is the inverse to $\overline{\varphi_i} \colon \overline{V_i} / \overline{\kernel \varphi_i} \to \overline{\image \varphi_i}$ showing the desired equality $\overline{\image \varphi_i} = \image \overline{\varphi_i}$.

Coming back to our functor $G \mapsto G \barotimes \IR$, we may now prove the following lemma:

\begin{lem}\label{lem:functor_exact}
Let
\[\ldots \longrightarrow G_i \stackrel{\varphi_i}\longrightarrow G_{i+1} \stackrel{\varphi_{i+1}}\longrightarrow G_{i+2} \longrightarrow \ldots\]
be an exact sequence of topological groups and continuous maps, which is in addition topologically exact, i.e., for all $i \in \IZ$ the from $\varphi_i$ induced map $G_i / \kernel \varphi_i \to \image \varphi_i$ is an isomorphism of topological groups.

Then
\[\ldots \longrightarrow G_i \barotimes \IR \longrightarrow G_{i+1} \barotimes \IR \longrightarrow G_{i+2} \barotimes \IR \longrightarrow \ldots\]
with the induced maps is an exact sequence of complete topological vector spaces, which is also topologically exact.
\end{lem}

\begin{proof}
We first tensor with $\IR$ (without the completion afterwards). This is known to be an exact functor and our sequence also stays topologically exact. To see this last claim, we need the following fact about tensor products: if $\varphi\colon M \to M^\prime$ and $\psi\colon N \to N^\prime$ are surjective, then the kernel of $\varphi \otimes \psi \colon M \otimes M^\prime \to N \otimes N^\prime$ is the submodule given by
\[\kernel (\varphi \otimes \psi) = (\iota_M \otimes 1) \big( (\kernel \varphi) \otimes N \big) + (1 \otimes \iota_N) \big( M \otimes (\kernel \psi) \big),\]
where $\iota_M \colon \kernel \varphi \to M$ and $\iota_N \colon \kernel \psi \to N$ are the inclusion maps. We will suppress the inclusion maps from now on to shorten the notation.

We apply this with $\varphi \colon M \to M^\prime$ being the quotient map $G_i \to G_i / \kernel \varphi_i$ and $\psi \colon N \to N^\prime$ being the identity $\id \colon \IR \to \IR$ to get
\[\kernel (\varphi_i \otimes \id) = (\kernel \varphi_i) \otimes \IR.\]
Since we have $(\image \varphi_i) \otimes \IR = \image (\varphi_i \otimes \id)$, we get that $\varphi \otimes \id \colon G_i \otimes \IR \to G_i \otimes \IR$ induces an algebraic isomorphism $(G_i / \kernel \varphi_i) \otimes \IR \to \image \varphi_i \otimes \IR$. But this has now an inverse map given by tensoring the inverse of $G_i / \kernel \varphi_i \to \image \varphi_i$ with $\id \colon \IR \to \IR$. So the isomorphism $(G_i / \kernel \varphi_i) \otimes \IR \cong \image \varphi_i \otimes \IR$ is also topological.

Now we apply the discussion before the lemma to show that the completion of this new sequence is still exact and also topologically exact.
\end{proof}

\section{Proof of the isomorphism theorem}\label{sec:proof_chern_iso}

We will need the following Theorem \ref{thm:triangulation_bounded_geometry} about manifolds of bounded geometry. To state it, we have to recall some notions:

\begin{defn}[Bounded geometry simplicial complexes]\label{defn:simplicial_complex_bounded_geometry}
A simplicial complex has \emph{bounded geometry} if there is a uniform bound on the number of simplices in the link of each vertex.

A subdivision of a simplicial complex of bounded geometry is called a \emph{uniform subdivision} if
\begin{itemize}
\item each simplex is subdivided a uniformly bounded number of times on its $n$-skeleton, where the $n$-skeleton is the union of the $n$-dimensional sub-simplices of the simplex, and
\item the distortion $\operatorname{length}(e) + \operatorname{length}(e)^{-1}$ of each edge $e$ of the subdivided complex is uniformly bounded in the metric given by barycentric coordinates of the original complex.
\end{itemize}
\end{defn}

\begin{defn}[Continuous quasi-isometries]
Two metric spaces $X$ and $Y$ are said to be \emph{quasi-isometric} if there is a homeomorphism $f\colon X \to Y$ with
\[\tfrac{1}{C} d_X(x,x^\prime) \le d_Y(f(x), f(x^\prime)) \le C d_X(x,x^\prime)\]
for all $x,x^\prime \in X$ and some constant $C > 0$.
\end{defn}

\begin{thm}[{\cite[Theorem 1.14]{attie_classification}}]\label{thm:triangulation_bounded_geometry}
Let $M$ be a manifold of bounded geometry. Then $M$ admits a triangulation as a simplicial complex of bounded geometry whose metric given by barycentric coordinates is quasi-isometric to the metric on $M$ induced by the Riemannian structure. This triangulation is unique up to uniform subdivision.

Conversely, if $M$ is a simplicial complex of bounded geometry which is a triangulation of a smooth manifold, then this smooth manifold admits a metric of bounded geometry with respect to which it is quasi-isometric to $M$.
\end{thm}

\begin{rem}
Attie uses in \cite{attie_classification} a weaker notion of bounded geometry as we do: additionally to a uniformly positive injectivity radius he only requires the sectional curvatures to be bounded in absolute value (i.e., the curvature tensor is bounded in norm), but he assumes nothing about the derivatives (see \cite[Definition 1.4]{attie_classification}). But going into his proof of \cite[Theorem 1.14]{attie_classification}, we see that the Riemannian metric constructed for the second statement of the theorem is actually of bounded geometry in our strong sense (i.e., also with bounds on the derivatives of the curvature tensor).

As a corollary we get that for any manifold of bounded geometry in Attie's weak sense there is another Riemannian metric of bounded geometry in our strong sense that is quasi-isometric the original one (in fact, this quasi-isometry is just the identity map of the manifold, as can be seen from the proof).
\end{rem}

We need the above theorem to prove the following technical lemma which will play a crucial role in our proof of the Chern Character Isomorphism Theorem.

\begin{lem}\label{lem:suitable_coloring_cover_M}
Let $M$ be a manifold of bounded geometry.

Then there is an $\varepsilon > 0$ and a countable collection of uniformly discretely distributed points $\{x_i\} \subset M$ such that $\{B_{\varepsilon}(x_i)\}$ is a uniformly locally finite cover of $M$.

Furthermore, it is possible to partition $\IN$ into a finite amount of subsets $I_1, \ldots, I_N$ such that for each $1 \le j \le N$ the subset $U_j := \bigcup_{i \in I_j} B_{\varepsilon}(x_i)$ is a disjoint union of balls that are a uniform distance apart from each other, and such that for each $1 \le K \le N$ the connected components of $U_K := U_1 \cup \ldots \cup U_k$ are also a uniform distance apart from each other (see Figure \ref{fig:not_allowed_cover} on the next page).
\end{lem}

\begin{figure}[htbp]
\centering
\includegraphics[scale=0.5]{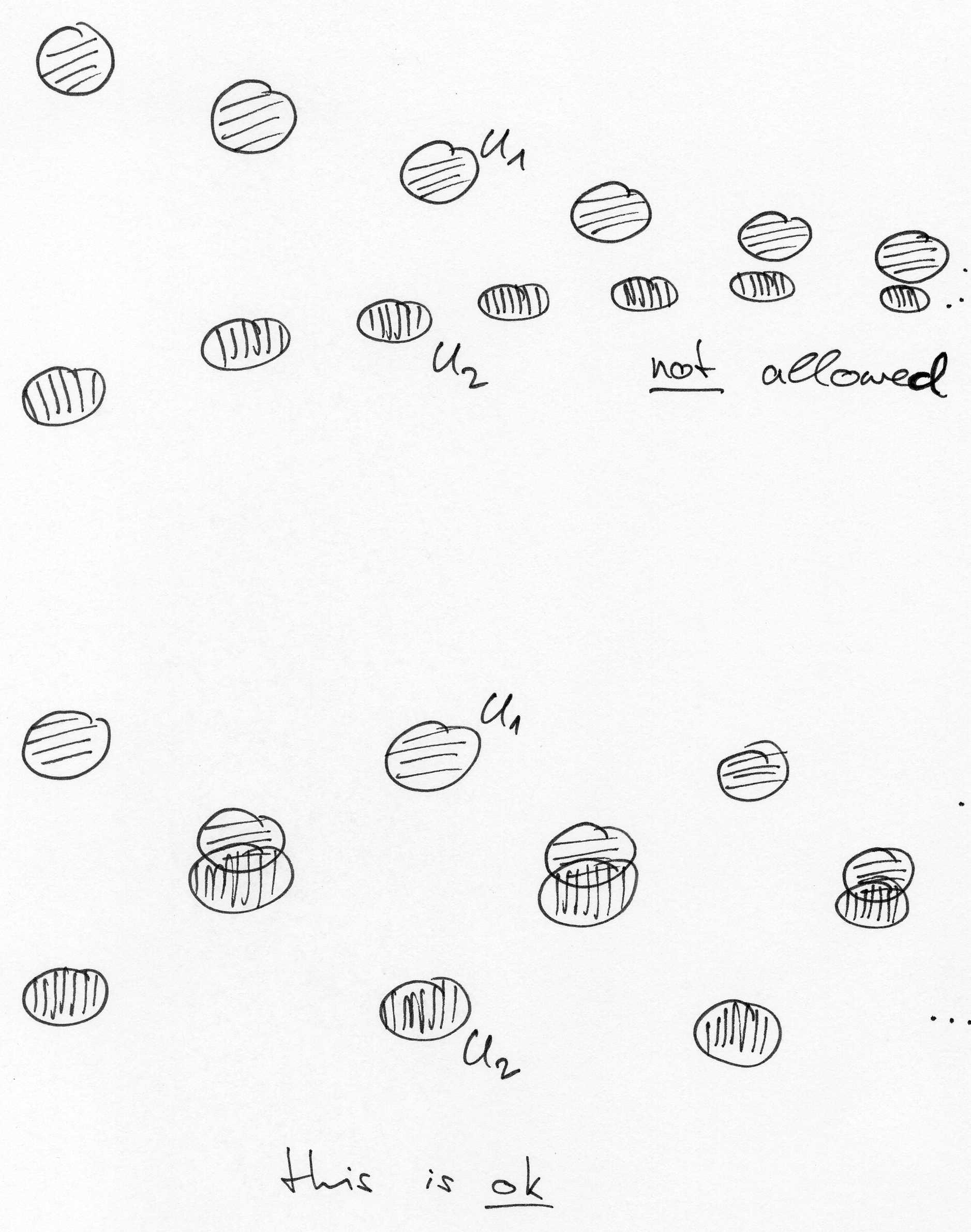}
\caption{Illustration for Lemma \ref{lem:suitable_coloring_cover_M}.}
\label{fig:not_allowed_cover}
\end{figure}

\begin{proof}
We triangulate $M$ via the above Theorem \ref{thm:triangulation_bounded_geometry}. Then we may take the vertices of this triangulation as our collection of points $\{x_i\}$ and set $\varepsilon$ to $2/3$ of the length of an edge multiplied with the constant $C$ which we get since the metric derived from barycentric coordinates is quasi-isometric to the metric derived from the Riemannian structure.

Two balls $B_\varepsilon(x_i)$ and $B_\varepsilon(x_j)$ for $x_i \not= x_j$ intersect if and only if $x_i$ and $x_j$ are adjacent vertices, and in the case that they are not adjacent, these balls are a uniform distance apart from each other. Hence it is possible to find a coloring of all these balls $\{B_\varepsilon(x_i)\}$ with finitely many colors having the claimed property.\footnote{see Footnote \ref{footnote:coloring}}
\end{proof}

To prove the Chern Character Isomorphism Theorem we will need Mayer--Vietoris sequences for uniform $K$-theory and bounded de Rham cohomology. But recall from Remark \ref{rem:mayer_vietoris_de_rham} that for the latter the existence of a Mayer--Vietoris sequence is very unlikely. We will work around this problem by introducing for open subsets $O \subset M$ slightly differently defined groups $K^\ast_{u, \mathrm{MV}}(O)$ and $H^\ast_{b, \mathrm{MV}}(O)$ (which coincide in the case $O = M$ with $K^\ast_u(M)$ and $H^\ast_{b, \mathrm{dR}}(M)$).

Let us start with bounded de Rham cohomology:

\begin{defn}
Let $O \subset M$ be an open subset, not necessarily connected. We define $\Omega_{b, \mathrm{MV}}^p(O)$ to consist of all bounded $p$-forms $\alpha \in \Omega_b^p(O)$ such that there exists a $\delta > 0$ (which will depend on $\alpha$) such that $\alpha$ has an extension to a bounded $p$-form $\widetilde{\alpha} \in \Omega_b^p(B_\delta(O))$ on a $\delta$-neighbourhood of $O$.
\end{defn}

Let the open subsets $U_j$ and $U_K$ for $1 \le j,K \le N$, of $M$ be as in Lemma \ref{lem:suitable_coloring_cover_M} and consider the sequence
\begin{equation*}
0 \to \Omega_{b, \mathrm{MV}}^p(U_K \cup U_{k+1}) \to \Omega_{b, \mathrm{MV}}^p(U_K) \oplus \Omega_{b, \mathrm{MV}}^p(U_{k+1}) \to \Omega_{b, \mathrm{MV}}^p(U_K \cap U_{k+1}) \to 0
\end{equation*}
(an analogue of the sequence usually used for the proof of the existence of the Mayer--Vietoris sequence for usual de Rham cohomology).

The hardest part in showing that the sequence is exact, is to show the surjectivity of $\Omega_{b, \mathrm{MV}}^p(U_K) \oplus \Omega_{b, \mathrm{MV}}^p(U_{k+1}) \to \Omega_{b, \mathrm{MV}}^p(U_K \cap U_{k+1})$. So given a form $\alpha \in \Omega_{b, \mathrm{MV}}^p(U_K \cap U_{k+1})$, we can always extend it to a form $\widetilde{\alpha} \in \Omega_{b, \mathrm{MV}}^p(U_K)$ in the following way: let $\delta$ be such that $\alpha$ has an extension to $B_\delta(U_K \cap U_{k+1})$, but not bigger than half the distance which are the balls of $U_{k+1}$ are apart from each other. Let $f \in C_b^\infty[0, \delta]$ be such that $f|_{[0, \delta / 3]} \equiv 1$ and $f|_{[\delta 2 / 3, \delta]} \equiv 0$ and define $F \in C_b^\infty(M)$ to be $f$ on every radial geodesic emenating orthogonally from the boundary of every ball of $U_{k+1}$, to be $1$ on the inside of every such ball and $0$ otherwise. Now we multiply $\alpha$ with $F$, which gives us the desired extension $\widetilde{\alpha}$ of $\alpha$ to $U_K$ (see Figure \ref{fig:extension_forms_MV}). So we get for bounded de Rham cohomology our claimed Mayer--Vietoris sequence, which we state in the next lemma.

\begin{figure}[htbp]
\centering
\includegraphics[scale=0.65]{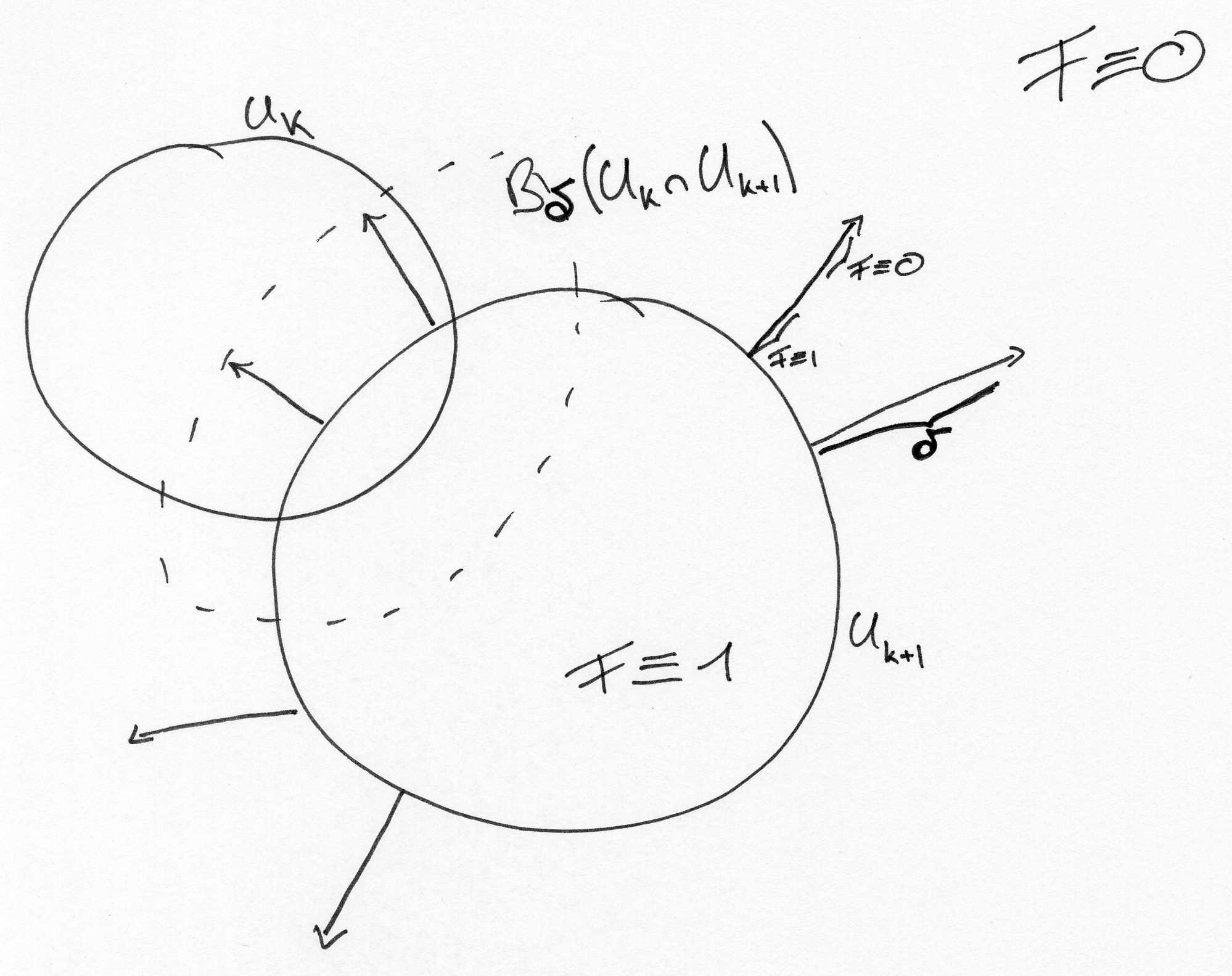}
\caption{Extension of $\alpha \in \Omega_{b, \mathrm{MV}}^p(U_K \cap U_{k+1})$ to $\widetilde{\alpha} \in \Omega_{b, \mathrm{MV}}^p(U_K)$.}
\label{fig:extension_forms_MV}
\end{figure}

\begin{lem}\label{lem:MV_bounded_dR}
Let the subsets $U_j$ and $U_K$ for $1 \le j,K \le N$, of $M$ be defined as in Lemma \ref{lem:suitable_coloring_cover_M}. Then we have Mayer--Vietoris sequences
\begin{alignat*}{4}
& && && \!\!\!\!\!\!\!\! \ldots \to H^{p-1}_{b, \mathrm{MV}}(U_K \cap U_{k+1}) \to \\
& \to H^p_{b, \mathrm{MV}}(U_K \cup U_{k+1}) && \to H^p_{b, \mathrm{MV}}(U_K) \oplus H^p_{b, \mathrm{MV}}(U_{k+1}) && \to H^p_{b, \mathrm{MV}}(U_K \cap U_{k+1}) \to \\
& \to H^{p+1}_{b, \mathrm{MV}}(U_K \cup U_{k+1}) && \to \ldots
\end{alignat*}
\end{lem}

\begin{rem}
The sequence here coincides in the setting of Example \ref{ex:computation_de_rham_real_line} with the sequence there. This shows that there is no version of this Mayer--Vietoris sequence here for the reduced groups $\bar{H}^\ast_{b, \mathrm{MV}}$.
\end{rem}

Now let us get to the Mayer--Vietoris sequence for uniform $K$-theory. Let $O \subset M$ be an open subset, not necessarily connected. We denote by $(M, d)$ the metric space $M$ endowed with the metric induced from the Riemannian metric $g$ on $M$, and by $C_u(O, d)$ we denote the $C^\ast$-algebra of all bounded, uniformly continuous functions on $O$, where we regard $O$ as a metric space equipped with the subset metric induced from $d$ (i.e., we do not equip $O$ with the induced Riemannian metric and consider then the corresponding induced metric structure; see Figure \ref{fig:different_metrics}).

\begin{figure}[htbp]
\centering
\includegraphics[scale=0.6]{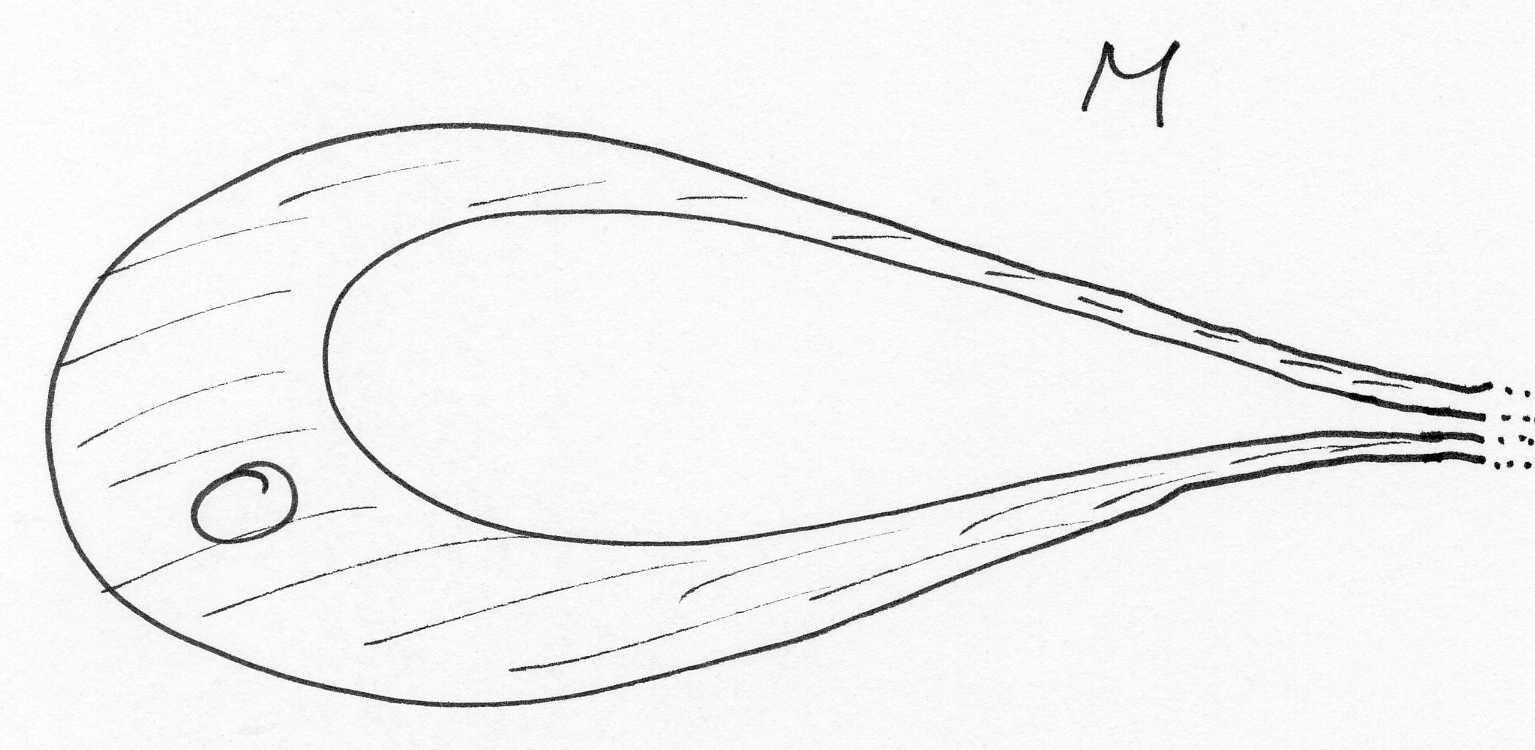}
\caption{The metric on $O$ induced from being a subset of the metric space $(M,d)$ may differ vastly from the induced Riemannian metric.}
\label{fig:different_metrics}
\end{figure}

\begin{defn}
Let $O \subset M$ be an open subset, not necessarily connected. We define $K^p_{u, \mathrm{MV}}(O) := K_{-p}(C_u(O,d))$.
\end{defn}

To deduce a Mayer--Vietoris sequence analogous to the one from the above lemma, we will need the following technical theorem about the existence of extensions of uniformly continuous functions:

\begin{lem}\label{lem:extension_samuel_compactification}
Let $O \subset M$ be open, not necessarily connected. Then every function $f \in C_u(O, d)$ has an extension to an $F \in C_u(M, d)$.
\end{lem}

\begin{proof}
For a metric space $X$ let $uX$ denote the Gelfand space of $C_u(X)$, i.e., this is a compactification of $X$ (the \emph{Samuel compactification}) with the following universal property: a bounded, continuous function $f$ on $X$ has an extension to a continuous function on $uX$ if and only if $f$ is uniformly continuous. We will use the following property of Samuel compactifications (see \cite[Theorem 2.9]{woods}): if $S \subset X \subset uX$, then the closure $\closure_{uX}(S)$ of $S$ in $uX$ is the Samuel compactification $uS$ of $S$.

So given $f \in C_u(O, d)$, we can extend it to a continuous function $\tilde{f} \in C(uO)$. Since $uO = \closure_{uM}(O)$, i.e., a closed subset of a compact Hausdorff space, we can extend $\tilde{f}$ by the Tietze extension theorem to a bounded, continuous function $\tilde{F}$ on $uM$. Its restriction $F := \tilde{F}|_M$ to $M$ is then a bounded, uniformly continuous function of $M$ extending $f$.
\end{proof}

\begin{lem}\label{lem:MV_uniform_K}
Let the subsets $U_j$, $U_K$ of $M$ for $1 \le j,K \le N$ be as in Lemma \ref{lem:suitable_coloring_cover_M}. Then we have Mayer--Vietoris sequences
\[\xymatrix{
K^0_{u, \mathrm{MV}}(U_K \cup U_{k+1}) \ar[r] & K^0_{u, \mathrm{MV}}(U_K) \oplus K^0_{u, \mathrm{MV}}(U_{k+1}) \ar[r] & K^0_{u, \mathrm{MV}}(U_K \cap U_{k+1}) \ar[d]\\
K^1_{u, \mathrm{MV}}(U_K \cap U_{k+1}) \ar[u] & K^1_{u, \mathrm{MV}}(U_K) \oplus K^1_{u, \mathrm{MV}}(U_{k+1}) \ar[l] & K^1_{u, \mathrm{MV}}(U_K \cup U_{k+1}) \ar[l]}\]
where the horizontal arrows are induced from the corresponding restriction maps.
\end{lem}

\begin{proof}
Recall the Mayer--Vietoris sequence for operator $K$-theory of $C^\ast$-algebras (see, e.g., \cite[Theorem 21.2.2]{blackadar}): given a commutative diagram of $C^\ast$-algebras
\[\xymatrix{P \ar[r]^{\sigma_1} \ar[d]_{\sigma_2} & A_1 \ar[d]^{\varphi_1}\\ A_2 \ar[r]^{\varphi_2} & B}\]
with $P = \{(a_1, a_2) \ | \ \varphi_1(a_1) = \varphi_2(a_2)\} \subset A_1 \oplus A_2$ and $\varphi_1$ and $\varphi_2$ surjective, then there is a long exact sequence (via Bott periodicity we get the $6$-term exact sequence)
\[\ldots \to K_n(P) \stackrel{({\sigma_1}_\ast,{\sigma_2}_\ast)}\longrightarrow K_n(A_1) \oplus K_n(A_2) \stackrel{{\varphi_2}_\ast - {\varphi_1}_\ast}\longrightarrow K_n(B) \to K_{n-1}(P) \to \ldots\]

We set $A_1 := C_u(U_K, d)$, $A_2 := C_u(U_{k+1}, d)$, $B := C_u(U_K \cap U_{k+1}, d)$ and $\varphi_1$, $\varphi_2$ the corresponding restriction maps. Due to the property of the sets $U_K$ as stated in the Lemma \ref{lem:suitable_coloring_cover_M} we get $P = C_u(U_K \cup U_{k+1}, d)$ and $\sigma_1$, $\sigma_2$ again just the restriction maps. To show that the maps $\varphi_1$ and $\varphi_2$ are surjective, we just have to use the above Lemma \ref{lem:extension_samuel_compactification}.
\end{proof}

The last ingredient that we will need for the proof of the Chern Character Isomorphism Theorem are the Chern character maps $\ch\colon K_{u, \mathrm{MV}}^\ast(O) \to H_{b, \mathrm{MV}}^\ast(O)$:

\begin{lem}\label{lem:MV_chern_character}
Let $O \subset M$ be open, not necessarily connected. Then there is a Chern character map $\ch\colon K_{u, \mathrm{MV}}^\ast(O) \to H_{b, \mathrm{MV}}^\ast(O)$ which coincides in the case $O = M$ with the usual one.
\end{lem}

\begin{proof}
We define the \Frechet space $C_{b, \mathrm{MV}}^\infty(O)$ as the set of all smooth functions $f$ on $O$ for which there is a $\delta > 0$ (depending on the function $f$) such that $f$ has an extension $F$ to $B_\delta(O)$, where $F$ and all its derivatives are bounded, i.e., $\|\nabla^i F\|_\infty < C_i$ for all $0 \le i$. Note that we have an inclusion $C_{b, \mathrm{MV}}^\infty(O) \subset C_u(O, d)$.

But we also know from Lemma \ref{lem:extension_samuel_compactification} that given a function $f \in C_u(O, d)$, there exists a bounded, uniformly continuous extension $F$ of it to all of $M$. We can approximate $F$ in sup-norm by functions from $C_b^\infty(M)$ (see Lemma \ref{lem:norm_completion_C_b_infty}) which gives us an approximation of $f$ by functions from $C_{b, \mathrm{MV}}^\infty(O)$, i.e., we have shown that $C_{b, \mathrm{MV}}^\infty(O)$ is dense in $C_u(O, d)$. Since $C_{b, \mathrm{MV}}^\infty(O)$ is a local $C^\ast$-algebra (Definition \ref{defn:local_Cstar_algebra}), the $K$-groups of it coincide with the ones of $C_u(O, d)$, i.e., $K^p_{u, \mathrm{MV}}(O) = K_{-p}(C_{b, \mathrm{MV}}^\infty(O))$.

Given an idempotent matrix $e \in \Idem_{N \times N}(C_{b, \mathrm{MV}}^\infty(O))$, we can define a vector bundle $E := \image e$ over $B_\delta(O)$ for some $\delta > 0$ and equip it with the metric and connection induced from the trivial bundle of which $E$ is a subbundle. Going through the proof of Statement 1 of Proposition \ref{prop:image_proj_matrix_complemented} we see that $E$ is $C_b^\infty$-complemented. Here we reinterpret the notion of ``$C^\infty$-bounded'' from Definition \ref{defn:C_infty_bounded} in the following way: for the synchronous framings we take the radii of the normal coordinate balls at most $\delta$ and demand the boundedness of the matrix entries of $\varphi$ and all their derivatives only with respect to synchronous framings which base points are located inside $O$. In fact, with this reinterpretation all results of Section \ref{sec:interpretation_uniform_k_theory} do hold. So we get a Chern character map $\ch\colon K_{u, \mathrm{MV}}^\ast(O) \to H_{b, \mathrm{MV}}^\ast(O)$.
\end{proof}

\begin{proof}[Proof of the Chern Character Isomorphism Theorem]
We invoke Lemma \ref{lem:suitable_coloring_cover_M} to get the subsets $U_j$ and $U_K = U_1 \cup \ldots \cup U_k$ for $1 \le j,K \le N$. We will do an induction over $k$.

We start with $k = 1$ and invoke the Lemmas \ref{lem:MV_bounded_dR} and \ref{lem:MV_uniform_K} to get the Mayer--Vietoris sequences and Lemma \ref{lem:MV_chern_character} to get the corresponding Chern character maps. Then we have the following commutative diagram:\footnote{Note that it may happen that $U_1 \cap U_2 = \emptyset$. In this case both the uniform $K$-theory and the bounded de Rham cohomology of the empty set is defined to be the trivial group. This is compatible (in the sense that the Mayer--Vietoris sequence gives the same result) with the computation of these groups for the spaces $U_1$, $U_2$ and $U_1 \cup U_2$, where the groups of the latter space turn out in this case to be exactly the direct sum of the groups of the former two spaces.}\vspace{\baselineskip}
\[\xymatrix{
K_{u, \mathrm{MV}}^\ast(U_{1} \cap U_2) \ar[r] \ar[d] & K_{u, \mathrm{MV}}^\ast(U_{1} \cup U_2) \ar[r] \ar[d] & K_{u, \mathrm{MV}}^\ast(U_1) \oplus K_{u, \mathrm{MV}}^\ast(U_2) \ar[d] \ar@/_1.5pc/[ll] \\
H_{b, \mathrm{MV}}^{\ast}(U_{1} \cap U_2) \ar[r] & H_{b, \mathrm{MV}}^{\ast}(U_{1} \cup U_2) \ar[r] & H_{b, \mathrm{MV}}^{\ast}(U_1) \oplus H_{b, \mathrm{MV}}^{\ast}(U_2) \ar@/^1.5pc/[ll]
}\vspace{\baselineskip}\]

The square containing the horizontal maps from the intersection to the union (i.e., the one with the boundary maps) is the only one where commutativity may not be immediately clear (since the other horizontal maps are induced by the corresponding restriction maps). So we have to give an argument why the above diagram commutes: in the proof of the above Lemma \ref{lem:MV_uniform_K} we have discussed the Mayer--Vietoris sequence for uniform $K$-theory by using the corresponding sequence for operator $K$-theory. A reference for this is, e.g., \cite[Theorem 21.2.2]{blackadar}, where we find the following description of the boundary map: it is given by composing the inverse of the suspension isomorphism with an isomorphism called there $\psi_\ast$ and finally composing with the map induced from a restriction map. The Chern character map clearly commutes with restriction maps. Furthermore, the map $\psi_\ast$ is an isomorphism because of the homotopy invariance of uniform $K$-theory and the Chern character commutes with homotopies. So it remains to show that it commutes with the suspension isomorphism. For bounded de Rham cohomology we have discussed the suspension isomorphism in Lemma \ref{lem:bounded_deRham_suspension_equal} and the smooth suspension of \Frechet algebras was discussed during the proof of Theorem \ref{thm:interpretation_K1u}. Interpreting the suspension map, defined by us only on the level of algebras, on the level of vector bundles, we get that the Chern character map commutes with taking suspensions. To prove that the Chern character commutes with the boundary maps in the corresponding Mayer--Vietoris sequences we finally have to show one last thing, namely that it commutes with the Bott periodicity map (since we use it to reduce the uniform $K$-theory sequence to a $6$-term sequence). But this map is given by multiplication with a certain element, namely the Bott generator. Since the Chern character is multiplicative and maps the Bott generator to the identity, we get that it commutes with the Bott periodicity map.

We equip the $K$-groups with the discrete topology and take the completed tensor product with $\IR$. Due to our choice of topology on the groups we conclude that the top horizontal sequence is topologically exact (we already know that it is algebraically exact). We will now show that the arising left and right vertical arrows are isomorphisms, since then we can conclude with the five lemma that the middle vertical arrow will also be an isomorphism (here we need Lemma \ref{lem:functor_exact} and the fact that we endow the uniform $K$-theory groups with the discrete topology to deduce that the top row stays exact). Now $U_1$, $U_2$ and $U_1 \cap U_2$ are each a disjoint union of geodesically convex sets which are a uniform distance from each other apart and have a uniform bound on their diameters. So all three sets are boundedly homotopy equivalent\footnote{\label{footnote:boundedly_homotopic}Let $f, g\colon M \to N$ be two maps of bounded dilatation. We say that they are \emph{boundedly homotopic}, if there is a homotopy $H\colon M \times [0,1] \to N$ from $f$ to $g$, which itself is of bounded dilatation. Recall that a map $h$ has \emph{bounded dilatation}, if $\|h_\ast V\| \le C \|V\|$ for all tangent vectors $V$. Bounded homotopy invariance of uniform $K$-theory follows from its definition $K_u^p(X) = K_{-p}(C_u(X))$ and the corresponding homotopy invariance of operator $K$-theory. Bounded homotopy invariance of bounded de Rham cohomology was shown in Corollary \ref{cor:homotopy_equivalence_de_rham}.} to a collection of uniformly discrete points.

So we have to show that the Chern character induces an isomorphism of vector spaces $K^\ast_u(X) \barotimes \IR \cong H^\ast_{b, \mathrm{dR}}(X)$, where $X$ is a countable set of uniformly discrete points. Since complex bundles over $S^1$ are trivial, we have $K^1_u(X) = 0$, and also clearly $H^{\mathrm{odd}}_{b, \mathrm{dR}}(X) = 0$, since $X$ is $0$-dimensional. Furthermore, a vector bundle of bounded geometry over $X$ is a Hermitian vector space over each point such that the dimensions are bounded from above. Since we can always find an isometric isomorphism between any two Hermitian vector spaces of the same dimension, $C_b^\infty$-isomorphism classes are completely determined by the dimensions of the vector spaces at each point. So we see that $K^0_u(X)$ is the group $\ell^\infty_\IZ(X)$ that consists of all bounded, integer-valued sequences indexed by $X$ (we already know this from Lemma \ref{lem:uniform_k_th_discrete_space}, but we need here the description of this isomorphism via vector bundles since we want to apply the Chern character). The Chern character of such a class in $K^0_u(X)$ is just the function assigning every point of $X$ the dimension of the corresponding vector space over this point. So the claim follows with the second example from Examples \ref{ex:completed_tensor_prod} since $H^{\mathrm{ev}}_{b, \mathrm{dR}}(X) \cong \ell^\infty(X)$. This completes the first step of the induction.

Now let $\ch\colon K_{u, \mathrm{MV}}^\ast(U_1 \cup \ldots \cup U_i) \barotimes \IR \to H_{b, \mathrm{MV}}^\ast(U_1 \cup \ldots \cup U_i)$ be an isomorphism for all $1 \le i \le k$ for some $k \ge 2$. We want to show that it is also an isomorphism for $i = k + 1$. We take a look at the diagram resulting from the Mayer--Vietoris sequences for the cover $\{U_K, U_{k+1}\}$ of $U_K \cup U_{k+1}$:
\vspace{\baselineskip}
\[\xymatrix{
K_{u, \mathrm{MV}}^\ast(U_K \cap U_{k+1}) \ar[r] \ar[d] & K_{u, \mathrm{MV}}^\ast(U_K \cup U_{k+1}) \ar[r] \ar[d] & K_{u, \mathrm{MV}}^\ast(U_K) \oplus K_{u, \mathrm{MV}}^\ast(U_{k+1}) \ar[d] \ar@/_1.5pc/[ll] \\
H_{b, \mathrm{MV}}^{\ast}(U_K \cap U_{k+1}) \ar[r] & H_{b, \mathrm{MV}}^{\ast}(U_K \cup U_{k+1}) \ar[r] & H_{b, \mathrm{MV}}^{\ast}(U_K) \oplus H_{b, \mathrm{MV}}^{\ast}(U_{k+1}) \ar@/^1.5pc/[ll]
}\vspace{\baselineskip}\]

Again, we form the completed tensor products with $\IR$ and then have to show that the arising left and right vertical arrows are isomorphisms, so that we can conclude by the five lemma that the middle one will be also an isomorphism. The resulting right vertical arrow is an isomorphism since for $U_K$ we know it from the induction hypothesis and for $U_{k+1}$ it is the same argument as in the first step of the induction, (i.e., $U_{k+1}$ is boundedly homotopy equivalent to a collection of uniformly discrete points). To see that the resulting left vertical arrow will be an isomorphism, we have to write $U_K \cap U_{k+1} = (U_1 \cap U_{k+1}) \cup \ldots \cup (U_k \cap U_{k+1})$. This is a union of $k$ geodesically convex open sets. So by a separate induction we get that the arising left vertical arrow is also an isomorphism. This completes the induction step and therefore the whole proof.
\end{proof}

\chapter{Index Theorem}\label{chapter:index_theorem}

In this chapter we will finally prove the generalization of Roe's Index Theorem to pseudodifferential operators as we have defined them. Our proof will rely heavily on \Poincare duality $K^\ast_u(M) \cong K^u_{m - \ast}(M)$, which we will prove in Section \ref{sec:poincare_duality}, where $m$ is the dimension of the \spinc manifold $M$. As for compact manifolds, the \Poincare duality map $K^\ast_u(M) \to K^u_{m - \ast}(M)$ is given by the cap product with the fundamental class $[M] \in K_m^u(M)$ of $M$ which is defined as the uniform $K$-homology class associated to the \spinc structure of $M$. We will define the cap product in Section \ref{sec:cap_product} and we will discuss \spinc structures for manifolds of bounded geometry in the next Section \ref{sec:spinc_manifolds}.

But first we will take \Poincare duality for granted and discuss the proof of the index theorem for pseudodifferential operators now.

Let us recall Roe's Index Theorem from \cite{roe_index_1}. So let $M$ be an oriented and $m$-dimensional manifold of bounded geometry and let $D$ be the Dirac operator of a graded Dirac bundle $S$ of bounded geometry over $M$ (see the next Section \ref{sec:spinc_manifolds} for a definition of Dirac bundles and their associated Dirac operators over manifolds of bounded geometry).

Using the asymptotic expansion of the integral kernel of the operator $e^{-t D^2}$, Roe defined the \emph{topological index class $I_t(D) \in H_{b, \mathrm{dR}}^m(M)$} of the operator $D$ analogously one does it in the proof of the local index theorem using the heat kernel method. We assume furthermore that $M$ is amenable. Then let $\theta$ be a fundamental class\footnote{see Definition \ref{defn:fundamental_class}} for $M$ corresponding to a choice of a F{\o}lner sequence of $M$ and to a functional $\tau \in (\ell^\infty)^\ast$ associated to a free ultrafilter on $\IN$.\footnote{That is, if we evaluate $\tau$ on a bounded sequence, we get the limit of some convergent subsequence.} The \emph{topological index of $D$} if now defined as $\theta(I_t(D)) \in \IR$. Note that it depends on $\theta$ contrary to the case of a compact manifold where the topological index is uniquely defined.

Let us recall Roe's definition of analytic indices. In \cite[Lemma 7.6]{roe_index_1} he shows that the operator $D$ is an abstractly elliptic operator between the $\mathcal{U}(S)$\footnote{$\mathcal{U}(S) = \bigcup_{k \in \IZ} \mathcal{U}_k(S)$ is a filtered algebra where $\mathcal{U}_k(S)$ are the quasilocal operators of order $k \in \IZ$.}-modules given by the eigen\-projections $(1+\epsilon)/2$ and $(1-\epsilon)/2$, where $\epsilon$ is the grading operator of $S$. So $D$ has an \emph{analytic index class $I_a(D) \in K_0^{\mathrm{alg}}(\mathcal{U}_{-\infty}(M))$}\footnote{Since Roe works with algebraic $K$-theory, he has no need to assume the algebra to be a $^\ast$-algebra. So he uses $\mathcal{U}_{-\infty}(M)$ whereas we use $\IU(M)$. Note furthermore that a priori $D$ has an analytic index class in $K_0^{\mathrm{alg}}(\mathcal{U}_{-\infty}(S))$ and that Roe showed that we have canonical maps from these $K$-groups to $K_0^{\mathrm{alg}}(\mathcal{U}_{-\infty}(M))$, analogously to our result from Proposition \ref{prop:analytic_index_map_quasiloc_smoothing}.}. The \emph{analytic index of $D$} is then defined as $\ind_\tau(I_a(D)) \in \IR$, where the analytic index map $\ind_\tau$ is defined analogously as the one of Proposition \ref{prop:analytic_index_map_quasiloc_smoothing}.

\begin{roesthm}[{\cite[Index Theorem 8.2]{roe_index_1}}]
Let $M$ be an amenable, oriented manifold of bounded geometry and $D$ a graded operator of Dirac type\footnote{see Definition \ref{defn:dirac_type_operator}}.

Then for all choices of F{\o}lner sequences for $M$ and all choices of functionals $\tau \in (\ell^\infty)^\ast$ associated to a free ultrafilter on $\IN$, we have
\[\ind_\tau(I_a(D)) = \theta(I_t(D)),\]
i.e., the analytic and the topological indices of $D$ coincide.
\end{roesthm}

We will now put the above theorem into our context. From Theorem \ref{thm:elliptic_symmetric_PDO_defines_uniform_Fredholm_module} we know that the operator $D$ defines a uniform $K$-homology class $[D] \in K_0^u(M)$ given by the uniform Fredholm module $(L^2(S), \rho, \chi(D))$, where $\rho \colon C_0(M) \to \IB(L^2(S))$ is the representation via multiplication operators and $\chi$ is a normalizing function. Comparing our computation of $\ind_\tau([D])$ from Section \ref{sec:index_maps_K_hom}, where $\ind_\tau \colon K_0^u(M) \to \IR$ is the index map defined there, with the formula from \cite[Lemma 4.1]{roe_index_1} combined with the arguments in the proof of \cite[Proposition 8.1]{roe_index_1}, we conclude
\begin{equation}\label{eq:equal_roe_our_ana_index}
\ind_\tau([D]) = \ind_\tau(I_a(D)).
\end{equation}

The uniform coarse assembly map $\mu_u\colon K_\ast^u(M) \to K_\ast(C_u^\ast(Y))$, where $Y \subset M$ is a uniformly discrete quasi-lattice in $M$, was constructed by \v{S}pakula in his article \cite[Section 9]{spakula_uniform_k_homology}. We know from Section \ref{sec:uniform_roe_algebra} that $\C(S)$ is a ``smooth version'' of $C_u^\ast(Y)$, which especially means $K_\ast(C_u^\ast(Y)) \cong K_\ast(\C(S))$, and we also know that the latter is isomorphic to $K_\ast(\C(M))$ (Corollary \ref{cor:natural_receptacle_iso}). So we get a \emph{uniform coarse index class $\mu_u([D]) \in K_0(\C(M))$ of $D$} and may conclude that it ``equals''\footnote{We have inclusions $\C(M) \subset \IU(M) \subset \mathcal{U}_{-\infty}(M)$, i.e., we may compare the classes using the induced maps on $K$-theory.} Roe's analytic index class $I_a(D) \in K_0^{\mathrm{alg}}(\mathcal{U}_{- \infty}(M))$ by comparing the construction of $I_a(D)$ from \cite[Section 4]{roe_index_1} with the one of $\mu_u$. From Section \ref{sec:index_maps_uniform_roe} we get
\[\ind_\tau([D]) = \ind_\tau(\mu_u([D]))\]
and combining this with the above Equation \eqref{eq:equal_roe_our_ana_index}, we get
\[\ind_\tau(I_a(D)) = \ind_\tau(\mu_u([D])).\]
Note that this was basically also shown by Roe in \cite[Proposition 8.1]{roe_index_1}. We summarize this discussion in the following lemma:

\begin{lem}
The analytic index $\ind_\tau(I_a(D))$ as defined by Roe coincides with the analytic indices $\ind_\tau([D])$ and $\ind_\tau(\mu_u([D]))$ defined in Sections \ref{sec:index_maps_K_hom}, \ref{sec:analytic_indices_quasiloc_smoothing} (resp. \ref{sec:index_maps_uniform_roe}).
\end{lem}

With this lemma we may now rephrase Roe's Index Theorem in a way that will be well suited for our extension of it to pseudodifferential operators. For the rephrasing we will also need the cap product pairing $\cap \colon K_u^p(M) \otimes K_q^u(M) \to K_{q-p}^u(M)$ from Section \ref{sec:cap_product} and its property $[E] \cap [D] = [D_E] \in K_0^u(M)$ for a vector bundle $E$ of bounded geometry and a graded operator $D$ of Dirac type, where $D_E$ denotes the twisted operator (see Definition \ref{defn:twisted_op}).

\begin{thm}[Roe's Index Theorem rephrased]\label{thm:roe_index_thm_rephrased}
Let $M$ be an amenable, oriented manifold of bounded geometry and let $D$ be a graded operator of Dirac type\footnote{see Definition \ref{defn:dirac_type_operator}}.

Then for all choices of F{\o}lner sequences for $M$ and all choices of functionals $\tau \in (\ell^\infty)^\ast$ associated to a free ultrafilter on $\IN$, the following diagram commutes:
\[\xymatrix{K^0_u(M) \ar[rr]^{\largecdot \cap [D]} \ar[drr]_{\theta(I_t(D_{\largecdot}))} & & K_0^u(M) \ar[d]^{\ind_\tau} \ar[rr]^{\mu_u} & & K_0(\C(M)) \ar[dll]^{\ind_\tau} \\ & & \IR}\]
\end{thm}

Note that for the statement of Roe's Index Theorem we just need the left hand side of the aobve diagram (which asserts that the topological and the analytic indices coincide). But we will need the right hand side of the diagram in order to conclude that the index \emph{classes} in $\bar{H}^m_{b, \mathrm{dR}}(M)$ coincide (the above theorem goes a step further and contains already the evaluations of these cohomology classes). We will prove this extension of the index theorem to the equality of classes in the Section \ref{sec:equality_classes}.

Recall that our goal is to get an index theorem for pseudodifferential operators. So let $P \in \Psi \mathrm{DO}_?^k(S)$ be an odd, $p$-multigraded, elliptic and symmetric pseudodifferential operator of positive order $k > 0$. By Theorem \ref{thm:elliptic_symmetric_PDO_defines_uniform_Fredholm_module} we get a uniform $K$-homology class $[P] \in K_p^u(M)$. Since Roe's Index Theorem applies only for $K_0^u(M)$, we assume that $p$ is even and use the formal periodicity $K_{p}^u(M) \cong K_{p+2}^u(M)$ to get a class $[P] \in K_0^u(M)$. Note that the assumption that $p$ is even is a very restrictive one (e.g., in the case of the usual geometric Dirac operators like the signature operator or the Atiyah--Singer operator $\slashed{D}$ associated to a spin manifold it restricts the index theorem to even-dimensional manifolds). Though there are index theorem for the case that $p$ is odd, they require substantially new ideas. See our corresponding discussion in Section \ref{sec:partitioned_manifold}.

Suppose that this class $[P] \in K_0^u(M)$ lies in the image of the cap product, i.e., there is an element $[E] - [F] \in K^0_u(M)$ and a graded operator $D$ of Dirac type, such that we have
\[[P] = ([E] - [F]) \cap [D] = [D_E] - [D_F] \in K_0^u(M).\]
We define the \emph{topological index $\tind_\tau(P) \in \IR$ of $P$} as $\theta(I_t(D_E) - I_t(D_F)) \in \IR$ and get from Theorem \ref{thm:roe_index_thm_rephrased} the following equality of indices of $P$:
\begin{equation}\label{eq:index_thm_pseudodiff}
\tind_\tau(P) = \ind_\tau([P]).
\end{equation}
Note that we did not write $\theta(I_t(P))$ for the topological index $\tind_\tau(P)$ of $P$ since we can not compute the index class $I_t(P)$ as we can do it for operators of Dirac type. The reason for this is that $I_t(D)$ is defined using the asymptotic expansion of the integral kernel of the operator $e^{-t D^2}$, and for pseudodifferential operators we generally do not have such an expansion (this even fails for compact manifolds, i.e, it is not a problem of our extension from compact to non-compact manifolds).

Of course it might be the case that there are different operators $D$ and elements $[E] - [F]$ such that $[P] = ([E] - [F]) \cap [D]$. So a priori the topological index $\tind_\tau(P)$ may depend on this choices. But Equation \eqref{eq:index_thm_pseudodiff} shows that this is not the case, i.e., the topological index $\tind_\tau(P)$ of $P$ is well-defined \emph{if $[P]$ lies in the image of the cap product}, and only in this case we do have an index theorem for $P$.

Since we now know that the topological index of $P$ is well-defined, we can state its definition:

\begin{defn}[Topological index $\tind_\tau(P)$]\label{defn:topological_index_P}
The \emph{topological index $\tind_\tau(P) \in \IR$ of the pseudodifferential operator $P$} is defined as $\theta(I_t(D_E) - I_t(D_F))$, where $D$ is a graded operator of Dirac type over $M$ and the element $[E]-[F] \in K^0_u(M)$ is such that we have $([E]-[F]) \cap [D] = [P] \in K_0^u(M)$.
\end{defn}

To solve the problem that $[P] \in K_0^u(M)$ must lie in the image of the cap product, we restrict ourselves to even-dimensional spin$^c$ manifolds (see the next section for a definition of spin$^c$ manifolds). In this case we have a distinguished element in the uniform $K$-homology of $M$, called the fundamental class $[M] \in K_m^u(M)$ of $M$, such that cap product with it induces an isomorphism $K^\ast_u(M) \cong K_{m-\ast}^u(M)$, where $m$ is the dimension of $M$. We will prove this \Poincare duality in Section \ref{sec:poincare_duality}. We assume that $m$ is even so that we may use formal periodicity to write \Poincare duality for $\ast = 0$ as $K_u^0(M) \cong K_0^u(M)$,\footnote{If $m$ is odd, then \Poincare duality reduces to $K^1_u(M) \cong K_0^u(M)$. But for $K^1_u(M)$ we no longer have an interpretation via vector bundles \emph{over $M$.}} and therefore we conclude that $[P] \in K_0^u(M)$ is \emph{always} given by the cap product of a graded operator of Dirac type with a formal difference of vector bundles of bounded geometry over $M$, i.e., the topological index of $P$ is always defined.

So we have finally proved the main result of this thesis: the generalization of Roe's Index Theorem to pseudodifferential operators.

\begin{indexthm}\label{thm:index_thm}
Let $M$ be an amenable and even-dimensional spin$^c$ manifold of bounded geometry and let $P \in \Psi \mathrm{DO}_?^k(S)$ be a graded, elliptic and symmetric pseudodifferential operator of positive order $k > 0$.

Then for all choices of F{\o}lner sequences for $M$ and all choices of functionals $\tau \in (\ell^\infty)^\ast$ associated to a free ultrafilter on $\IN$, we have
\[\tind_\tau(P)  = \ind_\tau([P]).\]
\end{indexthm}

\section{\texorpdfstring{Spin$^c$}{Spin-c} manifolds of bounded geometry}\label{sec:spinc_manifolds}

In this section we will define \spinc manifolds of bounded geometry. The definition of them is the same as for compact manifolds, but of course we have to assume that the involved bundles have bounded geometry. For the convenience of unfamiliar readers, we will recall all needed definitions.

Let $M$ be a manifold of bounded geometry. We will denote by $\ICl(M)$ the complexified Clifford algebra bundle of $M$, i.e., its fiber over a point $x \in M$ is the complexified Clifford algebra $Cl(T_x M) \otimes \IC$ of $T_x M$. Note that $\ICl(M)$ has a natural connection $\nabla$ extending the Levi--Civita connection of $M$, which is characterized by $\nabla(\varphi \psi) = (\nabla \varphi)\psi + \varphi \nabla \psi$ for all smooth sections $\varphi, \psi$ of $\ICl(M)$.

\begin{defn}[Dirac bundles of bounded geometry]
Let $S$ be a bundle of left modules over $\ICl(M)$. We will call $S$ a \emph{Dirac bundle of bounded geometry} if it is equipped with a Hermitian metric and compatible connection $\nabla^S$ turning it into a vector bundle of bounded geometry and such that:
\begin{itemize}
\item for each unit vector $v \in T_x M$ the Clifford multiplication $S_x \to S_x$, $s \mapsto v \cdot s$ is an isometry, and
\item for all smooth sections $\varphi$ of $\ICl(M)$ and $s$ of $S$ we have $\nabla^S (\varphi \cdot s) = \nabla \varphi \cdot s + \varphi \cdot \nabla^S s$.
\end{itemize}
If $S$ is equipped with an involution $\epsilon$ which is compatible with the connection and which anticommutes with the Clifford action of tangent vectors, we will call $S$ a \emph{graded} Dirac bundle of bounded geometry.

A \emph{$p$-multigraded Dirac bundle} is a graded Dirac bundle $S$ equipped with $p$ odd endomorphisms $\epsilon_1, \ldots, \epsilon_p$ such that
\[\epsilon_i^\ast = - \epsilon_i, \ \epsilon_i^2 = -1, \text{ and } \epsilon_i \epsilon_j = - \epsilon_j \epsilon_i \text{ for } i \not= j,\]
i.e., each fiber $S_x$ is a $p$-multigraded vector space, and such that each $\epsilon_i$ commutes with every Clifford multiplication operator on every fiber $S_x$.
\end{defn}

On every Clifford bundle $S$ of bounded geometry we have a natural first-order differential operator $D$ which is called the \emph{Dirac operator of $S$} and which is defined by the composition
\[C^\infty(S) \to C^\infty(T^\ast M \otimes S) \to C^\infty(TM \otimes S) \to C^\infty(S).\]
The first arrow is given by the connection, the second is derived from the Riemannian metric, and the third is Clifford multiplication. If $S$ is graded, then the operator $D$ will be odd, and of course if $S$ is $p$-multigraded, then $D$ is also $p$-multigraded.

If $(v_1, \ldots, v_n)$ is an orthonormal basis of $T_x M$, we get locally for $s \in C^\infty(S)$
\[(Ds)(x) = \sum_k v_k \cdot (\nabla^S_{v_k} s) (x)\]
and may therefore conclude due to the bounded geometry of $M$ and $S$ that $D$ is a pseudodifferential operator of order $1$.\footnote{Note that we have to check here the Uniformity Condition \ref{eq:uniformity_defn_PDOs}. It follows, of course, from the bounded geometry of $M$ and $S$.} Its symbol $p_D(x,\xi) \colon S_x \to S_x$ is given by Clifford multiplication with $\xi$ (converted to a tangent vector via the Riemannian metric) from which is follows that $D$ is an elliptic operator. Since $D$ is symmetric, we conclude with Theorem \ref{thm:elliptic_symmetric_PDO_defines_uniform_Fredholm_module} that $D$ defines a class in the uniform $K$-homology of $M$.

\begin{defn}[Operators of Dirac type]\label{defn:dirac_type_operator}
We will say that some operator $D$ is of \emph{Dirac type}, if it arises in the above way.
\end{defn}

Let $S$ be a Dirac bundle of bounded geometry and let $E$ be an arbitrary vector bundle of bounded geometry. Then the tensor product $S \otimes E$ becomes a bundle of left modules over $\ICl(M)$ via $\varphi \cdot (s \otimes e) := (\varphi \cdot s) \otimes e$ and equipping $S \otimes E$ with the tensor product metric and connection, $S \otimes E$ becomes a Dirac bundle of bounded geometry. If $S$ is $p$-multigraded, then $S \otimes E$ is again $p$-multigraded.

\begin{defn}[Twisted operators]\label{defn:twisted_op}
Let $D$ be an operator of Dirac type associated to the Dirac bundle $S$. If $E$ is any vector bundle of bounded geometry, we define the \emph{twisted operator $D_E$} to be the Dirac operator associated to the Dirac bundle $S \otimes E$.
\end{defn}

Now we will work towards the definition of \spinc manifolds of bounded geometry. First we need the following definition:

\begin{defn}[Complex Clifford algebra for $\IR^m$, {\cite[Definition 11.2.2]{higson_roe}}]
The \emph{complex Clifford algebra for $\IR^m$} is the complex $^\ast$-algebra $\IC_m$ generated by elements $e_1, \ldots, e_m$ corresponding to the standard orthonormal basis of $\IR^m$, such that
\[e_i^\ast = -e_i, \ e_i^2 = -1, \text{ and } e_i e_j = - e_j e_i \text{ for } i \not= j.\]
We define an inner product on $\IC_m$ by declaring the basis
\[\{ e_{i_1} \cdots e_{i_k} \ | \ i_1 < \ldots < i_k \text{ and } 0 \le k \le m\}\]
of $\IC_m$ to be orthonormal. The action of the algebra $\IC_m$ on the Hilbert space $\IC_m$ by left multiplication is a faithful $^\ast$-representation which gives $\IC_m$ the structure of a $C^\ast$-algebra. We grade $\IC_m$ by assigning each monomial $e_{i_1} \cdots e_{i_k}$ its degree mod $2$.
\end{defn}

If $M$ is an $m$-dimensional Riemannian manifold, we may define locally Dirac bundle over it. Let $e_1, \ldots, e_m$ be a local orthonormal frame for $T^\ast M$ defined over an open subset $U \subset M$, then the trivial bundle $U \times \IC_m$ over $U$ may be given the structure of an $m$-multigraded Dirac bundle in the following way: Clifford multiplication by an element $e_i$ of the frame is left multiplication by the $i$th generator of $\IC_m$, and the $m$-multigrading operators $\epsilon_1, \ldots, \epsilon_m$ of the bundle are given by right multiplication by the same generators.

Having defined canonical local Dirac bundles, we may now define what a complex spinor bundle on $M$ is: it is a Dirac bundle which is locally of the above described canonical type. Since we may recover from the canonical local Dirac bundles the orientation of frames, the definition of complex spinor bundles is only sensible for oriented manifolds. Note that we now have to incorporate the bounded geometry into the definition.

\begin{defn}[Complex spinor bundles, {\cite[Definition 11.2.3]{higson_roe}}]
Let $M$ be an $m$-dimensional oriented manifold of bounded geometry. A \emph{complex spinor bundle of bounded geometry on $M$} is an $m$-multigraded Dirac bundle $S$ of bounded geometry over $M$ with the following property: locally in synchronous framings of a fixed radius (i.e., around every point of $M$ we consider normal coordinates of a radius that does not depend on the point) it is isomorphic to the trivial bundle with fiber $\IC_m$ (i.e., at each fiber $S_x$ the synchronous framing is mapped to the standard orthonormal basis of $\IR^m$), and the Clifford multiplication is determined from a local synchronous framing $e_1, \ldots, e_m$ of $T^\ast M$ as above.
\end{defn}

Since $\IC_m$ has dimension $2^m$, it follows that complex spinor bundles have always fiber dimension $2^m$.

\begin{defn}[Spin$^c$-manifolds of bounded geometry]
Let $M$ be a manifold of bounded geometry. A \emph{\spinc structure on $M$} is the choice of a complex spinor bundle of bounded geometry over $M$. The \emph{fundamental class $[M] \in K_m^u(M)$} is the class of the Dirac operator associated to the spin$^c$ structure of $M$.
\end{defn}

Since our Index Theorem \ref{thm:index_thm} applies only to even-dimensional \spinc manifolds, we will discuss that case now a bit more thorough. But first we need the following observation:

\begin{lem}[{\cite[Lemma 11.3.1]{higson_roe}}]
Let $M$ be an $m$-dimensional manifold of bounded geometry. Then every $m$-multigraded Dirac bundle on $M$ of fiber dimension $2^m$ is a spinor bundle.
\end{lem}

Now we recall from Proposition \ref{prop:multigraded_categories_equiv} that the categories of $p$-multigraded and $(p+2)$-multigraded Hilbert spaces are equivalent. Transferring its proof to our present case and using the above lemma, we get the following result:

\begin{prop}[{\cite[Proposition 11.3.2]{higson_roe}}]
Let $M$ be a manifold of bounded geometry and of even dimension $m = 2k$. Then there is a one-to-one correspondence between isomorphism classes of complex spinor bundles on $M$ and isomorphism classes of Dirac bundles of dimension $2^k$.
\end{prop}

We may use this to show that complex manifolds are \spinc manifolds:

\begin{example}
Let $M$ be a complex manifold of bounded geometry and of complex dimension $k$ (so that its real dimension is $m = 2k$). Then the Dolbeault operator acts on a Dirac bundle of dimension $2^k$, i.e., we get a canonical \spinc structure. The uniform $K$-homology class $[D] \in K_0^u(M)$ of the Dolbeault operator corresponds under the periodicity isomorphisms $K_p^u(M) \cong K_{p+2}^u(M)$ to the fundamental class $[M] \in K_m^u(M)$ of $M$.
\end{example}

\section{Cap product}\label{sec:cap_product}

In this section we will define the cap product $\cap \colon K_u^p(X) \otimes K_q^u(X) \to K_{q-p}^u(X)$.

In the following we will need two facts: firstly, that we can canonically extend a representation $\rho$ of $C_0(X)$ to $C_u(X)$ (in fact, we may even extend it canonically to the bounded Borel functions $B_b(X)$), and secondly, that if $f \in \LLip_R(X)$ and $g \in {L^\prime\text{-}\operatorname{Lip}}_{R^\prime}(X)$, then their product $fg \in {N\text{-}\operatorname{Lip}}_{T}(X)$, where $T = \min\{R, R^\prime\}$ and $N = 2 \max\{L, L^\prime\}$. Recall that we have
\begin{equation*}
\LLip_R(X) := \{ f \in C_c(X) \ | \ f \text{ is }L\text{-Lipschitz}, \diam(\supp f) \le R \text{ and } \|f\|_\infty \le 1\}.
\end{equation*}

Let us first describe the cap product of $K_u^0(X)$ with $K^u_\ast(X)$ on the level of uniform Fredholm modules. The general definition of it will be given via dual algebras.

\begin{lem}\label{lem:proj_again_uniform_fredholm_module}
Let $P$ be a projection in $\Mat_{n \times n}(C_u(X))$ and let $(H, \rho, T)$ be a uniform Fredholm module.

We set $H_n := H \otimes \IC^n$, $\rho_n(\largecdot) := \rho(\largecdot) \otimes \id_{\IC^n}$, $T_n := T \otimes \id_{\IC^n}$ and by $\pi$ we denote the matrix $\pi_{ij} := \rho(P_{ij}) \in \Mat_{n \times n}(\IB(H)) = \IB(H_n)$.

Then $(\pi H_n, \pi \rho_n \pi, \pi T_n \pi)$ is again a uniform Fredholm module, with an induced (multi-)grading if $(H, \rho, T)$ was (multi-)graded.
\end{lem}

\begin{proof}
Let us first show that the operator $\pi T_n \pi$ is a uniformly pseudolocal one. Let $R, L > 0$ be given and we have to show that $\{[\pi T_n \pi, \pi \rho_n(f) \pi] \ | \ f \in \LLip_R(X)\}$ is uniformly approximable. This means that we must show that for every $\varepsilon > 0$ there exists an $N > 0$ such that for every $[\pi T_n \pi, \pi \rho_n(f) \pi]$ with $f \in \LLip_R(X)$ there is a rank-$N$ operator $k$ with $\|[\pi T_n \pi, \pi \rho_n(f) \pi] - k\| < \varepsilon$.

We have
\[[\pi T_n \pi, \pi \rho_n(f) \pi] = \pi [T_n, \pi \rho_n(f)] \pi,\]
because $\pi^2 = \pi$ and $\pi$ commutes with $\rho_n(f)$. So since $(\pi \rho_n(f))_{ij} = \rho(P_{ij} f) \in \IB(H)$, we get for the matrix entries of the commutator
\[([T_n, \pi \rho_n(f)])_{ij} = [T, \rho(P_{ij} f)].\]

Since the $P_{ij}$ are bounded and uniformly continuous, they can be uniformly approximated by Lipschitz functions, i.e., there are $P_{ij}^\varepsilon$ with
\[\|P_{ij} - P_{ij}^\varepsilon\|_\infty < \varepsilon / (4n^2 \|T\|).\]
Note that we have $P_{ij}^\varepsilon f \in {L_{ij}\text{-}\operatorname{Lip}}_{R}(X)$, where $L_{ij}$ depends only on $L$ and $P_{ij}^\varepsilon$. We define $L^\prime := \max\{L_{ij}\}$.

Now we apply the uniform pseudolocality of $T$, i.e., we get a maximum rank $N^\prime$ corresponding to $R, L^\prime$ and $\varepsilon / 2n^2$. So let $k_{ij}^\varepsilon$ be the rank-$N^\prime$ operators corresponding to the functions $P_{ij}^\varepsilon f$, i.e.,
\[\|[T, \rho(P_{ij}^\varepsilon f)] - k_{ij}^\varepsilon\| < \varepsilon / 2n^2.\]

We set $k := \pi (k_{ij}^\varepsilon) \pi$, where $(k_{ij}^\varepsilon)$ is viewed as a matrix of operators. Then $k$ has rank at most $N := n^2 N^\prime$. Then we compute
\begin{align*}
\|[\pi T_n & \pi, \pi \rho_n(f) \pi] - k\|\\
& = \|\pi[T_n, \pi \rho_n(f)]\pi - \pi (k_{ij}^\varepsilon) \pi\|\\
& \le \|\pi\|^2 \cdot n^2 \cdot \max_{i,j}\{\|[T, \rho(P_{ij} f)] - k_{ij}^\varepsilon\|\}\\
& \le \|\pi\|^2 \cdot n^2 \cdot \max_{i,j}\{\underbrace{\|[T, \rho(P_{ij} f)] - [T, \rho(P_{ij}^\varepsilon f)]\|}_{=\|[T, \rho(P_{ij} - P_{ij}^\varepsilon)\rho(f)]\|} + \underbrace{\|[T, \rho(P_{ij}^\varepsilon f)] - k_{ij}^\varepsilon\|\}}_{\le \varepsilon / 2n^2}\\
& \le \|\pi\|^2 \cdot n^2 \cdot \max_{i,j}\{2 \|T\| \cdot \underbrace{\|\rho(P_{ij} - P_{ij}^\varepsilon)\| \cdot \|\rho(f)\|}_{\le \varepsilon/(4n^2 \|T\|)} + \varepsilon / 2n^2\}\\
& \le \|\pi\|^2 \cdot \varepsilon,
\end{align*}
which concludes the proof of the uniform pseudolocality of $\pi T_n \pi$.

That $(\pi T_n \pi)^2 - 1$ and $\pi T_n \pi - (\pi T_n \pi)^\ast$ are uniformly locally compact can be shown analogously. Note that because $T$ is uniformly pseudolocal we may interchange the order of the operators $T_n$ and $\rho(P_{ij}^\varepsilon f)$ in formulas (since for fixed $R$ and $L$ the subset $\{[T_n, \rho(P_{ij}^\varepsilon f)] \ | \ f \in \LLip_R(X) \} \subset \IB(H_n)$ is uniformly approximable).

We have shown that $(\pi H_n, \pi \rho_n \pi, \pi T_n \pi)$ is a uniform Fredholm module. That it inherits a (multi-)grading from $(H, \rho, T)$ is clear and this completes the proof.
\end{proof}

That the construction from the above lemma is compatible with the relations defining $K$-theory and uniform $K$-homology and that it is bilinear is quickly deduced and completely analogous to the non-uniform case. So we get a well-defined pairing
\[\cap\colon K_u^0(X) \otimes K_\ast^u(X) \to K_\ast^u(X)\]
which exhibits $K_\ast^u(X)$ as a module over the ring $K_u^0(X)$.\footnote{Compatibility with the internal product on $K_u^0(X)$, i.e., $(P \otimes Q) \cap T = P \cap (Q \cap T)$, is easily deduced. It mainly uses the fact that the isomorphism $\Mat_{n \times n}(\IC) \otimes \Mat_{m \times m}(\IC) \cong \Mat_{nm \times nm}(\IC)$ is canonical up to the ordering of basis elements. But different choices of orderings result in isomorphisms that differ by inner automorphisms, which makes no difference at the level of $K$-theory.}

To define the cap product in its general form, we will use the dual algebra picture of uniform $K$-homology, i.e., Paschke duality, from Section \ref{sec:paschke_duality}. So let us therefore first recall the needed definitions:

\begin{defn}[{\cite[Definition 4.1]{spakula_uniform_k_homology}}]\label{defn:frakD_frakC}
Let $H$ be a separable Hilbert space and $\rho \colon C_0(X) \to \IB(H)$ a representation.

We denote by $\frakD^u_{\rho \oplus 0}(X) \subset \IB(H \oplus H)$ the $C^\ast$-algebra of all uniformly pseudolocal operators with respect to the representation $\rho \oplus 0$ of $C_0(X)$ on the space $H \oplus H$ and by $\frakC^u_{\rho \oplus 0}(X) \subset \IB(H \oplus H)$ the $C^\ast$-algebra of all uniformly locally compact operators.
\end{defn}

That the algebras $\frakD^u_{\rho \oplus 0}(X)$ and $\frakC^u_{\rho \oplus 0}(X)$ are indeed $C^\ast$-algebras was shown in \cite[Lemma 4.2]{spakula_uniform_k_homology}. There it was also shown that $\frakC^u_{\rho \oplus 0}(X) \subset \frakD^u_{\rho \oplus 0}(X)$ is a closed, two-sided $^\ast$-ideal.

The following lemma is a uniform analog of the fact \cite[Lemma 5.4.1]{higson_roe} and is essentially proven in \cite[Lemma 5.3]{spakula_uniform_k_homology} (by ``setting $Z := \emptyset$'' in that lemma).

\begin{lem}\label{lem:K_theory_frakC_zero}
We have
\[K_\ast(\frakC^u_{\rho \oplus 0}(X)) = 0\]
and so the quotient map $\frakD^u_{\rho \oplus 0}(X) \to \frakD^u_{\rho \oplus 0}(X) / \frakC^u_{\rho \oplus 0}(X)$ induces an isomorphism
\begin{equation}\label{eq:K_theory_frakC_zero}
K_\ast(\frakD^u_{\rho \oplus 0}(X)) \cong K_\ast(\frakD^u_{\rho \oplus 0}(X) / \frakC^u_{\rho \oplus 0}(X))
\end{equation}
due to the $6$-term exact sequence for $K$-theory.
\end{lem}

The last ingredient that we will need is the inclusion
\begin{equation}\label{eq:commutator_C_u_with_frakD}
[C_u(X), \frakD^u_{\rho \oplus 0}(X)] \subset \frakC^u_{\rho \oplus 0}(X).
\end{equation}
It is proven in the following way: let $\varphi \in C_u(X)$ and $T \in \frakD^u_{\rho \oplus 0}(X)$. We have to show that $[\varphi, T] \in \frakC^u_{\rho \oplus 0}(X)$. By approximating $\varphi$ uniformly by Lipschitz functions we may without loss of generality assume that $\varphi$ itself is already Lipschitz. Now the claim follows immediately from $f[\varphi, T] = [f \varphi, T] - [f,T]\varphi$ since $T$ is uniformly pseudolocal.

Now we are able to define the cap product. Consider the map
\[ \sigma \colon C_u(X) \otimes \frakD^u_{\rho \oplus 0}(X) \to \frakD^u_{\rho \oplus 0}(X) / \frakC^u_{\rho \oplus 0}(X), \ f \otimes T \mapsto [fT].\]
It is a multiplicative $^\ast$-homomorphism due to the above Equation \eqref{eq:commutator_C_u_with_frakD} and hence induces a map on $K$-theory
\begin{equation*}
\sigma_\ast \colon K_\ast(C_u(X) \otimes \frakD^u_{\rho \oplus 0}(X)) \to K_\ast(\frakD^u_{\rho \oplus 0}(X) / \frakC^u_{\rho \oplus 0}(X)).
\end{equation*}
Recall from Section \ref{sec:paschke_duality} that $K_\ast^u(X)$ is the direct limit of the groups $K_\ast^u(X; \rho \oplus 0)$ and that these are isomorphic to $K_{1+\ast}(\frakD^u_{\rho \oplus 0}(X))$. Now on $K_\ast^u(X; \rho \oplus 0)$ we may define the cap product as the composition
\begin{align*}
K_u^p(X) \otimes K_q^u(X; \rho \oplus 0) & \ = \ K_{-p}(C_u(X)) \otimes K_{1+q}(\frakD^u_{\rho \oplus 0}(X))\\
& \ \to \ \! K_{-p+1+q}(C_u(X) \otimes \frakD^u_{\rho \oplus 0}(X))\\
& \ \stackrel{\sigma_\ast}\to \ \! K_{-p+1+q}(\frakD^u_{\rho \oplus 0}(X) / \frakC^u_{\rho \oplus 0}(X))\\
& \stackrel{\eqref{eq:K_theory_frakC_zero}}\cong K_{-p+1+q}(\frakD^u_{\rho \oplus 0}(X))\\
& \ = \ K_{q-p}^u(X; \rho \oplus 0),
\end{align*}
where the first arrow is the external product on $K$-theory. Since all above occuring maps are compatible with the connecting homomorphisms in the directed system $K_\ast^u(X) \cong \underrightarrow{\lim} \ K_\ast^u(X; \rho \oplus 0)$, we get the cap product
\[\cap \colon K_u^p(X) \otimes K_q^u(X) \to K_{q-p}^u(X).\]

Let us state in a proposition some properties of it that we will need. The proofs of these properties are analogous to the non-uniform case.

\begin{prop}\label{prop:properties_general_cap_product}
The cap product has the following properties:
\begin{itemize}
\item the pairing of $K_u^0(X)$ with $K_\ast^u(X)$ coincides with the one in Lemma \ref{lem:proj_again_uniform_fredholm_module},
\item the fact that $K_\ast^u(X)$ is a module over $K_u^0(X)$ generalizes to
\begin{equation}\label{eq:general_cap_compatibility_module}
(P \otimes Q) \cap T = P \cap (Q \cap T)
\end{equation}
for all elements $P, Q \in K_u^\ast(X)$ and $T \in K_\ast^u(X)$, where $\otimes$ is the internal product on uniform $K$-theory,
\item if $X$ and $Y$ have jointly bounded geometry, then we have the following compatibility with the external products:
\begin{equation}\label{eq:compatibility_cap_external}
(P \times Q) \cap (S \times T) = (-1)^{qs} (P \cap S) \times (Q \cap T),
\end{equation}
where $P \in K_u^p(X)$, $Q \in K_u^q(Y)$ and $S \in K^u_s(X)$, $T \in K^u_t(Y)$, and
\item if we have a manifold of bounded geometry $M$, a vector bundle of bounded geometry $E \to M$ and an operator $D$ of Dirac type, then
\begin{equation}\label{eq:cap_twisted_Dirac}
[E] \cap [D] = [D_E] \in K_\ast^u(M),
\end{equation}
where $D_E$ is the twisted operator (see Definition \ref{defn:twisted_op}).
\end{itemize}
\end{prop}

\section{\Poincare duality}\label{sec:poincare_duality}

Let $M$ be an $m$-dimensional spin$^c$ manifold of bounded geometry. In this section we will show that the cap product $\largecdot \cap [M] \colon K_u^\ast(M) \to K^u_{m-\ast}(M)$ with its fundamental class $[M] \in K_m^u(M)$ is an isomorphism.

The proof is analogous to the proof of the Chern Character Isomorphism Theorem. Let us recall its basic idea: we showed in Lemma \ref{lem:suitable_coloring_cover_M} that we have a particular cover of $M$ with finitely many open subsets $U_j$, where each $U_j$ is a disjoint union of balls $B_\varepsilon(x_i)$. For these subsets we could show in Lemmas \ref{lem:MV_bounded_dR} and \ref{lem:MV_uniform_K} that we have corresponding Mayer--Vietoris sequences for (modified versions of) bounded de Rham cohomology and uniform $K$-theory and in Lemma \ref{lem:MV_chern_character} we showed that we have suitable Chern character maps between these Mayer--Vietoris sequences. This was all that was needed to do the induction step in the proof of the Chern Character Isomorphism Theorem. For the first step of the induction we needed two ingredients: the first was the homotopy invariance of bounded de Rham cohomology and uniform $K$-theory to reduce from the subsets $U_1$, $U_2$ and $U_1 \cap U_2$ to collections of uniformly discrete points, and the second ingredient was the direct computation for such a uniformly discrete collection of points, so that we could see that in this case the Chern character indeed induces the corresponding isomorphism.

So for our proof of \Poincare duality we have to show the following things: that uniform $K$-homology is homotopy invariant (we have actually already done this in Theorem \ref{thm:homotopy_equivalence_k_hom}) and that the cap product induces an isomorphism on a collection of uniformly discrete open balls (we will explain in the next paragraph why we have to show it for open balls instead of points). This will complete the first step of the induction. For the higher induction steps we have to show that we have a suitable version of the Mayer--Vietoris sequence for uniform $K$-homology (for our particular cover of $M$ with the subsets $U_j$) together with cap product maps that are compatible with these exact sequences, i.e., make the squares in the occuring diagrams commutative.

The induction starts with the subsets $U_1$, $U_2$ and $U_1 \cap U_2$, which are collections of uniformly discretely distributed open balls (resp., in the case of $U_1 \cap U_2$ it is a collection of intersections of open balls, which is boundedly homotopy equivalent to a collection of open balls). Now uniform $K$-theory of a space coincides with the uniform $K$-theory of its completion, and furthermore, uniform $K$-theory is a bounded homotopy invariant\footnote{See Footnote \ref{footnote:boundedly_homotopic} to recall the definition of a bounded homotopy.}. So the uniform $K$-theory of a collection of open balls is the same as the uniform $K$-theory of a collection of points. But uniform $K$-homology is homotopy invariant only with respect to bounded, \emph{proper} homotopies (see Theorem \ref{thm:homotopy_equivalence_k_hom}), and for totally bounded spaces it coincides with usual $K$-homology (see Proposition \ref{prop:compact_space_every_module_uniform}). So the uniform $K$-homology of a collection of open balls is not the same as the uniform $K$-homology of a collection of points. In fact, we have the following lemma:

\begin{lem}
Let $M$ be an $m$-dimensional manifold of bounded geometry and let $U \subset M$ be a subset consisting of uniformly discretely distributed geodesic balls in $M$ having radius less than the injectivity radius of $M$ (i.e., each geodesic ball is diffeomorphic to the standard ball in Euclidean space $\IR^m$). Let the balls be indexed by a set $Y$ (usually $Y \subset M$ is a quasi-lattice).

Then $K_m^u(U) \cong \ell^\infty_\IZ(Y)$, the group of all bounded, integer-valued sequences indexed by $Y$, and $K_p^u(U) = 0$ for $p \not= m$.
\end{lem}

\begin{proof}
The proof is analogous to the proof of Lemma \ref{lem:uniform_k_hom_discrete_space}. It uses the fact that for an open Ball $O \subset \IR^m$ we have $K_m(O) \cong \IZ$, and $K_p(O) = 0$ for $p \not= m$.
\end{proof}

Fortunately, the cap product has exactly this shift of indices built in, i.e., the cap product is an isomorphism $K_u^\ast(U) \cong K_{m-\ast}^u(U)$, where $U$ is as in the above lemma. Here we have to note that if $M$ is a \spinc manifold, then the restriction of its complex spinor bundle to any ball of $U$ is isomorphic to the complex spinor bundle on the open ball $O \subset \IR^m$. This means that the cap product on $U$ coincides on each open ball of $U$ with the usual cap product on the open ball $O \subset \IR^m$. This all, together with the bounded, proper homotopy invariance of uniform $K$-homology (used to treat the case of $U_1 \cap U_2$ since it does not consist of open balls, but of intersections of them), completes the first step of the induction.

So it remains to show that we have suitable Mayer--Vietoris sequences for uniform $K$-homology. Note that ``suitable'' means here that we need it for the wrong-way maps on uniform $K$-homology (since the arrows must go in the same direction as the arrows in the Mayer--Vietoris sequence for uniform $K$-theory, i.e., from the ``big'' into the ``small'' subset). Let a not necessarily connected subset $O \subset M$ be given. We define $K_\ast^{u, \mathrm{MV}}(O)$ to be generated by uniform Fredholm modules over $B_\delta(O)$, where $\delta > 0$ may depend on the concrete uniform Fredholm module at hand. Now the existence of suitable Mayer--Vietoris sequences for this modified uniform $K$-homology of the subsets in the covers $\{U_K, U_j\}$ of $U_K \cup U_{k+1}$ (recall that we used Lemma \ref{lem:suitable_coloring_cover_M} to get these subsets and note that these subsets are open, i.e., we do have wrong-way maps on uniform $K$-homology for them) follows by using excision for uniform $K$-homology which is proved by \v{S}pakula in \cite[Theorem 5.1]{spakula_uniform_k_homology}.

We denote by $[M]|_O \in K_m^{u, \mathrm{MV}}(O)$ the class of the Dirac operator associated to the restriction to a neighbourhood of $O$ of the complex spinor bundle of bounded geometry defining the spin$^c$-structure of $M$ (i.e., we equip the neighbourhood with the induced spin$^c$-structure). The cap product of $K^\ast_{u, \mathrm{MV}}(O)$ with $[M]|_O$ is then analogously defined as the one on $K^\ast_u(M)$.

It remains to argue why we get commutative squares between the Mayer--Vietoris sequences of uniform $K$-theory and uniform $K$-homology using the cap product. This is known for usual $K$-theory and $K$-homology; see, e.g., \cite[Exercise 11.8.11(c)]{higson_roe}. Since the cap product is in our uniform case completely analogously defined (see the second-to-last display before Proposition \ref{prop:properties_general_cap_product}), we may analogously conclude that we get commutative squares between our uniform Mayer--Vietoris sequences.

This completes the proof of \Poincare duality.

\section{Equality of the index classes}\label{sec:equality_classes}

We have defined the topological index $\tind_\tau(P)$ of a graded pseudodifferential operator $P$ as $\theta(I_t(D_E) - I_t(D_F))$, where the difference of vector bundles $[E] - [F] \in K_u^0(M)$ is such that $([E] - [F]) \cap [D] = [P] \in K_0^u(M)$ for a graded operator $D$ of Dirac type over $M$, and then we have shown the index theorem $\tind_\tau(P) = \ind_\tau([P])$. Now of course the question arises whether we can leave out the evaluation to $\IR$, i.e., if there is an analytic index class $\aind(P) \in H^m_{b, \mathrm{dR}}(M)$ with $\theta(\aind(P)) = \ind_\tau([P])$, and such that we can strengthen the index theorem to the equality $\tind(P) = \aind(P)$ with $\tind(P) := I_t(D_E) - I_t(D_F) \in H^m_{b, \mathrm{dR}}(M)$.

Our major ingredient for this will be the fact that every fundamental class $\theta$ for $M$ arises through \Folner sequences on $M$, i.e., every element in the dual space of $\bar{H}^m_{b, \mathrm{dR}}(M)$\footnote{It is clear that we have to restrict here to the reduced bounded de Rham cohomology of $M$ (see Remark \ref{rem:reduced_bounded_de_Rham_cohomology}) since fundamental classes are continuous.} can be constructed via \Folner sequences. This result follows basically from the ideas in \cite[Part II.§4]{sullivan}. Furthermore, it means that Roe's Index Theorem is equivalent to the one showing the equality of the corresponding index classes in $\bar{H}^m_{b, \mathrm{dR}}(M)$.

Let us make the following definition of the topological index class of $P$:

\begin{defn}[Topological index classes $\tind(P)$]
The \emph{topological index class $\tind(P) \in \bar{H}^m_{b, \mathrm{dR}}(M)$ of the graded pseudodifferential operator $P$} is defined as
\[\tind(P) := I_t(D_E) - I_t(D_F),\]
where $D$ is a graded operator of Dirac type over $M$ and $[E]-[F] \in K^0_u(M)$ is such that we have $([E]-[F]) \cap [D] = [P] \in K_0^u(M)$.
\end{defn}

Analogously as with the Definition \ref{defn:topological_index_P} of the topological index of $P$, this definition assumes that $[P]$ is in the image of the cap product. Furthermore, we have to show that it is well-defined: due to our Index Theorem \ref{thm:index_thm}, i.e., $\tind_\tau(P) = \ind_\tau([P])$, we know that possibly different index classes of $P$ coincide under all evaluations with fundamental classes associated to \Folner sequences since $\tind_\tau(P) = \theta(\tind(P))$. But the fundamental classes associated to \Folner sequences exhaust the whole dual space of $\bar{H}^m_{b, \mathrm{dR}}(M)$.

To define the analytic index class $\aind(P)$ of $P$ we will use the right hand side of the diagram in Theorem \ref{thm:roe_index_thm_rephrased}, i.e., $\ind_\tau([P]) = \ind_\tau(\mu_u([P]))$. Now we recall from Section \ref{sec:analytic_indices_quasiloc_smoothing} how the index map $\ind_\tau \colon K_0(\C(M)) \to \IR$ was defined: given a \Folner sequence $(M_i)_i$ of $M$ and an operator $A \in \C(S)$, we define a sequence $(m_i)_i$ via
\begin{equation*}
m_i := \frac{1}{\vol M_i} \int_{M_i} \trace k_A(x,x) dM,
\end{equation*}
where $k_A \in C_b^\infty(S \boxtimes S^\ast)$ is the uniformly bounded integral kernel of $A$. Since $(m_i)_i$ is bounded, we then may evaluate $\tau$ on it. This defines a trace\footnote{That is to say, it vanishes on commutators.} on $\C(S)$ and therefore we get an induced map $\ind_\tau \colon K_0(\C(S)) \to \IR$. Then we use the diagram in Proposition \ref{prop:analytic_index_map_quasiloc_smoothing} to go from $K_0(\C(S))$ to $K_0(\C(M))$.

So we could try to define an analytic index class map $K_0(\C(S)) \to \bar{H}^m_{b, \mathrm{dR}}(M)$ as the map induced from
\[\aind_{-\infty} \colon \C(S) \to \bar{H}^m_{b, \mathrm{dR}}(M; \IC) \text{, }A \mapsto \trace k_A(x,x) dM,\]
i.e., we would have to show that $\aind_{-\infty}$ is a trace. That the induced map then takes values in $\bar{H}^m_{b, \mathrm{dR}}(M)$ (i.e., that it is real-valued) follows from the fact that if $A$ is self-adjoint, then $\trace k_A(x,x) \in \IR$.

We know that $\theta(\aind_{-\infty}([A,B])) = \ind_\tau([A,B])$ via definition, and we know that $\ind_\tau$ is a trace. So $\theta_{-\infty}(\aind([A,B])) = 0$ for all fundamental classes $\theta$ associated to \Folner sequences, i.e., $\aind_{-\infty}([A,B]) = 0$ for all $A, B \in \C(S)$. So we get an induced map $K_0(\C(S)) \to \bar{H}^m_{b, \mathrm{dR}}(M)$ that we will also call $\aind_{-\infty}$, and may now define the analytic index class of a pseudodifferential operator:

\begin{defn}[Analytic index class $\aind(P)$]
We define the \emph{analytic index class $\aind(P) \in \bar{H}^m_{b, \mathrm{dR}}(M)$ of a graded pseudodifferential operator $P$} as
\[\aind(P) := \aind_{-\infty}(\mu_u([P])),\]
where we identify $K_0(\C(M))$ with $K_0(\C(S))$ via Corollary \ref{cor:natural_receptacle_iso} for the bundle $S$ on which the operator $P$ acts.
\end{defn}

Since $\theta(\tind(P)) = \tind_\tau(P)$ and also $\theta(\aind(P)) = \ind_\tau([P])$, the equality $\tind(P) = \aind(P)$ of the index classes of $P$ follows from the Index Theorem \ref{thm:index_thm} and the often mentioned fact that the dual of $\bar{H}^m_{b, \mathrm{dR}}(M)$ consists only of fundamental classes associated to \Folner sequences. So we have proved the following strengthening of the index theorem:

\begin{indexthm}
Let $M$ be an amenable, even-dimensional spin$^c$ manifold of bounded geometry and let $P \in \Psi \mathrm{DO}_?^k(S)$ be a graded, elliptic and symmetric pseudodifferential operator of positive order $k > 0$.

Then the topological and analytic index classes of $P$ coincide:
\[\tind(P) = \aind(P) \in \bar{H}^m_{b, \mathrm{dR}}(M).\]
\end{indexthm}

\chapter{Uniform coarse indices}\label{chap:uniform_coarse_indices}

Roe's Index Theorem, and therefore also ours, holds only for amenable manifolds, since if a manifold is not amenable, then its top-dimensional bounded de Rham cohomology $H^m_{b, \mathrm{dR}}(M)$ vanishes. But as we have already written in the introduction, the homotopy invariance of uniform $K$-homology may be seen as some sort of index theorem, and since homotopy invariance does not rely on amenability, we therefore have some sort of index theorem also for non-amenable manifolds.

The goal of this chapter is to make the above said precise, i.e., we will use homotopy invariance of uniform $K$-homoloy to investigate the uniform coarse assembly map $\mu_u \colon K_\ast^u(X) \to K_\ast(C_u^\ast(X))$ and the crucial conjecture about it: if $X$ is uniformly contractible and of coarsely bounded geometry, this map should be an isomorphism. This uniform coarse Baum--Connes conjecture, as one may call it, was already investigated by \Spakula in \cite{spakula_uniform_k_homology}, but since he did not have homotopy invariance to his disposal, he could not derive important implications from the conjecture like we will do here.

So this chapter rounds this thesis up by investigating the uniform coarse Baum--Connes conjecture as an index theorem which is also applicable to non-amenable spaces. Since our results rely heavily on the homotopy invariance that we have proved for uniform $K$-homology, this chapter fits perfectly into this thesis.

Since we are working with the uniform Roe algebra, all metric spaces in this section must be assumed to be proper.

\section{Assembly map}

We will describe here the construction of the \emph{uniform coarse assembly map}
\[\mu_u \colon K_\ast^u(X) \to K_\ast(C_u^\ast(Y)),\]
where $Y \subset X$ is a uniformly discrete quasi-lattice, using the techniques that we have introduced in this thesis.

\v{S}pakula was the first to construct it in \cite[Section 9]{spakula_uniform_k_homology}, but his presentation is vastly different from ours, though at the end we get the same map. Our construction will be analogous to the construction of the non-uniform coarse assembly map, as it may be found in, e.g., \cite[Section 12.3]{higson_roe}.

Let $X$ be a proper metric space of jointly bounded geometry\footnote{This assumption is only for simplification of the arguments, since with it Paschke duality assumes a simple form (Theorem \ref{thm:paschke_universal}), i.e., we do not have to work with direct limits (Proposition \ref{prop:direct_limit_version}).} (Definition \ref{defn:jointly_bounded_geometry}). From Section \ref{sec:paschke_duality} we know that if $\rho\colon C_0(X) \to \IB(H)$ is an ample representation\footnote{That is to say, $\rho$ is non-degenerate and $\rho(f) \in \IK(H)$ implies $f \equiv 0$.}, then $K_\ast^u(X)$ is isomorphic to $K_\ast^u(X; \rho \oplus 0)$, and the latter groups are isomorphic to $K_{1+\ast}(\frakD^u_{\rho \oplus 0}(X))$ for $\ast = -1,0$.

Recall that $\frakD^u_{\rho \oplus 0}(X) \subset \IB(H \oplus H)$ is the $C^\ast$-algebra of all uniformly pseudolocal operators with respect to the representation $\rho \oplus 0$ of $C_0(X)$ on the Hilbert space $H \oplus H$. Recall furthermore that $\frakC^u_{\rho \oplus 0}(X) \subset \frakD^u_{\rho \oplus 0}(X)$ is the closed, two-sided $^\ast$-ideal of all uniformly locally compact operators.

\begin{defn}[cf. Definition \ref{defn:_Du_Cu_spaces}]
Let $\rho\colon C_0(X) \to \IB(H)$ be an ample representation. We denote by $D_u^\ast(X) \subset \IB(H)$ the $C^\ast$-algebra generated by all uniformly pseudolocal operators having finite propagation, and by $C_u^\ast(X) \subset D_u^\ast(X)$ the closed, two-sided $^\ast$-ideal generated by all uniformly locally compact operators with finite propagation.
\end{defn}

We also have to recall Lemma \ref{lem:uniform_k_hom_via_dual_quot_2}. The additional assumption of jointly bounded geometry allows us to restate it by using Theorem \ref{thm:paschke_universal}: $K_\ast^u(X) \cong K_\ast^u(X; \rho \oplus 0)$, where $\rho\colon C_0(X) \to \IB(H)$ is an ample representation. Since we will need some facts from the proof of this lemma, we also quickly summarize its main steps.

\begin{lem}\label{lem:uniform_k_hom_via_dual_quot}
Let $X$ have jointly bounded geometry. Then
\[K_\ast^u(X) \cong K_{1 + \ast}(D_u^\ast(X) / C_u^\ast(X))\]
for $\ast = -1,0$.
\end{lem}

\begin{proof}
By Paschke duality together with Theorem \ref{thm:paschke_universal} we know
\[K_\ast^u(X) \cong K_\ast^u(X; \rho \oplus 0) \cong K_{1+\ast}(\frakD^u_{\rho \oplus 0}(X)).\]
We will now show that
\begin{equation}\label{eq:uniform_K_as_quotient_D*u/C*u}
K_\ast(\frakD_{\rho \oplus 0}^u(X)) \cong K_\ast(\frakD^u_\rho(X) / \frakC^u_\rho(X)),
\end{equation}
and that the inclusion of $D_u^\ast(X)$ into $\frakD^u_\rho(X)$ induces an isomorphism
\begin{equation}\label{eq:quotients_frakD_D*_equal}
D_u^\ast(X) / C_u^\ast(X) \cong \frakD^u_\rho(X) / \frakC^u_\rho(X)
\end{equation}
of $C^\ast$-algebras. These two statements together prove the claim.

To prove the first Isomorphism \eqref{eq:uniform_K_as_quotient_D*u/C*u} first note that
\[\frakD^u_{\rho \oplus 0}(X) = \begin{pmatrix}\frakD_\rho^u(X) & \fraklC^u_\rho(X) \\ \frakCr_\rho^u(X) & \IB(H)\end{pmatrix} \subset \IB(H \oplus H),\]
where $\fraklC^u_\rho(X)$ contains the operators that are uniformly locally compact \emph{from the left}, i.e., operators $T \in \IB(H)$ for which the collection $\{\rho(f) T \ | \ f \in \LLip_R(X)\}$ is uniformly approximable for all $R, L > 0$. Analogously, $\frakCr_\rho^u(X)$ is defined as the algebra containing all operators that are uniformly locally compact from the right.

We define $J \subset \frakD^u_{\rho \oplus 0}(X)$ as
\[J := \begin{pmatrix}\frakC_\rho^u(X) & \fraklC^u_\rho(X) \\ \frakCr_\rho^u(X) & \IB(H)\end{pmatrix} \subset \frakD^u_{\rho \oplus 0}(X),\]
which is a closed, two-sided $^\ast$-ideal in $\frakD^u_{\rho \oplus 0}(X)$. So we get a short exact sequence
\[ 0 \to J \to \frakD^u_{\rho \oplus 0}(X) \to \frakD^u_{\rho \oplus 0}(X) / J \to 0.\]
Now we may identify the quotient $\frakD^u_{\rho \oplus 0}(X) / J$ with $\frakD_\rho^u(X) / \frakC_\rho^u(X)$ and the quotient map $\frakD^u_{\rho \oplus 0}(X) \to \frakD^u_{\rho \oplus 0}(X) / J$ in the short exact sequence becomes the map
\[\frakD^u_{\rho \oplus 0}(X) \ni \begin{pmatrix}T_{11} & T_{12} \\ T_{21} & T_{22}\end{pmatrix} \mapsto [T_{11}] \in \frakD_\rho^u(X) / \frakC_\rho^u(X).\]
Furthermore, we have $J = \frakC_{\rho \oplus 0}^u(X)$. Hence the above short exact sequence becomes
\[0 \to \frakC_{\rho \oplus 0}^u(X) \to \frakD_{\rho \oplus 0}^u(X) \to \frakD_\rho^u(X) / \frakC_\rho^u(X) \to 0\]
and the claim that $K_\ast(\frakD_{\rho \oplus 0}^u(X)) \cong K_\ast(\frakD^u_\rho(X) / \frakC^u_\rho(X))$ now follows from the $6$-term exact sequence for $K$-theory and the fact that all the $K$-groups of $\frakC_{\rho \oplus 0}^u(X)$ vanish. This is a uniform analogue of the corresponding non-uniform statement which is proved in, e.g., \cite[Lemma 5.4.1]{higson_roe}, and this uniform analogue was essentially proven by \Spakula in \cite[Lemma 5.3]{spakula_uniform_k_homology} (by ``setting $Z := \emptyset$'' in that lemma).

The proof of the second Isomorphism \eqref{eq:quotients_frakD_D*_equal} is analogous to the proof of the corresponding non-uniform statement $D^\ast(X) / C^\ast(X) \cong \frakD_\rho(X) / \frakC_\rho(X)$ which may be found in, e.g., \cite[Lemma 12.3.2]{higson_roe}. This uniform version \eqref{eq:quotients_frakD_D*_equal} was also basically already shown by \Spakula in \cite[Section 7]{spakula_uniform_k_homology}.
\end{proof}

Now we are able to define the uniform coarse assembly map:

\begin{defn}
The \emph{uniform coarse assembly map}
\[\mu_u \colon K_\ast^u(X) \to K_\ast(C_u^\ast(X))\]
is the boundary map in the $6$-term exact sequence for the pair $(D^\ast_u(X), C^\ast_u(X))$, where we identify $K_\ast^u(X)$ with $K_{1 + \ast}(D_u^\ast(X) / C_u^\ast(X))$ via the above Lemma \ref{lem:uniform_k_hom_via_dual_quot} and we appeal to the formal $2$-periodicity to extend the map to all $\ast \in \IZ$.
\end{defn}

The following lemma is proved just by comparing our construction of the uniform coarse assembly map with the construction of \v{S}pakula.

\begin{lem}
If we identify $K_\ast(C_u^\ast(X))$ with $K_\ast(C_u^\ast(Y))$ (as in Proposition \ref{prop:IU(E)_dense_Cu*(E)} together with Lemma \ref{lem:iso_discrete_versions_uniform_roe}), where $Y \subset X$ is a uniformly discrete quasi-lattice, the above defined uniform coarse assembly map will coincide with the one defined by \Spakula in \cite[Section 9]{spakula_uniform_k_homology}.
\end{lem}

\section{Uniform coarse Baum--Connes conjecture}

Given a metric space $X$ of coarsely bounded geometry, we often used a quasi-lattice $Y \subset X$ in order to change from a continuous to a discrete version of the same space (in the same coarse equivalence class). Now we will have to go the other way round:

\begin{defn}[Rips complexes]
Let $Y$ be a discrete metric space and let $d \ge 0$. The \emph{Rips complex $P_d(Y)$ of $Y$} is a simplicial complex, where
\begin{itemize}
\item the vertex set of $P_d(Y)$ is $Y$, and
\item vertices $y_0, \ldots, y_q$ span a $q$-simplex if and only if we have $d(y_i, y_j) \le d$ for all $0 \le i, j \le q$.
\end{itemize}
Note that if $Y$ has coarsely bounded geometry, then the Rips complex $P_d(Y)$ is uniformly locally finite and finite dimensional and therefore also, especially, a simplicial complex of bounded geometry (i.e., the number of simplices in the link of each vertex is uniformly bounded). So if we equip $P_d(Y)$ with the metric derived from barycentric coordinates, $Y \subset P_d(Y)$ becomes a quasi-lattice (cf. Examples \ref{ex:coarsely_bounded_geometry}).
\end{defn}

Now we may state the \emph{uniform coarse Baum--Connes conjecture}:

\begin{conj}
Let $Y$ be a proper, uniformly discrete metric space with coarsely bounded geometry. Then
\[\mu_u \colon \lim_{d \to \infty} K_\ast^u(P_d(Y)) \to K_\ast(C_u^\ast(Y))\]
is an isomorphism.
\end{conj}

Let us relate the conjecture quickly to manifolds of bounded geometry. First we need the following notion:

\begin{defn}[Uniformly contractible spaces]
A metric space $X$ is called \emph{uniformly contractible}, if for every $r > 0$ there is an $s > 0$ such that every ball $B_r(x)$ can be contracted to a point in the ball $B_s(x)$.
\end{defn}

The for us most important examples of uniformly contractible spaces are universal covers of aspherical Riemannian manifolds equipped with the pull-back metric.

\begin{thm}
Let $M$ be a uniformly contractible manifold of bounded geometry and let $Y \subset M$ be a uniformly discrete quasi-lattice in $M$. Then we have a natural isomorphism
\[\lim_{d \to \infty} K^u_\ast(P_d(Y)) \cong K^u_\ast(M).\]
\end{thm}

The proof of this theorem is analogous to the corresponding non-uniform statement $\lim_{d \to \infty} K_\ast(P_d(Y)) \cong K_\ast(M)$ from \cite[Theorem 3.2]{yu_coarse_baum_connes_conj} and uses crucially the homotopy invariance of uniform $K$-homology that we developed in the first part of this thesis.

Let us now relate the uniform, coarse Baum--Connes conjecture to the usual Baum--Connes conjecture: let $\Gamma$ be a countable, discrete group and denote by $|\Gamma|$ the metric space obtained by endowing $\Gamma$ with a proper, left-invariant metric. Then $|\Gamma|$ becomes a proper, uniformly discrete metric space with coarsely bounded geometry. Note that we can always find such a metric and that any two of such metrics are quasi-isometric. If $\Gamma$ is finitely generated, an example is the word metric.

\Spakula proved in \cite[Corollary 10.3]{spakula_uniform_k_homology} the following equivalence of the uniform coarse Baum--Connes conjecture with the usual one: let $\Gamma$ be a torsion-free, countable, discrete group. Then the uniform coarse assembly map
\[\mu_u \colon \lim_{d \to \infty} K_\ast^u(P_d|\Gamma|) \to K_\ast(C_u^\ast|\Gamma|)\]
is an isomorphism if and only if the Baum--Connes assembly map
\[\mu \colon K_\ast^\Gamma(\underline{E}\Gamma; \ell^\infty(\Gamma)) \to K_\ast(C_r^\ast(\Gamma, \ell^\infty(\Gamma)))\]
for $\Gamma$ with coefficients in $\ell^\infty(\Gamma)$ is an isomorphism. For the definition of the Baum--Connes assembly map with coefficients the unfamiliar reader may consult the original paper \cite[Section 9]{baum_connes_higson}. Furthermore, the equivalence of the usual (i.e., non-uniform) coarse Baum--Connes conjecture with the Baum--Connes conjecture with coefficients in $\ell^\infty(\Gamma, \IK)$ was proved by Yu in \cite[Theorem 2.7]{yu_baum_connes_conj_coarse_geom}.

\Spakula mentioned in \cite[Remark 10.4]{spakula_uniform_k_homology} that the above equivalence does probably also hold without any assumptions on the torsion of $\Gamma$, but the proof of this would require some degree of homotopy invariance of uniform $K$-homology. So again we may utilize our proof of the homotopy invariance of uniform $K$-homology from the first part of this thesis and may therefore drop the assumption about the torsion of $\Gamma$:

\begin{thm}\label{thm:BC_equiv_uniform_coarse}
Let $\Gamma$ be a countable, discrete group.

Then the uniform coarse assembly map
\[\mu_u \colon \lim_{d \to \infty} K_\ast^u(P_d|\Gamma|) \to K_\ast(C_u^\ast|\Gamma|)\]
is an isomorphism if and only if the Baum--Connes assembly map
\[\mu \colon K_\ast^\Gamma(\underline{E}\Gamma; \ell^\infty(\Gamma)) \to K_\ast(C_r^\ast(\Gamma, \ell^\infty(\Gamma)))\]
for $\Gamma$ with coefficients in $\ell^\infty(\Gamma)$ is an isomorphism.
\end{thm}

\section{Obstruction against positive scalar curvature}\label{sec:obstructions_psc}

Recall the following obstruction against uniformly\footnote{This means that $\kappa(x) > c > 0$ for all $x \in M$ and some fixed constant $c > 0$, where $\kappa$ is the scalar curvature of $M$.} positive scalar curvature for non-compact manifolds: if $M$ is a complete, $m$-dimensional Riemannian spin manifold with uniformly positive scalar curvature, then $\mu([M]) = [0] \in K_m(C^\ast(M))$, where $[M] \in K_m(M)$ denotes the fundamental class of $M$ and $\mu$ the coarse assembly map. A proof of this may be found in, e.g., \cite[Proposition 12.3.7]{higson_roe}.

We will prove now an analogue of this obstruction for the uniform theory:

\begin{prop}\label{prop:obstruction_psc}
Let $M$ be an $m$-dimensional spin manifold of bounded geometry and denote by $[M] \in K_m^u(M)$ its fundamental class\footnote{Defined analogously as the \spinc fundamental classes: it is the uniform $K$-homology class of the Dirac operator $\slashed{D}$ associated to the spin structure of $M$.}. If $M$ has uniformly positive scalar curvature, then $\mu_u([M]) = [0] \in K_m(C_u^\ast(M))$.
\end{prop}

\begin{proof}
The proof is completely analogous to the one in the non-uniform case. But for the convenience of the reader, we will nevertheless write it down.

The main ingredient is, as usual for spin manifolds, the Lichnerowicz--Weitzenb{\"o}ck formula $\slashed{D}^2 s = \nabla^\ast \nabla s + \tfrac{1}{4} \kappa s$, where $\slashed{D}$ is the Dirac operator associated to the spin structure of $M$, $s$ a section of the spinor bundle, $\nabla^\ast \nabla$ the Laplace operator associated to the Dirac connection, and $\kappa$ the scalar curvature of $M$. From it we conclude the following: if the scalar curvature of $M$ is uniformly positive, then there is a whole interval around $0$ which does not meet the spectrum of $\slashed{D}$.

The uniform $K$-homology class $[M]$ is defined as the class of the uniform Fredholm module $(H, \rho, T)$, where $H := L^2(S)$ with the representation $\rho$ of $C_0(M)$ as multiplication operators, and $T := \chi(\slashed{D})$ for a normalizing function $\chi$. Since $\slashed{D}$ has a spectral gap around $0$, we may choose a normalizing function $\chi$ with $\chi(\lambda) = \pm 1$ for all $\lambda \in \sigma(\slashed{D})$. We conclude $T = T^\ast$ and $T^2 = 1$, i.e., $(H, \rho, T)$ is involutive.

We consider the case where $m$ is odd and use the formal periodicity to reduce to $[M] \in K_{-1}^u(M)$. Since $T$ is involutive, the operator $\frac{T+1}{2}$ is an orthogonal projection, and defines a projection in $\frakD_\rho^u(M) / \frakC_\rho^u(M)$. Combining Paschke duality with Equation \eqref{eq:uniform_K_as_quotient_D*u/C*u} we conclude that $[\frac{T+1}{2}] \in K_0(\frakD_\rho^u(M) / \frakC_\rho^u(M))$ represents $[M] \in K_{-1}^u(M)$ under duality. Equation \eqref{eq:quotients_frakD_D*_equal} states $D_u^\ast(M) / C_u^\ast(M) \cong \frakD^u_\rho(M) / \frakC^u_\rho(M)$ and using the fact that every normalizing function is the uniform limit of normalizing functions with compactly supported distributional Fourier transform, we conclude that $\frac{T+1}{2}$ actually belongs to $D_u^\ast(M)$.\footnote{This follows from Equation \ref{eq:schwartz_function_of_PDO} and the fact that $e^{it\slashed{D}}$ has \emph{finite} propagation.} Now by exactness of the sequence
\[\ldots \to K_0(D^\ast_u(M)) \to K_0(D_u^\ast(M) / C_u^\ast(M)) \stackrel{\mu_u}\to K_1(C_u^\ast(M)) \to \ldots\]
we conclude that $\mu_u([M]) = [0] \in K_1(C_u^\ast(M))$.

The case where $m$ is even is analogous (but note that Paschke duality now has a slightly more complicated form since we have to replace the graded uniform Fredholm module $(H, \rho, T)$ by a balanced\footnote{See Remark \ref{rem:defn_balanced_module} and the definition of the duality map after it.} one).
\end{proof}

We have a comparison map $K_\ast(C_u^\ast(M)) \to K_\ast(C^\ast(M))$ induced by the inclusion $C_u^\ast(M) \to C^\ast(M)$ which forgets the uniformness. It maps the uniform coarse index class $\mu_u([M])$ of $M$ to the coarse index class $\mu([M])$ of it. So we see that the above obstruction $\mu_u([M]) = [0]$ is a priori stronger than the already known obstruction $\mu([M]) = [0]$ since the comparison map may not be injective.

Using the above proposition we may conclude (analogously as in the non-uniform case) that compact aspherical manifolds do not carry metrics of positive scalar curvature, if the uniform coarse Baum--Connes conjecture holds. Due to Theorem \ref{thm:BC_equiv_uniform_coarse} we already know that this conjecture is true for a large class of groups.

\begin{thm}\label{thm:aspherical_no_psc}
Let $M$ be a compact aspherical manifold and assume that the uniform coarse Baum--Connes conjecture holds for $\pi_1(M)$. Then $M$ does not admit a metric of positive scalar curvature.
\end{thm}

\begin{proof}
Suppose that $M$ does admit a Riemannian metric with positive scalar curvature. We consider the universal cover $X$ of $M$ equipped with the pull-back metric. Since $X$ is also a spin manifold, it has a fundamental class $[X] \in K_m^u(X)$, which is not zero. This can be seen as follows: we have a comparison map $K_m^u(X) \to K_m(X)$ defined by just forgetting the uniformness of the Fredholm modules. It is clear that it maps the fundamental class of $X$ in $K_m^u(X)$ to its fundamental class in $K_m(X)$. But for the latter we know that it is non-zero (e.g., by \cite[Lemma 12.2.4]{higson_roe}), so we conclude $[X] \not= [0] \in K_m^u(X)$.

Since $X$ has uniformly positive scalar curvature, the above proposition gives us $\mu_u([X]) = [0] \in K_m(C_u^\ast(X))$, i.e., the uniform coarse assembly map $\mu_u$ would not be injective. A contradiction to the uniform coarse Baum--Connes conjecture.
\end{proof}

\chapter{Further questions}\label{chap:further_questions}

In this last chapter we are going to discuss some questions in relation to the present thesis that were left unanswered. We hope that we may give enough justification for the interest in each of the remaining questions so that they will be used by others for future research.

\section{Properties of uniform \texorpdfstring{$K$}{K}-homology}\label{sec:props_uniform_k_hom}

\Spakula was the first who defined uniform $K$-homology in his PhD thesis \cite{spakula_thesis} and there he proved the important properties that it admits Paschke duality, Mayer--Vietoris sequences and excision, and he constructed the uniform coarse assembly map into the $K$-theory of the uniform Roe algebra.

In the present thesis we have extended the list of developed properties of uniform $K$-homology to include the crucial construction of the external product and thereout arising its (weak) homotopy invariance. Furthermore, we defined the dual theory (uniform $K$-theory) and proved \Poincare duality using our construction of the cap product, and we constructed analytic index maps for uniform $K$-homology, which are compatible with the analytic index maps on the $K$-theory of the uniform Roe algebra via the uniform coarse assembly map.

Of course there are properties left that usual $K$-homology has and where we haven't proved the corresponding uniform analogue. We will give now three examples of such unproven properties.

\subsection*{Slant product}

Recall that the cap product between the $K$-theory and $K$-homology of $C^\ast$-algebras may be generalized to a slant product $\backslash \colon K_p(A) \otimes K^q(A \otimes B) \to K^{q-p}(B)$ (see, e.g., \cite[Exercise 9.8.9]{higson_roe} for its construction). It seems very reasonable that we do have an analogous slant product
\[\backslash \colon K^p_u(X) \otimes K_q^u(X \times Y) \to K^u_{q-p}(Y)\]
between the corresponding uniform theories which generalizes the corresponding cap product for them defined in Section \ref{sec:cap_product}.

\subsection*{Index pairings}

A second example of a property that we have not discussed yet are the index pairings $\langle \largecdot, \largecdot \rangle_\tau \colon K_u^i(X) \otimes K^u_i(X) \to \IR$ defined as $\langle P, T \rangle_\tau := \ind_\tau(P \cap T)$, where of course $X$ has to be amenable. From the Properties \eqref{eq:general_cap_compatibility_module} and \eqref{eq:compatibility_cap_external} of the cap product we get
\[ \langle P \otimes Q, T \rangle_\tau = \langle P, Q \cap T \rangle_\tau \]
whenever both sides are defined, and
\[ \langle P \times Q, S \times T \rangle_\tau = (-1)^{ij} \langle P, S \rangle_\tau \langle Q, T \rangle_\tau\]
for $P \in K_u^i(X)$, $S \in K^u_i(X)$ and $Q \in K_u^j(Y)$, $T \in K^u_j(Y)$. Note that in the latter equation we have to use on the left hand side the F{\o}lner sequence $(U_i \times V_i)_i$ for $\Gamma_X \times \Gamma_Y$ if we have on the right hand side $(U_i)_i$ for a quasi-lattice $\Gamma_X \subset X$ and $(V_i)_i$ for a quasi-lattice $\Gamma_Y \subset Y$. Furthermore, for its proof we need the formula $\ind_\tau (S \times T) = \ind_\tau (S) \cdot \ind_\tau (T)$. Since this formula might be of independent interest, let us state it as a proposition:

\begin{prop}
We have $\ind_\tau (S \times T) = \ind_\tau (S) \cdot \ind_\tau (T)$ where we use on the left hand side the product F{\o}lner sequence of the F{\o}lner sequences of the right hand side.
\end{prop}

The proof of the above proposition is analogous to the non-uniform case (see, e.g., \cite[Proposition 9.7.1]{higson_roe}) but more technical: we have to use the ideas developed in Section \ref{sec:index_maps_K_hom}, i.e., we have to see that the difference between the left and right hand sides is for each $i$ concentrated on the boundary of $U_i \times V_i$ due to finite propagation. Since the size of the boundary is uniformly bounded and the difference on the boundary between the left and right hand side is uniformly bounded against the number of elements of the boundary, we conclude that the difference vanishes in the limit under $\tau$ due to the amenability of the sequence $(U_i \times V_i)_i$.

At last, from Equation \ref{eq:cap_twisted_Dirac} we get
\[\langle E, D \rangle_\tau = \ind_\tau (D_E)\]
for an operator $D$ of Dirac type and a vector bundle $E$, where $D_E$ is the twisted operator.

\subsection*{$6$-term exact sequence}

As a last example we will mention here the $6$-term exact sequence for uniform $K$-homology. Though \Spakula has already defined relative uniform $K$-homology in \cite[Section 5]{spakula_uniform_k_homology}, he discussed there only Mayer--Vietoris sequences and excision. So the existence of a $6$-term exact sequence associated to a closed subspace $A \subset X$ remains unproved. Though the proof of the existence of this sequence may be easily derived using Paschke duality and the $6$-term exact sequence for $K$-theory, to show that the boundary maps are compatible with the index pairing would be certainly harder (see, e.g., \cite[Section 8.7]{higson_roe} for the corresponding proof in the case of usual $K$-homology).

\section{Global formulation of Roe's index theorem}

We will first explain why the global proof of the Atiyah--Singer index theorem as presented in \cite[Sections 11.4 \& 11.5]{higson_roe} does probably not generalize to our non-compact setting. Let us very briefly recall this proof: given a compact manifold $M$, crushing it to a point we get a map $K_0(M) \to K_0(\pt) \cong \IZ$. If $M$ is even-dimensional and \spinc we have a distinguished class $[D] \in K_0(M)$ of the Dirac operator associated to the \spinc structure. So we get the analytic index map $K^0(M) \to \IZ$ via $[E] \mapsto [D_E] \in K_0(\pt) \cong \IZ$. On the other hand, embedding $M$ into Euclidean space, applying a wrong way map for $K$-theory, and using Bott periodicity we get the topological index map. The Atiyah--Singer index theorem now asserts that these maps are identical, and to get the well-known cohomological formula for the topological index we have to apply the Chern character to the topological index map.

Now if $M$ is non-compact, we have two problems: we have to construct wrong way maps, but this is not always possible, and we can't use Bott periodicity as in the compact case. In order to construct the wrong way maps, we need an isometric embedding of $M$ into Euclidean space such that the normal bundle of $M$ has a uniformly thick disk in each fiber. It is not at all clear how to do this and it is, as mentioned, not always possible. We refer to the corresponding question \cite{MO_embedding_uniformly_thick_disk} at MathOverflow for a short discussion of this problem.

But of course there are manifolds which admit such an isometric embedding into Euclidean space with a uniformly thick disk in the normal bundle. So let us discuss how we would proceed with the proof of the index theorem in this case: in the compact case we would use Bott periodicity to deduce that there is essentially one classical index problem in each $\IR^n$ and for that one we would have to compute both the analytic and the topological index and see that they coincide. But if $M$ is non-compact, we do not have such a conclusion from Bott periodicity. And it is even worse: on each $\IR^n$ there are uncountably many such index problems (corresponding to different choices of \Folner exhaustions), so we would have to compute the analytic and the topological index maps on all these examples. A task that seems to be impossible to perform.

But we could probably circumvent this problem by just applying Roe's index theorem, since it is already proved! This means that we would not get a global proof of the index theorem since we would have to put its statement already into it, but what we would get is a new formulation of the index theorem: the global one which is quite similar to the well-known formulation for compact manifolds.

Let us briefly explain why we call it the global formulation of the index theorem: in the present case, in order to compute the index form of a Dirac operator $D$, we have to use the asymptotic expansion of the integral kernel of the operator $e^{-t D^2}$. This gives us a way of computing the index class locally, i.e., we can compute the differential form defining the index class of $D$ in neighbourhoods of points. But the usual formulation of the Atiyah--Singer index theorem defines the index class out of the symbol class $[\sigma(D)] \in K_{\mathrm{cpt}}^0(TM)$ of the operator $D$ in the compactly supported $K$-theory of the tangent bundle of $M$ via applying the Chern character to the symbol class, i.e., a global definition of the index class. So if we would succeed in proving all the above discussed, we would get a similar global formulation of Roe's index theorem. Though of course a formulation coinciding completely with the usual one would require us to reinterpret the $K$-theory $K_{\mathrm{cpt}}^0(TM)$ in a suitable way, i.e., find a suitable receptacle for the classes of the symbols of operators (note that if $M$ is not compact, then the symbol $\sigma(D)$ would not give us a \emph{compactly supported} class, i.e., we do not now where $[\sigma(D)]$ should live).

\section{Homological Chern characters}

\Spakula asked in \cite{spakula_thesis} whether there exists a Chern character map from uniform $K$-homology $K^u_\ast(X)$ to the uniformly finite homology $H^{\mathrm{uf}}_\ast(X)$ of Block and Weinberger. A property that we usually want a Chern character to fulfill is that it induces an isomorphism modulo torsion. But such a map $K^u_\ast(X) \to H^{\mathrm{uf}}_\ast(X)$ as asked for by \Spakula could never do this: in fact, the uniformly finite homology of Block and Weinberger is a coarse homology theory, i.e., the homology groups of a compact space are the groups of a point. But the uniform $K$-homology coincides for a compact space with usual $K$-homology which is certainly not trivial, resp. consists certainly not only of torsion elements, in almost all cases.

\subsection*{Uniform homological Chern character}

What we actually have (instead of the map that \Spakula asked for) is a uniform homological Chern character $K_\ast^u(M) \to H_\ast^\infty(M)$, where the latter groups denote simplicial $L^\infty$-homology, in the case that $M$ is a \spinc manifold of bounded geometry. We can define this Chern character in the following way: since $M$ is spin$^c$, it admits \Poincare duality $K_{m-\ast}^u(M) \cong K^{\ast}_u(M)$, where $m$ is the dimension of the manifold $M$. Then we apply the uniform cohomological Chern character $K^\ast_u(M) \to H^\ast_{b, \mathrm{dR}}(M)$, and follow this by \Poincare duality $H^\ast_{b, \mathrm{dR}}(M) \cong H_{m-\ast}^\infty(M)$ (\cite[Theorem 4]{attie_block_1}). Since we have proved in this thesis that the Chern character $K_u^\ast(M) \to H^\ast_{b, \mathrm{dR}}(M)$ induces an isomorphism modulo torsion, we conclude that the homological Chern character $K_\ast^u(M) \to H_\ast^\infty(M)$ also is an isomorphism modulo torsion.

\subsection*{Geometric picture of uniform $K$-homology}

Of course, it is desirable to have a more direct definition of the homological Chern character, i.e., without going through uniform $K$-theory via \Poincare duality (since this also leads to the undesirable fact that the homological Chern character depends on a choice of \spinc structure on $M$). In the compact case, such a direct definition can be given using a geometric picture of $K$-homology. It was developed by Baum and Douglas in \cite{baum_douglas} and a proof that it indeed coincides for finite CW-complexes with analytic $K$-homology is given by Baum, Higson and Schick in \cite{baum_higson_schick}. So one could try to give a geometric picture of uniform $K$-homology and then use this to define a uniform homological Chern character $\ch^u_\ast \colon K_\ast^u(M) \to H_\ast^\infty(M)$ which should of course be an isomorphism modulo torsion.

\subsection*{Uniform coarse Chern characters}

Inspired by \v{S}pakula's question if there is a Chern character map $K_\ast^u(M) \to H_\ast^{\mathrm{uf}}(M)$ (which probably not exists), we may instead ask if there is a uniform coarse Chern character map $\ch_\ast^{\mathrm{uf}} \colon K_\ast(C^\ast_u(M)) \to H_\ast^{\mathrm{uf}}(M)$ which induces an isomorphism modulo torsion? Note that here both the domain and target of this map are invariant under coarse equivalences of spaces.

The existence of such a map is claimed in \cite[First paragraph on page 562]{block_weinberger_large_scale} and an analogous map on \emph{co}homology was constructed in \cite[Section 4.2]{roe_coarse_cohomology}.

Suppose that there is such a map and that the diagram
\[\xymatrix{K_\ast^u(M) \ar[rr]^{\mu_u} \ar[d]_{\ch_\ast^u} & & K_\ast(C_u^\ast(M)) \ar[d]^{\ch_\ast^{\mathrm{uf}}} \\ H_\ast^\infty(M) \ar[rr] & & H_\ast^{\mathrm{uf}}(M)}\]
commutes and where we have also assumed that we have a uniform homological Chern character $\ch^u_\ast \colon K_\ast^u(M) \to H_\ast^\infty(M)$ as discussed above. We know that the lower map is an isomorphism if $M$ is uniformly contractible (\cite[Proposition 1]{attie_block_1}) and we assume that both Chern character maps are isomorphisms modulo torsion. Then we would conclude that the uniform coarse assembly map is an isomorphism modulo torsion, too! Note that all the versions of the Baum--Connes conjecture (the usual one, the coarse one, and the uniform versions of them) are either not known in full generality yet or are even false due to the existence of counterexamples. So the above result would be a remarkable theorem.

\section{Uniform coarse Baum--Connes conjecture}

We have seen in Proposition \ref{prop:obstruction_psc} that we have an obstruction against positive scalar curvature in the $K$-theory of $C_u^\ast(M)$, i.e., $\mu_u([M]) = [0] \in K_m(C_u^\ast(M))$ if $M$ is spin and has uniformly positive scalar curvature. This obstruction is a priori stronger than the known one $\mu([M]) = [0] \in K_m(C^\ast(M))$ since the comparison map $K_\ast(C_u^\ast(M)) \to K_\ast(C^\ast(M))$ maps the former class to the latter and the comparison map may not be injective. In the proof of Theorem \ref{thm:aspherical_no_psc} we have used the comparison map $K_\ast^u(M) \to K_\ast(M)$ and we have a commutative diagram
\[\xymatrix{K_\ast^u(M) \ar[rr]^{\mu_u} \ar[d] & & K_\ast(C_u^\ast(M)) \ar[d] \\ K_\ast(M) \ar[rr]^{\mu} & & K_\ast(C^\ast(M))}\]
Now two questions emerge: firstly, are there any sufficient conditions on a non-compact manifold $M$ such that the comparison maps will be injective, resp. surjective, and secondly, can we find an example of a spin manifold $M$ such that $\mu_u([M]) \not= [0]$, but $\mu([M]) = [0]$? Such an example would show that the obstruction against positive scalar curvature in $K_m(C_u^\ast(M))$ is indeed stronger than the one in $K_m(C^\ast(M))$.

Let us turn our attention to the conclusions that we get from the uniform coarse Baum--Connes conjecture. Using the above discussed obstrucion against positive scalar curvature, we have derived in Theorem \ref{thm:aspherical_no_psc} that compact aspherical manifolds do not carry metrics of positive scalar curvature.

Furthermore, we may deduce from the uniform coarse Baum--Connes conjecture the analytic Novikov conjecture: if $M$ is a finite\footnote{The author does not know whether we may drop the assumption that $M$ is a finite complex in this statement. The arguments from \cite[Chapter 12.6]{higson_roe}, where the analytic Novikov conjecture is derived from the coarse Baum--Connes conjecture, do not apply without the finiteness assumption. Since we use arguments similar to those to deduce the analytic Novikov conjecture from the uniform coarse Baum--Connes conjecture, the author does not know whether there is another deduction without the finiteness assumption that applies here in the uniform case.} aspherical complex with fundamental group $G$ and if the universal cover of $M$ satisfies the uniform coarse Baum--Connes conjecture, then the assembly map $K_p(M) \to K_p(C_r^\ast(G))$ is injective, i.e., $G$ satisfies the analytic Novikov conjecture. To prove this we may use a decent principle similar to the one for the non-uniform case (concretely, we may arrange the arguments given in \cite[Chapter 12.6]{higson_roe} so that they work for the uniform coarse Baum--Connes conjecture).

We see that we may derive from the uniform coarse Baum--Connes conjecture the same conclusions as from the usual coarse conjecture, so it is of the same interest to prove it. In Theorem \ref{thm:BC_equiv_uniform_coarse} we have proved that the uniform coarse Baum--Connes conjecture for a group $\Gamma$ is equivalent to the usual Baum--Connes conjecture for $\Gamma$ with coefficients in $\ell^\infty(\Gamma)$. So the uniform coarse Baum--Connes conjecture is true for a large class of groups.

For the usual coarse Baum--Connes conjecture there are also ways of proving it without using the Baum--Connes conjecture: Yu proved it in \cite{yu_finite_asymptotic_dimension} for proper metric spaces with finite asymptotic dimension, and in \cite{yu_embedding_Hilbert_space} for spaces which admit a uniform embedding into a Hilbert space. Now of course the question emerges by itself whether we may arrange Yu's proofs so that they hold true for the uniform coarse Baum--Connes conjecture.

In \cite{gong_wang_yu} Gong, Wang and Yu constructed a maximal coarse assembly map $K_\ast(X) \to K_\ast(C_{\mathrm{max}}^\ast(X))$ into the $K$-theory of the maximal Roe algebra and investigated the corresponding version of the coarse Baum--Connes conjecture, and in \cite{oyono_oyono_yu} Oyono-Oyono and Yu showed that this maximal version of the coarse Baum--Connes conjecture for a group $\Gamma$ is equivalent to the maximal version (i.e., using the maximal crossed product on the right hand side) of the Baum--Connes conjecture for $\Gamma$ with certain coefficients. Furthermore, in \cite{spakula_willett_maximal_uniform_roe_algebra} \Spakula and Willett investigated the relationship between the maximal and the reduced Roe algebra and also between the corresponding uniform versions. So the following question now emerges: is there a maximal version of the uniform coarse Baum--Connes conjecture and can we also prove that it is equivalent to the maximal version of the Baum--Connes conjecture with certain coefficients?

At last, let us get to the new formulation of the Baum--Connes conjecture and whether there is something analogous for the uniform coarse conjecture. In \cite{baum_guentner_willett} Baum, Guentner and Willett proposed a new formulation of the Baum--Connes conjecture by changing the crossed product that is used on the right hand side of it. This reformulation emerged out of the fact that there are counterexamples to the usual Baum--Connes conjecture that are confirming examples if one changes the reduced crossed product to the maximal crossed product on the right hand side of the conjecture (see the introduction of \cite{baum_guentner_willett} for a discussion of this and for references to all the involved papers).

So the important question now is whether we may define an intermediate (between maximal and reduced) version of the (uniform) Roe algebra, formulate a corresponding (uniform) coarse Baum--Connes conjecture and show that it is equivalent to the new formulation of the Baum--Connes conjecture with certain coefficients.

\section{Pseudodifferential operators of order zero}

We have seen in Proposition \ref{prop:basics_symbol_maps} that the principal symbol map $\sigma^k$ is an isomorphism of vector spaces $\Psi \mathrm{DO}_?^{k-[1]}(E,F) \cong \Symb^{k-[1]}(E,F)$ for all $k \in \IZ$ and vector bundles $E$ and $F$ of bounded geometry. For the case $k = 0$ and $E = F$ we furthermore know from Proposition \ref{prop:PsiDOs_filtered_algebra} that $\Psi \mathrm{DO}_?^{0-[1]}(E)$ is a commutative algebra and that $\sigma^0$ is an isomorphism of algebras.

In the case that the manifold $M$ is compact, it is known that $\sigma^0$ is continuous against the quotient norm\footnote{Which is induced from the operator norm on $\Psi \mathrm{DO}_?^0(E) \subset \IB(L^2(E))$. Since for $M$ compact we have $\overline{\Psi \mathrm{DO}_?^{-1}(E)} = \IK(L^2(E))$, the quotient norm on $\Psi \mathrm{DO}_?^{0-[1]}(E)$ is called the \emph{essential norm}.} on $\Psi \mathrm{DO}_?^{0-[1]}(E)$ and therefore $\sigma^0$ induces an isomorphism of $C^\ast$-algebras $\overline{\Psi \mathrm{DO}_?^{0-[1]}(E)} \cong \overline{\Symb^{0-[1]}(E)}$. So the question arises whether this is also the case if $M$ is non-compact.

To show this we would have to compare the quotient norms on $\Psi \mathrm{DO}_?^{0-[1]}(E)$ and on $\Symb^{0-[1]}(E)$. The first to prove similar results in the compact case were Seeley in \cite[Lemma 11.1]{seeley} and Kohn and Nirenberg in \cite[Theorem A.4]{kohn_nirenberg}, and two years later H{\"o}rmander provided in \cite[Theorem 3.3]{hormander_ess_norm} a proof of this for his class $S_{\rho, \delta}^0$ with $\delta < \rho$ of pseudodifferential operators of order $0$. Maybe one of these proofs generalizes to our case of pseudodifferential operators on open manifolds.

Provided that we could show $\overline{\Psi \mathrm{DO}_?^{0-[1]}(E)} \cong \overline{\Symb^{0-[1]}(E)}$, the next natural task would be to characterize this closures, especially the left-hand side. We know from Proposition \ref{prop:IU(E)_dense_Cu*(E)} that $\Psi \mathrm{DO}_u^{-\infty}(E) \subset \Psi \mathrm{DO}_u^{-1}(E) \subset C_u^\ast(E)$ are dense inclusions, i.e., the closure of $\Psi \mathrm{DO}_u^{-1}(E)$ is the $C^\ast$-algebra generated by the uniformly locally compact operators on $E$ with finite propagation.

But we do not know how the closure of the pseudodifferential operators of order $0$ looks like. Though we know from Proposition \ref{prop:PDO_order_0_l-uniformly-pseudolocal} that it is contained in $D_u^\ast(E)$, the $C^\ast$-algebra generated by the uniformly pseudolocal operators on $E$ with finite propagation, it is just a sub-$C^\ast$-algebra of it. Note that to the knowledge of the author, the question of characterizing the closure of $\Psi \mathrm{DO}_u^0(E)$ is even open for compact manifolds. Moreover, remember that it is important for this question to distinguish between the pseudodifferential operators as we have defined them (corresponding to H{\"o}rmanders class $S_{1,0}^0$) and the ``classical'' pseudodifferential operators $\Psi \mathrm{DO}^0_{\mathrm{cl}}(M)$ (the ones with a homogeneous symbol), since Melo showed in \cite{melo} that the closure of $\Psi \mathrm{DO}^0_{\mathrm{cl}}(M)$ does in general not contain all of $\Psi \mathrm{DO}_u^0(M)$.

A clue why we do not have that the closure of $\Psi \mathrm{DO}_u^0(E)$ is the whole of $D_u^\ast(E)$ may give us Example \ref{ex:reverse_implications_l-univ-pseudoloc-traceable}: there we showed that the closure of the operators fulfilling Points 2 and 2' from Lemma \ref{lem:kasparov_lemma_uniform_traceable} is strictly included in the closure of the operators fulfilling Points 3 and 3' of that lemma, and applying Lemma \ref{lem:D_tr_dense} to operators of finite propagation, we see that the operators fulfilling Points 3 and 3' are dense in $D_u^\ast(E)$. So if we could show that pseudodifferential operators of order $0$ have Properties 2 and 2' of Lemma \ref{lem:kasparov_lemma_uniform_traceable}, this would justify why their closure is not $D_u^\ast(M)$. Note that this arguments are also applicable if the manifold $M$ is compact since the difference between the properties in Lemma \ref{lem:kasparov_lemma_uniform_traceable} is a local one.

Since we are already discussing the different properties in Lemma \ref{lem:kasparov_lemma_uniform_traceable}, it would be also interesting to know whether the closure of the operators having Properties 1 and 1' includes all of the operators having Properties 2 and 2' (or if maybe Properties 1 and 1' are even equivalent to Properties 2 and 2'), and if we can somehow characterize the closure of the operators with Properties 2 and 2' (for a possible way how to do this, see the last paragraph of Example \ref{ex:reverse_implications_l-univ-pseudoloc-traceable}).

\section{Equality of \texorpdfstring{$\IU(E)$}{IU(E)} and \texorpdfstring{$\C(E)$}{IC(E)}}\label{sec:equality_smooth_uniform_roe_algebras}

Recall that $\IU(E)$ is defined as consisting of all quasilocal smoothing operators and that $\C(E)$ is defined as the \Frechet closure of all finite propagation smoothing operators. We clearly have $\C(E) \subset \IU(E)$ but it is an open question whether equality hold here. It might be conjectured that this only holds if the space in question has certain ``finite dimensionality'' properties like Property A or having finite asymptotic dimension. It is of course totally desirable to prove such a theorem.

Another natural question is whether the analytic results of Section \ref{sec:functions_of_PDOs}, which are proved only for operators $P \in \Psi \mathrm{DO}^{k \ge 1}(E)$ do also hold for operators from $\Psi \mathrm{DO}_u^{k \ge 1}(E)$? The main problem here is to prove Lemma \ref{lem:exp(itP)_quasilocal} for these operators, i.e., to prove that if $P \in \Psi \mathrm{DO}_u^{k \ge 1}(E)$ is symmetric and elliptic, then $e^{itP}$ may be approximated by operators of finite propagation.

Note that there are of course some related problems: for instance we could change the definition of the uniform Roe algebra $C_u^\ast(E)$ to consist of \emph{quasilocal} uniformly locally compact operators. Analogously to $\C(E) \subset C_u^\ast(E)$ being dense using the usual definition of $C_u^\ast(E)$, we can ask whether $\IU(E) \subset C_u^\ast(E)$ is dense using the new definition of the uniform Roe algebra. More generally, we may ask if the results of Section \ref{sec:uniformity_PDOs} do hold for $\Psi \mathrm{DO}^{k \le 0}(E)$ (without the subscript ``u'') if we change the definitions of $C_u^\ast(E)$ and $D_u^\ast(E)$.

In Corollary \ref{cor:natural_receptacle_iso} we showed that the natural map $K_\ast(\C(E)) \to K_\ast(\C(M))$ is an isomorphism by comparing both $K$-theories with the one of the uniform Roe algebra. Now we of course ask if the natural map $K_\ast(\IU(E)) \to K_\ast(\IU(M))$ is an isomorphism. That would follow from the above where we ask if $\IU(E) \subset C_u^\ast(E)$ is dense using the changed definition for $C_u^\ast(E)$.

If we change the definition of $C_u^\ast(X)$, we also get a uniform coarse Baum--Connes assembly map $\mu_u \colon K_\ast^u(X) \to K_\ast(C_u^\ast(X))$ mapping into the $K$-theory of the changed uniform Roe algebra. Now we may ask if this is an isomorphism, and especially, we may ask if this is an isomorphism in cases where the usual uniform coarse Baum--Connes assembly map might not be an isomorphism (if the latter happens at all).

\section{Geometric optics equation}

The main technical part in the proof of the Theorem \ref{thm:elliptic_symmetric_PDO_defines_uniform_Fredholm_module} that a pseudodifferential operator defines a class in uniform $K$-homology was to show that $\chi(P)$ is uniformly pseudolocal for $\chi$ a normalizing function. In Proposition \ref{prop:PDO_order_0_l-uniformly-pseudolocal} we have shown that pseudodifferential operators of order $0$ are automatically uniformly pseudolocal. So if we could show that the operator $\chi(P)$ is a pseudodifferential operator of order $0$, the proof of Theorem \ref{thm:elliptic_symmetric_PDO_defines_uniform_Fredholm_module} would be trivial.

For a compact manifold $M$ there are quite a few proofs that under certain conditions functions of pseudodifferential operators are again pseudodifferential operators: the first one to show such a result was seemingly Seeley in \cite{seeley_complex_powers}, where he proved it for complex powers of elliptic classical pseudodifferential operators. It was then considerably extended by Strichartz in \cite{strichartz} from complex powers to symbols in the sense of Definition \ref{defn:symbols_on_R}, and from classical operators to all of H{\"o}rmander's class $S^k_{1,0}(M)$. And last, let us mention the result \cite[Theorem 8.7]{dimassi_sjostrand} of Dimassi and Sj{\"o}strand for $h$-pseudodifferential operators in the semi-classical setting.

Now if we want to establish similar results in our setting, we get quite fast into trouble: e.g., the proof of Strichartz does not generalize to non-compact manifolds. He uses crucially that on compact manifolds we may diagonalize elliptic operators, which is not at all the case on non-compact manifolds (consider, e.g., the Laplace operator on Euclidean space). Looking for a proof that may be generalized to the non-compact setting, we stumble over Taylor's result from \cite[Chapter XII]{taylor_pseudodifferential_operators}. There he proves a result similar to Strichartz' but with quite a different proof, which may be possibly generalized to non-compact manifolds. An evidence for this is given by Cheeger, Gromov and Taylor in \cite[Theorem 3.3]{cheeger_gromov_taylor}, since this is exactly the result that we want to prove for our pseudodifferential operators, but in the special case of the operator $\sqrt{- \Delta}$, and their proof is a generalization of the one from the above cited book of Taylor. So it seems quite reasonable that we may probably extend the result of Cheeger, Gromov and Taylor to all pseudodifferential operators in our sense.

Let us briefly explain what this has to do with the geometric optics equation: this equation is treated by Taylor in \cite[Chapter VIII]{taylor_pseudodifferential_operators} for a compact manifold $M$ and it is one of the main ingredients in his proof that functions of pseudodifferential operators are again pseudodifferential operators. So if we want to extend Taylor's result to our pseudodifferential operators on non-compact manifolds, we will first have to solve the geometric optics equation on them. Since this is probably in itself a more interesting problem than the one about functions of pseudodifferential operators, we get a strong motivation for executing the above discussed ideas.

\section{Partitioned manifold index theorem}\label{sec:partitioned_manifold}

In order to detect elements in the $K$-theory $K_0(\C(M))$ of the smooth uniform Roe algebra $\C(M)$ of $M$, Roe constructed in \cite[Theorem 6.7]{roe_index_1} traces on $\C(M)$ via \Folner exhaustions, and in Section \ref{sec:index_maps_K_hom} we have constructed corresponding analytic index maps on the uniform $K$-homology $K_0^u(M)$ of $M$. This analytic maps are one of the main ingredients in the formulation of our Index Theorem \ref{thm:index_thm}. Since this analytic maps are constructed only for the even $K$-theory of the Roe algebra, resp. the even uniform $K$-homology, we have the restriction to evenly multigraded operators and to even-dimensional manifolds in our index theorem.

The question whether there is a similar index theorem in the odd case immediately arises. The answer is yes, but one has to find a way how to detect elements in the odd $K$-theory of the uniform Roe algebra. In \cite{roe_partitioning_non_compact_manifolds} Roe showed how one can do this for $K_1$ of $\C(M)$ using partitionings of manifolds and announced it for the general odd case in \cite{roe_exotic_cohomology}. In \cite{roe_coarse_cohomology} complete proofs for all even and odd cases are given, but now in the coarse category. The rough category was treated by Mavra in his PhD thesis \cite{mavra} written under the supervision of Roe.

So, of course, we may ask whether Roe's Index Theorem for the odd $K$-theory of $\C(M)$ connects to the theory that we have developed in this thesis. The main question would be if we can mimic Roe's constructions for detecting elements in $K_1^u(M)$, i.e., if we can construct analytic index maps $K_1^u(M) \to \IR$ that relate under the uniform coarse assembly map $\mu_u \colon K_\ast^u(M) \to K_\ast(\C(M))$ to ones constructed by Roe. Since we have proved \Poincare duality $K^\ast_u(M) \cong K_{m - \ast}^u(M)$ for all $\ast \in \IZ$, the extension of our index theorem to the odd case would immediately follow.

\section{Manifolds with boundary}

In our exposition we have not touched any manifolds with boundary. But we know that in the compact case there is an important generalization of the Atiyah--Singer index theorem to manifolds with boundary involving the so-called $\eta$-invariant. This version of the index theorem for compact manifolds with boundary is called the Atiyah--Patodi--Singer index theorem and was introduced in \cite{atiyah_patodi_singer_1}.

Of course the question whether such a theorem may also be proven in the non-compact case immediately arises. Firstly, note that there is a notion of manifolds of bounded geometry and with boundary: it is due to Schick from \cite{schick_bounded_geometry_boundary}. Secondly, note that our proof of the index theorem for pseudodifferential operators relies on the corresponding index theorem for Dirac operators by proving that every uniform $K$-homology class may be represented by a twisted Dirac operator, if the manifold is spin$^c$. So in order to do the same reduction in the case with boundary, we would have to prove the existence of a relative version of the $6$-term exact sequence for uniform $K$-theory and uniform $K$-homology for manifolds with boundary, and to prove a corresponding relative version of \Poincare duality. And we would of course need to prove the index theorem for manifolds with boundary somehow directly for Dirac operators.

A nice prove of the index theorem for manifolds with boundary for Dirac operators was given by Melrose in \cite{melrose_APS}. He invented the so-called $b$-calculus, a calculus for pseudodifferential operators on manifolds with boundary, and derived the Atiyah--Patodi--Singer index theorem from it via the heat kernel approach. So it would be desirable to extend his $b$-calculus to open manifolds with boundary (in the same way as we extended the calculus of pseudodifferential operators to open manifolds in a fruitful way) and then prove the extension of the Atiyah--Patodi--Singer index theorem to manifolds with boundary and of bounded geometry. Note that Roe's proof of his index theorem for open manifolds does also rely on the heat kernel approach, i.e., there is a real chance that we may generalize Melrose's proof to open manifolds with boundary and of bounded geometry.

\appendix

\chapter{\texorpdfstring{$K$}{K}-theory of dense subalgebras}\label{chapter:appendix_A}

In this appendix we will revisit the fact that certain dense $^\ast$-subalgebras of $C^\ast$-algebras (the so-called \emph{local $C^\ast$-algebras}) do define the same operator $K$-theory groups as their completions and then we will collect some sufficient, in many cases easy to prove conditions for a dense subalgebra to be local. Since two of these condition require the dense subalgebra to be a \Frechet subalgebra, we will end this appendix with a lemma which states that Phillips $K$-theory for $m$-convex \Frechet algebras coincides with the usual operator $K$-theory.

\begin{defn}[Local $C^\ast$-algebras, cf. {\cite[Definition 3.1.1]{blackadar}}]\label{defn:local_Cstar_algebra}
A normed $^\ast$-algebra $A$ is a \emph{local $C^\ast$-algebra}, if
\begin{itemize}
\item its completion $\overline{A}$ is a $C^\ast$-algebra,
\item $A$ is closed under holomorphic functional calculus, i.e., for all $a \in A$ and any holomorphic function $f$ on a neighbourhood of the spectrum of $a$ in the completion $\overline{A}$ (with $f(0) = 0$ if $A$ does not have a unit) the element $f(a)$, a priori an element of $\overline{A}$, lies in $A$, and
\item all matrix algebras over $A$ are also closed under holomorphic functional calculus.
\end{itemize}
\end{defn}

The importance of local $C^\ast$-algebras lies in the fact that in this case the inclusion of $A$ into its completion $\overline{A}$ induces an isomorphism on $K$-theory (e.g., Sections 5.1 and 8.1 in the book \cite{blackadar}):

\begin{lem}\label{lem:loc_algebra_same_k_theory}
Let $A$ be a local $C^\ast$-algebra. Then the inclusion $A \hookrightarrow \overline{A}$ induces isomorphisms $K_\ast(A) \cong K_\ast(\overline{A})$.
\end{lem}

Since it is annoying to check whether all matrix algebras over $A$ are also closed under holomorphic functional calculus, one wishes for a sufficient condition on $A$ which implies this. Fortunately, there is one if $A$ is a \Frechet subalgebra. Let us first define that notion before we state the result.

\begin{defn}[\Frechet algebras and subalgebras]\label{defn:Frechet_subalgebras}
An algebra $A$ is called a \emph{\Frechet algebra}, if $A$ is a \Frechet space\footnote{That is to say, a topological vector space whose topology is Hausdorff and induced by a countable family of semi-norms such that it is complete with respect to this family of semi-norms.} and multiplication is jointly continuous.

If the seminorms defining the \Frechet topology are submultiplicative, then $A$ is called an \emph{$m$-convex \Frechet algebra}.\footnote{Note that in this case multiplication becomes automatically jointly continuous.}

Now let $B$ be a normed algebra and $A \subset B$ a subalgebra of it. Then we will call $A$ an \emph{($m$-convex) \Frechet subalgebra} of $B$, if we can find some \Frechet topology on $A$ making it an ($m$-convex) \Frechet algebra and such that the \Frechet topology is at least as fine as the induced norm topology from $B$.
\end{defn}

\begin{lem}[{\cite[Corollary 2.3]{schweitzer}}]\label{lem:matrix_algebras_holomorphically_closed}
Let $\overline{A}$ be a $C^\ast$-algebra and $A$ a dense \Frechet subalgebra of it.

Then, if $A$ is closed under holomorphic functional calculus, it follows that the same also holds for all matrix algebras over $A$.
\end{lem}

In such a \Frechet subalgebra setting we have also an alternative way of showing that $A$ is closed under holomorphic functional calculus:

\begin{lem}[{\cite[Lemma 1.2]{schweitzer}}]\label{lem:local_algebra_equivalent}
Let $A$ be a dense \Frechet subalgebra of the unital Banach algebra $\overline{A}$ and such that $A$ contains the unit of $\overline{A}$. Then the following are equivalent:
\begin{itemize}
\item $A$ is closed under holomorphic functional calculus.
\item If $a \in A$ is invertible in $\overline{A}$, then it is also invertible in $A$ (i.e., $a^{-1} \in A$ if $a^{-1}$ exists in $\overline{A}$).
\end{itemize}
\end{lem}

To show that a normed algebra is closed under holomorphic functional calculus one has a priori to consider all holomorphic functions defined on some neighbourhood (that depends on the function) of the spectrum of an element. The following result of Schmitt shows that it suffices to consider only power series around $0 \in \IC$ which have a radius of convergence bigger than the norm of the element:

\begin{lem}[cf. {\cite[Theorem 2.1]{schmitt}}]\label{lem:stable_calculus_power_series}
Let $A$ be a normed algebra with the property: for all $a \in A$ and every power series $f$ around $0 \in \IC$ with radius of convergence bigger than $\|a\|$ (and with $f(0) = 0$ if $A$ is non-unital), we have $f(a) \in A$.

Then $A$ is closed under holomorphic functional calculus.
\end{lem}

Since we are already dealing with \Frechet algebras, we might ask ourselves whether the $K$-theory for $m$-convex \Frechet algebras that was developed by Phillips in \cite{phillips} does coincide with its operator $K$-theory. This is indeed the case:

\begin{lem}[{\cite[Corollary 7.9]{phillips}}]\label{lem:k_theory_frechet_algebras_coincide}
Let $A$ be a local $C^\ast$-algebra and additionally an $m$-convex \Frechet subalgebra of its completion $\overline{A}$.

Then Phillips $K$-theory of $A$ coincides with the operator $K$-theory of $A$.
\end{lem}

\bibliography{./Bibliographie_Dissertation}

\providecommand{\bysame}{\leavevmode\hbox to3em{\hrulefill}\thinspace}
\providecommand{\MR}{\relax\ifhmode\unskip\space\fi MR }
\providecommand{\MRhref}[2]{%
  \href{http://www.ams.org/mathscinet-getitem?mr=#1}{#2}
}
\providecommand{\href}[2]{#2}
\begin{thebibliography}{GWY08}

\bibitem[AB98]{attie_block_1}
O.~Attie and J.~Block, \emph{{P}oincar\'e {D}uality for ${L}^p$ {C}ohomology
  and {C}haracteristic {C}lasses}, DIMACS Technical Report (1998), no.~48,
  1--12.

\bibitem[AH61]{atiyah_hirzebruch}
M.~F. Atiyah and F.~Hirzebruch, \emph{{Vector Bundles and Homogeneous Spaces}},
  Proc. {S}ympos. {P}ure {M}ath., {V}ol. {III}, American Mathematical Society,
  1961, pp.~7--38.

\bibitem[AP68]{anselone_palmer}
P.~M. Anselone and T.~W. Palmer, \emph{Collectively compact sets of linear
  operators}, Pacific J. Math. \textbf{25} (1968), no.~3, 417--422.

\bibitem[APS75]{atiyah_patodi_singer_1}
M.~F. Atiyah, V.~K. Patodi, and I.~M. Singer, \emph{{Spectral asymmetry and
  Riemannian geometry. I}}, Math. Proc. Camb. Phil. Soc. \textbf{77} (1975),
  43--69.

\bibitem[AS68]{atiyah_singer_1}
M.~F. Atiyah and I.~M. Singer, \emph{{The Index of Elliptic Operators: I}},
  Ann. Math. \textbf{87} (1968), no.~3, 484--530.

\bibitem[Ati70]{atiyah_global_theory_elliptic_operators}
M.~F. Atiyah, \emph{{Global theory of elliptic operators}}, {Proc. Int. Conf.
  Funct. Anal. Rel. Topics, Tokyo 1969}, 1970, pp.~21--30.

\bibitem[Att94]{attie_classification}
O.~Attie, \emph{{Quasi-isometry classification of some manifolds of bounded
  geometry}}, Math. Z. \textbf{216} (1994), 501--527.

\bibitem[Aub98]{aubin_nonlinear_problems}
T.~Aubin, \emph{{Some Nonlinear Problems in Riemannian Geometry}}, Springer
  Monographs in Mathematics, Springer-Verlag, 1998.

\bibitem[BB03]{bierstedt_bonet}
K.D. Bierstedt and J.~Bonet, \emph{{Some aspects of the modern theory of
  Fr\'{e}chet spaces}}, Rev. R. Acad. Cien. Serie A. Mat. \textbf{97} (2003),
  no.~2, 159--188.

\bibitem[BCH94]{baum_connes_higson}
P.~Baum, A.~Connes, and N.~Higson, \emph{{Classifying space for proper actions
  and $K$-theory of group $C^\ast$-algebras}}, Contemporary Mathematics
  \textbf{167} (1994), 241--291.

\bibitem[BD82]{baum_douglas}
P.~Baum and R.~G. Douglas, \emph{{$K$-homology and index theory}}, {Operator
  Algebras and Applications} (R.~Kadison, ed.), Proc. Symp. Pure Math.,
  vol.~38, Amer. Math. Soc., 1982, pp.~117--173.

\bibitem[BGR77]{brown_green_rieffel}
L.~G. Brown, P.~Green, and M.~A. Rieffel, \emph{{Stable Isomorphism and Strong
  Morita Equivalence of $C^\ast$-Algebras}}, Pacific J. Math. \textbf{71}
  (1977), no.~2, 349--363.

\bibitem[BGW13]{baum_guentner_willett}
P.~Baum, E.~Guentner, and R.~Willett, \emph{{Expanders, exact crossed products,
  and the Baum-Connes conjecture}}, \url{http://arxiv.org/abs/1311.2343/},
  2013.

\bibitem[BHS07]{baum_higson_schick}
P.~Baum, N.~Higson, and T.~Schick, \emph{{On the Equivalence of Geometric and
  Analytic $K$-Homology}}, Pure and Applied Mathematics Quarterly \textbf{3}
  (2007), no.~1, 1--24.

\bibitem[Bla98]{blackadar}
B.~Blackadar, \emph{{$K$}-{T}heory for {O}perator {A}lgebras}, 2nd ed.,
  Cambridge University Press, 1998.

\bibitem[BNW07]{brodzki_niblo_wright}
J.~Brodzki, G.~A. Niblo, and N.~J. Wright, \emph{{Property A, partial
  translation structures, and uniform embeddings in groups}}, J. London Math.
  Soc. \textbf{76} (2007), no.~2, 479--497.

\bibitem[Bro81]{brooks}
R.~Brooks, \emph{{The fundamental group and the spectrum of the Laplacian}},
  Comment. Math. Helvetici \textbf{56} (1981), 581--598.

\bibitem[BW92]{block_weinberger_1}
J.~Block and S.~Weinberger, \emph{{A}periodic {T}ilings, {P}ositive {S}calar
  {C}urvature, and {A}menability of {S}paces}, J. Amer. Math. Soc. \textbf{5}
  (1992), no.~4, 907--918.

\bibitem[BW97]{block_weinberger_large_scale}
\bysame, \emph{{Large scale homology theories and geometry}}, AMS/IP Studies in
  Advanced Mathematics \textbf{2} (1997), 522--569.

\bibitem[CGT82]{cheeger_gromov_taylor}
J.~Cheeger, M.~Gromov, and M.~Taylor, \emph{{Finite Propagation Speed, Kernel
  Estimates for Functions of the Laplace Operator, and the Geometry of Complete
  Riemannian Manifolds}}, J. Diff. Geom. \textbf{17} (1982), 15--54.

\bibitem[DFW03]{dranishnikov_ferry_weinberger}
A.~N. Dranishnikov, S.~C. Ferry, and S.~Weinberger, \emph{{Large Riemannian
  manifolds which are flexible}}, Ann. Math. \textbf{157} (2003), 919--938.

\bibitem[DS99]{dimassi_sjostrand}
M.~Dimassi and J.~Sj{\"o}strand, \emph{Spectral asymptotics in the
  semi-classical limit}, London Mathematical Society Lecture Note Series, no.
  268, Cambridge University Press, 1999.

\bibitem[Eic91]{eichhorn_banach_manifold_structure}
J.~Eichhorn, \emph{{The Banach manifold structure of the space of metrics on
  noncompact manifolds}}, Diff. Geom. and its Appl. \textbf{1} (1991), 89--108.

\bibitem[Ele97]{elek}
G.~Elek, \emph{{$K$-Theory of Gromov's Translation Algebras and the Amenability
  of Discrete Groups}}, Proc. Amer. Math. Soc. \textbf{125} (1997), no.~9,
  2551--2553.

\bibitem[Gan10]{ganglberger}
V.~Ganglberger, \emph{{The Kernel Theorem and Microlocal Analysis for
  Distributions on Manifolds}}, {D}iploma {T}hesis, Universit\"{a}t Wien, 2010.

\bibitem[GJ08]{garrido_jaramillo}
M.~I. Garrido and J.~A. Jaramillo, \emph{{Lipschitz-type functions on metric
  spaces}}, J. Math. Anal. Appl. \textbf{340} (2008), 282--290.

\bibitem[Gra48]{graev}
M.~I. Graev, \emph{{Free topological groups}}, Izv. Akad. Nauk SSSR Ser. Mat.
  \textbf{12} (1948), no.~3, 279--324.

\bibitem[Gre78]{greene}
R.~E. Greene, \emph{{C}omplete metrics of bounded curvature on noncompact
  manifolds}, Archiv der Mathematik \textbf{31} (1978), no.~1, 89--95.

\bibitem[Gro81a]{gromov_curvature_diameter_betti_numbers}
M.~Gromov, \emph{{Curvature, diameter and Betti numbers}}, Comment. Math.
  Helvetici \textbf{56} (1981), 179--195.

\bibitem[Gro81b]{gromov_hyperbolic_manifolds_groups_actions}
\bysame, \emph{{Hyperbolic manifolds, groups and actions}}, Ann. Math. Studies
  \textbf{97} (1981), 183--215.

\bibitem[GWY08]{gong_wang_yu}
G.~Gong, Q.~Wang, and G.~Yu, \emph{{Geometrization of the Strong Novikov
  Conjecture for residually finite groups}}, Journal f{\"u}r die reine und
  angewandte Mathematik (Crelles Journal) \textbf{621} (2008), 159--189.

\bibitem[Hat09]{hatcher_VB}
A.~Hatcher, \emph{{Vector Bundles and $K$-Theory}}, available online on his web
  site \url{http://www.math.cornell.edu/~hatcher/VBKT/VBpage.html}, 2009.

\bibitem[Hig95]{higson_paschke_duality}
N.~Higson, \emph{{$C^\ast$-Algebra Extension Theory and Duality}}, J. Funct.
  Anal. \textbf{129} (1995), 349--363.

\bibitem[H{\"o}r67]{hormander_ess_norm}
L.~H{\"o}rmander, \emph{{Pseudo-differential Operators and Hypoelliptic
  Equations}}, {Proc. Symp. Pure Math.} \textbf{10} (1967), 138--183, Singular
  Integrals (1966).

\bibitem[HR00]{higson_roe}
N.~Higson and J.~Roe, \emph{{A}nalytic {K}-{H}omology}, Oxford University
  Press, New York, 2000.

\bibitem[Iva12]{MO_characterization_bounded_geometry}
S.~Ivanov, \emph{{Answer to ``characterization of bounded geometry --
  reference-request''}}, MathOverflow, 2012,
  \url{http://mathoverflow.net/q/113781/}.

\bibitem[Kaa13]{kaad}
J.~Kaad, \emph{{A Serre--Swan theorem for bundles of bounded geometry}}, J.
  Funct. Anal. \textbf{265} (2013), no.~10, 2465--2499.

\bibitem[Kak44]{kakutani}
S.~Kakutani, \emph{{Free topological groups and infinite direct product
  topological groups}}, Proc. Imp. Acad. \textbf{20} (1944), no.~8, 595--598.

\bibitem[Kar64]{karoubi_cartan_16}
M.~Karoubi, \emph{{Les isomorphismes de Chern et de Thom--Gysin en
  K-th{\'{e}}orie}}, {S{\'{e}}m. Henri Cartan} \textbf{16} (1963--1964), no.~2,
  1--16, Exp. No. 16.

\bibitem[Kas81]{kasparov_KK}
G.~G. Kasparov, \emph{{The Operator $K$-Functor and Extensions of
  $C^\ast$-Algebras}}, Math. USSR Izvestija \textbf{16} (1981), no.~3,
  513--572.

\bibitem[KN65]{kohn_nirenberg}
J.~J. Kohn and L.~Nirenberg, \emph{{An Algebra of Pseudo-Differential
  Operators}}, Comm. Pure Appl. Math. \textbf{18} (1965), 269--305.

\bibitem[Kor91]{kordyukov}
Yu.~A. Kordyukov, \emph{{$L^p$-Theory of Elliptic Differential Operators on
  Manifolds of Bounded Geometry}}, Acta Appl. Math. \textbf{23} (1991),
  223--260.

\bibitem[Lee03]{lee_smooth}
J.~M. Lee, \emph{{I}ntroduction to {S}mooth {M}anifolds}, 1st ed., Graduate
  Texts in Mathematics, vol. 218, Springer-Verlag, 2003.

\bibitem[LM89]{lawson_michelsohn}
H.~B. {Lawson, Jr.} and M.-L. Michelsohn, \emph{{Spin Geometry}}, Princeton
  University Press, 1989.

\bibitem[Mar41]{markoff_original}
A.~A. Markoff, \emph{{On free topological groups}}, Dokl. Akad. Nauk SSSR
  \textbf{31} (1941), 299--301.

\bibitem[Mar45]{markoff}
\bysame, \emph{{On free topological groups}}, Izv. Akad. Nauk SSSR Ser. Mat.
  \textbf{9} (1945), no.~1, 3--64.

\bibitem[Mav95]{mavra}
B.~Mavra, \emph{{Bounded Geometry Index Theory}}, Ph.D. thesis, University of
  Oxford, 1995.

\bibitem[Mel93]{melrose_APS}
R.~Melrose, \emph{{The Atiyah-Patodi-Singer Index Theorem}}, Research Notes in
  Mathematics, Vol 4, Peters, 1993.

\bibitem[Mel05]{melo}
S.~T. Melo, \emph{{Norm closure of classical pseudodifferential operators does
  not contain H{\"o}rmander's class}}, {Geometric Analysis of PDE and Several
  Complex Variables}, Contemporary Mathematics, vol. 368, American Mathematical
  Society, 2005, pp.~329--336.

\bibitem[Mel07]{melrose_microlocal_lecture}
R.~Melrose, \emph{{Introduction to Microlocal Analysis, Lecture Notes}}, 2007,
  online available at \url{http://math.mit.edu/~rbm/iml/}.

\bibitem[Mic13]{MO_embedding_uniformly_thick_disk}
P.~Michor, \emph{{Question ``Does a Riemannian manifold with bounded geometry
  admit an isometric proper embedding into Eulidean space with uniformly thick
  tubular neighbourhood''}}, MathOverflow, 2013, online at
  \url{http://mathoverflow.net/q/124840/}.

\bibitem[MM13]{mcintosh_morris}
A.~McIntosh and A.~J. Morris, \emph{{Finite propagation speed for first order
  systems and Huygens' principle for hyperbolic equations}}, Proc. Amer. Math.
  Soc. \textbf{141} (2013), 3515--3527.

\bibitem[Mor13]{morye}
A.~S. Morye, \emph{{Note on the Serre--Swan theorem}}, Math. Nachr.
  \textbf{286} (2013), no.~2--3, 272--278.

\bibitem[Nak43]{nakayama}
T.~Nakayama, \emph{{Note on free topological groups}}, Proc. Imp. Acad.
  \textbf{19} (1943), no.~8, 471--475.

\bibitem[OOY09]{oyono_oyono_yu}
H.~Oyono-Oyono and G.~Yu, \emph{{$K$-theory for the maximal Roe algebra of
  certain expanders}}, J. Funct. Anal. \textbf{257} (2009), no.~10, 3239--3292.

\bibitem[Pas81]{paschke}
W.~L. Paschke, \emph{{$K$-theory for commutants in the Calkin algebra}},
  Pacific J. Math. \textbf{95} (1981), no.~2, 427--434.

\bibitem[PCB87]{perezcarreras_bonet}
P.~P\'{e}rez~Carreras and J.~Bonet, \emph{{Barrelled Locally Convex Spaces}},
  {North-Holland Mathematics Studies}, vol. 131, Elsevier Science, 1987.

\bibitem[Phi91]{phillips}
N.~C. Phillips, \emph{{$K$-Theory for Fr\'{e}chet Algebras}}, Int. J. of Math.
  \textbf{2} (1991), no.~1, 77--129.

\bibitem[Roe88a]{roe_index_1}
J.~Roe, \emph{{A}n {I}ndex {T}heorem on {O}pen {M}anifolds, {I}}, J.
  Differential Geom. \textbf{27} (1988), 87--113.

\bibitem[Roe88b]{roe_index_2}
\bysame, \emph{{A}n {I}ndex {T}heorem on {O}pen {M}anifolds, {II}}, J.
  Differential Geom. \textbf{27} (1988), 115--136.

\bibitem[Roe88c]{roe_partitioning_non_compact_manifolds}
\bysame, \emph{{Partitioning non-compact manifolds and the dual Toeplitz
  problem}}, {Operator Algebras and Applications, 1} (D.~E. Evans and
  M.~Takesaki, eds.), London Mathematical Society Lecture Note Series, vol.
  135, Cambridge University Press, 1988, pp.~187--228.

\bibitem[Roe90]{roe_exotic_cohomology}
\bysame, \emph{{Exotic cohomology and index theory}}, Bulletin of the American
  Mathematical Society \textbf{23} (1990), no.~2, 447--453.

\bibitem[Roe93]{roe_coarse_cohomology}
\bysame, \emph{{Coarse Cohomology and Index Theory on Complete Riemannian
  Manifolds}}, vol. 104, Memoirs of the American Mathematical Society, no. 497,
  American Mathematical Society, 1993.

\bibitem[Roe96]{roe_index_coarse}
\bysame, \emph{{I}ndex {T}heory, {C}oarse {G}eometry, and {T}opology of
  {M}anifolds}, CBMS Regional Conference Series in Mathematics, vol.~90,
  American Mathematical Society, 1996.

\bibitem[Roe03]{roe_lectures_coarse_geometry}
\bysame, \emph{{Lectures on Coarse Geometry}}, University Lecture Series,
  vol.~31, American Mathematical Society, 2003.

\bibitem[Sak13]{sako}
H.~Sako, \emph{{Translation $C^\ast$-Algebras and Property A for Uniformly
  Locally Finite Spaces}}, \url{http://arxiv.org/abs/1212.5900/}, 2013.

\bibitem[Sar01]{sardanashvily}
G.~Sardanashvily, \emph{{Remark on the Serre--Swan theorem for non-compact
  manifolds}}, \url{http://arxiv.org/abs/math-ph/0102016/}, 2001.

\bibitem[Sch91]{schmitt}
L.~M. Schmitt, \emph{{Quotients of Local Banach Algebras are Local Banach
  Algebras}}, Publ. RIMS. Kyoto Univ. \textbf{27} (1991), 837--843.

\bibitem[Sch92]{schweitzer}
L.~B. Schweitzer, \emph{{A short proof that $M_n(A)$ is local if $A$ is local
  and Fr\'{e}chet}}, Int. J. Math. (1992), 581--589.

\bibitem[Sch01]{schick_bounded_geometry_boundary}
T.~Schick, \emph{{Manifolds with Boundary and of Bounded Geometry}},
  Mathematische Nachrichten \textbf{223} (2001), 103--120.

\bibitem[See65]{seeley}
R.~T. Seeley, \emph{{Integro-differential operators on vector bundles}},
  Transactions of The American Mathematical Society \textbf{117} (1965),
  167--167.

\bibitem[See67]{seeley_complex_powers}
\bysame, \emph{{Complex Powers of an Elliptic Operator}}, Proc. Symp. Pure
  Math. \textbf{10} (1967), 288--307.

\bibitem[Shu92]{shubin}
M.~A. Shubin, \emph{{S}pectral {T}heory of {E}lliptic {O}perators on
  {N}on-{C}ompact {M}anifolds}, Ast\'{e}risque \textbf{207} (1992), 35--108.

\bibitem[Siz13]{MathSE_codim_uncountable}
O.~Sizemore, \emph{{Answer to ``Codimension of bounded sequences with partial
  sums constituting again a bounded sequence''}}, Mathematics Stack Exchange,
  2013, \url{http://math.stackexchange.com/q/473333/}.

\bibitem[{\v{S}}pa08]{spakula_thesis}
J.~{\v{S}}pakula, \emph{{$K$-Theory of Uniform Roe Algebras}}, Ph.D. thesis,
  Vanderbilt University, Nashville, TN (USA), 2008.

\bibitem[{\v{S}}pa09]{spakula_uniform_k_homology}
\bysame, \emph{{Uniform $K$-homology theory}}, J. Funct. Anal. \textbf{257}
  (2009), 88--121.

\bibitem[{\v{S}}pa10]{spakula_universal_rep}
\bysame, \emph{{Uniform version of Weyl--von Neumann theorem}}, Archiv der
  Mathematik \textbf{95} (2010), 171--178.

\bibitem[Sti58]{stinespring}
W.~F. Stinespring, \emph{{A Sufficient Condition for an Integral Operator to
  Have a Trace}}, Journal f{\"u}r die reine und angewandte Mathematik (Crelles
  Journal) \textbf{200} (1958), 200--207.

\bibitem[Str72]{strichartz}
R.~S. Strichartz, \emph{{A Functional Calculus for Elliptic Pseudo-Differential
  Operators}}, American Journal of Mathematics \textbf{94} (1972), no.~3,
  711--722.

\bibitem[STY02]{skandalis_tu_yu}
G~Skandalis, J.L. Tu, and G.~Yu, \emph{{The coarse Baum--Connes conjecture and
  groupoids}}, Topology \textbf{41} (2002), 807--834.

\bibitem[Sul76]{sullivan}
D.~Sullivan, \emph{{Cycles for the Dynamical Study of Foliated Manifolds and
  Complex Manifolds}}, Invent. math. \textbf{36} (1976), 225--255.

\bibitem[{\v{S}}W11]{spakula_willett_maximal_uniform_roe_algebra}
J.~{\v{S}}pakula and R.~Willett, \emph{{Maximal and reduced Roe algebras of
  coarsely embeddable spaces}}, \url{http://arxiv.org/abs/1110.1531/}, accepted
  for publication in ``Journal f{\"u}r die reine und angewandte Mathematik
  (Crelles Journal)'', 2011.

\bibitem[{\v{S}}W13]{spakula_willett}
\bysame, \emph{{On Rigidity of Roe Algebras}}, preprint available on the arXiv
  at \url{http://arxiv.org/abs/1110.1532/}, 2013.

\bibitem[Tau10]{MO_elliptic_essentially_self_adjoint}
D.~Tausk, \emph{{``Essential self-adjointness of differential operators on
  compact manifolds''}}, MathOverflow, 2010,
  \url{http://mathoverflow.net/q/47123/}.

\bibitem[Tay81]{taylor_pseudodifferential_operators}
M.~E. Taylor, \emph{{Pseudodifferential Operators}}, Princeton Mathematical
  Series, vol.~34, Princeton University Press, Princeton, New Jersey, 1981.

\bibitem[Tay08]{taylor_pseudodifferential_operators_lectures}
\bysame, \emph{{Pseudodifferential Operators}}, 2008, Four Lectures at MSRI.

\bibitem[Wan11]{wang_thesis}
H.~Wang, \emph{{$L^2$-index formula for proper cocompact group actions}}, Ph.D.
  thesis, Vanderbilt University, 2011.

\bibitem[Wan14]{wang}
\bysame, \emph{{$L^2$-index formula for proper cocompact group actions}}, J.
  Noncommut. Geom. \textbf{8} (2014), 393--432.

\bibitem[Why01]{whyte}
K.~Whyte, \emph{{I}ndex {T}heory with {B}ounded {G}eometry, the {U}niformly
  {F}inite {$\hat A$} {C}lass, and {I}nfinite {C}onnected {S}ums}, J. Diff.
  Geom. \textbf{59} (2001), 1--14.

\bibitem[Woo95]{woods}
R.~G. Woods, \emph{{The Minimum Uniform Compactification of a Metric Space}},
  Fund. Math. \textbf{147} (1995), 39--59.

\bibitem[Yu95a]{yu_baum_connes_conj_coarse_geom}
G.~Yu, \emph{{Baum--Connes Conjecture and Coarse Geometry}}, $K$-Theory
  \textbf{9} (1995), 223--231.

\bibitem[Yu95b]{yu_coarse_baum_connes_conj}
\bysame, \emph{{Coarse Baum--Connes Conjecture}}, $K$-Theory \textbf{9} (1995),
  199--221.

\bibitem[Yu98]{yu_finite_asymptotic_dimension}
\bysame, \emph{{The Novikov conjecture for groups with finite asymptotic
  dimension}}, Ann. Math. \textbf{147} (1998), 325--355.

\bibitem[Yu00]{yu_embedding_Hilbert_space}
\bysame, \emph{{The coarse Baum--Connes conjecture for spaces which admit a
  uniform embedding into Hilbert space}}, Invent. math. \textbf{139} (2000),
  no.~1, 201--240.

\end{thebibliography}
\bibliographystyle{amsalpha}

\end{document}